%% file: article.tex
\pdfoutput=1
\documentclass[10pt,a4paper]{smfbook}

\usepackage[french, english]{babel}
\usepackage[T1]{fontenc}
\usepackage[latin1]{inputenc}

\usepackage{lmodern}
\usepackage{ucs} % \unichar

\let\printindex\undef

\usepackage{splitidx}
\makeindex

\usepackage{pdftexcmds}
\usepackage{xpatch}
\usepackage{xspace}
\usepackage{ifmtarg}

\usepackage{amsmath}
\usepackage{amssymb}
\usepackage{mathtools} % \MoveEqLeft, \intertext

\usepackage{xy}
\xyoption{arrow}
\xyoption{matrix}
\SilentMatrices
\xyoption{tips}
\SelectTips{cm}{10}
\xyoption{2cell}
\UseAllTwocells

\usepackage[hidelinks,pdfencoding=unicode,plainpages=false,hyperindex]{hyperref}

\usepackage{enumitem}
\setlist{nosep}
\setlist[enumerate, 1]{label={\rm (}\emph{\alph*}{\rm )}}
\setlist[enumerate, 2]{label={\rm (}\emph{\alph{enumi}.\arabic*}{\rm )}}

\makeatletter

\newcommand{\@m@mchapter}[1][]{%
  \def\ch@pt@c{#1}% capture first optional arg
  \@ifnextchar[{\@chapter}{\@chapter[]}%
}
\def\m@m@empty{\@empty}

\def\chapter{\cleardoublepage
  \thispagestyle{empty}\global\@topnum\z@
  \@afterindenttrue
  \@ifstar{\@schapter}{\@m@mchapter}}

\def\@chapter[#1]#2{\refstepcounter{chapter}%
  \def\f@rbdy{#2}%
  \ifx\ch@pt@c\@empty % no optional args
    \def\f@rtoc{#2}%
    \def\f@rhdr{#2}%
  \else                  % at least one opt arg
    \let\f@rtoc\ch@pt@c
    \ifx\@empty#1\@empty
      \let\f@rhdr\ch@pt@c
    \else
      \def\f@rhdr{#1}%
    \fi
  \fi
  \ifnum\c@secnumdepth<\z@ \let\@secnumber\@empty
  \else \let\@secnumber\thechapter \fi
  \typeout{\chaptername\space\@secnumber}%
  \def\@toclevel{0}%
  \ifx\chaptername\appendixname
     \@tocwriteb\tocappendix{chapter}{\f@rtoc}%
  \else \@tocwriteb\tocchapter{chapter}{\f@rtoc}\fi
  \chaptermark{\f@rhdr}%
  \addtocontents{lof}{\protect\addvspace{10\p@}}%
  \addtocontents{lot}{\protect\addvspace{10\p@}}%
  \@makechapterhead{\f@rbdy}\@afterheading}

\def\smf@titleii{\clearpage
  \thispagestyle{empty}
  \ifx\@empty\addresses\else \@setaddresses\fi
  \vfill
  \ifx\@empty\@subjclass
     \ifx\@empty\@keywords\else
        \hrule\medskip\@setkeywords\medskip\hrule\vfill\fi
  \else
     \hrule\medskip\@setsubjclass\par\medskip
     \ifx\@empty\@keywords\else\@setkeywords\par\medskip\fi
     \ifx\@empty\@keywords\else\@setaltkeywords\par\medskip\fi
     \hrule\vfill
  \fi
  \ifx\@empty\@translators\else\@settranslators\bigskip\fi
  \ifx\@empty\thankses\else\@setthanks\par\fi
  \ifx\@empty\@date\else\@setdate\fi
  \vfilneg
}

\patchcmd{\smf@captionsfrench}{Bibliographie}{Références}{}{}

\makeatother

\input{macros}

\author{Dimitri Ara}
\address{Aix~Marseille~Univ,~CNRS,~Centrale~Marseille,~I2M,~Marseille,~France}
\email{dimitri.ara@univ-amu.fr}
\urladdr{\href{http://www.i2m.univ-amu.fr/perso/dimitri.ara/}{http://www.i2m.univ-amu.fr/perso/dimitri.ara/}}

\author{Georges Maltsiniotis}
\address{%
Institut de Math\'ematiques de Jussieu\\
Universit\'e Paris 7 Denis Diderot\\
Case Postale 7012\\
B\^atiment Sophie Germain\\
75205 Paris Cedex 13\\
France}
\email{georges.maltsiniotis@imj-prg.fr}
\urladdr{\href{https://webusers.imj-prg.fr/~georges.maltsiniotis/}
{http://webusers.imj-prg.fr/%
\raise -3.3pt\vbox{\hbox{$\widetilde{ \ }\,$}}georges.maltsiniotis/}}

\title[Joint et tranches pour les \pdfoo-catégories strictes]{Joint et tranches
pour les\\ \pdfoo-catégories strictes}
\alttitle{Join and slices for strict \pdfoo-categories}

\begin{document}

\frontmatter

\begin{abstract}
  Le but de cet article est de développer une théorie du joint et des
  tranches pour les \oo-catégories strictes. À deux \oo-catégories strictes,
  on en associe une troisième qu'on appelle leur joint. Cette opération est
  compatible au joint usuel des catégories à troncation près. On montre que
  le joint définit une structure de catégorie monoïdale sur la catégorie des
  \oo-catégories strictes et qu'il commute aux limites inductives connexes
  en chaque variable. En particulier, on obtient l'existence de certains
  adjoints à droite ; ces adjoints définissent des tranches
  \oo-catégoriques, en un sens généralisé. On énonce des conjectures de
  fonctorialité du joint et des tranches par rapport aux transformations lax
  et oplax supérieures et on démontre des premiers résultats dans ce sens.
  Ces résultats sont utilisés dans un autre travail pour établir un théorème
  A de Quillen \oo-catégorique. Enfin, dans un appendice, on revisite le
  produit tensoriel de Gray \oo-catégorique. Un des principaux outils
  utilisés dans ce travail est la théorie des complexes dirigés augmentés de
  Steiner.
\end{abstract}

\begin{altabstract}
  The goal of this paper is to develop a theory of join and slices for
  strict \oo-categories. To any pair of strict \oo-categories, we associate
  a third one that we call their join. This operation is compatible with the
  usual join of categories up to truncation. We show that the join defines a
  monoidal category structure on the category of strict \oo-categories and
  that it respects connected inductive limits in each variable. In
  particular, we obtain the existence of some right adjoints; these adjoints
  define \oo-categorical slices, in a generalized sense. We state some
  conjectures about the functoriality of the join and the slices with
  respect to higher lax and oplax transformations and we prove some first
  results in this direction.
  These results are used in another paper to establish a Quillen Theorem A
  for strict \oo-categories.
  Finally, in
  an appendix, we revisit the Gray tensor product of strict \oo-categories.
  One of the main tools used in this paper is Steiner's theory of augmented
  directed complexes.
\end{altabstract}



\altkeywords{Gray $\infty$-categories, strict $\infty$-categories, monoidal
categories, locally biclosed monoidal categories, augmented directed
complexes, join, Street's nerve, orientals, Gray tensor product, slices, lax
transformations}

\maketitle

\tableofcontents
\renewcommand\bibname{Références}
\mainmatter

% augmente légèrement la taille de la page pour éviter le underfull vbox
% dans les pages sans glue
\setlength{\dimen0}{\textheight}
\addtolength{\dimen0}{-\topskip}
\divide\dimen0\baselineskip
\setlength{\textheight}{\number\dimen0 \baselineskip}
\addtolength{\textheight}{\topskip}

\chapter*{Introduction}

% TOCHECK
\kern-40pt

Ce travail, même s'il en est essentiellement indépendant, est issu d'un
projet consacré à la théorie homotopique des \oo-catégories strictes, projet
constitué actuellement des textes \cite{AraMaltsiNThom}, \cite{Ara2Thom},
\cite{AraMaltsiCondE}, \cite{AraMaltsiThmAI}, \cite{AraMaltsiThmAII},
\cite{AraThmB} et \cite{AraMaltsiNerfs}, ainsi que de l'article \cite{Gagna}
de Gagna. Ces travaux sont motivés par le fait que les \oo-catégories
strictes fournissent des modèles des types d'homotopie plus proches de
l'intuition géométrique que ceux fournis par les catégories.
Une description détaillée de ce projet et de ses motivations se
trouve dans l'introduction de \cite{AraMaltsiThmAI}. C'est en travaillant
sur un théorème A de Quillen pour les \oo-catégories strictes, résultat
principal de \cite{AraMaltsiThmAI} et \cite{AraMaltsiThmAII}, que le besoin
de définir une théorie du joint et des tranches \oo-catégoriques
généralisées s'est fait sentir. En effet, non seulement l'énoncé même du
théorème A fait intervenir des tranches du type~$\cotr{C}{c}$, où $c$ est un
objet de $C$, mais surtout, sa démonstration, déjà pour les
$2$\nbd-catégories strictes, fait intervenir des tranches de la forme
$\cotr{C}{c}$, où $c$ est un $n$-simplexe du nerf de~$C$ (voir la preuve
de~\hbox{\cite[théorème 2.16]{ChicheThmA}}, preuve inspirée des références
originales~\cite{BullCegGeom2Cat} et~\cite{CegThmB}).  La motivation
initiale du présent travail était de fournir les outils pour définir et
étudier ces tranches généralisées pour les \oo-catégories en vue d'une
démonstration d'un théorème~A \oo-catégorique. Néanmoins, les notions de
joint et de tranches \oo-catégoriques sont, nous semble-t-il, des notions
fondamentales de la théorie des \oo-catégories strictes dont l'intérêt
dépasse largement les applications qui les ont motivées.

\medskip

Commençons par rappeler la situation pour le joint et les tranches en
dimension~$1$, c'est-à-dire pour les catégories. Si $A$ et $B$ sont deux
catégories, on définit une nouvelle catégorie $A \joint B$, appelée le
\ndef[]{joint} de $A$ et $B$, de la manière suivante. Le graphe sous-jacent à
$A \joint B$ est le graphe sous-jacent à la somme disjointe $A \amalg B$
auquel on adjoint une flèche $j_{b,a}$ de $a$ vers $b$ pour tout couple $(a,
b)$ formé d'un objet $a$ de $A$ et d'un objet $b$ de $B$. Les identités et
la composition sont définies de la manière évidente. On obtient ainsi un
foncteur
\[
  \begin{split}
    \Cat \times \Cat & \to \Cat \\
    (A, B) \,\,\, & \mapsto A \joint B \pbox{,}
  \end{split}
\]
où $\Cat$ désigne la catégorie des petites catégories.
On vérifie facilement que le joint définit une structure de catégorie
monoïdale sur $\Cat$ d'unité la catégorie vide. Cette structure
n'est pas bifermée mais est localement bifermée au sens suivant : pour toutes
petites catégories $A$ et $B$, les foncteurs
\[
  \begin{split}
    \Cat & \to \cotr{\Cat}{A} \\
    B & \mapsto (A \joint B, \iota_1 : A \to A \joint B)
  \end{split}
  \quadet
  \begin{split}
    \Cat & \to \cotr{\Cat}{B} \\
    A & \mapsto (A \joint B, \iota_2 : B \to A \joint B) \pbox{,}
  \end{split}
\]
où
\[
  A \xto{\iota_1} A \joint B \xot{\iota_2} B
\]
désignent les foncteurs canoniques, admettent des adjoints à droite. On
obtient ainsi des foncteurs
\[
  \begin{split}
    \cotr{\Cat}{A} & \to \Cat \\
    (C, u : A \to C) & \mapsto \cotr{C}{u}
  \end{split}
  \quadet
  \begin{split}
    \cotr{\Cat}{B} & \to \Cat \\
    (C, v : B \to C) & \mapsto \tr{C}{v}
  \end{split}
\]
qu'on appelle respectivement les foncteurs \ndef[]{tranches généralisées
au-dessous} et \ndef[]{au-dessus}. Explicitement, si $u : A \to C$ est un
foncteur, les objets de la catégorie $\cotr{C}{u}$ sont les cônes inductifs
sur le diagramme $u$ et les morphismes sont les morphismes de cônes en un
sens évident. De même pour la catégorie $\tr{C}{v}$ et les cônes projectifs.
Si $c$ est un objet de $C$, on peut considérer $c$ comme un foncteur $c : e
\to C$, où $e$ désigne la catégorie finale, et la catégorie $\cotr{C}{c}$ au
sens précédent n'est autre que la tranche usuelle.  Ainsi, les tranches
$\cotr{C}{u}$ sont des tranches généralisées au sens où on considère des
tranches au-dessous d'un diagramme quelconque à valeur dans $C$ et non pas
seulement d'un objet de $C$.

Ce point de vue sur les tranches ne joue traditionnellement pas un rôle
important en théorie des catégories, sans doute parce que les objets en jeu
sont simples à décrire explicitement. Néanmoins, le joint et ses deux
adjoints sont au c\oe{}ur de la théorie des quasi-catégories (voir par
exemple \cite{JoyalQuasiKan}) et sont un outil précieux pour obtenir la
structure de catégorie de modèles de Joyal \cite{JoyalQCatAppl} sur les
ensembles simpliciaux.

\medbreak

Le résultat principal de cet article est la généralisation du formalisme du
joint et des tranches à la catégorie $\ooCat$ des \oo-catégories strictes.
On définit un foncteur
\[
  \begin{split}
    \ooCat \times \ooCat & \to \ooCat \\
    (A, B) \qquad & \mapsto A \joint B
  \end{split}
\]
qu'on appelle le \ndef[]{joint \oo-catégorique}. Ce foncteur est compatible au
joint $1$-catégorique au sens suivant : si $A$ et $B$ sont deux catégories
vues comme des \oo-catégories, alors leur joint \oo-catégorique $A \joint B$
est une $3$-catégorie dont le tronqué $1$-catégorique, obtenu en appliquant
l'adjoint à gauche du foncteur d'inclusion de $\Cat$ dans $\ooCat$, est le
joint $1$-catégorique usuel. En effet, le joint \oo-catégorique $A \joint B$
s'obtient à partir du joint $1$-catégorique en ajoutant, pour tout objet $a$
de $A$ et toute flèche $g : b \to b'$ de~$B$, une $2$-flèche dans le
triangle formé de $g$, $j_{b, a}$ et $j_{b', a}$, pour toute flèche $f : a
\to a'$ de $A$ et tout objet~$b$ de~$B$, une $2$-flèche dans le triangle
formé de $f$, $j_{b, a}$ et $j_{b, a'}$ et, pour toute flèche de $A$ et
toute flèche de $B$, une $3$-flèche dans le tétraèdre formé des triangles du
type précédent, et en quotientant par les relations évidentes. Notre joint
\oo-catégorique est par ailleurs compatible au joint des ensembles
simpliciaux (voir par exemple~\cite[section 3]{JoyalQuasiKan}) au sens où,
si $X$ et $Y$ sont deux ensembles simpliciaux et si $c_\infty$ désigne
l'adjoint à gauche du nerf de Street \cite{StreetOrient}, on a un
isomorphisme canonique de \oo-catégories $c_\infty(X \joint Y) \simeq
c_\infty(X) \joint c_\infty(Y)$.

On montre que le joint \oo-catégorique définit une structure de catégorie
monoïdale sur $\ooCat$ d'unité la \oo-catégorie vide et que cette structure
est localement bifermée. Ainsi, pour toute \oo-catégorie~$A$ et toute
\oo-catégorie $B$, les foncteurs
\[
  \begin{split}
    \ooCat & \to \cotr{\ooCat}{A} \\
    B & \mapsto (A \joint B, \iota_1 : A \to A \joint B)
  \end{split}
  \quadet
  \begin{split}
    \ooCat & \to \cotr{\ooCat}{B} \\
    A & \mapsto (A \joint B, \iota_2 : B \to A \joint B) \pbox{,}
  \end{split}
\]
où
\[
  A \xto{\iota_1} A \joint B \xot{\iota_2} B
\]
désignent des \oo-foncteurs canoniques, admettent des adjoints à droite. On
obtient ainsi des foncteurs
\[
  \begin{split}
    \cotr{\ooCat}{A} & \to \ooCat \\
    (C, u : A \to C) & \mapsto \cotr{C}{u}
  \end{split}
  \quadet
  \begin{split}
    \cotr{\ooCat}{B} & \to \ooCat \\
    (C, v : B \to C) & \mapsto \trm{C}{v} \pbox{.}
  \end{split}
\]
Le premier foncteur définit les \ndef[]{tranches \oo-catégoriques généralisées
au-dessous}. Si $c$ est un objet de $C$, on peut considérer $c$ comme un
\oo-foncteur $c : e \to C$, où $e$ désigne la \oo-catégorie finale, et on
obtient une tranche $\cotr{C}{c}$ de $C$ au-dessous de l'objet $c$. Dans ce
cas particulier, on décrit explicitement la \oo-catégorie $\cotr{C}{c}$ et
on retrouve les formules connues décrivant les cônes $n$-catégoriques (voir
\cite{MetPolRes} pour le cas, plus général, des cylindres).

Nous avons décidé de réserver la notation $\tr{C}{v}$, ainsi que la terminologie «
tranches au-dessus », à une variante de la \oo-catégorie
$\smash{\trm{C}{v}}$. La catégorie $\ooCat$ possède un automorphisme
remarquable qui envoie une \oo-catégorie $C$ sur la \oo-catégorie $C^\op$
obtenue en renversant le sens des $i$-flèches de $C$ pour tout $i > 0$. Le
joint n'est pas compatible à cet automorphisme et on obtient un
foncteur
\[
  \begin{split}
    \ooCat \times \ooCat & \to \ooCat \\
    (A, B) \qquad & \mapsto (B^\op \joint A^\op)^\op
  \end{split}
\]
qui définit une seconde structure de catégorie monoïdale sur la catégorie des
\oo-catégories. Ce foncteur permet d'obtenir comme ci-dessus des foncteurs
\[
  \begin{split}
    \cotr{\ooCat}{A} & \to \ooCat \\
    (C, u : A \to C) & \mapsto \cotrm{C}{u}
  \end{split}
  \quadet
  \begin{split}
    \cotr{\ooCat}{B} & \to \ooCat \\
    (C, v : B \to C) & \mapsto \tr{C}{v} \pbox{.}
  \end{split}
\]
C'est ce second foncteur qui définit les \ndef[]{tranches
\oo-catégoriques généralisées au-dessus}.

Le choix de privilégier les \oo-catégories $\cotr{C}{u}$ et $\tr{C}{v}$ par
rapport aux \oo-ca\-tégories~$\smash{\cotrm{C}{u}}$ et $\smash{\trm{C}{v}}$
est motivé par le fait que les premières admettent une description beaucoup
plus simples que les dernières. Par ailleurs, on a des isomorphismes
canoniques
\[
    \tr{C}{v} \simeq (\cotr{C^\op}{v^\op})^\op
    \quadet
    \cotrm{C}{u} \simeq (\trm{C^\op}{u^\op})^\op.
\]

Si $A$ et $B$ sont deux $n$-catégories strictes, pour un $n \ge 0$, leur
joint est une \hbox{$(2n+1)$}\nbd-caté\-gorie stricte. En tronquant cette
$(2n+1)$-catégorie en dimension $n$ (c'est-à-dire en appliquant l'adjoint à
gauche du foncteur d'inclusion), on obtient une $n$\nbd-catégorie $A \joint_n B$
qu'on appelle leur joint $n$-catégorique. On montre qu'on définit ainsi une
structure de catégorie monoïdale sur $\nCat{n}$, la catégorie des
$n$\nbd-catégories strictes, et que cette structure est localement bifermée
comme ci-dessus. Lorsque $n = 1$, on retrouve la structure de catégorie
monoïdale définie par le joint $1$-catégorique usuel.

\medbreak

Pour construire ce joint \oo-catégorique, on adopte une stratégie inspirée,
d'une part, d'une esquisse de construction du produit tensoriel de Gray
\oo-catégorique donnée par Street dans \cite[section 9]{StreetDescent} et,
d'autre part, d'idées de Steiner pour construire ce même produit tensoriel
basées sur sa théorie des complexes dirigés augmentés \cite{Steiner}.

La théorie de Steiner associe à tout complexe dirigé augmenté, c'est-à-dire
à tout complexe de chaînes de groupes abéliens en degrés positifs augmenté et
muni en chaque degré d'un sous-monoïde des chaînes, une \oo-catégorie
stricte. Un des résultats importants de \cite{Steiner} donne des conditions
suffisantes pour que la \oo-catégorie ainsi associée soit libre au sens des
polygraphes.  La théorie de Steiner permet ainsi de décrire en termes de
complexes de chaînes une sous-catégorie pleine de $\ooCat$, que nous
appellerons la catégorie des \oo-catégories de Steiner fortes, qui contient
notamment la catégorie $\Theta$ de Joyal \cite{JoyalTheta}, les orientaux de
Street \cite{StreetOrient} et les cubes $n$-catégoriques. (Une théorie
alternative permettant de décrire ces \oo-catégories de manière combinatoire
est la théorie des complexes de parité de Street \cite{StreetParComp,
StreetParCompCorr}.)

Afin de construire le joint \oo-catégorique, on commence par décrire, en
termes de produit tensoriel de complexes de chaînes, le joint de deux
complexes dirigés augmentés, s'inspirant d'une construction analogue due à
Street \cite{StreetParComp} dans le cadre des complexes de parité.  On
obtient alors une structure de catégorie monoïdale sur la catégorie des
complexes dirigés augmentés. On montre que cette structure induit une
structure de catégorie monoïdale sur la catégorie des \oo-catégories de
Steiner fortes. Il s'agit ensuite d'étendre cette structure à la catégorie
de toutes les \oo-catégories strictes. Pour cela, suivant la stratégie de
Street pour définir le produit tensoriel de Gray
\cite[section~9]{StreetDescent}, on utilise un théorème d'extension de
structures de catégorie monoïdale à la Day \cite{DayFunct, DayRefl}.
Notre structure de catégorie monoïdale n'étant pas bifermée mais seulement
localement bifermée, les résultats de Day ne s'appliquent pas tels quels. On
est donc conduit à généraliser un théorème de Day au cas localement
bifermé.  Enfin, la partie la plus délicate de la construction du joint
\oo-catégorique est la vérification des hypothèses de ce théorème à la Day.
Il s'agit essentiellement de construire à la main les tranches $\cotr{C}{u}$
et $\smash{\trm{C}{v}}$ dans le cas où les sources de $u$ et $v$ sont des
\oo-catégories de Steiner fortes, et de vérifier que ces tranches ont les
propriétés universelles attendues.

\medbreak

Dans un appendice, on construit le produit tensoriel de Gray \oo-catégorique
en suivant une stratégie analogue. Rappelons que le produit tensoriel de
Gray \oo-catégorique est une généralisation \oo-catégorique du produit
tensoriel introduit par Gray dans~\cite{GrayFCT} sur la catégorie des
$2$-catégories strictes. Ce produit définit une structure de catégorie
monoïdale bifermée sur $\ooCat$. Le produit tensoriel de Gray
\oo-catégorique a été construit pour la première fois par Al-Agl et Steiner
\cite{AlAglSteiner}, généralisant une construction analogue pour les
\oo-groupoïdes due à Brown et Higgins \cite{BrownTensor}. Deux
constructions alternatives sont données par Crans dans sa thèse
\cite{CransThese}.  Dans \cite{Steiner}, Steiner propose une nouvelle
construction qui a l'avantage d'être relativement explicite. Néanmoins, la
preuve de~\cite{Steiner} est incomplète. En effet, dans la preuve de son
théorème 7.3, Steiner affirme qu'il est évident que son produit tensoriel
commute aux limites inductives, sous-entendant que cette commutation est
formelle, ce qui n'est pas le cas. Steiner nous a cependant affirmé savoir
compléter cette preuve, esquissant un argument \cite{SteinerTensor}. Dans
l'appendice~\ref{sec:Gray}, on propose une démonstration alternative en
adoptant la stratégie qu'on a utilisé pour construire le joint. En
particulier, on utilise de manière cruciale un théorème de Day et, comme
dans le cas du joint, la partie la plus délicate du travail consiste à
vérifier que les hypothèses du théorème de Day sont satisfaites.

\medbreak

On formule de très générales conjectures de fonctorialité du joint et des
tranches. Pour cela, on introduit la notion de \oo-catégorie de Gray,
catégorie enrichie dans la catégorie des \oo-catégories strictes munie du
produit tensoriel de Gray. (On définit également au passage la notion de
\oo-sesquicatégorie.) La \oo-catégorie de Gray paradigmatique est $\ooCatGr$
dont les objets sont les \oo-catégories strictes, les $1$-flèches les
\oo-foncteurs stricts, les $2$-flèches les transformations oplax et les
flèches supérieures les transformations oplax supérieures. De même, on
introduit la notion de catégorie de Gray gauche, catégorie enrichie dans la
catégorie des \oo-catégories strictes munie du produit tensoriel $(C, D)
\mapsto D \otimes C$ obtenu en transposant celui de Gray. La \oo-catégorie
de Gray gauche paradigmatique est $\ooCatGrg$ obtenue en utilisant cette
fois les transformations lax. On conjecture l'existence de tranches pour les
\oo-catégories de Gray et les \oo-catégories de Gray gauches au-dessus ou
au-dessous d'un objet. En particulier, si $C$ désigne une \oo-catégorie
stricte, on disposerait de \oo-catégories de Gray
\[
  \cotrm{\ooCatGr}{C}
  \quadet
  \tr{\ooCatGr}{C}
\]
et de \oo-catégories de Gray gauches
\[
  \cotr{\ooCatGrg}{C}
  \quadet
  \trm{\ooCatGrg}{C}.
\]
On conjecture que les foncteurs $\var \joint C$ et $C \joint \var$
proviennent de \oo-foncteurs de Gray et de Gray gauches (c'est-à-dire des
foncteurs enrichis) respectivement
\[
  \begin{split}
    \var \joint C & : \ooCatGr \to \cotrm{\ooCatGr}{C} \\
    C \joint \var & : \zbox{$\ooCatGrg$}\phantom{\ooCatGr} \to \cotr{\ooCatGrg}{C}
  \end{split}
\]
et que l'association $(A, A \xto{u} C) \mapsto \cotr{C}{u}$, selon qu'on la
considère comme un foncteur en $C$ ou en $A$, provient de
\oo-foncteurs de Gray gauches et de Gray respectivement
\[
  \begin{split}
    \cotr{\ooCatGrg}{A} & \to \ooCatGrg  \\
    (\trm{\ooCatGrg}{C})^\op & \to \ooCatGr,
  \end{split}
\]
où $^\op$ est une dualité transformant une \oo-catégorie de Gray gauche en
une \oo-catégorie de Gray (et réciproquement). Dans ce texte, on s'intéresse
spécifiquement à la dernière conjecture, c'est-à-dire au \oo-foncteur de
Gray \smash{$(\trm{\ooCatGrg}{C})^\op \to \ooCatGr$}, car c'est celui-ci qui
nous permet de démontrer un théorème A \oo-catégorique
dans~\cite{AraMaltsiThmAII}. Explicitons en petite dimension à quoi correspond
l'action d'un tel \oo-foncteur sur les cellules :
\[
  \shorthandoff{;:}
  \begin{split}
  \raisebox{20pt}{
  \xymatrix{
    A \ar[d]_u \\
    C
  }
  }
  \hskip4.2em
  & \mapsto
  \hskip4.7em
  \cotr{C}{u}\pbox{,} \\
  \raisebox{20pt}{
    \xymatrix@C=1.5pc{
      A \ar[rr]^v \ar[dr]_{u}_{}="f" & & A' \ar[dl]^(0.42){u'} \\
      & C
      \ar@{}"f";[ur]_(.15){}="ff"
      \ar@{}"f";[ur]_(.55){}="oo"
      \ar@<-0.5ex>@2"ff";"oo"^{\alpha}
    }
  }
  \quad
  & \mapsto
  \hskip2.2em
  \xymatrix{\cotr{C}{u'} \ar[r]^{(v, \alpha)^\ast} & \cotr{C}{u} \pbox{,}} \\
  \raisebox{20pt}{
    \xymatrix@C=1.5pc@R=3pc{
    A \ar@/^2ex/[rr]^(.33){v'}_{}="1" \ar@/_2ex/[rr]^(.30)v_{}="0"
    \ar[dr]_{}="f"_{\phantom{u'}u}
    \ar@2"0";"1"_{\beta}
    & & A' \ar[dl]^{u'} \\
    & C
    \ar@{}"f";[ur]_(.15){}="ff"
    \ar@{}"f";[ur]_(.55){}="oo"
    \ar@<-0.5ex>@/^1ex/@{:>}"ff";"oo"^(.18){\alpha'\!\!}_(.30){}="h'"
    \ar@<-2.0ex>@/^-1ex/@2"ff";"oo"_(.36){\alpha}_(.80){}="h"
    \ar@3"h";"h'"_(.20){\,\,\Lambda}
    }
  }
  \quad
  & \mapsto
  \quad
      \xymatrix@C=4.5pc{
        \cotr{C}{u'}
        \ar@/^2.5ex/[r]^{(v', \alpha')^\ast}_{}="1"
        \ar@/_2.5ex/[r]_{(v, \alpha)^\ast}_{}="0"
        \ar@2"1";"0"^(.55){\,(\beta, \Lambda)^\ast}
        &
       \cotr{C}{u'} \pbox{,}
      }
  \end{split}
\]
où les $2$-flèches et $3$-flèches à gauche du signe <<~$\mapsto$~>>
sont des transformations lax et des $2$-transformations lax respectivement,
et la $2$-flèche à droite de ce signe est une transformation oplax.

\medbreak

Ces conjectures sont étayées par le fait que les conjectures analogues pour
les complexes dirigés augmentés (vérifiant une hypothèse anodine) sont
vraies. Dans ce texte, nous démontrons ces analogues en petites dimensions.
Plus précisément, nous montrons qu'on peut définir une correspondance comme
dans l'illustration ci-dessus en remplaçant les \oo-catégories par des
complexes dirigés augmentés, les \oo-foncteurs par des morphismes de
complexes dirigés augmentés, les transformations oplax par des homotopies et
les transformations lax par des antihomotopies, notion duale à celle
d'homotopie. On montre par ailleurs la sesquifonctorialité de cette
correspondance.

Un des résultats principaux de cet article est l'extension de cette
correspondance en basse dimension aux \oo-catégories strictes,
malheureusement sous des hypothèses restrictives (mais suffisantes pour
obtenir notre théorème A \oo-catégorique dans \cite{AraMaltsiThmAII}). On
montre qu'on peut définir une correspondance comme dans l'illustration
ci-dessus si on se restreint aux diagrammes obtenus en composant un
diagramme comme dans l'illustration, où $A$, $A'$ et $C$ sont des
\oo-catégories de Steiner fortes et $u'$ est un monomorphisme vérifiant une
hypothèse technique, par un \oo-foncteur quelconque~$C \to C'$ (en
particulier, on ne demande pas que $C'$ soit une \oo-catégorie de Steiner
forte). On montre que cette correspondance est sesquifonctorielle.

Ces résultats, dans le cas des complexes dirigés augmentés, sont obtenus en
produisant des formules explicites et par le calcul. Cette approche s'adapte
a priori mal au cas des \oo-catégories mais nous montrons qu'on peut
déduire le cas des \oo-catégories de celui des complexes dirigés augmentés.
Pour ce faire, on utilise, d'une part, la densité de la sous-catégorie
pleine de $\ooCat$ formée des \oo-catégories de Steiner fortes et, d'autre
part, un résultat, démontré dans ce texte, de stabilité des \oo-catégories
de Steiner fortes par certaines sommes amalgamées. C'est des hypothèses de
ce dernier résultat que provient la condition restrictive sur $u'$.

\medbreak

Après que nous avons rendu publique une première version de ce texte,
Dominic Verity nous a signalé que ses travaux sur les ensembles compliciaux
permettraient également de définir un joint \oo-catégorique. Dans
\cite[section 3]{VerityWeakCompI}, Verity définit un joint des ensembles
stratifiés, ensembles simpliciaux munis de la donnée supplémentaire de
simplexes distingués dits fins. Il nous a affirmé qu'on pourrait montrer, en
utilisant un théorème de Day, que ce joint induit un joint sur la
sous-catégorie réflexive formée des ensembles compliciaux. Or, en vertu de
la conjecture de Street-Roberts qu'il a démontrée \cite{VerityComplicial},
les ensembles compliciaux sont équivalents aux \oo-catégories strictes. Il
est probable que ce joint \oo-catégorique serait isomorphe au nôtre, même si
la comparaison ne semble pas triviale.

\begin{organisation}
Le premier chapitre est consacré aux préliminaires sur les \oo-catégories
strictes. On fixe nos notations et conventions et on rappelle quelques
constructions classiques. On définit la notion standard de \oo-catégorie
stricte engendrée librement au sens des polygraphes et on dégage un
résultat de commutation aux limites inductives de ces \oo-catégories qui
jouera un rôle important dans la suite du texte. On définit la notion de
\oo-transformation oplax en donnant des formules explicites et on étudie
quelques moyens de composer ces \oo-transformations oplax.

Dans le deuxième chapitre, on rappelle la théorie des complexes dirigés
augmentés de Steiner et on y apporte quelques compléments. On définit les
foncteurs $\lambda$ et $\nu$ de Steiner qui forment un couple de foncteurs
adjoints entre la catégorie des \oo-catégories strictes et la catégorie des
complexes dirigés augmentés. On introduit les termes «~complexe de Steiner~»
et «~complexe de Steiner fort~» pour désigner ce que Steiner appelle les
complexes dirigés augmentés à base unitaire sans boucle ou fortement sans
boucle. On rappelle les deux résultats fondamentaux de la théorie de Steiner
: le foncteur $\nu$ restreint aux complexes de Steiner est, d'une part,
pleinement fidèle et, d'autre part, à valeurs dans la catégorie des
\oo-catégories libres au sens des polygraphes.  On apporte ensuite les
compléments suivants. On définit la notion de complexe dirigé augmenté
décent, notion peut-être plus pertinente que celle de complexe dirigé
augmenté. On définit des endofoncteurs de dualité de la catégorie des
complexes dirigés augmentés et des notions de troncations pour ces
complexes. On étudie les compatibilités entre ces constructions et les
foncteurs $\lambda$ et $\nu$ de Steiner. Enfin, on introduit les notions
d'homotopie et d'antihomotopie entre morphismes de complexes dirigés
augmentés et, plus généralement, celles de $n$-homotopie et de
$n$-antihomotopie, et on décrit quelques opérations sur ces objets.

Dans le troisième chapitre, on décrit les limites inductives de complexes de
Steiner et on étudie la commutation du foncteur $\nu$ à celles-ci. On
introduit la notion de système de Steiner et on montre que le foncteur $\nu$
commute aux limites inductives de ces systèmes. On définit la notion
d'inclusion rigide ordonnée entre complexes de Steiner et on montre qu'une
somme amalgamée de complexes de Steiner forts faisant intervenir des
inclusions rigides ordonnées forme un système de Steiner. Le premier
résultat joue un rôle important dans la construction même du joint et le
second dans les propriétés de fonctorialité du joint.

Le quatrième chapitre est consacré à la catégorie $\Theta$ de Joyal.
On la définit comme une sous-catégorie pleine de la catégorie des
\oo-catégories strictes en utilisant la notion de sommes globulaires.
On indique comment les préfaisceaux sur $\Theta$ compatibles à ces sommes
globulaires induisent des \oo-catégories. On étudie la restriction du
foncteur~$\lambda$ de Steiner à $\Theta$. On montre que $\Theta$ est dans
l'image du foncteur $\nu$ de Steiner en utilisant les résultats du chapitre
précédent (ce résultat était déjà connu de Steiner \cite{SteinerTheta}).

Dans le cinquième chapitre, on s'intéresse à la question suivante : pour
définir une structure de catégorie monoïdale sur une catégorie, sous quelles
conditions suffit-il de le faire sur une sous-catégorie dense ? Un théorème
de Day répond à cette question pour les catégories monoïdales bifermées. On
généralise ce résultat à ce qu'on appelle les catégories monoïdales
localement bifermées qui, si la catégorie sous-jacente est localement
présentable, sont les structures pour lesquelles le produit tensoriel
commute aux limites inductives \emph{connexes} en chaque variable. C'est ce
résultat qu'on utilisera dans le chapitre suivant pour construire le joint.

Le sixième chapitre est le chapitre central de l'article. C'est dans
celui-ci qu'on construit le foncteur joint et les tranches généralisées. On
commence par des rappels sur la structure de catégorie monoïdale bifermée
sur la catégorie des complexes dirigés augmentés définie par le produit
tensoriel. On en déduit, grâce à un jeu de suspensions déjà présent sous une
autre forme dans \cite{StreetParComp}, une nouvelle structure de catégorie
monoïdale sur cette même catégorie ; on appelle le produit de cette
structure le joint. On montre que cette structure est localement bifermée et
qu'elle induit une structure de catégorie monoïdale sur la catégorie des
complexes de Steiner forts. On transporte cette structure sur la
sous-catégorie pleine de $\ooCat$ formée des \oo-catégories de
\hbox{Steiner} fortes et on vérifie que les hypothèses de notre théorème à
la Day dégagé dans le chapitre précédent sont satisfaites ; c'est le
contenu de la proposition~\ref{prop:pu_cotranche}, proposition clé de ce
chapitre. On en déduit l'existence d'une structure de catégorie monoïdale
localement bifermée sur la catégorie des \oo-catégories strictes. On obtient
ainsi un foncteur joint et des foncteurs tranches. On finit le chapitre par
une étude des compatibilités du joint et des tranches avec les dualités de
la catégorie des \oo-catégories strictes.

Dans le septième chapitre, on montre comment on peut définir les orientaux de
Street et le nerf de Street formellement à partir des propriétés de notre
joint \oo-catégorique, idée apparaissant déjà dans \cite{StreetParComp}.
Plus précisément, on montre que la \oo-catégorie finale est munie d'une
unique structure de monoïde pour la structure de catégorie monoïdale définie
par le joint et que l'objet cosimplicial associé à ce monoïde est l'objet
cosimplicial de Street. On démontre par ailleurs que l'adjoint à gauche du
nerf de Street est un foncteur monoïdal pour le joint simplicial et le joint
\oo-catégorique.

Dans le huitième chapitre, on montre que la structure de catégorie monoïdale
sur la catégorie des \oo-catégories strictes définie par le joint induit par
troncation une structure de catégories monoïdales sur la catégorie des
$n$-catégories strictes pour tout~$n \ge 0$. On obtient de plus que cette
structure est localement bifermée. On vérifie que pour~$n = 1$ on retrouve
le joint et les tranches $1$-catégoriques usuels.

Le but du neuvième chapitre est de décrire par des formules explicites les
tranches au-dessous d'un objet qui ont été définies abstraitement dans le
sixième chapitre. En particulier, pour $n = 2$, on retrouve les formules
définissant les tranches qui apparaissent déjà dans la littérature.

Le dixième chapitre est consacré aux tranches des complexes dirigés
augmentés et à leurs propriétés de fonctorialité. On définit explicitement
des tranches généralisées pour les complexes dirigés augmentés et on montre
que celles-ci ont la propriété universelle attendue. On étudie des
propriétés de fonctorialité de l'opération tranche. On montre qu'à tout
triangle commutatif à une antihomotopie près, on peut associer un morphisme
entre les tranches associées et qu'à tout cône commutatif à une
$2$\nbd-antihomotopie près, on peut associer une homotopie entre les morphismes
entre les tranches associées. On étudie les compatibilités de ces
constructions aux différentes compositions.

Le onzième chapitre contient d'importants résultats de fonctorialité des
tranches \oo-catégoriques, analogues \oo-catégoriques des résultats du
chapitre précédent. On montre qu'à tout triangle de \oo-foncteurs commutatif
à une transformation lax près se factorisant par un triangle de complexes de
Steiner forts (et satisfaisant une hypothèse technique),
on peut associer un \oo-foncteur entre les tranches \oo-catégoriques
associées. De même, on montre qu'à tout cône satisfaisant une condition
analogue on peut associer une transformation oplax. Enfin, on étudie les
compatibilités de ces constructions aux différentes compositions. Tous ces
résultats sont obtenus à partir des résultats analogues pour les complexes
dirigés augmentés. Cette réduction au cas des complexes dirigés augmentés
utilise de manière cruciale les résultats sur les sommes amalgamées de
complexes de Steiner forts établis dans le troisième chapitre.

L'appendice~\ref{sec:Gray} est consacré au produit tensoriel de Gray
\oo-catégorique. On applique les techniques utilisées précédemment pour le
joint pour construire le produit tensoriel de Gray et montrer qu'il définit
une structure de catégorie monoïdale bifermée sur la catégorie des
\oo-catégories strictes. On définit les \oo-catégories strictes des
foncteurs lax et oplax comme $\Hom$ interne à droite et à gauche. On discute
les compatibilités du produit tensoriel et de ces deux $\Hom$ internes
avec les dualités de la catégorie des \oo-catégories strictes. On montre que,
comme dans le cas du joint, le produit tensoriel induit par troncation une
structure de catégorie monoïdale bifermée sur la catégorie des
$n$-catégories strictes pour tout $n \ge 0$.

Le but de l'appendice~\ref{sec:trans_oplax} est de montrer l'équivalence
entre la définition concrète des \oo-foncteurs oplax donnée dans le premier
chapitre et la notion abstraite définie en termes du produit tensoriel
\oo-catégorique. Ceci nous permet de définir la composition verticale des
transformations oplax sans donner explicitement de formules. On utilise par
ailleurs cette équivalence pour associer à toute homotopie de complexes
dirigés augmentés une transformation oplax de \oo-catégories et on étudie
les propriétés de cette correspondance. Pour finir, on étudie le lien entre
les \oo-catégories $\Dn{1} \otimes C$, $\Dn{0} \joint C$ et $\susp{C}$, où
$\Dn{0}$ et $\Dn{1}$ désignent les \oo-catégories associées aux ensembles
ordonnés $\{0\}$ et $\{0 < 1\}$, et $\susp{}$ désigne la suspension,
introduite dans cet appendice.

Enfin, dans l'appendice~\ref{sec:conj}, on dégage des conjectures de
fonctorialités du joint et des tranches dont les résultats du douzième
chapitre sont des cas particuliers. On introduit la notion de
$\Catmod$-sesquicatégorie, où $\Catmod$ est une catégorie munie d'un
foncteur vers la catégorie des ensembles. Dans le cas où $\Catmod$ est la
catégorie $\ooCat$ des \oo-catégories strictes munie du foncteur ensemble
des objets, on obtient la notion de \oo-sesquicatégorie. On introduit par
ailleurs les notions de \oo-catégorie de Gray et de \oo-catégorie de Gray
gauche, \oo-catégories enrichies dans $\ooCat$ munie respectivement du
produit tensoriel de Gray et du transposé de ce produit tensoriel. On donne
l'exemple de $\ooCatGr$ (resp. de $\ooCatGrg$), \oo-catégorie de Gray (resp.
de Gray gauche) dont les objets sont les \oo-catégories strictes, les
$1$-flèches les \oo-foncteurs stricts et les $i$-flèches pour~$i > 1$ les
$(i-1)$-transformations oplax (resp. lax). On conjecture l'existence d'une
\oo-catégorie de Gray (resp. de Gray gauche) tranche au-dessus ou au-dessous
d'un objet d'une \oo-catégorie de Gray (resp. de Gray gauche). On conjecture
que les foncteurs joint et tranches \oo-catégoriques, vus comme foncteurs en
une de leurs deux variables, proviennent de \oo-foncteurs de Gray ou de Gray
gauches (c'est-à-dire de foncteurs enrichis) entre des tranches de
$\ooCatGr$ et $\ooCatGrg$.
\end{organisation}

\begin{remerciements}
  Les auteurs remercient Andrea Gagna pour les nombreuses coquilles qu'il a
  débusquées dans une version préliminaire de ce texte, ainsi que pour
  l'exemple~\ref{exem:contre-ex_rig} qui lui est dû et qui montre que la
  première rédaction du troisième chapitre était légèrement incorrecte.
\end{remerciements}

\begin{notations}
  Si $\C$ est une catégorie, on notera \nnot{$\Ob(\C)$} la classe de ses objets.
  La catégorie opposée à $\C$ sera notée \nnot{$\C^\op$}. Si $\D$ est une seconde
  catégorie, la catégorie des foncteurs de $\C$ vers $\D$ sera notée
  \nnot{$\Homi(\C, \D)$}. On notera \nnot{$\Ens$} la catégorie des ensembles
  et \nnot{$\Ab$} la catégorie des groupes abéliens. Si $A$ est une petite
  catégorie, on notera \nnot{$\pref{A}$} la catégorie des préfaisceaux sur
  $A$, c'est-à-dire la catégorie $\Homi(A^\op, \Ens)$ des foncteurs de
  $A^\op$ vers la catégorie des ensembles.

  On dira qu'une catégorie est connexe si elle $0$-connexe. Autrement dit,
  on ne considérera pas la catégorie vide comme une catégorie connexe.
  Ainsi, lorsqu'on parlera de limites inductives connexes, on supposera que
  la catégorie source du système inductif est non vide.
\end{notations}

\chapter{Préliminaires \pdfoo-catégoriques}

\begin{paragr}\label{paragr:conv_ooCat}
  On notera \nnot{$\ooCat$} la catégorie des petites \oo-catégories
  \emph{strictes} et des \oo-foncteurs \emph{stricts} entre celles-ci.
  Toutes les \oo-catégories et tous les \oo-foncteurs considérés dans ce
  texte seront stricts et on se permettra donc d'omettre l'adjectif
  «~strict~».  Sauf mention expresse du contraire, ces \oo-catégories seront
  de plus supposées petites. Si $C$ est une \oo-catégorie, pour $i \ge 0$,
  on notera \nnot{$C_i$} l'ensemble de ses $i$-flèches ou $i$-cellules.
  \termindex{$i$-flèche!d'une $\infty$-catégorie}%
  Si $x$ est une $i$-flèche pour $i > 0$, on notera
  respectivement~\nnot[$s(x)$, $s_j(x)$]{$s(x)$} et \nnot[$t(x)$,
  $t_j(x)$]{$t(x)$} les $(i-1)$\nbd-flèches source et but de $x$. Plus
  généralement, pour $j$ tel que $0 \le j \le i$, on notera $s_j(x)$
  et~$t_j(x)$ la \ndef[$j$-source d'une $i$-flèche!d'une
  $\infty$-catégorie]{$j$-source} et le \ndef[$j$-but d'une $i$-flèche!d'une
  $\infty$-catégorie]{$j$-but} de $x$, c'est-à-dire les $j$\nbd-flèches
  source et but itérés de $x$ en dimension $j$. On dira que deux
  $i$\nbd-flèches $x$ et~$y$ sont \ndef[$i$-flèches $j$-composables!d'une
  $\infty$-catégorie]{$j$-composables} si on a $s_j(x) = t_j(y)$. Si $x$ et
  $y$ sont $j$-composables, on notera~\nnot{$x \comp_j y$} leur composé
  qu'on appellera parfois \ndef[$j$-composé de $i$-flèches!d'une
  $\infty$-catégorie]{$j$-composé}.  Si $x$ est une $i$-flèche de $C$, on
  notera~\nnot{$\id{x}$}
  \termindex{unité d'une cellule!d'une $\infty$-catégorie}%
  la $(i+1)$-flèche identité de $x$. On dira parfois qu'une $(i+1)$-flèche
  de la forme $\id{x}$, pour $x$ une $i$-flèche, est
  \ndef[$i$-flèche!triviale]{triviale}.

  Afin de simplifier certaines formules, on adoptera les conventions
  suivantes. Pour tous $i, j > k \ge 0$, si $x$ est une $i$-flèche de $C$, $y$
  une $j$-flèche de $C$ et qu'on a l'égalité $s_k(x) = t_k(y)$, on notera $x
  \comp_k y$ la $l$-flèche, où $l$ est le maximum de $\{i, j\}$, définie par
  \[
  x \comp_k y =
  \begin{cases}
    \id{x} \comp_k y & \text{si $i \le j$,}\\
    x \comp_k \id{y} & \text{si $i \ge j$,}
  \end{cases}
  \]
  où $\id{z}$, pour $z = x, y$, désigne l'identité itérée de $z$ en dimension
  $l$. Par ailleurs, si $i < j$, on considérera que l'opération $\comp_i$
  est prioritaire sur l'opération $\comp_j$ de sorte qu'on a
  \[
  x \comp_i y \comp_j z
  =
  (x \comp_i y) \comp_j z
  \quadet
  x \comp_j y \comp_i z
  =
  x \comp_j (y \comp_i z),
  \]
  lorsque ces formules ont un sens.
\end{paragr}

\begin{paragr}\label{paragr:tronque}
  Soit $n \ge 0$. On considérera une $n$-catégorie $C$ comme une \oo-catégorie
  dont les ensembles de $j$-flèches $C_j$ pour $j > n$ ne contiennent que des
  cellules triviales. On notera \nnot{$\nCat{n}$} la sous-catégorie pleine de
  $\ooCat$ formée des $n$-catégories.

  Le foncteur d'inclusion $\nCat{n} \hookto \ooCat$ des $n$-catégories dans les
  \oo-catégories admet un adjoint à gauche et un adjoint à droite qu'on notera
  respectivement
  \[
    \ti{n} : \ooCat \to \nCat{n}
    \quadet
    \tb{n} : \ooCat \to \nCat{n}.
  \]
  \notindex{$\ti{n}(C), \tb{n}(C)$}%
  On appellera $\ti{n}$ le foncteur \ndef[$n$-tronqué!intelligent!d'une
  $\infty$-catégorie]{$n$-tronqué intelligent} et $\tb{n}$
  le foncteur \ndef[$n$-tronqué!bête!d'une $\infty$-catégorie]{$n$-tronqué
  bête}.  Explicitement, si $C$ est une \oo-catégorie, on a
  \[
    \ti{n}(C)_i = \begin{cases}
      C_i & \text{si $i < n$,} \\
      {C_n}\slash{\sim} & \text{si $i \ge n$,} \\
    \end{cases}
  \]
  où $\sim$ est la relation d'équivalence sur $C_n$ définie par « $x \sim y$
  s'il existe un zigzag de $(n+1)$\nbd-flèches de $C$ entre $x$ et $y$ », et
  \[
    \tb{n}(C)_i = \begin{cases}
      C_i & \text{si $i <  n$,} \\
      C_n & \text{si $i \ge n$,}
    \end{cases}
  \]
  les compositions et identités étant héritées de celles de $C$ de la
  manière évidente.

  On identifiera souvent $\tb{n}(C)$ à une sous-\oo-catégorie de $C$
  \forlang{via} le morphisme d'adjonction $\tb{n}(C) \hookto C$.

  Par ailleurs,  le foncteur $\tb{n}$ admet lui-même un adjoint à droite : le
  foncteur qui associe à une $n$-catégorie $C$ la $(n+1)$-catégorie dont le
  $n$-tronqué bête est $C$ et telle que pour toutes $n$-flèches $x$ et $y$, il
  existe exactement une $(n+1)$-flèche de $x$ vers~$y$. En particulier, le
  foncteur $\tb{n}$ commute aussi bien aux limites projectives qu'aux limites
  inductives.
\end{paragr}

\begin{paragr}
  Soit $C$ une \oo-catégorie. Un \ndef[ensemble multiplicatif de
  cellules]{ensemble multiplicatif} de cellules de $C$ est un ensemble $M$
  de cellules de $C$ satisfaisant aux deux conditions suivantes :
  \begin{enumerate}
    \item pour tout $i \ge 0$ et toute $i$-flèche $x$ de $C$, si $x$
      appartient à $M$, alors la \hbox{$(i+1)$}\nbd-flèche $\id{x}$ appartient à $M$;
    \item pour tous $i > j \ge 0$ et tout couple $x, y$ de $i$-flèches
      $j$-composables de $C$, si $x$ et~$y$ appartiennent à $M$, alors
      la $i$\nbd-flèche composée $x \comp_j y$ appartient à $M$.
  \end{enumerate}
  On dit qu'un ensemble $E$ de cellules d'une $\infty$\nbd-catégorie $C$
  \emph{engendre $C$ par compositions} si l'ensemble de toutes les cellules de
  $C$ est le plus petit ensemble multiplicatif de cellules de $C$ contenant
  $E$. Dans ce cas, $E$ contient nécessairement l'ensemble des objets de~$C$
  et, pour tout $i>0$, toute $i$\nbd-flèche de $C$ est un composé d'un nombre
  fini de $i$\nbd-flèches qui sont dans $E$ ou sont des identités itérées de
  cellules appartenant à $E$.
\end{paragr}

\begin{paragr}\label{paragr:def_pol}
  Soient $C$ une \oo-catégorie et $E$ un ensemble de cellules de $C$. Pour $i
  \ge 0$, posons $E_i = E \cap C_i$. On dira que
  \ndef[$\infty$-catégorie!engendrée librement au sens des polygraphes]{$E$
  engendre librement $C$ au sens des polygraphes} si
  \begin{enumerate}
    \item\label{item:pol_a} $E_0 = C_0$ ;
    \item\label{item:pol_b} pour tout $i \ge 0$, toute \oo-catégorie $D$, tout
      \oo-foncteur $u : \tb{i}(C) \to D$ et toute application $f : E_{i+1} \to
      D_{i+1}$ tels que, pour tout $x$ dans $E_{i+1}$, on ait
      \[ s(f(x)) = u(s(x)) \quad\text{et}\quad t(f(x)) = u(t(x)), \]
      il existe un unique \oo-foncteur $u' : \tb{i+1}(C) \to D$ tel que
      \[ u'_{|\tb{i}(C)} = u \quad\text{et}\quad u'_{|E_{i+1}} = f. \]
  \end{enumerate}
\end{paragr}

\begin{prop}\label{prop:eng_pol_comp}
  Soient $C$ une $\infty$\nbd-catégorie et $E$ un ensemble de cellules de $C$
  qui l'engendre librement au sens des polygraphes. Alors $E$ engendre $C$ par
  compositions.
\end{prop}

\begin{proof}
  Pour $i \ge 0$, on pose $E_i = E\cap C_i$ et \smash{$E_{\le
  i}=\mathop{\cup}\nolimits_{\,0\le j\le i}E_j$}. On va montrer par
  récurrence sur $i \ge 0$ que l'ensemble $E_{\le i}$ engendre $\tb{i}(C)$
  par compositions, ce qui prouvera la proposition. Pour $i=0$, l'assertion
  est évidente grâce à la condition~\ref{item:pol_a} du
  paragraphe~\ref{paragr:def_pol}. Supposons donc l'assertion démontrée pour
  $i$ et démontrons-la pour~$i+1$. Soit~$D$ la sous-$(i+1)$\nbd-catégorie de
  $\tb{i+1}(C)$ dont les $j$\nbd-flèches, pour $0\le j\le i$, sont les
  $j$\nbd-flèches de~$C$ et dont les $(i+1)$\nbd-flèches sont les composés
  d'éléments de~$E_{i+1}$ et d'identités itérées de $j$\nbd-flèches de $C$
  pour $0\le j\le i$. Par hypothèse de récurrence, l'ensemble $E_{\le i+1}$
  engendre $D$ par compositions. Il suffit donc de montrer que
  $\tb{i+1}(C)=D$ ou, de manière équivalente, que l'inclusion
  $v:D\hookto\tb{i+1}(C)$ admet une section. Par définition, on a une
  inclusion $u:\tb{i}(C)\hookto D$ et une inclusion
  $f:E_{i+1}\hookto D_{i+1}$ satisfaisant aux hypothèses de la
  condition~\ref{item:pol_b} du paragraphe~\ref{paragr:def_pol}. On en
  déduit l'existence d'un $\infty$\nbd-foncteur $u':\tb{i+1}(C)\to D$ tel
  que
  \[ u'_{|\tb{i}(C)}=u \quadet u'_{|E_{i+1}}=f. \]
  La restriction de $vu'$ à $\tb{i}(C)$ est donc l'inclusion
  $\tb{i}(C)\hookto\tb{i+1}(C)$ et la restriction à $E_{i+1}$ l'inclusion
  $E_{i+1}\hookto C_{i+1}$, ce qui, en vertu de la partie unicité de la
  condition~\ref{item:pol_b} du paragraphe~\ref{paragr:def_pol}, prouve que
  $u'$ est une section de $v$, ce qu'il fallait démontrer.
\end{proof}

\begin{prop}\label{prop:iso_pol}
  Soient $P : A \to \ooCat$ un foncteur de source une petite catégorie $A$,
  $C$ une \oo-catégorie et $\alpha$~une transformation naturelle de $P$ vers
  le foncteur constant de valeur $C$. On suppose fixé, pour tout objet $a$
  de $A$, un ensemble $E_a$ engendrant librement la \oo-catégorie $P(a)$ au
  sens des polygraphes, ainsi qu'un ensemble $E$ engendrant librement la
  \oo-catégorie $C$ au sens des polygraphes. On suppose que
  \begin{enumerate}
    \item\label{item:iso_pol_a} pour tout objet $a$ de $A$, on a
      $\alpha_a(E_a) \subset E$ ;
    \item\label{item:iso_pol_b} pour toute flèche $f : a \to b$ de $A$, on a
      $P(f)(E_a) \subset E_b$ ;
    \item\label{item:iso_pol_c} l'application $\limind_{a \in A} E_a \to E$
      induite par $\alpha$ est une bijection.
  \end{enumerate}
  Alors le \oo-foncteur $\limind_{a \in A} P(a) \to C$ induit par $\alpha$
  est un isomorphisme de \oo-caté\-gories.
\end{prop}

\begin{proof}
  % TOCHECK
  \abovedisplayskip=3.5pt
  \belowdisplayskip=3.5pt
  \tolerance=500
  Pour tout $i\ge0$, on pose $E_i=E\cap C_i$ et, pour tout objet $a$ de~$A$,
  $E_{a,i}=E_a\cap P(a)_i$. Par hypothèse, pour tout $i \ge 0$, l'application
  $\limind_A E_{a,i}\to E_i$ induite par $\alpha$ est bijective. On va montrer
  par récurrence sur $i$ que, pour tout $i\ge0$, le tronqué bête
  (défini au paragraphe~\ref{paragr:tronque})
  \[
    \tb{i}(\limind_A P)\simeq\limind_A\tb{i} P
    \longto
    \tb{i}(C)
  \]
  du \oo-foncteur $\limind_A P \to C$ induit par $\alpha$ est un
  isomorphisme, ce qui prouvera la proposition.  Pour $i=0$, cela résulte de
  la bijectivité de l'application \hbox{$\limind_AE_{a,0}\to E_0$} et du fait que,
  en vertu de la condition~\ref{item:pol_a} du
  paragraphe~\ref{paragr:def_pol}, on a $C_0=E_0$ et~\hbox{$P(a)_0 =
  E_{a,0}$} pour tout $a$ dans $A$. Supposons l'assertion établie pour un $i
  \ge 0$ et prouvons-la pour~\hbox{$i+1$}. Soient $D$ une $\infty$\nbd-catégorie et
  \hbox{$\beta:\tb{i+1} P \to D$} une transformation naturelle de but le
  foncteur constant de valeur~$D$. Il s'agit de montrer qu'il existe un
  unique $\infty$\nbd-foncteur $v:\tb{i+1}(C)\to D$ tel que, pour tout objet
  $a$ de~$A$, on ait \hbox{$\beta_a=v\circ\tb{i+1}(\alpha_a)$}. En vertu de
  l'hypothèse de récurrence, il existe un unique $\infty$\nbd-foncteur
  $u:\tb{i}(C)\to D$ tel que, pour tout objet $a$ de~$A$,
  la restriction de~$\beta_a$ à $\tb{i}(P(a))$ soit égale à $u\circ\tb{i}(\alpha_a)$.
  D'autre part, la
  bijectivité de l'application \hbox{$\limind_A E_{a,i+1} \to E_{i+1}$} implique
  l'existence d'une unique application $f:E_{i+1}\to D_{i+1}$ telle que,
  pour tout objet $a$ de $A$ et tout $x_a$ dans $E_{a,i+1}$, on ait
  $\beta_a(x_a)=f(\alpha_a(x_a))$. Or, pour tout $x$ dans $E_{i+1}$, il
  existe un objet $a$ de $A$ et $x_a$ dans $E_{a,i+1}$ tel que
  $x=\alpha_a(x_a)$. On a donc
  \[
    \begin{split}
      s(f(x)) & = s(f(\alpha_a(x_a)))=s(\beta_a(x_a))=\beta_a(s(x_a)) \\
      & = u\alpha_a(s(x_a))=u(s(\alpha_a(x_a)))=u(s(x))
    \end{split}
  \]
  et, de même, $t(f(x))=u(t(x))$. Comme $E$ engendre librement $C$ au sens
  des polygraphes, en vertu de la condition~\ref{item:pol_b} du
  paragraphe~\ref{paragr:def_pol}, il existe donc un unique
  $\infty$\nbd-foncteur $v:\tb{i+1}(C)\to D$ tel que
  \[ v_{|\tb{i}(C)} = u \quadet v_{|E_{i+1}} = f. \]
  Il reste à prouver que, pour tout objet $a$ de $A$, on a
  $\beta_a=v\circ\tb{i+1}(\alpha_a)$. Or, on a les égalités
  \[
    {\beta_a}_{|\tb{i}(P(a))} = u \circ \tb{i}(\alpha_a)
    = v_{|\tb{i}(C)} \circ \tb{i}(\alpha_a) = (v \circ
    \tb{i+1}(\alpha_a))_{|\tb{i}(P(a))}
  \]
  et, pour tout $x_a$ dans $E_{a,i+1}$, on a
  $\beta_a(x_a)=f(\alpha_a(x_a))=v(\alpha_a(x_a))$ et donc
  \[
    {\beta_a}_{|E_{a,i+1}}= (v \circ \tb{i+1}({\alpha_a}))_{|E_{a,i+1}}.
  \]
  Comme $E_{a}$ engendre librement $P(a)$ au sens des polygraphes, en vertu de
  la partie unicité de la condition~\ref{item:pol_b} du
  paragraphe~\ref{paragr:def_pol}, on a donc
  $\beta_a=v\circ\tb{i+1}(\alpha_a)$, ce qui achève la démonstration.
\end{proof}

\begin{rem}
  On vérifie facilement que la proposition précédente reste vraie si on
  remplace l'hypothèse « $E_a$ engendre librement $P(a)$ au sens des
  polygraphes » par l'hypothèse plus faible « $E_a$ engendre $P(a)$ par
  compositions ». En revanche, on doit toujours supposer que l'ensemble $E$
  engendre librement $C$ au sens des polygraphes. On n'aura pas besoin de ce
  résultat dans la suite.
\end{rem}

\begin{paragr}\label{paragr:dual_ooCat}
  Soit $C$ une \oo-catégorie. On définit une \oo-catégorie
  \nnot[$C^\op$, $C^\opp$, $C^\co$]{$C^\op$}, appelée le \ndef[dual!d'une
  $\infty$-catégorie]{dual} de~$C$, en inversant le sens des $i$-flèches de
  $C$ pour tout $i > 0$. On vérifie immédiatement que l'application $C
  \mapsto C^\op$ définit un endofoncteur involutif $\dual{} : \ooCat \to
  \ooCat$.

  Plus généralement, pour toute partie $J$ de $\N^\ast = \N \sauf \{0\}$, on
  définit un endofoncteur involutif $\dual{J} : \ooCat \to \ooCat$
  \notindex{$\dual{J}(C)$}%
  qui envoie une \oo-catégorie sur la \oo-catégorie obtenue en inversant le
  sens des $i$-flèches pour tout $i$ dans $J$. On appellera $\dual{J}(C)$ le
  \ndef[$j$-z@$J$-dual!d'une $\infty$-catégorie]{$J$-dual} de $C$. Par
  définition, on a $\dual{} = \dual{\N^\ast}$.

  Deux autres dualités jouent un rôle particulièrement important. Si $J$ est
  l'ensemble des entiers strictement positifs impairs, on appellera le
  $J$-dual le \ndef[dual impair!d'une $\infty$-catégorie]{dual impair} et on
  désignera $\dual{J}(C)$ par $C^\opp$.
  Si $J$ est l'ensemble des entiers strictement positifs pairs, on parlera
  du \ndef[dual pair!d'une $\infty$-catégorie]{dual pair} et on désignera
  $\dual{J}(C)$ par $C^\co$. Notons que si $C$ est une \oo-catégorie, on a
  $C^\op = (C^\opp)^\co = (C^\co)^\opp$.  Pour cette raison, la
  \oo-catégorie $C^\op$ est parfois notée~$C^\coop$ dans la littérature.
\end{paragr}

\begin{paragr}\label{paragr:def_trans}
  % TOCHECK
  \abovedisplayskip=3.0pt
  \belowdisplayskip=3.0pt
  Soient $u, v : C \to D$ deux \oo-foncteurs. Une \ndef{prétransformation
  oplax} $\alpha$ de $u$ vers $v$ consiste en la donnée, pour tout $i \ge 0$
  et toute $i$-flèche $x$ de $C$, d'une $(i+1)$-flèche
  \[
    \alpha_x :
    \alpha_{t_{i-1}(x)} \comp_{i-1} \cdots \comp_1 \alpha_{t_0(x)} \comp_0 u(x)
    \longto
    v(x) \comp_0 \alpha_{s_0(x)} \comp_1 \cdots \comp_{i-1}
    \alpha_{s_{i-1}(x)}
  \]
  \notindex{$\alpha_x$}%
  de $D$. (Pour lire ces formules, il faut garder en tête les conventions
  énoncées au paragraphe~\ref{paragr:conv_ooCat} : en particulier, pour
  $k < l$, l'opération $\comp_k$ est prioritaire sur
  l'opération~$\comp_l$.)

  Ainsi, si $x$ est un objet de $C$, on dispose d'une $1$-flèche
  \[
    \xymatrix{
      u(x) \ar[d]_{\alpha_x} \\ v(x)
    }
  \]
  de $D$; si $x$ est une $1$-flèche de $C$, on dispose d'une $2$-flèche
  \[
    \shorthandoff{;:}
    \xymatrix{
      u(s_0(x)) \ar[d]_{\alpha_{s_0(x)}} \ar[r]^{u(x)} &
      u(t_0(x)) \ar[d]^{\alpha_{t_0(x)}} \\
      v(s_0(x)) \ar[r]_{v(x)} &
      v(t_0(x))
      \ar@{}[u];[l]_(.35){}="x"
      \ar@{}[u];[l]_(.65){}="y"
      \ar@2"x";"y"_{\alpha_x}
    }
  \]
  de $D$; si $x$ est une $2$-flèche de $C$, on dispose d'une $3$-flèche
  \[
    \shorthandoff{;:}
    \xymatrix@C=4pc@R=4pc{
      u(s_0(x))
      \ar@/^2ex/[r]^{u(s_1(x))}_{}="0"
      \ar@/_2ex/[r]_(0.60){u(t_1(x))}_{}="1"
      \ar[d]_{}="f"_{\alpha_{s_0(x)}}
      \ar@2"0";"1"_{u(x)}
      &
      u(t_0(x))
      \ar[d]^{\alpha_{t_0(x)}} \\
      v(s_0(x))
      \ar@{.>}@/^2ex/[r]^(0.40){v(s_1(x))}_{}="0"
      \ar@/_2ex/[r]_{v(t_1(x))}_{}="1"
      \ar@2{:>}"0";"1"_{v(x)}
      &
      v(t_0(x))
      \ar@{}[u];[l]_(.40){}="x"
      \ar@{}[u];[l]_(.60){}="y"
      \ar@<-1ex>@/_1ex/@{:>}"x";"y"_(0.60){\alpha_{s_1(x)}}_{}="0"
      \ar@<1ex>@/^1ex/@2"x";"y"^(0.40){\alpha_{t_1(x)}}_{}="1"
      \ar@{}"1";"0"_(.05){}="z"
      \ar@{}"1";"0"_(.95){}="t"
      \ar@3"z";"t"_{\alpha_x}
    }
  \]
  de $D$ de source $\alpha_{t_1(x)} \comp_1 (\alpha_{t_0(x)} \comp_0 u(x))$ et
  de but $(v(x) \comp_0 \alpha_{s_0(x)}) \comp_1 \alpha_{s_1(x)}$; etc.

  Une telle prétransformation oplax est une
  \ndef[transformation!oplax]{transformation oplax}
  si elle satisfait aux axiomes de fonctorialité suivants :
  \begin{enumerate}
    \item pour tout $i \ge 0$ et toute $i$-flèche $x$ de $C$, on a
      \[ \alpha^{}_{\id{x}} = \id{\alpha^{}_x}; \]
    \item pour tous $i > j \ge 0$ et tout couple $x, y$ de $i$-flèches
      $j$-composables de $C$, on a
      \[
        \begin{split}
          \alpha_{x \comp_j y} & =
          \left(v(t_{j+1}(x)) \comp_0 \alpha_{s_0(y)} \comp_1 \cdots
          \comp_{j-1} \alpha_{s_{j-1}(y)} \comp_j \alpha_y\right) \\
          & \phantom{=1} \qquad
          \comp_{j+1} \left(\alpha_x \comp_j \alpha_{t_{j-1}(x)} \comp_{j-1}
          \cdots \comp_1 \alpha_{t_0(x)} \comp_0 u(s_{j+1}(y))\right).
        \end{split}
      \]
  \end{enumerate}
\end{paragr}

\begin{rem}
  On montrera dans l'appendice~\ref{sec:trans_oplax} (voir notamment le
  corollaire~\ref{coro:trans_oplax_abs}) que les
  transformations oplax entre \oo-foncteurs de $C$ vers $D$ peuvent
  s'interpréter en termes du produit tensoriel de Gray \oo-catégorique. Plus
  précisément, en notant $\Dn{1}$ la catégorie associée à l'ensemble
  ordonné $0 < 1$, de telles transformations oplax correspondent aux
  \oo-foncteurs \hbox{$\Dn{1} \otimes C \to D$}, où $\otimes$ désigne le
  produit tensoriel de Gray  (voir le paragraphe~\ref{paragr:def_tens}), ou
  encore, par adjonction, aux \oo-foncteurs $C \to \HomLax(\Dn{1}, D)$
  (voir le paragraphe~\ref{paragr:def_HomOpLax}).
\end{rem}

\begin{rem}
  La \oo-catégorie $\HomLax(\Dn{1}, D)$ de la remarque précédente est
  isomorphe à la \oo-catégorie $HD$ des cylindres dans $D$ introduite par
  Métayer dans~\cite{MetPolRes} (voir notre remarque~\ref{rem:HC}).
  Ainsi, les formules définissant les transformations oplax se trouvent déjà
  dans \cite{MetPolRes}.
\end{rem}

\begin{paragr}
  Soient $u, v : C \to D$ deux \oo-foncteurs. Une
  \ndef[transformation!lax]{transformation lax}
  de $u$ vers~$v$ est une transformation oplax de $u^\co : C^\co \to D^\co$
  vers $v^\co : C^\co \to D^\co$.
\end{paragr}

\begin{rem}
  On verra dans l'appendice~\ref{sec:trans_oplax} (voir le
  corollaire~\ref{coro:trans_lax_abs}) que les transformations lax entre
  \oo-foncteurs de $C$ vers $D$ correspondent aux \oo-fonc\-teurs~$C \otimes
  \Dn{1} \to D$.
\end{rem}

\begin{rem}\label{rem:dual_trans}
  Si $u, v : C \to D$ sont deux \oo-foncteurs, on vérifie facilement qu'une
  transformation oplax de $u^\op$ vers $v^\op$ n'est autre qu'une
  transformation oplax de $v$ vers $u$ ; on en déduit qu'une
  transformation oplax de $u^\opp$ vers $v^\opp$ est une transformation lax
  de $v$ vers $u$.
\end{rem}

Dans ce texte, on privilégiera les transformations oplax aux transformations
lax. On laisse le soin au lecteur d'expliciter la notion de transformation
lax et de dualiser les énoncés concernant les transformations oplax en des
énoncés sur les transformations lax.

\begin{paragr}\label{paragr:def_trans_id}
  Soit $u : C \to D$ un \oo-foncteur. On vérifie immédiatement qu'on définit
  une transformation oplax $\id{u}$ de $u$ vers $u$ en posant, pour toute
  cellule $x$ de $C$,
  \[
    (\id{u})_x = \id{u(x)}.
  \]
  \notindex{$\id{u}$}%
  On appellera cette transformation oplax la
  \ndef[transformation!oplax!identité]{transformation oplax identité} de
  $u$.
\end{paragr}

\begin{paragr}\label{paragr:def_trans_comp}
  Soient $v_0, v_1 : C \to D$ deux \oo-foncteurs et $\alpha$ une
  transformation oplax de~$v_0$ vers $v_1$. Soit $u : B \to C$ un
  \oo-foncteur. On vérifie immédiatement qu'on définit une transformation
  oplax $\alpha \comp u$
  \notindex{$\alpha \comp u$, $u \comp \alpha$}%
  de $v_0u$ vers $v_1u$ (qui sont des \oo-foncteurs
  de $B$ vers~$D$) en posant, pour $x$ une cellule de $B$,
  \[ (\alpha \comp u)_x = \alpha_{u(x)}. \]

  De même, si $w : D \to E$ est un \oo-foncteur, on définit une
  transformation oplax $w \comp \alpha$ de~$wv_0$ vers $wv_1$ (qui sont des
  \oo-foncteurs de $C$ vers $E$) en posant, pour $x$ une cellule de $C$,
  \[ (w \comp \alpha)_x = w(\alpha^{}_x). \]
\end{paragr}

\begin{paragr}
  Fixons une \oo-catégorie $C$. Soient $(A, v : A \to C)$ et $(B, w : B \to
  C)$ deux \oo-catégories au-dessus de $C$ et $u_0, u_1 : A \to B$ deux
  \oo-foncteurs au-dessus de $C$. Autrement dit, on a un diagramme
  \[
    \shorthandoff{;:}
    
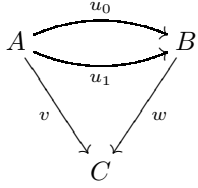
\xymatrix@C=1.5pc@R=3pc{
      A \ar@/^2ex/[rr]^{u_0} \ar@/_2ex/[rr]_{u_1}
      \ar[dr]_v
      & & B \ar[dl]^w \\
      & C &
    }
  \]
  de \oo-foncteurs formé de deux triangles commutatifs. On dira qu'une
  transformation oplax $\alpha$ de $u_0$ vers $u_1$ est
  \ndef[transformation!oplax!au-dessus d'une $\infty$-catégorie]{au-dessus
  de~$C$} si $w \comp \alpha$ est la transformation oplax identité de~$v$.
\end{paragr}

\chapter[Rappels et compléments sur la théorie de Steiner]{Rappels et
compléments sur\\ la théorie de Steiner}

Le but de ce chapitre est d'exposer la théorie développée par Steiner dans
\cite{Steiner} et d'y apporter quelques compléments.

\begin{paragr}\label{paragr:conv_comp}
  Dans ce texte, par « complexe de chaînes » on entendra toujours « complexe
  de chaînes de groupes abéliens en degrés positifs ». On rappelle qu'un
  \emph{élément homogène} d'un complexe de chaînes $K$ est un élément de la
  somme disjointe ensembliste des~$K_n$ pour $n \ge 0$. Le \ndef[degré d'un
  élément homogène d'un complexe de chaînes]{degré} d'un élément homogène
  $x$, c'est-à-dire l'unique $n$ tel que~$x$ appartienne à $K_n$, sera noté
  \nnot[$\vert x \vert$]{$|x|$}.
\end{paragr}

\begin{paragr}
  Un \ndef{complexe dirigé augmenté} est un triplet $(K, K^\ast, e)$, où
  \[
    K =
    \xymatrix{\cdots \ar[r]^-{d_{n+1}} & K_n \ar[r]^-{d_n} & K_{n-1}
    \ar[r]^-{d_{n-1}} & \cdots \ar[r]^{d_2} & K_1 \ar[r]^{d_1} & K_0}
  \]
  est un complexe de chaînes (de groupes abéliens en degrés positifs),
  \nnot{$e : K_0 \to \Z$} est une augmentation (on a donc $e d_1 = 0$) et $K^\ast =
  (K^\ast_n)_{n \ge 0}$ est la donnée pour tout~$n \ge 0$ d'un sous-monoïde
  \nnot{$K^\ast_n$} du groupe abélien $K_n$. (On ne demande aucune compatibilité
  entre la différentielle et ces sous-monoïdes.) On appellera les
  sous-monoïdes~$K^\ast_n$, pour~$n \ge 0$, les \ndef[sous-monoïde de
  positivité]{sous-monoïdes de positivité} de $K$.

  Un \ndef{morphisme de complexes dirigés augmentés} est un morphisme de
  complexes de chaînes augmentés qui respecte les sous-monoïdes de
  positivité. Plus précisément, si $(K, K^\ast, e)$ et $(K', K'^\ast, e')$
  sont deux complexes dirigés augmentés, un morphisme du premier vers le
  second est donné par un morphisme de complexes de chaînes $f : K \to K'$
  tel que $e'f_0 = e$ et $f_n(K^\ast_n) \subset K'^\ast_n$ pour tout $n \ge
  0$. On notera \nnot{$\Cda$} la catégorie des complexes dirigés augmentés.

  On désignera souvent, par abus de notation, un complexe dirigé augmenté par
  son complexe de chaînes sous-jacent.
\end{paragr}

\begin{paragr}
  On définit un foncteur
  \[ \lambda : \ooCat \to \Cda \]
  \notindex{$\lambda(C)$}%
  de la manière suivante.

  Soit $C$ une \oo-catégorie. Pour $i \ge 0$, le groupe abélien
  $\lambda(C)_i$ est défini par les générateurs
  \[ [x], \qquad\text{pour $x$ une $i$-flèche de $C$}, \]
  \notindex{$[x]$}%
  soumis aux relations
  \[
    [x \ast_j y] = [x] + [y], \qquad\text{pour $0 \le j < i$ et $x$ et $y$
    des $i$-flèches $j$-composables.}
  \]
  Le sous-monoïde de positivité $\lambda(C)^\ast_i$ est le sous-monoïde
  engendré par les générateurs~$[x]$ pour $x$ une $i$-flèche de $C$. Pour $i > 0$, la
  différentielle $d_i :
  \lambda(C)_i \to \lambda(C)_{i-1}$ est définie par
  \[
    d([x]) = [t(x)] - [s(x)],
    \qquad\text{pour $x$ une $i$-flèche de $C$}.
  \]
  Enfin, l'augmentation $\lambda(C)_0 \to \Z$ est l'unique morphisme qui
  envoie, pour tout objet~$x$ de $C$, le générateur $[x]$ sur $1$. On vérifie
  immédiatement qu'on obtient bien ainsi un complexe dirigé augmenté.

  Si $u : C \to D$ est un \oo-foncteur, on définit, pour tout $i \ge 0$, un
  morphisme de groupes abéliens $\lambda(u)_i : \lambda(C)_i \to \lambda(D)_i$ en
  envoyant un générateur $[x]$, pour $x$ une $i$-flèche de $C$, sur le
  générateur $[u(x)]$. On vérifie immédiatement qu'on obtient ainsi un
  morphisme de complexes dirigés augmentés $\lambda(u) : \lambda(C) \to
  \lambda(D)$ et que $\lambda$ ainsi défini est bien un foncteur.
\end{paragr}

\begin{paragr}\label{paragr:def_nu}
  On définit un foncteur
  \[
    \nu : \Cda \to \ooCat
  \]
  \notindex{$\nu(K)$}%
  de la manière suivante.

  Soit $K$ un complexe dirigé augmenté. Pour $i \ge 0$, les $i$-flèches de
  $\nu(K)$ sont les tableaux
  \[
    \tabld{x}{i}
  \]
  \notindex{$x^\e_i$, $x_i$}%
  tels que
  \begin{enumerate}
    \item $x^\epsilon_k$ appartient à $K^\ast_k$ pour $\epsilon = 0, 1$ et $0
      \le k \le i$ ;
    \item $d(x^\epsilon_k) = x^1_{k-1} - x^0_{k-1}$ pour $\epsilon = 0, 1$ et $0
      < k \le i$ ;
    \item $e(x^\epsilon_0) = 1$ pour $\epsilon = 0, 1$ ;
    \item $x_i^0 = x_i^1$.
  \end{enumerate}
  La structure de \oo-catégorie sur $\nu(K)$ est définie de la manière
  suivante. Soit
  \[
    x = \tabld{x}{i},
  \]
  pour $i \ge 0$, une $i$-flèche de $\nu(K)$. Si $i > 0$, les sources et
  buts de $x$ sont donnés par les tableaux
  \[
    s(x)=\tablnu{x}{i-2}{x^0_{i-1}}
    \quad\text{et}\quad
    t(x)=\tablnu{x}{i-2}{x^1_{i-1}}.
  \]
  Pour tout $i \ge 0$, l'identité de $x$ est le tableau
  \[
    \id{x} = \tablnu{x}{i}{0}.
  \]
  Enfin, si
  \[
    x = \tabld{x}{i}
    \quad\text{et}\quad
    y = \tabld{y}{i}
  \]
  sont deux $i$-flèches $j$-composables pour $i > j \ge 0$, on pose
  \[
    x \comp_j y =
    \begin{pmatrix}
      y^0_0
      &\dots
      &y^0_j
      &x^0_{j+1}+y^0_{j+1}
      &\dots
      &x^0_{i}+y^0_{i}
      \cr\noalign{\vskip 3pt}
      x^1_0
      &\dots
      &x^1_j
      &x^1_{j+1}+y^1_{j+1}
      &\dots
      &x^1_{i}+y^1_{i}
    \end{pmatrix}.
  \]
  On vérifie qu'on définit bien ainsi une \oo-catégorie.

  Si $x$ est une $i$-flèche de $\nu(K)$ pour un $i \ge 0$, on notera
  $x^\e_k$, pour $0 \le k \le i$ et $\e = 0, 1$, les composantes du
  tableau
  \[
    x = \tabld{x}{i}.
  \]
  On désignera par $x_i$ l'élément $x^0_i = x^1_i$ et, pour $k > i$
  et $\e = 0, 1$, on posera $x^\e_k = 0$.

  Si maintenant $f : K \to K'$ est un morphisme de complexes dirigés
  augmentés, on vérifie qu'on définit un \oo-foncteur $\nu(f) : \nu(K) \to
  \nu(K')$ par la formule
  \[
    x =
    \tabld{x}{i}
    \mapsto
    f(x) =
    \begin{pmatrix}
      f(x^0_1)
      &\dots
      &f(x^0_{i-1})
      &f(x^0_{i})
      \cr\noalign{\vskip 3pt}
      f(x^1_1)
      &\dots
      &f(x^1_{i-1})
      &f(x^1_{i})
    \end{pmatrix}.
  \]
  On vérifie immédiatement que $\nu$ ainsi défini est bien un foncteur.
\end{paragr}

\begin{thm}[Steiner]
  Les foncteurs
  \[
    \lambda : \ooCat \to \Cda, \qquad \nu : \Cda \to \ooCat
  \]
  forment un couple de foncteurs adjoints.
\end{thm}

\begin{proof}
  Voir \cite[théorème 2.11]{Steiner}.
\end{proof}

\begin{paragr}
  Une \ndef[base d'un complexe dirigé augmenté]{base} d'un complexe dirigé
  augmenté $K$ est un ensemble gradué \hbox{$B = (B_i)_{i \ge 0}$} tel que,
  pour tout $i \ge 0$,
  \begin{enumerate}
    \item $B_i$ est une base du $\Z$-module $K_i$ ;
    \item $B_i$ engendre le sous-monoïde $K^\ast_i$ de $K_i$.
  \end{enumerate}
  On identifiera souvent une base $B = (B_i)_{i \ge 0}$ à l'ensemble
  $\coprod_{i \ge 0} B_i$.

  Soit $K$ un complexe dirigé augmenté. On définit une relation de préordre
  $\le$ sur $K_i$ en posant
  \[
    x \le y \quaddefssi y - x \in K_i^\ast.
  \]
  On vérifie immédiatement que si $K$ admet une base, cette relation de
  préordre est une relation d'ordre et que les éléments de $B_i$ sont
  les éléments minimaux de~\hbox{$(K_i^\ast \sauf \{0\}, \le)$}. Ainsi, si
  $K$ admet une base, cette base est unique.

  On dira que le complexe dirigé augmenté $K$ est \ndef[complexe dirigé
  augmenté!à base]{à base} s'il admet une base (nécessairement unique).
\end{paragr}

\begin{paragr}
  Fixons $A$ un groupe abélien libre muni d'une base $B$. Soit
  \[ x = \sum_{b \in B} x_b\, b \]
  un élément de $A$. On appellera \ndef[support d'un élément d'un groupe
  abélien libre]{support} de $x$ l'ensemble
  \[
    \supp(x) = \{b \in B \mid x_b \neq 0\}.
  \]
  \notindex{$\supp(x)$}%
  Notons $A^\ast$ le sous-monoïde de $A$ engendré par $B$.
  On définit deux éléments $x_+$ et $x_-$ de $A^\ast$ par
  \[
    x_+ = \sum_{\substack{b \in B \\ x_b > 0}} x_b\,b
    \qquad\text{et}\qquad
    x_- = -\sum_{\substack{b \in B \\ x_b < 0}} x_b\,b.
  \]
  \notindex{$x_+, x_-$}%
  On a alors $x = x_+ - x_-$.

  En particulier, si $K$ est un complexe dirigé augmenté admettant une base
  \hbox{$B = (B_i)_{i \ge 0}$}, pour tout $x$ dans $K_i$ pour un $i \ge 0$,
  on dispose, en appliquant le paragraphe précédent au groupe abélien $K_i$
  muni de sa base $B_i$, d'une notion de support de $x$ et d'éléments $x_+$
  et $x_-$ de $K^\ast_i$.
\end{paragr}

\begin{paragr}\label{paragr:def_atome}
  Soit $K$ un complexe dirigé augmenté admettant une base $B = (B_i)_{i
  \ge 0}$. Pour $i \ge 0$ et $x$ un élément de $K_i$, on définit un tableau
  \[
    \atom{x}=\tabll{\atom{x}}{i},
  \]
  \notindex{$\atom{x}, \atom{x}^\e_i$}%
  où les $\atom{x}^\epsilon_k$ sont définis par récurrence descendante sur $k$
  de $i$ à $0$ :
  \begin{itemize}
    \item $\atom{x}^0_i = x = \atom{x}^1_i$ ;
    \item $\atom{x}^0_{k - 1} = d(\atom{x}^0_k)_-$ et $\atom{x}^1_{k - 1} =
      d(\atom{x}^1_k)_+$ pour $0 < k \le i$.
  \end{itemize}
  On vérifie facilement que ce tableau est une $i$-flèche de $\nu(K)$ si et
  seulement si, d'une part, $x$ appartient à $K^\ast_i$ et, d'autre part, on
  a les égalités $e(\atom{x}^0_0) = 1 = e(\atom{x}^1_0)$.
  On posera par ailleurs $\atom{x}_i = x$ et $\atom{x}^\e_k = 0$ pour tout
  $k > i$ et $\e = 0, 1$. Lorsque $\atom{x}$ est une $i$-flèche, ces
  conventions sont compatibles avec celles du
  paragraphe~\ref{paragr:def_nu}.

  On dit que la base $B$ de $K$ est \ndef[base d'un complexe dirigé
  augmenté!unitaire]{unitaire} si, pour tout $i \ge
  0$ et tout $x$ dans~$B_i$, le tableau $\atom{x}$ est une $i$-flèche de
  $\nu(K)$, ce qui revient à dire qu'on a l'égalité \hbox{$e(\atom{x}^0_0) =
  1 = e(\atom{x}^1_0)$}.

  On dit qu'un complexe dirigé augmenté est \ndef[complexe dirigé augmenté!à
  base!unitaire]{à base unitaire} s'il est
  à base et que son unique base est unitaire. Si un complexe dirigé augmenté
  $K$ est à base unitaire, pour tout élément $x$ de la base de $K$, on
  appelle la cellule $\atom{x}$ de $\nu(K)$ l'\ndef[atome d'un complexe
  dirigé augmenté à base unitaire]{atome} associé à $x$.
\end{paragr}

\begin{paragr}
  Soit $K$ un complexe dirigé augmenté admettant une base $B$. Pour $i \ge 0$,
  on note \nnot{$\le_i$} la plus petite relation de préordre sur $B$ ($=
  \coprod_i B_i$) satisfaisant
  \[
    x \le_i y \quad\text{si}\quad
    \text{$|x| > i$, $|y| > i$ et
    $\supp(\atom{x}^1_i) \cap \supp(\atom{y}^0_i)
    \neq \vide$}.
  \]
  On dit que la base $B$ est \ndef[base d'un complexe dirigé augmenté!sans
  boucle]{sans boucle} si, pour tout $i \ge 0$, la relation de préordre
  $\le_i$ est une relation d'ordre.

  On dit qu'un complexe dirigé augmenté est \ndef[complexe dirigé augmenté!à
  base!sans boucle]{à base sans boucle} s'il est à base et si sa base est
  sans boucle.

\end{paragr}

\begin{paragr}
  On appellera \ndef[complexe!de Steiner]{complexe de Steiner} un complexe
  dirigé augmenté à base dont la base est unitaire et sans boucle.

  On appellera \ndef[$\infty$-catégorie!de Steiner]{\oo-catégorie de
  Steiner} une
  \oo-catégorie dans l'image essentielle de la restriction du foncteur $\nu
  : \Cda \to \ooCat$ à la catégorie des complexes de Steiner. Le théorème
  suivant affirme que le foncteur $\nu$ induit une équivalence de catégories
  entre la catégorie des complexes de Steiner et celle des \oo-catégories de
  Steiner.
\end{paragr}

\begin{thm}[Steiner]\label{thm:Steiner}
  Pour tout complexe de Steiner $K$, le morphisme d'adjonction
  \[
    \lambda(\nu(K)) \to K
  \]
  est un isomorphisme. En particulier, la restriction du foncteur $\nu : \Cda
  \to \ooCat$ à la catégorie des complexes de Steiner est un foncteur
  pleinement fidèle.
\end{thm}

\begin{proof}
  Voir \cite[théorème 5.6]{Steiner}.
\end{proof}

\begin{thm}[Steiner]\label{thm:Steiner_pol}
  Soit $K$ un complexe de Steiner. Alors la \oo-caté\-gorie~$\nu(K)$
  est engendrée librement au sens des polygraphes par ses atomes,
  c'est-à-dire par les $\atom{x}$, où $x$ varie dans la base de $K$.
\end{thm}

\begin{proof}
  Voir \cite[théorème 6.1]{Steiner}.
\end{proof}

\begin{paragr}\label{paragr:def_le_N}
  Soit $K$ un complexe dirigé augmenté admettant une base $B$. On notera
  \nnot{$\leN$} la plus petite relation de préordre sur $B$ satisfaisant
  \[
    x \leN y \quad\text{si}\quad x \in \supp(d(y)_-) \text{ ou }
    y \in \supp(d(x)_+),
  \]
  où, par convention, $d(b) = 0$ si $b$ est dans $B_0$.
  On dit que la base $B$ est \ndef[base d'un complexe dirigé
  augmenté!fortement sans boucle]{fortement sans boucle} si la relation de
  préordre $\leN$ est une relation d'ordre.

  On dit qu'un complexe dirigé augmenté est \ndef[complexe dirigé
  augmenté!à base!fortement sans boucle]{à base fortement sans boucle} s'il
  est à base et si sa base est fortement sans boucle.
\end{paragr}

\begin{prop}[Steiner]
  Soit $K$ un complexe dirigé augmenté admettant une base $B$. Si la base $B$
  est fortement sans boucle, alors elle est sans boucle.
\end{prop}

\begin{proof}
  Voir \cite[proposition 3.7]{Steiner}.
\end{proof}

\begin{paragr}\label{paragr:def_ooCat_Stf}
  On appellera \ndef[complexe!de Steiner fort]{complexe de Steiner fort} un
  complexe dirigé augmenté à
  base dont la base est unitaire et fortement sans boucle. En vertu de la
  proposition précédente, un complexe de Steiner fort est un complexe de
  Steiner. On notera \nnot{$\Stf$} la sous-catégorie pleine de la catégorie des
  complexes dirigés augmentés formée des complexes de Steiner forts.

  On appellera \ndef[$\infty$-catégorie!de Steiner forte]{\oo-catégorie de
  Steiner forte} une \oo-catégorie dans l'image essentielle de la
  restriction du foncteur $\nu : \Cda \to \ooCat$ à
  la catégorie des complexes de Steiner forts. En vertu du
  théorème~\ref{thm:Steiner}, le foncteur $\nu$ induit donc une équivalence
  de la catégorie des complexes de Steiner forts vers la catégorie des
  \oo-catégories de Steiner fortes.
\end{paragr}

\begin{exem}
  L'exemple suivant montre que la classe des \oo-catégories de Steiner
  fortes est strictement incluse dans celle des \oo-catégories de Steiner.
  Considérons la $3$-catégorie $C$ définie par le diagramme
  \[
    \shorthandoff{;}
    \xymatrix@C=3pc@R=3pc{
      a \ar[r]^u_{}="ab" \ar[d]_f &
      b \ar[d]^g
      &
      a \ar[r]^u_{}="ab'" \ar[d]_f_{}="t3"
        &
      b \ar[d]^g
      \\
      c \ar[r]_v_{}="cd" \ar[ur]^{h\!\!\!} & d \ar@{}[u]_{}="s3"
      &
      c \ar[r]_v_{}="cd'" \ar[ur]^{h'\!\!\!\!\!} & d
      \ar@{}"s3";"t3"_(0.25){}="ss3"_(0.75){}="tt3"
      \ar@3"ss3";"tt3"^{\Lambda\,\,\,}
      \ar@{}"ab";[lll]_(0.15){}="sa"_(0.55){}="ta"
      \ar@2"sa";"ta"_(0.45){\alpha}
      \ar@{}[ull];"cd"_(0.45){}="sb"_(0.85){}="tb"
      \ar@2"sb";"tb"^{\beta}
      \ar@{}"ab'";[l]_(0.15){}="sa'"_(0.55){}="ta'"
      \ar@2"sa'";"ta'"_(0.45){\alpha'\!\!}
      \ar@{}[u];"cd'"_(0.45){}="sb'"_(0.85){}="tb'"
      \ar@2"sb'";"tb'"^{\beta'}
      \pbox{.}
    }
  \]
  On vérifie facilement que $\lambda(C)$ est un complexe de Steiner,
  de base les éléments de~$\lambda(C)$ correspondant aux cellules «
  génératrices » de $C$ (c'est-à-dire celles qui sont nommées dans le diagramme
  ci-dessus), et que le morphisme canonique $C \to \nu(\lambda(C))$ est un
  isomorphisme.  Néanmoins, il n'est pas vrai que $\lambda(C)$ soit un
  complexe de Steiner fort : en effet, on a
  \[ \beta \leN \Lambda \leN \alpha' \leN h' \leN b \leN g \leN \beta, \]
  où $\leN$ désigne la relation du paragraphe~\ref{paragr:def_le_N} et où on
  a identifié les cellules génératrices de $C$ avec les éléments de la base
  de $\lambda(C)$ leur correspondant.
\end{exem}

Dans la suite de ce chapitre, on introduit quelques compléments sur
la notion de complexe dirigé augmenté.

\begin{paragr}\label{paragr:def_decent}
  On dira qu'un complexe dirigé augmenté $K$ est \ndef[complexe dirigé
  augmenté!décent]{décent} si, pour tout
  $x$ dans~$K^\ast_0$, on a $e(x) \ge 0$. Notons que tout complexe dirigé
  augmenté $K$ à base unitaire, et donc en particulier tout complexe de Steiner
  (et donc tout complexe de Steiner fort), est décent puisque, par
  définition, si $x$ est dans la base de $K^\ast_0$, on a $e(x) = 1$.
\end{paragr}

\begin{paragr}\label{paragr:def_dual_Cda}
  Soient $K$ un complexe dirigé augmenté et $J$ une partie de $\N^\ast = \N
  \sauf \{0\}$. On définit un complexe dirigé augmenté \nnot{$\dual{J}(K)$}, appelé
  le \ndef[$j$-z@$J$-dual!d'un complexe dirigé augmenté]{$J$-dual} de
  $K$, de la manière suivante :
  \begin{itemize}
    \item pour tout $n \ge 0$, on a
    \[ \dual{J}(K)_n = K_n \quadet \dual{J}(K)^\ast_n = K^\ast_n \pbox{;}\]
    \item $\dual{J}(K)$ a même augmentation que $K$ ;
    \item la différentielle $d'$ de $\dual{J}(K)$ est
      donnée, pour $n \ge 1$, par
      \[ d'_n =
        \begin{cases}
          -d_n & \text{si $n$ appartient à $J$,} \\
          d_n & \text{sinon.}
        \end{cases}
      \]
  \end{itemize}
  On vérifie immédiatement que $K \mapsto \dual{J}(K)$ définit un
  endofoncteur involutif \hbox{$\dual{J} : \Cda \to \Cda$} de la catégorie
  des complexes dirigés augmentés.

  Comme dans le cas des \oo-catégories, trois dualités jouent un rôle
  particulièrement important. Si $J$ est l'ensemble des entiers strictement
  positifs impairs, on appellera le $J$-dual le \ndef[dual impair!d'un
  complexe dirigé augmenté]{dual impair} et on
  désignera \nnot[$K^\op$, $K^\opp$, $K^\co$]{$\dual{J}(K)$} par $K^\opp$.
  Si $J$ est l'ensemble des entiers strictement positifs pairs, on parlera
  du \ndef[dual pair!d'un complexe dirigé augmenté]{dual pair} et on
  désignera $\dual{J}(K)$ par $K^\co$. Enfin, si $J$ est l'ensemble des
  entiers strictement positifs, on appellera le $J$-dual le \ndef[dual!d'un
  complexe dirigé augmenté]{dual} et on désignera $\dual{J}(K)$ par $
  K^\op$.
\end{paragr}

\begin{prop}\label{prop:dual_lambda_nu}
  Pour toute partie $J$ de $\N^\ast$, les carrés de foncteurs
  \[
    \xymatrix{
    \ooCat \ar[r]^{\dual{J}} \ar[d]_\lambda  & \ooCat \ar[d]^\lambda \\
    \Cda \ar[r]_{\dual{J}} & \Cda \pbox{,} \\
    }
    \qquad
    \xymatrix{
    \Cda \ar[r]^{\dual{J}} \ar[d]_\nu & \Cda \ar[d]^\nu \\
    \ooCat \ar[r]_{\dual{J}} & \ooCat \pbox{,}
    }
  \]
  où les foncteurs horizontaux sont ceux définis au paragraphe précédent et
  au paragraphe~\ref{paragr:dual_ooCat}, sont commutatifs à isomorphisme
  canonique près. En particulier, pour toute \oo-catégorie $C$, on a des
  isomorphismes canoniques
  \[
    \lambda(C^\opp) \simeq (\lambda(C))^\opp,
    \quad
    \lambda(C^\co) \simeq (\lambda(C))^\co
    \quadet
    \lambda(C^\op) \simeq (\lambda(C))^\op,
  \]
  et, pour tout complexe dirigé augmenté $K$, on a des isomorphismes
  canoniques
  \[
    \nu(K^\opp) \simeq (\nu(K))^\opp,
    \quad
    \nu(K^\co) \simeq (\nu(K))^\co
    \quadet
    \nu(K^\op) \simeq (\nu(K))^\op.
  \]
\end{prop}

\begin{proof}
  Puisque $\dual{J}$ est une involution, ce foncteur est son propre adjoint
  à gauche comme à droite. La commutativité à isomorphisme près des deux
  carrés de l'énoncé est donc équivalente par adjonction et il suffit de
  vérifier la commutativité du premier.

  Soit $C$ une \oo-catégorie. Par définition, pour $i \ge 0$, le groupe
  abélien $\lambda(\dual{J}(C))_i$ est engendré par les $[x]$, pour $x$
  dans $\dual{J}(C)_i = C_i$, soumis aux relations $[x \comp'_j y] = [x] +
  [y]$, pour $x$ et $y$ des $i$-flèches composables pour la
  $j$-composition~$\comp'_j$ de~$\dual{J}(C)$. Si $j + 1$ n'appartient pas à
  $J$, la composition $\comp'_j$ est la composition $\comp_j$ de~$C$. Si $j
  + 1$ appartient à $J$, alors $x$ et $y$ sont composables pour $\comp'_j$
  si et seulement si $y$ et $x$ sont composables pour $\comp_j$ et on a alors
  $x \comp'_j y = y \comp_j x$. Puisque $[x] + [y] = [y] + [x]$, les
  relations définissant $\lambda(\dual{J}(C))_i$ sont donc les mêmes que
  celles qui définissent $\lambda(C)_i$. Puisque par définition, on a
  $\dual{J}(\lambda(C))_i = \lambda(C)_i$, on obtient donc l'égalité
  \[ \lambda(\dual{J}(C))_i = \dual{J}(\lambda(C))_i. \]
  On a également
  \[ \lambda(\dual{J}(C))^\ast_i = \dual{J}(\lambda(C))^\ast_i \]
  puisque ces sous-monoïdes sont engendrés par les mêmes générateurs. Enfin, si
  $i > 0$ et $x$~appartient à $C_i$, la différentielle $d'$ de
  $\lambda(\dual{J}(C))$ envoie $[x]$ sur $[t'(x)] - [s'(x)]$, où $s'$
  et~$t'$ désignent la source et le but dans $\dual{J}(C)$. Si $i$
  n'appartient pas à $J$, on a \hbox{$s'(x) = s(x)$} et $t'(x) = t(x)$, et donc
  $d'([x]) = d([x])$, où $d$ désigne la différentielle de~$\lambda(C)$.  Si
  $i$ appartient à $J$, on a $s'(x) = t(x)$ et $t'(x) = s(x)$, et donc
  \hbox{$d'([x]) = -d([x])$}, ce qui montre que $\lambda(\dual{J}(C))$ et
  $\dual{J}(\lambda(C))$ ont même différentielle. L'égalité des
  augmentations étant évidente, ceci achève la démonstration.
\end{proof}

\begin{paragr}
  Soit $K$ un complexe dirigé augmenté. On appellera \ndef[dimension d'un
  complexe dirigé augmenté]{dimension} de $K$
  le plus petit entier $n \ge 0$ tel que $K_i = 0$ pour tout $i > n$. Si un tel
  $n$ n'existe pas, on dira que $K$ est de \ndef[]{dimension infinie}. Pour $n
  \ge 0$, on notera \nnot{$\nCda{n}$} la sous-catégorie pleine de $\Cda$ formée des
  complexes dirigés augmentés de dimension au plus $n$.

  Le foncteur d'inclusion $\nCda{n} \hookto \Cda$ admet un adjoint à gauche
  et un adjoint à droite qu'on notera respectivement
  \[
    \ti{n} : \Cda \to \nCda{n}
    \quadet
    \tb{n} : \Cda \to \nCda{n}.
  \]
  \notindex{$\ti{n}(K)$, $\tb{n}(K)$}%
  On appellera $\ti{n}$ le foncteur \ndef[$n$-tronqué!intelligent!d'un
  complexe dirigé augmenté]{$n$-tronqué intelligent} et $\tb{n}$ le foncteur
  \ndef[$n$-tronqué!bête!d'un complexe dirigé augmenté]{$n$-tronqué bête}.

  Explicitement, si $K$ est un complexe dirigé augmenté, on a
  \[
    \ti{n}(K)_i =
    \begin{cases}
      K_i & \text{si $i < n$,} \\
      K_n\slash d(K_{n+1}) & \text{si $i = n$,} \\
      0 & \text{si $i > n$,} \\
    \end{cases}
    \quadet
    \ti{n}(K)^\ast_i =
    \begin{cases}
      K^\ast_i & \text{si $i < n$,} \\
      \overline{K^\ast_n} & \text{si $i = n$,} \\
      0 & \text{si $i > n$,} \\
    \end{cases}
  \]
  où $\overline{K^\ast_n}$ désigne l'image de $K^\ast_n$ dans
  $K_n/d(K_{n+1})$,
  et
  \[
    \tb{n}(K)_i = \begin{cases}
      K_i & \text{si $i \le  n$,} \\
      0 & \text{si $i > n$,}
    \end{cases}
    \quadet
    \tb{n}(K)^\ast_i = \begin{cases}
      K^\ast_i & \text{si $i \le  n$,} \\
      0 & \text{si $i > n$,}
    \end{cases}
  \]
  les différentielles et augmentations étant héritées de celles de $K$ de
  la manière évidente.
\end{paragr}

\begin{paragr}
  Fixons un entier $n \ge 0$. Si $C$ est une $n$-catégorie, il est immédiat que
  $\lambda(C)$ est un complexe dirigé augmenté de dimension au plus $n$. De
  même, si $K$ est un complexe dirigé augmenté de dimension~$n$, il est
  immédiat que $\nu(K)$ est une $n$-catégorie. Les foncteurs~$\lambda$ et
  $\nu$ induisent donc des foncteurs
  \[
    \lambda : \nCda{n} \to \nCat{n}
    \quadet
    \nu : \nCat{n} \to \nCda{n}
  \]
  et ceux-ci forment un couple de foncteurs adjoints. Par définition, les carrés
  \[
     \xymatrix{
       \nCat{n} \ar[r] \ar[d]_\lambda & \ooCat \ar[d]^\lambda \\
       \nCda{n} \ar[r] & \Cda \pbox{,}
     }
     \qquad
     \xymatrix{
       \nCda{n} \ar[r] \ar[d]_\nu & \Cda \ar[d]^\nu \\
       \nCat{n} \ar[r] & \ooCat \pbox{,}
     }
  \]
  où les flèches horizontales désignent les foncteurs d'inclusion,
  sont commutatifs.
\end{paragr}

\begin{prop}\label{prop:tr_lambda}
  Pour tout entier $n \ge 0$, les carrés
  \[
    \xymatrix{
      \ooCat \ar[r]^{\tb{n}} \ar[d]_\lambda & \nCat{n} \ar[d]^\lambda \\
      \Cda \ar[r]_{\tb{n}} & \nCda{n} \pbox{,}
    }
    \qquad
    \xymatrix{
      \ooCat \ar[r]^{\ti{n}} \ar[d]_\lambda & \nCat{n} \ar[d]^\lambda \\
      \Cda \ar[r]_{\ti{n}} & \nCda{n} \pbox{,}
    }
    \qquad
    \xymatrix{
      \Cda \ar[r]^{\tb{n}} \ar[d]_\nu & \nCda{n} \ar[d]^\nu \\
      \ooCat \ar[r]_{\tb{n}}  & \nCat{n}
    }
  \]
  sont commutatifs à isomorphisme canonique près.
\end{prop}

\begin{proof}
  L'existence des flèches de comparaison dans ces carrés résulte formellement
  de la commutativité des deux carrés du paragraphe précédent. Les
  commutativités à isomorphisme près du premier et du troisième carrés sont
  évidentes. Enfin, la commutativité à isomorphisme près du deuxième carré
  résulte par adjonction de la commutativité du second carré du paragraphe
  précédent.
\end{proof}

\begin{prop}\label{prop:tr_nu}
  Soit $K$ un complexe dirigé augmenté à base. Pour tout entier $n \ge 0$,
  le morphisme canonique
  \[ \ti{n}\nu(K) \to \nu\ti{n}(K) \]
  est un isomorphisme.
\end{prop}

\begin{proof}
  Explicitons la flèche de comparaison \hbox{$\alpha : \ti{n}\nu(K) \to
  \nu\ti{n}(K)$}.  Par adjonction, puisque $\nu\ti{n}(K)$ est une
  $n$\nbd-catégorie, il s'agit de définir un \oo-fonc\-teur~$\nu(K) \to
  \nu\ti{n}(K)$. On obtient ce \oo-foncteur en appliquant le foncteur $\nu$
  au morphisme canonique $K \to \ti{n}(K)$.  Montrons que \hbox{$\alpha :
  \ti{n}\nu(K) \to \nu\ti{n}(K)$} est un isomorphisme. Pour~$i < n$, il est
  évident que $\alpha_i$ est une bijection. Pour conclure, il suffit donc de
  montrer que $\alpha_n$ est une bijection.  Décrivons explicitement
  \hbox{$\alpha_n : (\ti{n}\nu(K))_n \to (\nu\ti{n}(K))_n$}. Une $n$-flèche
  de $\ti{n}\nu(K)$ est la classe d'équivalence d'un élément
  \[ \tabld{x}{n} \]
  de $\nu(K)_n$ pour la plus petite relation d'équivalence identifiant
  \[
  x = \tabld{x}{n} \quadet y = \tabld{y}{n}
  \]
  s'il existe
  \[
    z =
    \begin{pmatrix}
      z^0_0 &\dots & z^0_n & z^0_{n+1}
      \cr\noalign{\vskip 3pt}
      z^1_0 &\dots &z^1_n & z^1_{n+1}
    \end{pmatrix}
  \]
  dans $\nu(K)_{n+1}$ tel que $x = s(z)$ et $y = t(z)$. Concrètement, cette
  dernière condition signifie qu'on a
  \[
    \tablm{x}{n-1} = \tablm{y}{n-1}
  \]
  et qu'il existe $z_{n+1}$ dans $K^\ast_{n+1}$ tel que $d(z_{n+1}) = y^1_n
  - x^0_n$. On vérifie facilement que la condition définissant la relation
  d'équivalence engendrée s'obtient de la même manière mais en demandant que
  $z_{n+1}$ soit dans le sous-groupe de $K_{n+1}$ engendré
  par~$K^\ast_{n+1}$. (Plus précisément, la clôture symétrique s'obtient en
  demandant que $z_{n+1}$ soit dans~$K^*_{n+1}$ ou dans $-K^*_{n+1}$, et la
  clôture transitive de la clôture symétrique en demandant que $z$ soit dans
  $K^*_{n+1} + (-K^*_{n+1})$ qui n'est autre que le sous-groupe engendré par
  $K^*_{n+1}$.) Or, puisqu'on a supposé le complexe $K$ à base, ce
  sous-groupe est~$K_{n+1}$ tout entier. Par ailleurs, une $n$-flèche de
  $\nu\ti{n}(K)$ est un tableau
  \[
    \begin{pmatrix}
      x^0_0 &\dots & x^0_{n-1} & t^0_n
      \cr\noalign{\vskip 3pt}
      x^1_0 &\dots &x^1_{n-1} & t^1_n
    \end{pmatrix} \pbox{,}
  \]
  où les $x_i^\epsilon$, pour $0 \le i < n$ et $\epsilon = 0,1$, sont dans
  $K^\ast_i$, et $t^0_n = t^1_n$ est dans l'image de~$K^\ast_n$ dans $K_n\slash
  d(K_{n+1})$, satisfaisant aux conditions du
  paragraphe~\ref{paragr:def_nu}. L'application~$\alpha_n$ envoie la
  $n$-flèche donnée par la classe de
  \[ \tabld{x}{n} \]
  sur la $n$-flèche
   \[
    \begin{pmatrix}
      x^0_0 &\dots & x^0_{n-1} & \overline{x^0_n}
      \cr\noalign{\vskip 3pt}
      x^1_0 &\dots &x^1_{n-1} & \overline{x^1_n}
    \end{pmatrix} \pbox{,}
  \]
  où $\overline{x^0_n} = \overline{x^1_n}$ désigne l'image de $x^0_n =
  x^1_n$ dans $K_n\slash d(K_{n+1})$. Notre description de la relation
  d'équivalence définissant les $n$-flèches de $\ti{n}\nu(K)$ montre
  que cette application est une bijection, ce qui achève la démonstration.
\end{proof}

\begin{rem}
  Il n'est pas vrai en général que le morphisme de l'énoncé précédent soit
  un isomorphisme si on ne fait aucune hypothèse sur le complexe dirigé
  augmenté $K$ ; un contre-exemple pour $n = 0$ est donné par le
  complexe de chaînes normalisé de l'ensemble simplicial~$\Deltan{1}$ muni
  des sous-monoïdes de positivité $\N^2 \subset \Z^2$ et $0 \subset \Z$ en
  degré~$0$ et $1$ respectivement.
\end{rem}

\begin{paragr}
  Soient $f, g : K \to L$ deux morphismes de complexes de chaînes. Rappelons
  qu'une \ndef[homotopie!de complexes de chaînes]{homotopie} $h$ de $f$ vers
  $g$ consiste en la donnée de morphismes de groupes abéliens
  \[ h_i : K_i \to L_{i+1}, \quad\text{pour $i \ge 0$,} \]
  satisfaisant à la condition
  \[ d_{i+1}h_i + h_{i-1}d_i = g_i - f_i, \qquad\text{pour $i \ge 0$,} \]
  en convenant que $h_{-1} = 0$ et $d_0 = 0$. Pour $i = 0$, on a donc
  \[ d_1h_0 = g_0 - f_0. \]
\end{paragr}

\begin{paragr}\label{paragr:def_antihomotopie}
  Soient $f, g : K \to L$ deux morphismes de complexes dirigés augmentés.
  Une \ndef[homotopie!de complexes dirigés augmentés]{homotopie} $h$ de $f$
  vers $g$ est une homotopie entre les morphismes de complexes de chaînes
  sous-jacents satisfaisant à la condition supplémentaire suivante :
  \[ h^{}_i(K^\ast_i) \subset L^\ast_{i+1}, \quad\text{pour $i \ge 0$.} \]

  Soient $f, g : K \to L$ deux morphismes de complexes dirigés augmentés. Une
  \ndef{antihomotopie} de $f$ vers $g$ est une homotopie de $f^\co :
  K^\co \to L^\co$ vers $g^\co : K^\co \to L^\co$. Explicitement, une
  antihomotopie $h$ de $f$ vers $g$ consiste en la donnée de morphismes de
  groupes abéliens
  \[ h_i : K_i \to L_{i+1}, \quad\text{pour $i \ge 0$,} \]
  satisfaisant aux conditions suivantes :
  \begin{itemize}
    \item pour tout $i \ge 0$, on a
      \[ d_{i+1}h_i - h_{i-1}d_i = (-1)^i(g_i - f_i), \]
      en convenant toujours que $h_{-1} = 0$ et $d_0 = 0$ (en particulier,
      pour $i = 0$, on a $d_1h_0 = g_0 - f_0$) ;
    \item pour tout $i \ge 0$, on a
      \[ h^{}_i(K^\ast_i) \subset L^\ast_{i+1}. \]
  \end{itemize}
\end{paragr}

\begin{rem}
  On verra dans l'appendice~\ref{sec:trans_oplax} (voir le
  paragraphe~\ref{paragr:homot_abs}) qu'une homotopie entre morphismes de
  complexes dirigés augmentés de $K$ vers $L$ correspond à la donnée d'un
  morphisme de $\lambda(\Dn{1}) \otimes K$ vers~$L$, où $\Dn{1}$ désigne la
  catégorie associée à l'ensemble ordonné $\{0 < 1\}$ et
  $\otimes$ le produit tensoriel des complexes dirigés augmentés (voir le
  paragraphe~\ref{paragr:def_produit_cda}). On en déduira qu'on obtient la
  notion d'antihomotopie en inversant les deux facteurs de $\lambda(\Dn{1})
  \otimes K$, c'est-à-dire en considérant des morphismes de $K \otimes
  \lambda(\Dn{1})$ vers~$L$.
\end{rem}

\begin{paragr}\label{paragr:def_n-homot}
  Plus généralement, on va définir une notion de $n$\nbd-homotopie de
  complexes dirigés augmentés pour tout $n \ge 0$. Soient $K$ et $L$ deux
  complexes dirigés augmentés. On définit simultanément par récurrence sur
  $n \ge 0$ les notions de $n$-homotopie de~$K$ vers $L$ et de
  $n$-homotopies de $K$ vers $L$ parallèles comme suit. Une
  \ndef[]{$0$-homotopie} de~$K$ vers $L$ est un morphisme de complexes dirigés
  augmentés de $K$ vers $L$. Deux telles $0$-homotopies sont par définition
  toujours \ndef[]{parallèles}. Si $n > 0$, une
  \ndef[$n$-homotopie]{$n$\nbd-homotopie}~$H$ de~$K$ vers $L$ est la donnée
  de deux $(n-1)$-homotopies $h$ et $k$ parallèles de $K$ vers $L$, ainsi
  que de morphismes de groupes abéliens
  \[ H_i : K_i \to L_{i+n}, \quad\text{pour $i \ge 0$,} \]
  satisfaisant aux conditions suivantes :
  \begin{itemize}
    \item pour tout $i \ge 0$, on a
    \[ d_{i+n}H_i - (-1)^nH_{i-1}d_i = k_i - h_i, \]
    en convenant que $H_{-1} = 0$ et $d_0 = 0$ (en particulier,
      pour $i = 0$, on a l'égalité $d_nH_0 = k_0 - h_0$) ;
    \item pour tout $i \ge 0$, on a
    \[ H^{}_i(K^\ast_i) \subset L^\ast_{i+n}. \]
  \end{itemize}
  On dit alors également que $H$ est une $n$-homotopie \ndef[]{de $h$ vers
  $k$}. Deux $n$-homotopies sont \ndef[$n$-homotopies
  parallèles]{parallèles} si elles vont toutes deux d'une même
  $(n-1)$-homotopie vers une même $(n-1)$-homotopie.

  Pour $n = 1$, on retrouve la notion d'homotopie de morphismes de complexes
  dirigés augmentés.

  On définit de même une notion de \ndef{$n$-antihomotopie} pour tout $n \ge 0$ en
  remplaçant la relation liant $H$ à $h$ et $k$ dans la définition d'une
  $n$-homotopie par l'égalité
  \[
     d_{i+n}H_i - H_{i-1}d_i = (-1)^i(k_i - h_i),
      \quad\text{pour $i \ge 0$,}
  \]
  en convenant toujours que $H_{-1} = 0$ et $d_0 = 0$ (en particulier, pour
  $i = 0$, on a $d_nH_0 = k_0 - h_0$). On retrouve également, pour
  $n = 1$, la notion d'antihomotopie.
\end{paragr}

\begin{rem}
  Les notions de $n$\nbd-homo\-topie et $n$\nbd-antihomotopie peuvent
  également s'interpréter en termes du produit tensoriel des complexes
  dirigés augmentés (voir la remarque~\ref{rem:n-homot_abs}).
\end{rem}

On verra au paragraphe~\ref{paragr:conj_Cda} que les complexes dirigés
augmentés munis des $n$\nbd-homotopies (resp. des $n$\nbd-antihomotopies),
pour $n \ge 0$, forment une \oo\nbd-sesqui\-catégorie au sens du
paragraphe~\ref{paragr:def_oosesqui} et même une \oo-catégorie de Gray
(resp. une \oo-catégorie de Gray gauche) au sens du
paragraphe~\ref{paragr:Graycat}. Dans la suite de ce chapitre, on décrit
quelques unes des opérations de ces structures.

\medbreak

\emph{Dans la suite de ce chapitre, on fixe $K$ et $L$ deux complexes
dirigés augmentés.}

\begin{paragr}\label{paragr:def_antih_id}
  Soit $h$ une $n$-homotopie (resp. une $n$-antihomotopie) de $K$ vers $L$,
  pour un~$n \ge 0$. On définit une \hbox{$(n+1)$}\nbd-homotopie (resp. une
  $(n+1)$\nbd-anti\-homo\-topie)~\nnot[$\id{f}$, $\id{h}$]{$\id{h}$} de~$h$
  vers $h$ en posant, pour $i \ge 0$,
  \[ (\id{h})_i = 0. \]
  On appellera $\id{h}$ l'\ndef[homotopie!identité]{homotopie identité}
  (resp. l'\ndef[antihomotopie!identité]{antihomotopie identité}) de $h$.

  En particulier, si $f : K \to L$ est un morphisme de complexes dirigés
  augmentés, on dispose d'une homotopie $\id{f}$ de $f$ vers $f$ et d'une
  antihomotopie, également notée~$\id{f}$, de $f$ vers $f$.
\end{paragr}

\begin{paragr}\label{paragr:def_antih_sesqui}
  Soient $h$ et $h'$ deux $(n-1)$-homotopies (resp. deux
  $(n-1)$-antihomotopies) de $K$ vers $L$, pour un $n \ge 1$, et soit $H$
  une $n$-homotopie (resp. une $n$\nbd-antihomotopie) de $h$ vers $h'$. Si
  $g : J \to K$ est un morphisme de complexes dirigés augmentés, on vérifie
  immédiatement, par récurrence, qu'on définit une $n$-homotopie (resp. une
  $n$\nbd-antihomotopie) \nnot[$Hg$, $gH$]{$Hg$} de~$h'g$ vers~$hg$ (qui
  vont de $J$ vers $L$) en posant, pour $i \ge 0$,
  \[ (Hg)_i = H_ig_i. \]

  De même, si on a un morphisme de complexes dirigés augmentés $g : L \to
  M$, on définit une $n$-homotopie (resp. une $n$-antihomotopie) $gH$ de
  $gh$ vers $gh'$ (qui vont de~$K$ vers $M$) en posant,
  pour $i \ge 0$,
  \[ (gH)_i = g_{i+n}H_i. \]
\end{paragr}

\begin{paragr}\label{paragr:def_antih_comp}
  Soient $h_0$, $h_1$ et $h_2$ trois $(n-1)$-homotopies (resp. trois
  $(n-1)$\nbd-anti\-homo\-topies) de $K$ vers $L$, pour un $n \ge 1$,
  soient $H$ une $n$\nbd-homo\-topie (resp.
  une $n$\nbd-antihomotopie) de $h_0$ vers $h_1$ et $H'$ une $n$-homotopie
  (resp. une $n$\nbd-antihomotopie) de $h_1$ vers $h_2$
  \[
    \shorthandoff{;:}
    \xymatrix@C=4pc{
    K
    \ar@/^4.5ex/[r]^(.35){h_0}_{}="0"
    \ar[r]_(.30){h_1}_{}="1"
    \ar@/_4.5ex/[r]_(.35){h_2}_{}="2"
    \ar@2"0";"1"^{\,H^{}}
    \ar@2"1";"2"^{\,H'}
    &
    L \pbox{.}
    }
  \]
  On vérifie immédiatement qu'on définit une
  $n$-homotopie (resp. une $n$-antihomotopie) \nnot{$H' + H$} de $h_0$ vers
  $h_2$ en posant, pour~$i \ge 0$,
  \[ (H' + H)^{}_i = H'_i + H^{}_i. \]
\end{paragr}

\begin{paragr}\label{paragr:def_homot_cod2}
  Soient $f_0$, $f_1$ et $f_2$ trois $(n-2)$-homotopies (resp. trois
  $(n-2)$-anti\-homo\-topies) de $K$ vers $L$,  pour un $n \ge 2$, soient
  $h$ et~$k$ deux $(n-1)$-homotopies (resp. deux $(n-1)$-antihomotopies) de
  $f_0$ vers~$f_1$, et $h'$ et $k'$ deux $(n-1)$-homotopies (resp. deux
  $(n-1)$-antihomotopies) de $f_1$ vers $f_2$, et soient enfin $H$ une
  $n$-homotopie (resp. une $n$-antihomotopie) de $h$ vers $k$ et $H'$ une
  $n$-homotopie (resp. une $n$-antihomotopie) de~$h'$ vers $k'$
  \[
    \shorthandoff{;:}
    \xymatrix@C=6pc{
    K
    \ar@/^5.5ex/[r]^(.35){}_(.35){}="0"_(.65){}="0p"^(.20){f_0}
    \ar[r]_(.30){}_(.35){}="1"_(.65){}="1p"^(.15)*+<-.3em>{\labelstyle f_1}
    \ar@/_5.5ex/[r]_(.35){}_(.35){}="2"_(.65){}="2p"_(.20){f_2}
    \ar@2"0";"1"_{h^{}\,}_{}="h"
    \ar@2"1";"2"_{h'\,}_{}="hp"
    \ar@2"0p";"1p"^{\,k^{}}_{}="k"
    \ar@2"1p";"2p"^{\,k'}_{}="kp"
    \ar@{}"h";"k"_(.10){}="s"
    \ar@{}"h";"k"_(.90){}="t"
    \ar@{}"hp";"kp"_(.10){}="sp"
    \ar@{}"hp";"kp"_(.90){}="tp"
    \ar@3"s";"t"^{H^{}}
    \ar@3"sp";"tp"^{H'}
    &
    L \pbox{.}
    }
  \]
  On vérifie immédiatement qu'on définit une $n$-homotopie (resp. une
  $n$-antihomotopie) \nnot{$H'+H$} de $h' + h$ vers $k' + k$ (voir le
  paragraphe précédent) en posant, pour $i \ge 0$,
  \[ (H' + H)^{}_i = H'_i + H^{}_i. \]

  On prendra garde au fait qu'on représente par « + » deux opérations
  différentes sur les $n$-homotopies et $n$-antihomotopies (voir le
  paragraphe précédent pour la première), ces opérations correspondant
  aux composition en codimension $1$ et $2$ respectivement. Néanmoins, le
  contexte rendra
  toujours clair l'opération désignée par « + ». Par ailleurs, dans ce
  texte, nous n'utiliserons l'opération introduite dans ce paragraphe que
  dans le cas particulier où $H$ ou $H'$ est une identité.
\end{paragr}

\begin{paragr}\label{paragr:def_homot_Gray}
  Soient $f, g : K \to L$ et $f', g' : L \to M$ des morphismes de complexes
  dirigés augmentés et soient $h$ une homotopie (resp. une antihomotopie) de
  $f$ vers $g$ et $h'$ une homotopie (resp. une antihomotopie) de $f'$ vers
  $g'$
  \[
    \shorthandoff{;:}
    \xymatrix@C=3pc{
      K \ar@/^2.3ex/[r]^{f}_{}="1"
        \ar@/_2.3ex/[r]_{g\phantom{'}}_{}="0"
      &
      L \ar@/^2.3ex/[r]^{f'}_{}="2"
        \ar@/_2.3ex/[r]_{g'}_{}="3"
      &
      \zbox{$M$}
      \phantom{L}
      \ar@2"1";"0"^{\,h^{}}
      \ar@2"2";"3"^{\,h'}
      \pbox{.}
    }
  \]
  On définit une $2$-homotopie (resp. une $2$-antihomotopie) \nnot{$h'h$} de
  $h'g + f'h$ vers $g'h + h'f$ (resp. de $g'h + h'f$ vers $h'g + f'h$)
   en posant, pour $i \ge 0$,
  \[ (h'h)_i = h'_{i+1}h^{}_i. \]
  Vérifions-le dans le cas des antihomotopies (c'est celui que nous utiliserons
  dans ce texte). Pour $i \ge 0$, avec les conventions adéquates pour le cas
  $i = 0$, on a
  \[
    \begin{split}
      d^{}_{i+2}h'_{i+1}h^{}_i - h'_{i}h^{}_{i-1}d^{}_i
      & =
      d^{}_{i+2}h'_{i+1}h^{}_i - h'_id^{}_{i+1}h^{}_i
      + h'_id_{i+1}h_i - h'_{i}h^{}_{i-1}d_i \\
      & =
      (d^{}_{i+2}h'_{i+1} - h'_id^{}_{i+1})h^{}_i
      + h'_i(d^{}_{i+1}h^{}_i - h^{}_{i-1}d^{}_i) \\
      & =
      (-1)^{i+1}(g'_{i+1} - f'_{i+1})h^{}_i
      + (-1)^i h'_i(g^{}_i - f^{}_i) \\
      & =
      (-1)^i[(h'_i g^{}_i + f'_{i+1}h^{}_i) - (g'_{i+1}h^{}_i +
      h'_if^{}_i)],
    \end{split}
  \]
  d'où l'assertion.
\end{paragr}

\chapter{Limites inductives de complexes de Steiner}
\label{sec:lim_ind_Steiner}

Le but de ce chapitre est de dégager des résultats de commutation
du foncteur $\nu : \Cda \to \ooCat$ de Steiner à certaines classes de
limites inductives.

\begin{paragr}
  On peut montrer que la catégorie des complexes dirigés augmentés est
  localement présentable. En particulier, elle est cocomplète. Dans la suite
  du texte, nous aurons besoin d'une description explicite de ses limites
  inductives. Soit donc $F : I \to \Cda$ un foncteur de source une petite
  catégorie $I$. Notons $(K(i), K(i)^\ast, e(i))$ le complexe dirigé
  augmenté $F(i)$. On définit un complexe dirigé augmenté $(K, K^\ast, e)$
  de la manière suivante. Le complexe de chaînes~$K$ est la limite inductive
  des complexes de chaînes $K(i)$.  Rappelons que cette limite inductive est
  calculée degré par degré. Pour $n \ge 0$, le sous-monoïde $K^\ast_n$ de
  $K_n$ est le sous-monoïde engendré par les images des $K(i)^\ast_n$ par
  les morphismes $K(i)_n \to K_n$. Enfin, l'augmentation $e$ est donnée par
  la propriété universelle de $K_0 = \limind_{i \in I} K(i)_0$. On vérifie
  facilement que $(K, K^\ast, e)$ est bien la limite inductive de $F$.
\end{paragr}

\begin{paragr}\label{paragr:def_rigide}
  Soient $K$ et $L$ deux complexes dirigés augmentés à base. Un morphisme
  \hbox{$f : K \to L$} sera dit \ndef[morphisme de complexes dirigés
  augmentés!prérigide]{prérigide} s'il envoie tout élément de
  la base de $K$ sur un élément de la base de $L$. (On rappelle que si un
  complexe dirigé augmenté admet une base, cette base est unique.) On dira
  que $f$ est \ndef[morphisme de complexes dirigés augmentés!rigide]{rigide}
  si, de plus, pour tout élément $b$ de la base de~$K$, on a
  $f(\atom{b}^\e_k) = \atom{f(b)}^\e_k$, pour $\e = 0, 1$ et $0
  \le k \le |b|$. Lorsque $K$ est à base unitaire, cette dernière condition
  signifie exactement que, pour tout élément $b$ de la base de~$K$, on a
  $\nu(f)(\atom{b}) = \atom{f(b)}$. On verra plus loin
  (proposition~\ref{prop:mono_rigide}) que les notions de morphismes
  prérigides et rigides coïncident lorsque $f$ est un monomorphisme.
\end{paragr}

\begin{rem}
  Si $f : K \to L$ est un morphisme rigide entre complexes de
  \hbox{Steiner}, de sorte que $\nu(K)$ et $\nu(L)$ sont en vertu du
  théorème~\ref{thm:Steiner_pol} engendrés librement au sens des polygraphes
  par leurs atomes, le \oo-foncteur $\nu(f) : \nu(K) \to \nu(L)$ est rigide
  au sens où il provient d'un morphisme de polygraphes. Autrement dit, ce
  \oo-foncteur respecte les générateurs au sens des polygraphes de $\nu(L)$
  et $\nu(K)$ donnés par les atomes.
\end{rem}

\begin{exem}\label{exem:contre-ex_rig}
  L'exemple suivant montre qu'un morphisme prérigide n'est pas
  nécessairement rigide, même entre complexes de Steiner forts. Considérons
  les $2$\nbd-catégories
  \[
    \shorthandoff{;:}
    C = \raisebox{1.6pc}{\xymatrix{
      \bullet \ar[d]_{h_1} \ar[r]^{h_0} &
      \bullet \ar[d]^{k} \\
      \bullet \ar[r]_{l} &
      \bullet
      \ar@{}[u];[l]_(.30){}="x"
      \ar@{}[u];[l]_(.70){}="y"
      \ar@2"x";"y"_{\alpha}
    }}
    \quadet
    S = \xymatrix@C=3pc@R=3pc{\bullet \ar[r]^{h'} & \bullet
    \ar@/^2.5ex/[r]_{}="0"^{k'}
      \ar@/_2.5ex/[r]_{}="1"_{l'}
      \ar@2"0";"1"_{\alpha'}
    &  \bullet \pbox{,}}
  \]
  et le $2$-foncteur $u : C \to S$ défini par
  \[
    u(h_0) = u(h_1) = h', \quad
    u(k) = k', \quad
    u(l) = l' \quadet
    u(\alpha) = \alpha' \comp_0 h'.
  \]
  On vérifie facilement que le morphisme $f = \lambda(u) : \lambda(C) \to
  \lambda(S)$ est prérigide. En particulier, on a $f([\alpha]) = [\alpha']$.
  Néanmoins, on n'a pas $\nu(f)(\atom{[\alpha]}) =
  \atom{[\alpha']}$. Par exemple, on a $f(\atom{[\alpha]}^0_1) = [k'] + [h']$
  mais $\atom{[\alpha']}^0_1 = [k']$. On verra plus loin que les
  $2$-catégories~$C$ et $S$ sont des \oo-catégories de Steiner fortes (cela
  résulte des propositions~\ref{prop:Theta_Steiner}
  et~\ref{prop:tens_Steiner}, ainsi que de l'isomorphisme $C \simeq
  \nu(\lambda(\Dn{1}) \otimes \lambda(\Dn{1}))$, où $\Dn{1}$ désigne la
  catégorie associée à l'ensemble ordonné $\{0 < 1\}$). En particulier, les
  complexes $\lambda(C)$ et~$\lambda(S)$ sont de Steiner forts et le
  \oo-foncteur $\nu(f)$ s'identifie à $u$.
\end{exem}

\begin{paragr}
  On dira qu'un foncteur $F : I \to \Cda$ de source une petite catégorie $I$
  est un \ndef[système!prérigide]{système prérigide} (resp. un
  \ndef[système!rigide]{système rigide}) s'il
  satisfait aux conditions suivantes :
  \begin{enumerate}
    \item pour tout objet $i$ de $I$, le complexe dirigé augmenté $F(i)$ est
      à base ;
    \item pour tout morphisme $f : i \to i'$ de $I$, le morphisme $F(f) : F(i)
      \to F(i')$ est prérigide (resp. rigide).
  \end{enumerate}
\end{paragr}

\begin{prop}\label{prop:limind_syst_base}
  Si $F : I \to \Cda$ est un système prérigide, alors le complexe dirigé
  augmenté $\limind_{i \in I} F(i)$ est à base. Plus précisément, l'ensemble
  gradué
  \[ \Big(\limind_{i \in I} B(i)_n\Big)_{n \ge 0}, \]
  où $(B(i)_n)_{n \ge 0}$ désigne la base de $F(i)$, fournit une base de ce
  complexe dirigé augmenté. En particulier, les morphismes canoniques
  $F(i_0) \to \limind_{i \in I} F(i)$, pour $i_0$ un objet de~$I$, sont
  prérigides.
\end{prop}

\begin{proof}
  Fixons $n \ge 0$. Par hypothèse, on a
  \[
    F(i)_n \simeq \Z^{(B(i)_n)}
    \quad\text{et}\quad
    F(i)^\ast_n \simeq \N^{(B(i)_n)}.
  \]
  Puisque $F$ est un système prérigide, le foncteur $i \mapsto F(i)_n$ se
  factorise par le foncteur $\Z$-module libre. Or celui-ci commute aux
  limites inductives et on a donc
  \[
    (\limind F)_n \simeq \limind_{i \in I} F(i)_n
    \simeq \limind_{i \in I} \Z^{(B(i)_n)}
    \simeq \Z^{\left(\limind_{i \in I} B(i)_n\right)}.
  \]
  Par ailleurs, $(\limind F)^\ast_n$ est le sous-monoïde engendré par les
  $\N^{(B(i)_n)}$ qui n'est autre que $\N^{\left(\limind_{i \in I}
  B(i)_n\right)}$, d'où le résultat.
\end{proof}

\begin{paragr}\label{paragr:def_syst_Steiner}
  On dira qu'un foncteur $F : I \to \Cda$ de source une petite catégorie $I$
  est un \ndef[système!de Steiner]{système de Steiner} (resp. un
  \ndef[système!de Steiner fort]{système de Steiner fort}) si $F$ satisfait
  aux conditions suivantes :
  \begin{enumerate}
    \item $F$ est un système rigide ;
    \item pour tout objet $i$ de $I$, le complexe $F(i)$ est de Steiner (resp.
      de Steiner fort) ;
    \item le complexe $\limind_{i \in I} F(i)$ est de Steiner (resp. de Steiner
      fort) ;
    \item pour tout objet $i_0$ de $I$, le morphisme canonique $F(i_0) \to
    \limind_{i \in I} F(i)$ est rigide.
  \end{enumerate}
  Si de plus la petite catégorie $I$ est connexe, on parlera de
  \ndef[système!de Steiner!connexe]{système de Steiner connexe} (resp. de
  \ndef[système!de Steiner fort!connexe]{système de Steiner fort connexe}).
\end{paragr}

\begin{thm}\label{thm:nu_syst_Steiner}
  Le foncteur $\nu : \Cda \to \ooCat$ commute aux limites inductives des
  systèmes de Steiner.
\end{thm}

\begin{proof}
  Soit $F : I \to \Cda$ un système de Steiner. Il s'agit de montrer que le
  morphisme canonique
  \[ \limind_{i \in I} \nu(F(i)) \to \nu\big(\limind_{i \in I} F(i)\big) \]
  est un isomorphisme de \oo-catégories. Pour ce faire, nous allons
  appliquer la proposition~\ref{prop:iso_pol}. En vertu du
  théorème~\ref{thm:Steiner_pol}, les \oo-catégories $\nu(F(i))$, pour $i$
  un objet de~$I$, et~$\nu(\limind F)$ sont engendrées librement au sens des
  polygraphes par les bases de $F(i)$ et $\limind F$ respectivement (ou,
  plus précisément, par les atomes associés à ces bases). Par hypothèse,
  pour tout objet $i_0$ de $I$, le morphisme canonique $F(i_0) \to \limind F$
  est rigide. Cela implique la condition \ref{item:iso_pol_a} de la
  proposition~\ref{prop:iso_pol}. De même, le fait que, par hypothèse, pour
  tout morphisme $i \to i'$ de $I$, le morphisme $F(i) \to F(i')$ est rigide
  entraîne la condition~\ref{item:iso_pol_b}. Enfin, la
  condition~\ref{item:iso_pol_c} résulte de la
  proposition~\ref{prop:limind_syst_base}. On peut donc appliquer la
  proposition~\ref{prop:iso_pol}, ce qui achève la
  démonstration.
\end{proof}

Dans la suite de ce chapitre, on va dégager des conditions suffisantes
permettant d'appliquer ce théorème.

\begin{lemme}\label{lemme:rig_comp_pm}
  Soit $f : K \to L$ un monomorphisme prérigide entre complexes dirigés
  augmentés à base. Alors, pour tout élément homogène $x$ de $K$, on a
  \[ f(x_-) = f(x)_- \quadet f(x_+) = f(x)_+. \]
\end{lemme}

\begin{proof}
  Écrivons $x = x_+ - x_-$. Il s'agit de montrer que
  $f(x_+) - f(x_-)$ est la décomposition $f(x) = f(x)_+ - f(x)_-$,
  c'est-à-dire que les coefficients apparaissant dans l'écriture de $f(x_+)$
  et $f(x_-)$ selon la base de $L$ sont positifs, et que les supports de
  $f(x_+)$ et $f(x_-)$ sont disjoints. Le premier point est immédiat et le
  second résulte du fait que $f$ est un monomorphisme prérigide, d'où le
  résultat.
\end{proof}

\begin{prop}\label{prop:mono_rigide}
  Un monomorphisme prérigide entre complexes dirigés augmentés à base est
  rigide.
\end{prop}

\begin{proof}\label{lemme:atom_mon_rig}
  Soit $f : K \to L$ un monomorphisme prérigide entre complexes dirigés
  augmentés à base. Il s'agit de montrer que, pour tout élément $b$ de la
  base de~$K$, on a $f(\atom{b}^\e_k) = \atom{f(b)}^\e_k$,
  pour $\e = 0, 1$ et $0 \le k \le |b|$. Par définition, les $\atom{b}^\e_k$
  sont obtenus en itérant la différentielle et les opérations $x \mapsto
  x_-$ ou $x \mapsto x_+$ selon la valeur de $\e$. Or, le morphisme $f$ est
  compatible aux différentielles par définition et aux opérations $x \mapsto
  x_-$ et $x \mapsto x_+$ en vertu du lemme précédent, d'où le résultat.
\end{proof}

\begin{paragr}
  On dira qu'un foncteur $F : I \to \Ens$ d'une petite catégorie $I$ vers la
  catégorie des ensembles est \ndef[foncteur séparant]{séparant} si, pour
  tout objet $i_0$ de $I$, l'application canonique \hbox{$F(i_0) \to
  \limind_{i \in I} F(i)$} est une injection.

  On dira qu'un foncteur $F : I \to \Cda$ est un
  \ndef[système!séparant]{système séparant} s'il satisfait aux conditions
  suivantes :
  \begin{enumerate}
    \item $F$ est un système prérigide ;
    \item pour tout $n \ge 0$, le foncteur $I \to \Ens$ qui envoie $i$ sur la
      base de $F(i)_n$ est séparant.
  \end{enumerate}
\end{paragr}

\begin{prop}\label{prop:limind_comp_mono}
  Si $F : I \to \Cda$ est un système séparant, alors, pour tout
  objet~$i_0$ de $I$, le morphisme canonique $F(i_0) \to \limind_{i \in I}
  F(i)$ est un monomorphisme rigide.
\end{prop}

\begin{proof}
  Fixons un objet $i_0$ de $I$. On a déjà vu
  (proposition~\ref{prop:limind_syst_base}) que le morphisme $F(i_0) \to
  \limind F$ est prérigide. En vertu de la
  proposition~\ref{prop:mono_rigide}, il s'agit donc de montrer que ce
  morphisme est un monomorphisme. Notons ${(B(i)_n)}_{n \ge 0}$ la base du
  complexe $F(i)$.
  En vertu de la proposition \ref{prop:limind_syst_base},
  \[ \Big(\limind_{i \in I} B(i)_n\Big)_{n \ge 0} \]
  est une base du complexe $\limind F$. Il s'agit donc de montrer que, pour
  tout $n \ge 0$, le morphisme
  \[
    \Z^{(B(i_0)^{}_{n})} \to \Z^{\big(\limind_{i \in I} B(i)^{}_{n}\big)}
  \]
  est un monomorphisme. Mais, puisque le système $F$ est séparant,
  l'application $B(i_0)_{n} \to \limind_{i \in I} B(i)_{n}$ est une
  injection et on obtient le résultat puisque le foncteur $\Z$-module libre
  préserve les monomorphismes.
\end{proof}

\begin{prop}\label{prop:rig_sep}
  Soit $F : I \to \Cda$ un système séparant. Pour tout morphisme
  $f : i \to i'$ de $I$, le morphisme $F(f) : F(i) \to F(i')$ est un
  monomorphisme. En particulier, un tel système est rigide.
\end{prop}

\begin{proof}
  On a un triangle commutatif
  \[
    \xymatrix@C=1.5pc{
      F(i) \ar[rr]^{F(f)} \ar[dr] & & F(i') \ar[dl] \\
      & \limind F & \pbox{,}
    }
  \]
  où, en vertu de la proposition précédente, les deux flèches obliques
  sont des monomorphismes. On en déduit qu'il en est de même de $F(f)$. La
  seconde assertion résulte de la première en vertu de la
  proposition~\ref{prop:mono_rigide}.
\end{proof}

\begin{prop}\label{prop:limind_comp_unitaire}
  Si $F : I \to \Cda$ est un système séparant à valeurs dans les
  complexes dirigés augmentés à base unitaire, alors le complexe $\limind_{i
  \in I} F(i)$ est à base unitaire.
\end{prop}

\begin{proof}
  Notons ${(B(i)_n)}_{n \ge 0}$ la base du complexe $F(i)$. En vertu de la
  proposition~\ref{prop:limind_syst_base}, le complexe $\limind F$ est à
  base et sa base est
  \[
    \Big(\limind_{i \in I} B(i)_n\Big)_{n \ge 0}.
  \]
  Fixons $n \ge 0$. Soit $b$ un élément de $\limind_{i \in I} B(i)_{n}$. Il
  s'agit de montrer qu'on a
  \[
    e(\atom{b}_0^0) = 1
    \quad\text{et}\quad
    e(\atom{b}_0^1) = 1,
  \]
  où $e$ est l'augmentation de $\limind F$. L'élément $b$ provient
  d'un élément $b_0$ de $B(i_0)_n$ pour un certain objet $i_0$ de $I$
  par l'application canonique $B(i_0)_n \to \limind_{i \in I} B(i)_n$. De
  plus, puisque $B(i_0)$ est à base unitaire, on a
  \[
    e(i_0)(\atom{b_0}_0^0) = 1
    \quad\text{et}\quad
    e(i_0)(\atom{b_0}_0^1) = 1,
  \]
  où $e(i_0)$ désigne l'augmentation de $F(i_0)$. En vertu de la
  proposition~\ref{prop:limind_comp_mono}, le morphisme $F(i_0) \to
  \limind_{i \in I} F(i)$ est rigide, ce qui permet de conclure.
\end{proof}

\begin{prop}\label{prop:syst_Steiner_sep}
  Soit $F : I \to \Cda$ un foncteur satisfaisant aux conditions suivantes :
  \begin{enumerate}
    \item\label{item:syst_St_sep_a} $F$ est un système séparant ;
    \item\label{item:syst_St_sep_c} pour tout objet $i$ de $I$, le complexe
    $F(i)$ est de Steiner (resp. de Steiner fort) ;
    \item\label{item:syst_St_sep_d} le complexe $\limind_{i \in I} F(i)$ est
    à base sans boucle (resp. à base fortement sans boucle).
  \end{enumerate}
  Alors $F$ est un système de Steiner (resp. un système de Steiner fort). En
  particulier, le foncteur $\nu : \Cda \to \ooCat$ commute à la limite
  inductive de $F$.
\end{prop}

\begin{proof}
  En vertu de la proposition~\ref{prop:rig_sep}, la
  condition~\ref{item:syst_St_sep_a} entraîne que $F$ est un système rigide.
  Par ailleurs, d'après la proposition~\ref{prop:limind_comp_unitaire}, dont
  les hypothèses découlent des conditions~\ref{item:syst_St_sep_a}
  et~\ref{item:syst_St_sep_c}, le complexe $\limind F$ est à base unitaire.
  La condition~\ref{item:syst_St_sep_d} implique donc que ce complexe est de
  Steiner (resp. de Steiner fort). Enfin, la
  proposition~\ref{prop:limind_comp_mono}, dont l'hypothèse est précisément
  la condition~\ref{item:syst_St_sep_a}, montre que, pour tout objet~$i_0$
  de $I$, le morphisme canonique $F(i_0) \to \limind F$ est rigide, ce qui
  achève de prouver la première assertion. La seconde résulte de la première
  et du théorème~\ref{thm:nu_syst_Steiner}.
\end{proof}

Terminons ce chapitre par l'étude du cas des sommes amalgamées.

\begin{paragr}\label{paragr:def_rigide_ordonnee}
  Soit $f : K \to L$ un morphisme de complexes dirigés augmentés à base. On
  dira que $f$ est une \ndef{inclusion rigide ordonnée} si
  \begin{enumerate}
    \item\label{item:rig_ord_a} $f$ est un monomorphisme rigide ;
    \item\label{item:rig_ord_b} si $x$ et $y$ sont des éléments de la base de
      $K$, alors on a
      \[ x \leN y \quadssi f(x) \leN f(y), \]
      où $\leN$ désigne la relation de préordre du
      paragraphe~\ref{paragr:def_le_N}.
  \end{enumerate}
  Notons que, en vertu du lemme~\ref{lemme:rig_comp_pm}, le sens
  direct dans l'équivalence de la condition~\ref{item:rig_ord_b} est
  automatique si la condition~\ref{item:rig_ord_a} est satisfaite ; le
  contenu de cette condition se trouve donc dans l'implication réciproque.
\end{paragr}

\begin{lemme}\label{lemme:ord_total}
  Soit $K$ un complexe dirigé augmenté à base fortement sans boucle pour
  lequel la relation d'ordre $\leN$ est totale. Alors tout monomorphisme
  rigide de~$K$ vers un complexe dirigé augmenté à base fortement sans
  boucle est une inclusion rigide ordonnée.
\end{lemme}

\begin{proof}
  Cela résulte immédiatement du fait que si $i : E \to F$ est une injection
  croissante d'un ensemble totalement ordonné $E$ vers un ensemble
  ordonné~$F$, alors, pour $x$ et $y$ dans $E$, on a $x \le y$ si et
  seulement si $i(x) \le i(y)$.
\end{proof}

\begin{lemme}\label{lemme:amalg_Steiner}
  Soit
  \[
    \xymatrix@C=1.5pc{
      K & M \ar[l]_i \ar[r]^j & L
    }
  \]
  un diagramme de complexes dirigés augmentés à base fortement sans boucle
  et d'inclusions rigides ordonnées. Alors la somme amalgamée $K \amalg_M L$
  est un complexe à base fortement sans boucle et les morphismes canoniques
  $K \to K \amalg_M L$ et $L \to K \amalg_M L$ sont des inclusions rigides
  ordonnées.
\end{lemme}

\begin{proof}
  La proposition~\ref{prop:limind_syst_base} et la description explicite des
  sommes amalgamées ensemblistes entraînent immédiatement que ce diagramme,
  vu comme un foncteur de \hbox{$I =  \bullet \ot \bullet \to
  \bullet$} vers $\Cda$, est un système séparant. Il résulte donc de
  la proposition~\ref{prop:limind_comp_mono} que les morphismes canoniques
  \[ K \to K \amalg_M L \quadet L \to K \amalg_M L \]
  sont des monomorphismes rigides. On va considérer dans la suite de cette
  démonstration les monomorphismes en jeu comme des inclusions.

  Montrons que le complexe $K \amalg_M L$ est fortement sans boucle. 
  En vertu de la proposition~\ref{prop:limind_syst_base}, la base de ce
  complexe est la somme amalgamée des bases de $K$ et $L$ au-dessous de
  celle de $M$. On définit une relation $\lec$ sur cette base en posant
  \[
    x \lec y \quaddefssi
    {
    \setstretch{0}
    {
    \begin{cases}
      & \text{$x$, $y$ sont dans la base de $K$ et $x \leN^K y$} \\
      \text{ou} \\
      & \text{$x$, $y$ sont dans la base de $L$ et $x \leN^L y$,}
    \end{cases}
    }
    }
  \]
  où on a noté $\leN^K$ et $\leN^L$ la relation d'ordre du
  paragraphe~\ref{paragr:def_le_N} pour $K$ et $L$ respectivement.  Notons
  que si $x$ et $y$ sont à la fois dans la base de $K$ et celle de $L$, cela
  signifie qu'ils sont dans la base de $M$ et le fait que $i$ et $j$ sont
  des inclusions rigides ordonnées montre qu'on a
  \[
    x \lec y \quadssi x \leN^M y.
  \]
  En particulier, si $x$ et $y$ sont dans la base de $K$, on a
  \[
    x \lec y \quadssi x \leN^K y.
  \]
  De même, si $x$ et $y$ sont dans la base de $L$, on a
  \[
    x \lec y \quadssi x \leN^L y.
  \]
  La clôture transitive de la relation $\lec$ n'est autre que la relation de
  préordre~$\leN$ du paragraphe~\ref{paragr:def_le_N} pour $K \amalg_M L$.
  En effet, cela résulte du fait que pour $x$ et $y$ deux éléments de la
  base de $K \amalg_M L$, si $x$ appartient à $\supp(d(y)_-)$ ou si $y$
  appartient à~$\supp(d(x)_+)$, alors $x$ et $y$ sont soit tous les deux
  dans la base de $K$, soit tous les deux dans la base de $L$.
  Pour conclure, il suffit donc de montrer que la relation $\lec$ est sans
  cycle non trivial. Soit donc
  \[
    x_0 \lec x_1 \lec \cdots \lec x_n = x_0
  \]
  un cycle non trivial, c'est-à-dire tel que $n \ge 2$, qu'on va supposer
  minimal. Par symétrie, on peut supposer que $x_0$ est dans la base de $K$.
  La relation $\le^K_\N$ étant sans cycle non trivial, il existe un élément
  du cycle qui n'est pas dans la base de $K$ et est donc dans celle de $L$.
  Soit~$i$ le plus petit indice tel que $x_i$ ne soit pas dans la base
  de~$K$. On a nécessairement $0 < i < n$.  Puisqu'on a $x_{i-1} \lec x_i
  \lec x_{i+1}$ et que $x_i$ n'est pas dans la base de~$K$, les
  éléments~$x_{i-1}$ et~$x_{i+1}$ sont nécessairement dans la base de~$L$.
  Les éléments $x_{i-1} \lec x_i \lec x_{i+1}$ sont donc tous les trois dans
  la base de $L$. Si~$n = 2$, cela montre que $x_0 \lec x_1 \lec x_2$ est un
  cycle non trivial pour la relation $\le^L_\N$, ce qui est exclu par
  hypothèse. Si $n > 2$, par transitivité de la relation~$\le^L_\N$, on
  obtient un cycle non trivial plus court en supprimant $x_i$ de notre cycle
  qui était pourtant supposé minimal, ce qui est absurde.

  Montrons maintenant que $K \to K \amalg_M L$ est une inclusion rigide
  ordonnée (le cas de $L \to K \amalg_M L$ en résultera par symétrie).
  Il s'agit donc de montrer que si on a
  \[ x = x_0 \lec x_1 \lec \cdots \lec x_n = y \]
  avec $x$ et $y$ dans la base de $K$, alors on a $x \leN^K y$, c'est-à-dire
  $x \lec y$. Ceci se démontre par un argument similaire à celui qu'on a
  utilisé pour montrer que~$\lec$ est sans cycle non trivial : on considère
  une chaîne minimale pour laquelle on n'a pas $x \lec y$; elle doit
  nécessairement contenir un élément qui n'est pas dans la base de~$K$ ; on
  en déduit une chaîne plus courte et donc une contradiction.
\end{proof}

\begin{thm}\label{thm:nu_somme_amalg}
  Un diagramme
  \[
    \xymatrix@C=0.5pc@R=1pc{
      K_1 & & K_2 & & K_3 & & \cdots & & K_{l-1} & & K_{l} \\
      & L_1 \ar[ul]^{f_1} \ar[ur]_(.43){g_1} & & L_2 \ar[ul]^{f_2} \ar[ur]_(.43){g_2}
      & & \cdots & & \cdots & & L_{l-1} \ar[ul]^{f^{}_{l-1}}
      \ar[ur]_(.44){g^{}_{l-1}}
    }
  \]
  de complexes de Steiner forts dont les morphismes sont des inclusions
  rigides ordonnées est un système de Steiner fort. En particulier, la
  limite inductive d'un tel diagramme est un complexe de Steiner fort et le
  foncteur $\nu : \Cda \to \ooCat$ commute à cette
  limite inductive.
\end{thm}

\begin{proof}
  En décomposant la limite inductive du diagramme en une somme amalgamée
  itérée, le lemme précédent et le stabilité des inclusions rigides
  ordonnées par composition impliquent que la limite inductive du diagramme
  est un complexe à base fortement sans boucle. La
  proposition~\ref{prop:limind_syst_base} et la description explicite des
  sommes amalgamées itérées ensemblistes entraînent que ce diagramme est un
  système séparant. Les hypothèses de la
  proposition~\ref{prop:syst_Steiner_sep} sont donc satisfaites et on
  conclut en invoquant cette proposition.
\end{proof}

\begin{coro}
Soit
\[
  \xymatrix@C=1.5pc{
    K & M \ar[l] \ar[r] & L
  }
\]
un diagramme de complexes de Steiner forts et de monomorphismes rigides.
On suppose que la relation d'ordre $\leN$ sur la base de $M$ est totale.
Alors $K \amalg_M L$ est un complexe de Steiner fort et le morphisme
canonique $\nu(K) \amalg_{\nu(M)} \nu(L) \to \nu(K \amalg_M L)$ est un
isomorphisme.
\end{coro}

\begin{proof}
  En vertu du lemme~\ref{lemme:ord_total}, les morphismes du diagramme sont
  des inclusions rigides ordonnées et le résultat découle donc du théorème
  précédent.
\end{proof}

\chapter{La catégorie \pdfTheta{} de Joyal}

% TOCHECK
\kern-22pt

Ce chapitre est consacré à des rappels sur la catégorie $\Theta$,
introduite par Joyal dans~\cite{JoyalTheta}.

\begin{paragr}\label{paragr:def_disque}
  Pour $i \ge 0$, on notera \nnot{$\Dn{i}$} la \oo-catégorie qui coreprésente le
  foncteur \nnot[$\Fl_i(C)$]{\hbox{$\Fl_i : \ooCat \to \Ens$}} qui envoie
  une \oo-catégorie $C$ sur l'ensemble $C_i$ de ses $i$-flèches. Cette
  \oo-catégorie est en fait une $i$-catégorie. Elle admet une unique
  $i$-flèche non triviale qu'on appellera sa \ndef[cellule principale de
  $\Dn{i}$]{cellule principale}. Pour tout $k$ tel que $0 \le k < i$, la
  $i$\nbd-catégorie~$\Dn{i}$ admet exactement deux $k$-flèches non triviales
  qui sont la source et le but itérés en dimension $k$ de sa cellule
  principale.  Voici les graphes sous-jacents (sans les identités) de
  $\Dn{i}$ pour $i = 0, 1, 2, 3$:
  \[
    \shorthandoff{;:}
    \Dn{0} = \xymatrix{\bullet}
    \text{,}\quad
    % \quad,\quad
    \Dn{1} = \xymatrix{\bullet \ar[r] & \bullet}
    \text{,}\quad
    % \quad,\quad
    \Dn{2} = \xymatrix@C=3pc@R=3pc{\bullet \ar@/^2.5ex/[r]_{}="0"
      \ar@/_2.5ex/[r]_{}="1"
      \ar@2"0";"1"
    &  \bullet}
    \quadet
    \Dn{3} = \xymatrix@C=3pc@R=3pc{\bullet \ar@/^3ex/[r]_(.47){}="0"^(.53){}="10"
      \ar@/_3ex/[r]_(.47){}="1"^(.53){}="11"
      \ar@<2ex>@2"0";"1"_{}="2" \ar@<-2ex>@2"10";"11"^{}="3"
      \ar@3"3";"2"_{}
    &  \bullet} \text{.}
  \]
  Il sera parfois utile d'étendre la notation $\Dn{i}$ au cas $i = -1$ en
  convenant que $\Dn{-1}$ est la \oo-catégorie vide.

  Pour $i > 0$, on a des \oo-foncteurs $\sigma_i, \tau_i : \Dn{i-1} \to
  \Dn{i}$
  \notindex{$\sigma_i : \Dn{i-1} \to \Dn{i}$, $\sigma_i^j : \Dn{j} \to
  \Dn{i}$}%
  \notindex{$\tau_i : \Dn{i-1} \to \Dn{i}$, $\tau_i^j : \Dn{j} \to
  \Dn{i}$}%
  qui coreprésentent respectivement les transformations naturelles
  source et but $\Fl_i \to \Fl_{i-1}$. Explicitement, \hbox{$\sigma_i :
  \Dn{i-1} \to \Dn{i}$} (resp. $\tau_i : \Dn{i-1} \to \Dn{i}$) envoie la
  cellule principale de $\Dn{i-1}$ sur la source (resp. le but) de la cellule
  principale de $\Dn{i}$.

  Pour $i \ge j \ge 0$, on notera $\sigma_j^i, \tau_j^i : \Dn{j} \to \Dn{i}$
  les \oo-foncteurs définis par
  \[
    \sigma_j^i = \sigma_i \cdots \sigma_{j+2} \sigma_{j+1}
    \quadet
    \tau_j^i = \tau_i \cdots \tau_{j+2} \tau_{j+1}.
  \]
\end{paragr}

\begin{paragr}\label{paragr:sch_comp_glob}
  Soient $l \ge 1$ et $i_1, \dots, i_l$, $j_1, \dots, j_{l-1}$ des entiers
  positifs vérifiant les inégalités
  \[ i_k > j_k < i_{k+1}, \qquad\text{pour $0 < k < l$.} \]
  On associe à ces entiers le diagramme
  \[
    \xymatrix@C=1pc@R=1pc{
      \Dn{i_1} & & \Dn{i_2} & & \Dn{i_3} & & \cdots & & \Dn{i_{l-1}} & &
      \Dn{i_l} \\
      & \Dn{j_1} \ar[ul]^{\sigma_{j_1}^{i_1}} \ar[ur]_{\tau_{j_1}^{i_2}}
      & & \Dn{j_2} \ar[ul]^{\sigma_{j_2}^{i_2}} \ar[ur]_{\tau_{j_2}^{i_3}}
      & & \cdots & & \cdots & &\Dn{j_{l-1}}
      \ar[ul]^{\sigma_{j_{l-1}}^{i_{l-1}}} \ar[ur]_{\tau_{j_{l-1}}^{i_l}}
    }
  \]
  \notindex{$\Dn{i_1} \amalg_{\Dn{j_1}} \dots \amalg_{\Dn{j_{l-1}}}
  \Dn{i_l}$}%
  dans $\ooCat$. On appellera \ndef{somme globulaire} la limite inductive
  d'un tel diagramme et on la notera simplement
  \[ \Dn{i_1} \amalg_{\Dn{j_1}} \dots \amalg_{\Dn{j_{l-1}}} \Dn{i_l} . \]
  Les \oo-catégories obtenues de cette manière seront appelées des
  \ndef[schéma de composition globulaire]{schémas de composition globulaires}.
\end{paragr}

\begin{paragr}\label{paragr:def_kappa_nabla}
  Pour $i \ge 0$, on notera
  \[ \kappa_i : \Dn{i+1} \to \Dn{i} \]
  \notindex{$\kappa_i : \Dn{i+1} \to \Dn{i}$, $\kappa_i^j : \Dn{i} \to
  \Dn{j}$}%
  le \oo-foncteur coreprésentant la transformation naturelle identité $\Fl_i
  \to \Fl_{i+1}$. Concrètement, $\kappa_i$ envoie la cellule principale de
  $\Dn{i+1}$ sur l'identité de la cellule principale de~$\Dn{i}$.
  Pour $i \ge j \ge 0$, on notera $\kappa_i^j : \Dn{i} \to \Dn{j}$ le
  \oo-foncteur défini par
  \[ \kappa_i^j = \kappa_j \cdots \kappa_{i-2} \kappa_{i-1}. \]

  Pour $i > j \ge 0$, on notera
  \[ \nabla^i_j : \Dn{i} \to \Dn{i} \amalgDn{j} \Dn{i} \]
  \notindex{$\nabla^i_j : \Dn{i} \to \Dn{i} \amalgDn{j} \Dn{i}$}%
  le \oo-foncteur qui coreprésente la transformation naturelle $\comp^i_j :
  \Fl_i \times_{\Fl_j} \Fl_i \to \Fl_i$ de $j$-composition des $i$-flèches.
  Concrètement, $\nabla^i_j$ envoie la cellule principale de $\Dn{i}$ sur le
  $j$-composé des deux $i$-flèches de $\Dn{i} \amalgDn{j} \Dn{i}$
  correspondant aux cellules principales des deux copies de $\Dn{i}$.
\end{paragr}

\begin{paragr}\label{paragr:def_Theta}
  La catégorie \nnot[$\Theta$, $\ThetaAug$]{$\Theta$} de Joyal est la
  sous-catégorie pleine de $\ooCat$ formée des schémas de composition
  globulaires. On notera $\ThetaAug$ la sous-catégorie pleine de $\ooCat$
  dont les objets sont les schémas de composition globulaires ainsi que la
  \oo-catégorie vide.
\end{paragr}

\begin{rem}
  La catégorie $\Theta$ a été introduite par Joyal dans \cite{JoyalTheta} avec
  un point de vue différent de celui adopté ici. C'est Berger
  \cite{BergerNerve} et Makkai-Zawadowski \cite{MakZaw} qui ont prouvé
  indépendamment l'équivalence des deux définitions.
\end{rem}

\begin{prop}\label{prop:Theta_dense}
  La catégorie $\Theta$ (et donc la catégorie $\ThetaAug$) est une
  sous-catégorie dense de $\ooCat$. Autrement dit, si $C$ est une
  \oo-catégorie, le morphisme canonique
  \[
    \limind_{(S, S \to C) \in \tr{\Theta}{C}} S \longto C
  \]
  est un isomorphisme de \oo-catégories.
\end{prop}

\begin{proof}
  Voir le théorème 1.12 de \cite{BergerNerve} ou le théorème 4.10 de
  \cite{WeberNerve} appliqué à l'exemple 4.18.
\end{proof}

\begin{prop}\label{prop:dual_Theta}
  Soit $J$ une partie de $\N^\ast$. Pour tout objet $S$ de $\ThetaAug$, la
  \oo-catégorie $\dual{J}(S)$ \noemph{(voir le
  paragraphe~\ref{paragr:dual_ooCat})} est isomorphe à un objet de
  $\ThetaAug$.
\end{prop}

\begin{proof}
  Cela résulte, par exemple, du fait que les \oo-catégories isomorphes à un
  objet de $\Theta$ peuvent être caractérisées de manière intrinsèque à la
  catégorie $\ooCat$ munie du foncteur d'oubli vers les \oo-graphes (voir
  par exemple \cite[exemple 4.18]{WeberNerve}). On donnera une démonstration
  directe dans \cite{AraMaltsiNerfs} basée sur la description de $\Theta$ en
  termes de produit en couronnes de \cite{BergerWreath} (voir également
  \cite{CisMaltsiTheta}).
\end{proof}

\begin{paragr}\label{paragr:pu_Theta}
  La catégorie $\Theta^\op$ est naturellement munie d'un objet \oo-catégorie interne.
  Plus précisément,  les objets
  \[ \Dn{i}, \quad\text{pour $i \ge 0$,} \]
  et les morphismes
  \begin{alignat*}{2}
    \sigma_i & : \Dn{i-1} \to \Dn{i}, && \quad\text{pour $i > 0$,}\\
    \tau_i & : \Dn{i-1} \to \Dn{i}, && \quad\text{pour $i > 0$,}\\
    \kappa_i & : \Dn{i+1} \to \Dn{i}, && \quad\text{pour $i \ge 0$,}\\
    \nabla^i_j & : \Dn{i} \to \Dn{i} \amalg_{\Dn{j}} \Dn{i}, && \quad\text{pour
    $i > j \ge 0$,}
  \end{alignat*}
  de $\Theta$, correspondant respectivement aux cellules, sources, buts,
  unités et compositions, vérifient les axiomes duaux à ceux des
  \oo-catégories. Il en résulte que tout préfaisceau $X : \Theta^\op \to \Ens$
  sur $\Theta$ qui \ndef[envoyer les sommes globulaires sur des produits
  globulaires!pour un préfaisceau sur $\Theta$]{envoie les sommes
  globulaires sur des produits globulaires}, au sens où, pour toute somme
  globulaire $\Dn{i_1} \amalgDn{j_1} \dots \amalgDn{j_{l-1}} \Dn{i_l}$,
  l'application canonique
  \[
    F(\Dn{i_1} \amalgDn{j_1} \dots \amalgDn{j_{l-1}} \Dn{i_l})
    \to
    F(\Dn{i_1}) \times_{F(\Dn{j_1})} \dots
      \times_{F(\Dn{j_{l-1}})} F(\Dn{i_l})
  \]
  est une bijection, définit une \oo-catégorie. On notera
  \nnot[$X(\cocatD)$, $f_{\cocatD}$]{$X(\cocatD)$} cette \oo-catégorie.
  L'ensemble de ses $i$-cellules, pour $i \ge 0$, est $X(\Dn{i})$ et les
  sources, buts, unités et compositions sont donnés par les applications
  $X(\sigma_i)$, $X(\tau_i)$, $X(\kappa_i)$ et $X(\nabla^i_j)$ pour des
  valeurs de $i$ et $j$ convenables. De plus, tout morphisme $f : X \to Y$
  entre des préfaisceaux sur~$\Theta$ envoyant les sommes globulaires sur
  des produits globulaires définit un \oo-foncteur $f_{\cocatD} : X(\cocatD)
  \to Y(\cocatD)$ entre les \oo-catégories associées. Ainsi, en notant
  \nnot{$\faisc{\Theta}$} la sous-catégorie pleine de $\pref{\Theta}$ formée
  des préfaisceaux envoyant les sommes globulaires sur des produits
  globulaires, on obtient un foncteur
  \[ \faisc{\Theta} \to \ooCat. \]
  Nous n'en aurons pas besoin dans la suite de ce texte mais on peut
  vérifier facilement que ce foncteur est une équivalence de catégories,
  un quasi-inverse étant induit par le \ndef{nerf cellulaire}~$\ooCat \to
  \pref{\Theta}$, foncteur « nerf » associé à l'inclusion~$\Theta \hookto \ooCat$.
\end{paragr}

\begin{paragr}\label{paragr:pu_ThetaAug}
  La catégorie $\ThetaAug$ possède une propriété analogue à celle de
  $\Theta$ exposée au paragraphe précédent.
  Soit \hbox{$X_+ : (\ThetaAug)^\op \to \Ens$} un préfaisceau sur $\ThetaAug$ qui
  \ndef[envoyer les sommes globulaires sur des produits globulaires!pour un
  préfaisceau sur $\ThetaAug$]{envoie
  les sommes globulaires sur des produits globulaires},
  c'est-à-dire tel que sa restriction $X$ à~$\Theta$ envoie les sommes
  globulaires sur des produits globulaires au sens du paragraphe précédent.
  En vertu de ce même paragraphe, un tel préfaisceau définit une
  \oo-catégorie~$X(\cocatD)$. De plus, les morphismes $\vide \to \Dn{i}$ de
  $\ThetaAug$, pour $i \ge 0$, permettent de munir cette \oo-catégorie d'un
  \oo-foncteur vers l'ensemble $X_+(\vide)$, vu comme une \oo-catégorie dont
  les $i$-cellules sont triviales pour $i > 0$. On notera
  \nnot[$X_+(\cocatD)$, $X_+(\cocatD)_x$]{$X_+(\cocatD)$} cette
  \oo-catégorie au-dessus de $X_+(\vide)$. Si
  $x$ est un élément de $X_+(\vide)$, on notera $X_+(\cocatD)_x$ la fibre en
  $x$ du \oo-foncteur $X(\cocatD) \to X_+(\vide)$. On a ainsi
  \[ X(\cocatD) = \coprod_{x \in X_+(\vide)} X_+(\cocatD)_x. \]
  Si $f_+ : X_+ \to Y_+$ est un morphisme entre préfaisceaux sur $\ThetaAug$
  envoyant les sommes globulaires sur des produits globulaires, la
  restriction $f$ de $f_+$ à $\Theta$ définit un \oo-foncteur
  $f_{\cocatD} : X(\cocatD) \to Y(\cocatD)$ rendant le carré
  \[
    \xymatrix{
      X(\cocatD) \ar[d] \ar[r]^{f_{\cocatD}} & Y(\cocatD) \ar[d] \\
      X_+(\vide) \ar[r]_{(f_+)^{}_{\vide}} & Y_+(\vide)
    }
  \]
  commutatif. En particulier, si $x$ appartient à $X_+(\vide)$ et $y =
  (f_+)_\vide(x)$, alors on dispose d'un \oo-foncteur $X_+(\cocatD)_x \to
  Y_+(\cocatD)_y$.

  Ainsi, en notant \nnot{$\faiscp{\ThetaAug}$} la sous-catégorie pleine de
  $\pref{\ThetaAug}$ formée des préfaisceaux envoyant les sommes globulaires
  sur des produits globulaires, on obtient un foncteur
  \[ \faiscp{\ThetaAug} \to \ooCat \downarrow \Ens, \]
  où $\ooCat \downarrow \Ens$ désigne la catégorie des \oo-catégories au-dessus
  d'un ensemble. Nous n'en aurons pas besoin dans la suite de ce texte mais
  on peut déduire facilement de l'énoncé analogue pour $\Theta$ que ce
  foncteur est une équivalence de catégories.
\end{paragr}

Nous allons maintenant étudier l'image de $\Theta$ par le
foncteur $\lambda : \ooCat \to \Cda$.

\begin{paragr}\label{paragr:desc_lambda_Dn}
  Fixons $i \ge 0$ et considérons la $i$-catégorie $\Dn{i}$. Notons $x$ sa
  cellule principale. Pour $k$ tel que $0 \le k < i$, les $k$-flèches non
  triviales de $\Dn{i}$ sont exactement $s_k(x)$ et~$t_k(x)$. Ces cellules
  seront également notées \nnot[$x^\e_k$, $x_i$]{$x^0_k$} et $x^1_k$. De
  même, on notera parfois $x^0_i$, $x^1_i$ ou $x_i$ la cellule $x$.

  On vérifie immédiatement que le complexe dirigé augmenté $\lambda(\Dn{i})$
  peut se décrire de la manière suivante. Pour $k \ge 0$, on a
  \[
    \lambda(\Dn{i})_k
    =
    \begin{cases}
      \Z^{(\{x^0_k, x^1_k\})} & \text{si $0 \le k < i$,} \\
      \Z^{(\{x_i\})} & \text{si $k = i$,} \\
      0 & \text{si $k > i$.}
    \end{cases}
  \]
  On identifiera souvent les cellules non triviales de $\Dn{i}$ avec les
  éléments de la base de~$\lambda(\Dn{i})$. De plus, on posera parfois
  $x^0_k = 0 = x^1_k$ pour $k > i$.
  Pour $k$ tel que $0 < k \le i$, la différentielle envoie $x^0_k$ et $x^1_k$
  sur $x^1_{k-1} - x^0_{k-1}$. Les sous-monoïdes $\lambda(\Dn{i})^\ast_k$ sont
  les sous-monoïdes engendrés par les bases canoniques et l'augmentation
  est donnée par la somme des coefficients. Le complexe dirigé
  augmenté~$\lambda(\Dn{i})$ est donc à base.

  Soient $k$ tel que $0 < k \le i$ et $\epsilon = 0,1$. On vérifie
  immédiatement qu'on a
  \[ \atom{x^\epsilon_k}^\eta_l = x^\eta_l \]
  pour $\eta = 0,1$ et $l$ tel que $0 \le l < k$ (voir le paragraphe
  \ref{paragr:def_atome} pour la notation $\atom{z}$). En particulier, on a
  \[
    \atom{x^\epsilon_k}^0_0 = x^0_0
    \quad\text{et}\quad
    \atom{x^\epsilon_k}^1_0 = x^1_0,
  \]
  et la base de $\lambda(\Dn{i})$ est unitaire.

  Par ailleurs, la relation de préordre $\le_\N$ sur $\lambda(\Dn{i})$ est
  l'ordre total
  \[
    x^0_0 \le_\N x^0_1 \le_\N \dots \le_\N x^0_{i-1} \le_\N x \le_\N
    x^1_{i-1} \le_\N \dots \le_\N x^1_1 \le_\N x^1_0,
  \]
  ce qui prouve que la base de $\lambda(\Dn{i})$ est fortement sans boucle. Le
  complexe $\lambda(\Dn{i})$ est donc un complexe de Steiner fort.
\end{paragr}

\begin{paragr}\label{paragr:desc_lambda_cocat}
  Fixons deux entiers positifs $i$ et $j$. On notera $x$ et $y$ les cellules
  principales respectives de $\Dn{j}$ et $\Dn{i}$.

  Supposons $i \ge j$. Le morphisme~$\lambda(\sigma_j^i) :
  \lambda(\Dn{j}) \to \lambda(\Dn{i})$ est donné par
  \[
    \begin{split}
      x^\epsilon_k & \mapsto y^\epsilon_k, \quad\text{pour $0 \le k < j$ et
      $\epsilon = 0, 1$,} \\
      x_j & \mapsto y^0_j.
    \end{split}
  \]
  De même, le morphisme $\lambda(\tau_j^i) : \lambda(\Dn{j}) \to
  \lambda(\Dn{i})$ est donné par
  \[
    \begin{split}
      x^\epsilon_k & \mapsto y^\epsilon_k, \quad\text{pour $0 \le k < j$ et
      $\epsilon = 0, 1$,} \\
      x_j & \mapsto y^1_j.
    \end{split}
  \]

  Supposons maintenant $i \le j$. Le morphisme $\lambda(\kappa_j^i) :
  \lambda(\Dn{j}) \to \lambda(\Dn{i})$ est donné par
  \[
    x^\epsilon_k \mapsto y^\epsilon_k, \quad\text{pour $0 \le k \le j$ et
    $\epsilon = 0, 1$}.
  \]
  Ainsi, $x^0_k$ et $x^1_k$, pour $k = i$, sont envoyés sur $y_i$ et, pour $k >
  i$, sur $0$.

  Soient maintenant $i > j \ge 0$ deux entiers. Considérons le morphisme de
  cocomposition \hbox{$\nabla^i_j : \Dn{i} \to \Dn{i} \amalgDn{j} \Dn{i}$}.
  Notons $x$ la cellule principale de la source $\Dn{i}$ de~$\nabla^i_j$
  et~$y$ et $z$ les cellules principales des deux copies de $\Dn{i}$
  apparaissant de gauche à droite dans $\Dn{i} \amalgDn{j} \Dn{i}$, vues
  comme des cellules de~$\Dn{i} \amalgDn{j} \Dn{i}$. On utilisera les
  notations du paragraphe~\ref{paragr:desc_lambda_Dn} pour la cellule $x$ et
  on les étendra de la manière évidente à $y$ et $z$ : on notera
  $y^\e_k$ et $z^\e_k$, pour $0 \le k \le i$ et $\e = 0, 1$, les sources et
  buts itérés de $y$ et $z$, et on identifiera ces cellules avec l'élément
  de $\lambda(\Dn{i} \amalgDn{j} \Dn{i})$ qu'elles définissent. Puisque
  $s_j(y) = t_j(z)$, on a
  \[
    y^0_j = z^1_j
    \quadet
    y^\e_k = z^\e_k \quad\text{pour $0 \le k < j$ et $\e = 0, 1$.}
  \]
  On vérifie que les $y^\e_i$, $z^\e_i$, pour $0 \le k \le i$ et $\e = 0,
  1$, modulo les relations ci-dessus et les relations triviales $y^0_i =
  y^1_i$ et $z^0_i = z^1_i$, forment une base de $\lambda(\Dn{i} \amalgDn{j}
  \Dn{i})$. Avec ces notations, le morphisme $\lambda(\nabla^i_j)$ est
  donné par
  \begin{alignat*}{3}{}
    x^0_k & \mapsto z^0_k, && \quad\text{pour $0 \le k \le j$,}\\
    x^1_k & \mapsto y^1_k, && \quad\text{pour $0 \le k \le j$,}\\
    x^\epsilon_k & \mapsto y^\epsilon_k + z^\epsilon_k, && \quad\text{pour $j
    < k \le i$ et $\epsilon = 0, 1$.}
  \end{alignat*}
\end{paragr}

\begin{paragr}\label{paragr:syst_Steiner_Theta}
  Soit $S = \Dn{i_1} \amalg_{\Dn{j_1}} \dots \amalg_{\Dn{j_{l-1}}}
  \Dn{i_l}$ un schéma de composition globulaire. Par définition, la
  \oo-catégorie $S$ est limite inductive du foncteur \nnot[$F_S : I_S \to
  \ooCat$]{\hbox{$F_S : I_S \to \ooCat$}} associé au diagramme décrit au
  paragraphe~\ref{paragr:sch_comp_glob}. On a donc un isomorphisme
  \[
    \lambda(S) \simeq \lambda(\Dn{i_1}) \amalg_{\lambda(\Dn{j_1})} \dots
    \amalg_{\lambda(\Dn{j_{l-1}})} \lambda(\Dn{i_l}),
  \]
  où le membre de droite est la limite inductive du foncteur \nnot[$G_S :
  I_S \to \Cda$]{\hbox{$G_S : I_S \xto{F_S} \ooCat \xto{\lambda} \Cda$}}.
  Cet isomorphisme s'étend au cas $S = \vide$ en considérant les uniques
  foncteurs \hbox{$F_\vide : \vide \to \ooCat$} et $G_\vide : \vide \to
  \Cda$.
\end{paragr}

\begin{prop}\label{prop:Theta_Steiner}
  Pour tout objet $S$ de $\ThetaAug$, le foncteur $G_S$ est un système de
  Steiner fort. En particulier, le complexe~$\lambda(S)$ est un
  complexe de Steiner fort. De plus, le morphisme d'adjonction $S \to
  \nu\lambda(S)$ est un isomorphisme.
\end{prop}

\begin{proof}
  Il est immédiat que les morphismes
  \[
    \lambda(\sigma_j^i), \lambda(\tau_j^i) : \lambda(\Dn{j}) \to \lambda(\Dn{i}),
    \qquad\text{pour $i > j \ge 0$,}
  \]
  sont des inclusions rigides ordonnées. Le foncteur $G_S : I_S \to \Cda$
  vérifie donc les hypothèses du théorème~\ref{thm:nu_somme_amalg}, ce qui
  entraîne la première assertion, ainsi que le fait que $\lambda(S)$ est un
  complexe de Steiner fort. Ce même théorème donne de plus un isomorphisme
  canonique
  \[
    \nu\lambda(S) \simeq \nu\lambda(\Dn{i_1}) \amalg_{\nu\lambda(\Dn{j_1})} \dots
    \amalg_{\nu\lambda(\Dn{j_{l-1}})} \nu\lambda(\Dn{i_l}).
  \]
  Pour obtenir la dernière assertion, il suffit donc de montrer que, pour
  tout \hbox{$k \ge 0$}, le morphisme d'adjonction $\Dn{k} \to
  \nu\lambda(\Dn{k})$ est un isomorphisme, ce qui se vérifie facilement par
  un calcul direct.
\end{proof}

\begin{rem}
  Le fait que $\lambda(S)$, pour $S$ un schéma de composition, est un complexe de
  Steiner fort et que le morphisme d'adjonction $S \to \nu\lambda(S)$ est un
  isomorphisme est dû à Steiner \cite{SteinerTheta}. On renvoie à ce même
  texte pour une étude plus complète des liens entre complexes dirigés
  augmentés et schémas de composition globulaires.
\end{rem}

\chapter[Extension de structures de catégorie monoïdale à la Day]{Extension
de structures de\\ catégorie monoïdale à la Day}

Le but de ce chapitre est d'étendre un théorème dû à Day d'extension de
structures de catégorie monoïdale bifermée au cadre plus général des
catégories monoïdales \emph{localement} bifermées, qu'on introduit dans le
chapitre.

\begin{paragr}\label{paragr:def_dense}
  Soit $\D$ une sous-catégorie pleine d'une catégorie $\C$. Rappelons qu'on
  dit que $\D$ est une \ndef{sous-catégorie dense} de $\C$ si tout objet de
  $\C$ est limite inductive canonique d'objets de $\D$. Cela signifie la
  chose suivante. Soit $X$ un objet de $\C$. Notons~\nnot{$\tr{\D}{X}$} la
  catégorie dont les objets sont les couples $(d, f : d \to X)$, où $d$ est
  un objet de~$\D$ et $f$ un morphisme de $\C$, et dont les morphismes entre
  deux tels couples $(d, f)$ et~$(d', f')$ sont les morphismes $g : d \to
  d'$ de $\C$ tels que $f'g = f$. On a un foncteur canonique $\tr{\D}{X} \to
  \C$ qui envoie $(d, d \to X)$ sur $d$. Ce foncteur est muni
  d'une transformation naturelle canonique vers le foncteur constant de
  valeur $X$. On dit que $X$ est \ndef[limite inductive canonique]{limite
  inductive canonique d'objets de $\D$} si cette transformation naturelle
  est un cône inductif universel ou, autrement dit, si la limite inductive
  de ce foncteur existe et le morphisme canonique
  \[ \limind_{(d, d \to X) \in \tr{\D}{X}} d \longto X \]
  est un isomorphisme.

  On vérifie facilement que si une sous-catégorie pleine $\D$ de $\C$
  contient une sous-catégorie pleine dense dans $\C$, alors $\D$ est dense
  dans $\C$.

  Rappelons enfin que si la sous-catégorie pleine $\D$ est une petite
  sous-catégorie $D$, alors $D$ est dense dans $\C$ si et seulement si le
  foncteur de $\C$ vers les préfaisceaux $\pref{D}$ sur $D$ défini par
  \[
  X \longmapsto \big(d \mapsto \Hom_\C(d, X)\big)
  \]
  est pleinement fidèle.
\end{paragr}

\begin{paragr}\label{paragr:biferme}
  Soit $\C$ une catégorie monoïdale de produit tensoriel $\otimes$. On
  rappelle que $\C$ est dite \ndef[catégorie monoïdale!fermée à
  droite]{fermée à droite} si, pour tout
  objet~$Y$ de $\C$, le foncteur $\var \otimes Y : \C \to \C$ admet un
  adjoint à droite. De même, on dit que $\C$ est \ndef[catégorie
  monoïdale!fermée à gauche]{fermée à gauche} si, pour tout objet $X$
  de~$\C$, le foncteur $X \otimes \var : \C \to \C$ admet un adjoint à
  droite. On dit que $\C$ est \ndef[catégorie monoïdale!bifermée]{bifermée}
  si elle est à la fois fermée à gauche et à droite.

  En vertu d'un résultat classique sur les adjonctions paramétrées (voir par
  exemple \cite[chapitre IV, section 7, théorème 3]{MacLane}), si $\C$ est
  fermée à droite, il existe un foncteur $\Homd_\C : \C^\op \times \C \to
  \C$
  \notindex{$\Homd_\C, \Homg_\C : \C^\op \times \C \to \C$}%
  et des bijections
  \[
    \Hom_\C(X \otimes Y, Z) \simeq \Hom_\C(X, \Homd_\C(Y, Z)),
  \]
  naturelles en $X$, $Y$ et $Z$ dans $\C$. De même, si $\C$ est fermée à
  gauche, il existe un foncteur $\Homg_\C : \C^\op \times \C \to \C$ et des
  bijections
  \[
    \Hom_\C(X \otimes Y, Z) \simeq \Hom_\C(Y, \Homg_\C(X, Z)),
  \]
  naturelles en $X$, $Y$ et $Z$ dans $\C$.
\end{paragr}

\begin{thm}[Day]\label{thm:Day}
  Soient $\C$ une catégorie complète et cocomplète, et $\D$ une
  sous-catégorie pleine de $\C$ munie d'une structure de catégorie
  monoïdale. On suppose qu'il existe une petite sous-catégorie pleine de
  $\D$ dense dans $\C$, des foncteurs
  \[ H, H' : \D^\op \times \C \to \C \]
  et des bijections
  \[
    \Hom_\C(S, H(S', Z)) \simeq
      \Hom_\C(S \otimes S', Z) \simeq \Hom_\C(S', H'(S, Z)),
  \]
  naturelles en $S$, $S'$ dans $\D$ et $Z$ dans $\C$.

  Alors il existe une et une seule structure monoïdale sur $\C$ (à unique
  isomorphisme monoïdal près) donnée par un produit tensoriel commutant aux
  petites limites inductives en chaque variable et pour laquelle le foncteur
  d'inclusion $\D \hookto \C$ s'étend en un foncteur monoïdal. De plus,
  cette structure monoïdale est bifermée.
\end{thm}

\begin{proof}
  Soit $D$ une petite sous-catégorie pleine de $\D$ dense dans $\C$.  On va
  commencer par montrer que si $S$ et $S'$ sont deux objets de $\D$, on a un
  isomorphisme canonique
  \[
    S \otimes S' \simeq
    \limind_{\substack{s \,\to S\, \in \tr{D}{S}\\
                       \,s'\! \to S'\! \in \tr{D}{S'}}}
                        s \otimes s',
  \]
  où la limite inductive, comme toutes les limites inductives dans cette
  preuve, est calculée dans~$\C$. On a, pour tout objet $Z$ de $\C$,
  {
    \allowdisplaybreaks
    \begin{align*}
      \Hom_\C(\limind_{s \to S} (s \otimes S'), Z)
      & \simeq \limproj_{s \to S} \Hom_\C(s \otimes S', Z) \\
      & \simeq \limproj_{s \to S} \Hom_\C(s, H(S', Z)) \\
      & \simeq \Hom_\C(\limind_{s \to S} s, H(S', Z)) \\
      & \simeq \Hom_\C(S, H(S', Z)) \\
      & \simeq \Hom_\C(S \otimes S', Z),
    \end{align*}
  }%
  d'où un isomorphisme canonique
  \[
    S \otimes S' \simeq \limind_{s \to S} (s \otimes S').
  \]
  De même, en utilisant $H'$, on obtient un isomorphisme canonique
  \[
    S \otimes S' \simeq \limind_{s' \to S'} (S \otimes s').
  \]
  On a donc
  \[
    S \otimes S' \simeq \limind_{s \to S} (s \otimes S')
    \simeq \limind_{s \to S} \limind_{s' \to S'} s \otimes s',
  \]
  d'où l'isomorphisme annoncé.

  Étendons maintenant le produit tensoriel de $\D$ à $\C$.
  On est contraint de poser, pour~$X$ et $Y$ dans $\C$,
  \[
    X \otimes Y =
    \limind_{\substack{s\, \to X \in \tr{D}{X}\\ s' \to Y \in \tr{D}{Y}}}
    s \otimes s'.
  \]
  Le paragraphe précédent montre qu'on prolonge bien ainsi le produit
  tensoriel de $\D$. De même, on va étendre les foncteurs $H$ et $H'$ en des
  foncteurs~$\C^\op \times \C \to \C$. Pour~$X$ et $Y$ des objets de $\C$,
  on pose
  \[
    H(X, Y) = \limproj_{s \to X \in \tr{D}{X}}H(s, Y)
    \quadet
    H'(X, Y) = \limproj_{s \to X \in \tr{D}{X}}H'(s, Y).
  \]
  On va montrer qu'on a des bijections
  \[
    \Hom_\C(X, H(Y, Z)) \simeq \Hom_\C(X \otimes Y, Z) \simeq \Hom_\C(Y,
    H'(X, Z)),
  \]
  naturelles en $X$, $Y$ et $Z$ dans $\C$.
  Faisons-le pour la bijection de gauche, celle de droite se traitant de
  manière analogue. On a
  {
    \allowdisplaybreaks
    \begin{align*}
      \Hom_\C(X \otimes Y, Z)
      & \simeq \Hom_\C(\limind_{s \to X, s' \to Y} s \otimes s', Z) \\
      & \simeq \limproj_{s \to X, s' \to Y} \Hom_\C(s \otimes s', Z) \\
      & \simeq \limproj_{s \to X, s' \to Y} \Hom_\C(s, H(s', Z)) \\
      & \simeq \Hom_\C(\limind_{s \to X} s, \limproj_{s' \to Y} H(s', Z)) \\
      & \simeq \Hom_\C(X, H(Y, Z)),
    \end{align*}
  }%
  ce qu'il fallait démontrer. Ceci montre que le foncteur $\otimes :
  \C \times \C \to \C$ commute aux limites inductives en chaque variable.

  Pour conclure, il suffit de vérifier que la structure monoïdale sur $\D$
  induit une structure monoïdale sur la catégorie $\C$ munie du produit
  qu'on vient de définir. Notons $I$ l'unité du produit tensoriel de $\D$.
  On a, pour tout objet $X$ de $\C$,
  \[
    I \otimes X \simeq I \otimes \limind_{s \to X} s
      \simeq \limind_{s \to X} (I \otimes s) \simeq \limind_{s \to X} s \simeq X
  \]
  et, de même, $X \otimes I \simeq X$. Enfin, en utilisant le fait que le
  produit sur $\C$ commute aux limites inductives en chaque variable, on
  obtient, pour tous objets $X$, $Y$ et~$Z$ de~$\C$, les isomorphismes
  \[
    \begin{split}
      (X \otimes Y) \otimes Z
      &\simeq \limind_{s \to X, s' \to Y, s'' \to Z} (s \otimes s') \otimes s''
      \\
      &\simeq \limind_{s \to X, s' \to Y, s'' \to Z} s \otimes (s' \otimes s'')
      \simeq X \otimes (Y \otimes Z),
    \end{split}
  \]
  ce qui achève la démonstration.
\end{proof}

\begin{rem}
  Le résultat ci-dessus, purement formel, est attribué à Day par Street
  dans~\cite{StreetDescent}. Les références données par celui-ci sont
  \cite{DayFunct} et \cite{DayRefl}.
\end{rem}

Le but de la suite du chapitre est de généraliser le théorème précédent
en un énoncé qui puisse s'appliquer à des produits tensoriels dont on
suppose uniquement qu'ils commutent aux limites inductives \emph{connexes}.

\begin{paragr}
  Soit $\C$ une catégorie monoïdale. Une \ndef[coaugmentation!à
  gauche]{coaugmentation à gauche} de $\C$ est
  la donnée d'un endofoncteur $M$ de $\C$ et de morphismes
  \[
    M(X) \xto{\iota_1} X \otimes Y,  \phantom{\xot{\iota_2} N(Y)}
  \]
  \notindex{$\iota_1 : M(X) \to X \otimes Y$, $\iota_2 : N(Y) \to X \otimes
  Y$}%
  naturels en $X$ et $Y$ dans $\C$. Plus formellement, on demande que
  $\iota_1$ soit une transformation naturelle
  \[
    Mp \xto{\iota_1} \otimes \phantom{\xot{\iota_2} Nq}
  \]
  de foncteurs de $\C \times \C$ vers $\C$, où $p : \C \times \C \to \C$
  désigne la première projection. De même, une \ndef[coaugmentation!à
  droite]{coaugmentation à droite} de $\C$ est la donnée d'un endofoncteur
  $N$ de $\C$ et de morphismes
  \[
    \phantom{M(X) \xto{\iota_1}} X \otimes Y \xot{\iota_2} N(Y),
  \]
  naturels en $X$ et $Y$ dans $\C$. Plus formellement, on demande que
  $\iota_2$ soit une transformation naturelle
  \[
    \phantom{Mp \xto{\iota_1}} \otimes \xot{\iota_2} Nq
  \]
  de foncteurs de $\C \times \C$ vers $\C$, où $q : \C \times \C \to \C$
  désigne la seconde projection. Enfin, on appellera \ndef{bicoaugmentation}
  de $\C$ la donnée d'une coaugmentation à gauche et d'une coaugmentation à
  droite
  \[
    M(X) \xto{\iota_1} X \otimes Y \xot{\iota_2} N(Y).
  \]

  Par abus de langage, on dira parfois que $M$ est une coaugmentation à
  gauche, sous-entendant ainsi la transformation naturelle $\iota_1$. De
  même pour les coaugmentations à droite et les bicoaugmentations.
\end{paragr}

\begin{exem}\label{exem:coaug}
  Soit $\C$ une catégorie monoïdale possédant un objet initial $\vide$. Il
  existe deux bicoaugmentations naturelles sur $\C$ :
  \begin{enumerate}
    \item la bicoaugmentation \ndef[bicoaugmentation!triviale]{triviale} :
    $M$ et $N$ sont les foncteurs constants de valeur $\vide$ et, pour tous
    objets $X$ et $Y$ de $\C$, les morphismes
      \[
        \vide \xto{\iota_1} X \otimes Y \xot{\iota_2} \vide
      \]
      sont l'unique morphisme $\vide \to X \otimes Y$ ;
    \item la bicoaugmentation
    \ndef[bicoaugmentation!pseudo-locale]{pseudo-locale} : $M$ et $N$ sont
    les foncteurs $\bullet \otimes \vide$ et $\vide \otimes \bullet$
    respectivement et, pour tous objets $X$ et $Y$ de $\C$, les morphismes
      \[
        X \otimes \vide \xto{\iota_1} X \otimes Y \xot{\iota_2} \vide
        \otimes Y
      \]
      sont induits par fonctorialité par les morphismes $\vide \to Y$ et
      $\vide \to X$.
  \end{enumerate}
  Notons que si la catégorie monoïdale $\C$ est bifermée, ces deux
  bicoaugmentations sont isomorphes en un sens évident.

  Par ailleurs, si l'unité de la catégorie monoïdale $\C$ est l'objet
  initial $\vide$, alors la bicoaugmentation pseudo-locale induit une
  bicoaugmentation isomorphe :
  \begin{enumerate}[resume]
    \item la bicoaugmentation \ndef[bicoaugmentation!locale]{locale} : $M$
    et $N$ sont l'endofoncteur identité de $\C$ et, pour tous objets $X$ et
    $Y$ de $\C$, les morphismes
      \[
        X \xto{\iota_1} X \otimes Y \xot{\iota_2} Y
      \]
      sont induits, à partir de la bicoaugmentation pseudo-locale, par les
      contraintes d'unité $X \simeq X \otimes \vide$ et $\vide \otimes Y
      \simeq Y$.
  \end{enumerate}
\end{exem}

\begin{paragr}\label{paragr:def_loc_ferm}
  %TOCHECK
  \abovedisplayskip3.7pt
  \belowdisplayskip3.7pt
  Soit $\C$ une catégorie monoïdale munie d'une coaugmentation à droite $N$.
  On dira que $\C$ est \ndef[catégorie monoïdale!localement fermée à
  droite]{localement fermée à droite} si, pour tout objet
  $Y$ de $\C$, le foncteur
  \[
    \begin{split}
      \C & \to \cotr{\C}{N(Y)} \\
      X & \mapsto (X \otimes Y, \iota_2 : N(Y) \to X \otimes Y)
    \end{split}
  \]
  admet un adjoint à droite. De même, si $\C$ est munie d'une coaugmentation
  à gauche~$M$, on dira que $\C$ est \ndef[catégorie monoïdale!localement
  fermée à gauche]{localement fermée à gauche} si,
  pour tout objet $X$ de $\C$, le foncteur
  \[
    \begin{split}
      \C & \to \cotr{\C}{M(X)} \\
      Y & \mapsto (X \otimes Y, \iota_1 : M(X) \to X \otimes Y)
    \end{split}
  \]
  admet un adjoint à droite. Enfin, si $\C$ est munie d'une
  bicoaugmentation, on dira que~$\C$ est \ndef[catégorie
  monoïdale!localement bifermée]{localement bifermée} si elle
  est localement fermée à gauche et à droite.

  Si $\C$ est une catégorie monoïdale d'unité un objet initial, on
  appliquera parfois le vocabulaire précédent sans préciser les
  coaugmentations auxquelles on fait référence. On utilisera alors
  implicitement la bicoaugmentation locale de l'exemple~\ref{exem:coaug}.

  Si une catégorie monoïdale coaugmentée à droite $\C$ est localement fermée
  à droite, alors, pour tout objet $Y$ de $\C$, le foncteur $\var \otimes Y
  : \C \to \C$ commute aux limites inductives connexes. En effet, celui-ci
  se décompose en
  \[
    \C \to \cotr{\C}{N(Y)} \xto{U} \C,
  \]
  où le premier foncteur est le foncteur canonique déjà considéré et $U$ est
  le foncteur d'oubli. Or, puisque $\C$ est localement fermée à droite, ce
  premier foncteur admet un adjoint à droite. On obtient l'assertion puisque
  le second commute toujours aux limites inductives connexes (voir par
  exemple \cite[proposition 3.3.8]{Riehl}). De même, si
  $\C$ est localement fermée à gauche, alors, pour tout objet~$X$ de~$\C$,
  le foncteur $X \otimes \var : \C \to \C$ commute aux limites inductives
  connexes. Ainsi, si $\C$ est localement bifermée, son produit tensoriel
  commute aux limites inductives connexes en chaque variable.
\end{paragr}

\begin{rem}
  Notons que si $\C$ est une catégorie monoïdale avec un objet
  initial~$\vide$ munie de la bicoaugmentation triviale (voir
  l'exemple~\ref{exem:coaug}), alors $\C$ est bifermée si et seulement si
  elle est localement bifermée.
\end{rem}

\begin{rem}\label{rem:loc_ferm_th_adj}
  Soit $\C$ une catégorie monoïdale d'unité un objet initial $\vide$. On
  munit $\C$ de la bicoaugmentation locale. On a vu que si $\C$ est
  localement bifermée, alors le produit tensoriel de $\C$ commute aux
  limites inductives connexes en chaque variable. La réciproque est
  également vraie si on suppose que la catégorie $\C$ est localement
  présentable. En effet, on vérifie tout d'abord que, pour tout objet $Y$ de
  $\C$, le foncteur $\C \to \cotr{\C}{Y}$ du
  paragraphe~\ref{paragr:def_loc_ferm} commute aux limites inductives. Cela
  résulte du fait qu'il commute aux limites inductives connexes puisque le
  foncteur d'oubli $U : \cotr{\C}{Y} \to \C$ est conservatif et commute aux
  limites inductives connexes, et du fait qu'il envoie l'objet initial~$\vide$ de
  $\C$ sur $(\vide \otimes Y, \iota_2) \simeq (Y, \id{Y})$ qui est un objet
  initial de~$\cotr{\C}{Y}$. Par ailleurs, puisque la catégorie $\C$ est
  localement présentable, il en est de même de la catégorie~$\cotr{\C}{Y}$.
  Or, il est bien connu qu'un foncteur entre catégories localement
  présentables qui commute aux limites inductives admet un adjoint à droite.
  Ceci montre que $\C$ est localement fermée à droite. On montre de même que
  $\C$ est localement fermée à gauche. Nous n'utiliserons pas ce résultat
  dans la suite de ce texte.
\end{rem}

\begin{paragr}
  Soit $F : \C \to \D$ un foncteur. On notera \nnot{$\FlTord(F)$} la
  catégorie dont les objets sont les triplets $(X, f : F(X) \to Y, Y)$, où
  $X$ est un objet de $\C$, $Y$ un objet de $\D$ et
  \hbox{$f : F(X) \to Y$} un morphisme de $\D$, et dont les morphismes de
  $(X, f, Y)$ vers~$(X', f', Y')$ sont les couples $(g, h)$, où $g : X' \to
  X$ est un morphisme de $\C$ et $h : Y \to Y'$ un morphisme de $\D$ tels
  que le carré
  \[
    \xymatrix{
      F(X) \ar[d]_f & F(X') \ar[l]_{F(g)} \ar[d]^{f'} \\
      Y  \ar[r]_h & Y'
    }
  \]
  soit commutatif.

  Notons que, dans le cas où $F$ est le foncteur constant de valeur un objet
  initial de~$\D$, la catégorie $\FlTord(F)$ est canoniquement isomorphe à
  la catégorie $\C^\op \times \D$.
\end{paragr}

\begin{paragr}\label{paragr:def_FlTordD}
  Soit $\C$ une catégorie monoïdale munie d'une coaugmentation à droite $N$.
  On notera \nnot[$\FlTordD(\C)$, $\FlTordG(\C)$]{$\FlTordD(\C)$} la
  catégorie $\FlTord(N : \C \to \C)$. Une variante du théorème sur les
  adjonctions paramétrées déjà cité montre que, si $\C$ est localement
  fermée à droite, alors il existe un foncteur
  \[ H : \FlTordD(\C) \to \C \]
  et des bijections
  \[
    \Hom_{\cotr{\C}{N(Y)}}((X \otimes Y, \iota_2), (Z, g))
      \simeq \Hom_\C(X, H(Y, g, Z)),
  \]
  naturelles en $X$ dans $\C$ et $(Y, g, Z)$ dans $\FlTordD(\C)$.

  De même, si $\C$ est munie d'une coaugmentation à gauche $M$, on notera
  $\FlTordG(\C)$ la catégorie $\FlTord(M : \C \to \C)$ et on montre que, si
  $\C$ est localement fermée à gauche, alors il existe un foncteur
  \[ H' : \FlTordG(\C) \to \C \]
  et des bijections
  \[
    \Hom_{\cotr{\C}{M(X)}}((X \otimes Y, \iota_1), (Z, f))
      \simeq \Hom_\C(Y, H'(X, f, Z)),
  \]
  naturelles en $Y$ dans $\C$ et $(X, f, Z)$ dans $\FlTordG(\C)$.
\end{paragr}

\begin{lemme}\label{lemme:limind_tranche}
  Soient $F, G : I \to \C$ deux foncteurs de source une petite catégorie~$I$ et de
  but une catégorie cocomplète~$\C$, et $\alpha : F \to G$ une transformation
  naturelle. Alors, pour tout objet $Z$ de $\C$ et tout cône inductif $(f_i
  : F(i) \to Z)_{i \in I}$ de $F$ vers $Z$, on a une bijection canonique
  \[
    \Hom_{\cotr{\C}{\limind F}}((\limind G, \limind \alpha),
    (Z, f))
    \simeq
    \limproj_{i \in I} \Hom_{\cotr{\C}{F(i)}}((G(i), \alpha_i), (Z, f_i)),
  \]
  où $f : \limind F \to Z$ désigne le morphisme déduit du cône inductif
  $(f_i)_{i \in I}$.
\end{lemme}

\begin{proof}
  Un élément du membre de droite est donné par une famille de morphismes
  $(g_i : G(i) \to Z)_{i \in I}$ satisfaisant $g_i \alpha_i = f_i$ pour $i$
  dans $I$  et qui est de plus dans la limite projective en question,
  c'est-à-dire qui est telle que, pour tout $h : i \to i'$ dans $I$, on ait
  $g_{i'}G(h) = g_i$.  Cette dernière condition signifie exactement que
  $(g_i)_{i \in I}$ est un cône inductif. On peut donc associer à la famille
  $(g_i)_{i \in I}$ un foncteur $\limind G \to Z$. Dire que ce
  foncteur est au-dessous de $\limind F$ signifie exactement qu'on a $g_i
  \alpha_i = f_i$ pour tout $i$ dans $I$, d'où l'assertion.
\end{proof}

\begin{thm}\label{thm:Day_loc}
  Soit $\C$ une catégorie complète et cocomplète munie de deux endofoncteurs
  $M, N : \C \to \C$ commutant aux petites limites inductives connexes.
  Soit~$\D$ une sous-catégorie pleine de $\C$, stable par les endofoncteurs
  $M$ et $N$, munie d'une structure de catégorie monoïdale et d'une
  bicoaugmentation dont les endofoncteurs sous-jacents sont les restrictions
  de $M$ et $N$ à $\D$. On suppose que $\D$ contient un objet initial
  de~$\C$, qu'il existe une petite sous-catégorie pleine de $\D$ dense dans
  $\C$, des foncteurs
  \[
    H : \FlTordD(\D, \C) \to \C
    \quad\text{et}\quad
    H' : \FlTordG(\D, \C) \to \C,
  \]
  où on a posé
  \[
    \FlTordD(\D, \C) = \FlTord(N_{\vert \D} : \D \to \C)
    \quadet
    \FlTordG(\D, \C) = \FlTord(M_{\vert \D} : \D \to \C),
  \]
  des bijections
  \[
    \Hom_{\cotr{\C}{N(S')}}((S \otimes S', \iota_2), (Z, g))
    \simeq \Hom_\C(S, H(S', g, Z)),
  \]
  naturelles en $S$ dans $\D$ et $(S', g, Z)$ dans $\FlTordD(\D, \C)$,
  et des bijections
  \[
    \Hom_{\cotr{\C}{M(S)}}((S \otimes S', \iota_1), (Z, f))
    \simeq \Hom_\C(S', H'(S, f, Z)),
  \]
  naturelles en $S'$ dans $\D$ et $(S, f, Z)$ dans $\FlTordG(\D, \C)$.

  Alors il existe une et une seule structure de catégorie monoïdale sur $\C$
  (à unique isomorphisme monoïdal près) donnée par un produit tensoriel
  commutant aux petites limites inductives connexes en chaque variable et
  pour laquelle le foncteur d'inclusion $\D \hookto \C$ s'étend en un
  foncteur monoïdal. De plus, cette structure de catégorie monoïdale
  sur~$\C$ admet une unique bicoaugmentation d'endofoncteurs sous-jacents
  $M$ et~$N$ prolongeant la bicoaugmentation de $\D$, et la structure est
  localement bifermée pour cette bicoaugmentation.
\end{thm}

\begin{proof}
  Soit $D$ une petite sous-catégorie pleine de $\D$ dense dans $\C$
  contenant un objet initial $\vide$ de $\C$. On va commencer par montrer
  que si $S$ et $S'$ sont deux objets de $\D$, on a un
  isomorphisme canonique
  \[
    S \otimes S' \simeq
    \limind_{\substack{s \,\to S\, \in \tr{D}{S}\\
                       \,s'\! \to S'\! \in \tr{D}{S'}}}
                        s \otimes s',
  \]
  où la limite inductive, comme toutes les limites inductives dans cette
  preuve, est calculée dans~$\C$ (ou dans une tranche de $\C$). Notons que
  ces limites inductives sont connexes puisque $D$ admet un objet
  initial. On a, pour tout objet $(Z, g)$ de $\cotr{\C}{N(S')}$,
  {
    \allowdisplaybreaks
    \begin{align*}
      \Hom_{\cotr{\C}{N(S')}}(\limind_{s \to S} (s \otimes S', \iota_2),
        (Z, g))
      & \simeq
      \limproj_{s \to S} \Hom_{\cotr{\C}{N(S')}}
        ((s \otimes S', \iota_2), (Z, g)) \\
      & \simeq
      \limproj_{s \to S} \Hom_{\C}(s, H(S', g, Z)) \\
      & \simeq
      \Hom_{\C}(\limind_{s \to S} s, H(S', g, Z)) \\
      & \simeq
      \Hom_{\C}(S, H(S', g, Z)) \\
      & \simeq
      \Hom_{\cotr{\C}{N(S')}}((S \otimes S', \iota_2), (Z, g)),
    \end{align*}
  }%
  d'où un isomorphisme canonique
  \[
    (S \otimes S', \iota_2) \simeq \limind_{s \to S} (s \otimes S', \iota_2)
  \]
  dans $\cotr{\C}{N(S')}$
  et donc un isomorphisme canonique
  \[
    S \otimes S' \simeq \limind_{s \to S} (s \otimes S')
  \]
  dans $\C$ puisque la limite inductive en jeu est connexe et que le
  foncteur d'oubli
  $\cotr{\C}{N(S')} \to \C$ commute aux limites inductives connexes.
  De même, en utilisant $H'$, on obtient des isomorphismes canoniques
  \[
    (S \otimes S', \iota_1) \simeq \limind_{s' \to S'} (S \otimes s',
    \iota_1)
  \]
  dans $\cotr{C}{N(S)}$ et
  \[
    S \otimes S' \simeq \limind_{s' \to S'} (S \otimes s')
  \]
  dans $\C$. On a donc
  \[
    S \otimes S' \simeq \limind_{s \to S} (s \otimes S')
    \simeq \limind_{s \to S} \limind_{s' \to S'} s \otimes s',
  \]
  d'où l'isomorphisme annoncé.

  Étendons maintenant le produit tensoriel de $\D$ à $\C$.
  Puisque pour tout objet~$Z$ de~$\C$, la limite inductive
  canonique
  \[
    \limind_{s \to Z \in \tr{D}{Z}} s \simeq Z
  \]
  est connexe, on est contraint de poser, pour~$X$ et $Y$ dans $\C$,
  \[
    X \otimes Y =
    \limind_{\substack{s\, \to X \in \tr{D}{X}\\ s' \to Y \in \tr{D}{Y}}}
    s \otimes s'.
  \]
  Le paragraphe précédent montre qu'on prolonge bien ainsi le produit
  tensoriel de~$\D$. Par ailleurs, puisque l'endofoncteur $M$ commute aux
  petites limites inductives connexes,
  la transformation naturelle $\iota_1$ induit des morphismes
  \[
    M(X) \simeq \limind_{s \to X} M(s) \simeq \limind_{s \to X, s' \to Y} M(s)
      \xlongto{\limind \iota_1}
      \limind_{s \to X, s' \to Y} s \otimes s' = X \otimes Y,
  \]
  naturels en $X$ et $Y$ dans $\C$, qu'on notera également $\iota_1$. De même,
  la transformation naturelle $\iota_2$ induit des morphismes
  \[
    X \otimes Y \xot{\iota_2} N(Y),
  \]
  naturels en $X$ et $Y$ dans $\C$. Il résulte également du paragraphe
  précédent que ces morphismes $\iota_1$ et $\iota_2$ prolongent ceux
  définis sur $\D$. Étendons maintenant les foncteurs~$H$ et $H'$ en des
  foncteurs $\FlTordD(\C) \to \C$ et $\FlTordG(\C) \to \C$ respectivement.
  Pour $(X, g : N(X) \to Y, Y)$ un objet de $\FlTordD(\C)$ et $(X, f : M(X)
  \to Z, Z)$ un objet de~$\FlTordG(\C)$, on pose
  \[
      H(X, g, Y) = \limproj_{s \xto{k} X \in \tr{D}{X}} H(s, g \circ N(k), Y)
  \]
  et
  \[
      H'(X, f, Z) = \limproj_{s \xto{h} X \in \tr{D}{X}}H'(s, f \circ M(h), Z).
  \]
  On va montrer qu'on a des bijections
  \[
    \Hom_{\cotr{\C}{N(Y)}}((X \otimes Y, \iota_2), (Z, g))
      \simeq \Hom_\C(X, H(Y, g, Z)),
  \]
  naturelles en $X$ dans $\C$ et $(Y, g, Z)$ dans $\FlTordD(\C)$,
  et des bijections
  \[
    \Hom_{\cotr{\C}{M(X)}}((X \otimes Y, \iota_1), (Z, f))
      \simeq \Hom_\C(Y, H'(X, f, Z)),
  \]
  naturelles en $Y$ dans $\C$ et $(X, f, Z)$ dans $\FlTordG(\C)$.
  Faisons-le pour la première famille de bijections, la seconde se traitant
  de manière analogue. On a
  {
    \allowdisplaybreaks
    \begin{align*}
      \Hom_{\cotr{\C}{N(Y)}}((X \otimes Y, \iota_2), (Z, g))
      & \simeq
      \Hom_{\cotr{\C}{\limind\limits_{s' \to Y} N(s')}}
        \big((\limind_{s \to X, s' \to Y} s \otimes s', \iota_2),
        (Z, g)\big) \\
      & \simeq
      \limproj_{s \to X} \Hom_{\cotr{\C}{\limind\limits_{s' \to Y} N(s')}}
        \big((\limind_{s' \to Y} s \otimes s', \limind_{s' \to Y} \iota_2),
        (Z, g)\big) \\
      & \simeq
      \limproj_{s \to X, s' \xto{k} Y} \Hom_{\cotr{\C}{N(s')}}
        \big((s \otimes s', \iota_2), (Z, g \circ N(k))\big) \\*
      & \phantom{\simeq 1} \text{(en vertu du
      lemme~\ref{lemme:limind_tranche})} \\
      & \simeq
      \limproj_{s \to X, s' \xto{k} Y} \Hom_\C(s, H(s', g\circ N(k), Z)) \\
      & \simeq
      \Hom_\C(\limind_{s \xto{\phantom{k}} X} s, \limproj_{s' \xto{k} Y} H(s',
      g\circ N(k), Z)) \\
      & \simeq
      \Hom_\C(X, H(Y, g, Z)),
    \end{align*}
  }%
  ce qu'on voulait démontrer. Ceci montre que le foncteur $\otimes : \C
  \times \C \to \C$ commute aux limites inductives connexes en chaque
  variable.

  Pour conclure, il suffit de vérifier que la structure monoïdale sur $\D$
  induit une structure monoïdale sur $\C$ munie du produit tensoriel qu'on
  vient de définir. On procède exactement comme dans la preuve du
  théorème~\ref{thm:Day}.
\end{proof}

\begin{coro}\label{coro:Day_loc}
  Soient $\C$ une catégorie complète et cocomplète, et $\D$ une
  sous-catégorie pleine de $\C$ munie d'une structure monoïdale d'unité un
  objet initial de $\C$ et contenant une petite sous-catégorie pleine dense
  dans $\C$. On suppose qu'il existe des foncteurs
  \[
    H : \FlTordD(\D, \C) \to \C
    \quad\text{et}\quad
    H' : \FlTordG(\D, \C) \to \C,
  \]
  où on a posé
  \[
    \FlTordD(\D, \C) = \FlTord(\D \hookto \C) = \FlTordG(\D, \C),
  \]
  des bijections
  \[
    \Hom_{\cotr{\C}{S'}}((S \otimes S', \iota_2), (Z, g))
    \simeq \Hom_\C(S, H(S', g, Z)),
  \]
  naturelles en $S$ dans $\D$ et $(S', g, Z)$ dans $\FlTordD(\D, \C)$,
  et des bijections
  \[
    \Hom_{\cotr{\C}{S}}((S \otimes S', \iota_1), (Z, f))
    \simeq \Hom_\C(S', H'(S, f, Z)),
  \]
  naturelles en $S'$ dans $\D$ et $(S, f, Z)$ dans $\FlTordG(\D, \C)$.

  Alors il existe une et une seule structure de catégorie monoïdale sur
  $\C$ (à unique isomorphisme monoïdal près) donnée par un produit tensoriel
  commutant aux petites limites inductives connexes en chaque variable et
  pour laquelle le foncteur d'inclusion $\D \hookto \C$ s'étend en un
  foncteur monoïdal. De plus, cette structure monoïdale est localement
  bifermée.
\end{coro}

\begin{proof}
   On applique le théorème précédent pour $M$ et $N$ l'endofoncteur
   identité et pour la bicoaugmentation locale de l'exemple~\ref{exem:coaug}.
\end{proof}

\begin{rem}
  On peut également obtenir le théorème~\ref{thm:Day} comme corollaire du
  théorème~\ref{thm:Day_loc}. Si $\D$ contient un objet initial de $\C$, il
  suffit d'appliquer ce théorème à $\D$ munie de la bicoaugmentation
  triviale (voir l'exemple~\ref{exem:coaug}). Si $\D$ ne contient pas
  d'objet initial, il suffit de faire de même en remplaçant $\D$ par la
  catégorie $\D'$ obtenue en ajoutant à $\D$ un objet initial de $\C$. On
  vérifie alors qu'on peut prolonger la structure de catégorie monoïdale et
  les foncteurs $H$ et $H'$ à $\D'$ en posant $\vide \otimes S \simeq \vide
  \simeq S \otimes \vide$ pour tout $S$ dans $\D$ et $H(\vide, Z) = \ast =
  H'(\vide, Z)$, où $\ast$ est un objet final de~$\C$, pour tout $Z$
  dans~$\C$.
\end{rem}

\chapter{Joint et tranches \pdfoo-catégoriques}
\label{sec:joint}

Le but de ce chapitre est de définir le joint et les tranches
\oo-catégoriques. Pour ce faire, on commence par définir un joint pour les
complexes de Steiner forts qu'on étend aux \oo-catégories grâce à un
théorème à la Day établi dans le chapitre précédent.

\begin{paragr}\label{paragr:def_produit_comp}
  Soient $K$ et $L$ deux complexes de chaînes (de groupes abéliens en degrés
  positifs, selon notre convention du paragraphe~\ref{paragr:conv_comp}).
  Rappelons que leur produit tensoriel \nnot{$K \otimes L$} est défini de la
  manière suivante. On pose
  \[
    (K \otimes L)_p =
      \bigoplus_{\substack{i + j = p\\i \ge 0,\, j \ge 0}} K_i \otimes L_j,
      \quad\text{pour $p \ge 0$,}
  \]
  et, pour $x$ dans $K_i$ et $y$ dans $L_j$ avec $i + j > 0$,
  \[
    d(x \otimes y) = d(x) \otimes y + (-1)^{|x|} x \otimes d(y),
  \]
  où $|x|$ désigne le degré de $x$, c'est-à-dire $i$.
  Ce produit tensoriel définit une structure de catégorie monoïdale sur la
  catégorie des complexes de chaînes. L'unité est donnée par le complexe
  concentré en degré $0$ de valeur $\Z$. On dispose d'un morphisme
  \hbox{$\gamma^{}_{K, L} : K \otimes L \to L \otimes K$} envoyant $x \otimes y$,
  pour $x$ un élément homogène de $K$ et $y$ un élément homogène de $L$, sur
  $(-1)^{|x||y|}y \otimes x$ qui définit une symétrie pour cette structure
  de catégorie monoïdale. De plus, cette catégorie monoïdale symétrique est
  fermée.

  Ce produit tensoriel s'étend aux complexes de chaînes augmentés de la
  manière suivante.  Si~$K$ et $L$ sont deux complexes de chaînes augmentés,
  on définit une augmentation sur $K \otimes L$ en posant
  \[ e(x \otimes y) = e(x)e(y), \]
  pour $x$ dans $K_0$ et $y$ dans $L_0$. On obtient ainsi une structure de
  catégorie monoïdale symétrique fermée sur la catégorie des complexes de
  chaînes augmentés. L'unité est l'unité du produit tensoriel des complexes
  de chaînes munie de l'augmentation identité~$\id{\Z} : \Z \to \Z$.
\end{paragr}

\begin{paragr}\label{paragr:def_produit_cda}
  Dans \cite[exemple 3.10]{Steiner}, Steiner étend le produit tensoriel des
  complexes de chaînes augmentés aux complexes dirigés augmentés de la
  manière suivante. Si $K$ et $L$ sont deux complexes dirigés augmentés, le
  complexe de chaînes augmenté sous-jacent à \nnot{$K \otimes L$} est le produit
  tensoriel des complexes de chaînes augmentés sous-jacents à~$K$ et $L$
  et, pour $p \ge 0$, le sous-monoïde de positivité $(K \otimes L)^\ast_p$
  est le sous-monoïde de~$(K \otimes L)_p$ engendré par les éléments de la
  forme $x \otimes y$, avec $x$ dans~$K^\ast_i$, $y$ dans $L^\ast_j$ et~$i +
  j = p$.

  Steiner montre \cite[exemple 3.10 et p.~198]{Steiner} qu'on obtient ainsi
  une structure de catégorie monoïdale bifermée (voir le
  paragraphe~\ref{paragr:biferme}) sur la catégorie des complexes dirigés
  augmentés. L'unité~\nnot{$\Zdec$} est définie de la manière suivante : le
  complexe sous-jacent est le complexe concentré en degré $0$ de
  valeur~$\Z$, l'augmentation est l'identité~$\id{\Z}$ et on a $\Zdec_0^\ast
  = \N$. On vérifie facilement que $\Zdec$ est un objet final de la
  catégorie des complexes dirigés augmentés \emph{décents} (voir le
  paragraphe~\ref{paragr:def_decent}). Notons par ailleurs que $\nu(\Zdec)$
  est la \oo-catégorie finale. On prendra garde que cette structure
  de catégorie monoïdale n'est \emph{pas} symétrique. De fait, la symétrie
  $\gamma_{K, L}$ du paragraphe~\ref{paragr:def_produit_comp} n'induit pas
  un morphisme de complexes dirigés augmentés car elle ne respecte pas
  les sous-monoïdes de positivité à cause des signes apparaissant dans sa
  définition. De plus, on peut exhiber des complexes dirigés augmentés $K$ et
  $L$ tels que $K \otimes L$ et $L \otimes K$ ne soient pas isomorphes (par
  exemple $K = \lambda(\Dn{1})$ et $L = \lambda(\Dn{2})$).
\end{paragr}

\begin{paragr}\label{paragr:susp_comp}
  Soit $K$ un complexe de chaînes augmenté de différentielle $d$ et
  d'augmentation~$e$. On notera \nnot[$\Sigma K$, $\Sigma^{-1}$]{$\Sigma K$}
  le complexe de chaînes de différentielle $d'$ défini par
  \[
    (\Sigma K)_p =
    \begin{cases}
      \Z & \text{si $p = 0$}, \\
      K_{p-1}   & \text{si $p > 0$},
    \end{cases}
  \]
  et
  \[
    d'_p =
    \begin{cases}
      \aug & \text{si $p = 1$,} \\
      d_{p-1} & \text{si $p > 1$}.
    \end{cases}
  \]
  On appellera ce complexe la \ndef[suspension!d'un complexe de chaînes
  augmenté]{suspension} de $K$.

  Il est immédiat que $\Sigma$ définit un isomorphisme de la catégorie des
  complexes de chaînes augmentés vers la sous-catégorie non pleine de la
  catégorie des complexes de chaînes formée des complexes $K$ tels que $K_0
  = \Z$ et des morphismes $f$ tels que~$f_0 = \id{\Z}$. On notera
  $\Sigma^{-1}$ l'inverse du foncteur $\Sigma$.
\end{paragr}

\begin{rem}\label{rem:conv_susp}
  Si on considère un complexe de chaînes augmenté $K$ comme un complexe de
  chaînes muni d'un morphisme vers le complexe concentré en degré~$0$ de
  valeur $\Z$, en suspendant ce morphisme avec la convention de signe
  usuelle, on obtient un complexe de chaînes qui diffère du complexe
  $\Sigma K$ qu'on a décrit ci-dessus : ce complexe s'obtient en multipliant
  par~$-1$ les différentielles $d'_p$, pour $p > 1$, du complexe $\Sigma K$.
  Notre choix de signe est dicté par une compatibilité aux
  orientaux de Street (voir la remarque~\ref{rem:orient_lax}).
\end{rem}

\begin{paragr}\label{paragr:def_joint_aug}
  Soient $K$ et $L$ deux complexes de chaînes augmentés. On définit un
  nouveau complexe de chaînes augmenté \nnot{$K \joint L$}, appelé le
  \ndef[joint!de complexes de chaînes augmentés]{joint} de $K$ et
  $L$, par la formule
  \[
    K \joint L = \Sigma^{-1}(\Sigma K  \otimes \Sigma L ).
  \]
  (Cette formule a un sens puisque $(\Sigma K \otimes \Sigma L)_0 = \Z \otimes
  \Z \simeq \Z$.)

  On obtient ainsi un foncteur
  \[  (K, L) \mapsto K \joint L. \]
  Les contraintes d'associativité et d'unité du produit tensoriel des
  complexes de chaînes induisent des contraintes pour le joint. Le joint
  définit donc une structure de catégorie monoïdale sur la catégorie des
  complexes de chaînes augmentés. L'unité est le complexe nul (muni de
  l'unique augmentation possible).

  Explicitement, pour $p \ge 0$, on a
  \[
    (K \joint L)_p = \bigoplus_{\substack{i + 1 + j = p\\i \ge -1,\, j \ge -1}} K_i \otimes L_j,
  \]
  où on a posé $K_{-1} = \Z$ et $L_{-1} = \Z$. Ainsi, on a
  \[
    (K \joint L)_p \simeq
      K_p \oplus (K_{p-1} \otimes L_0) \oplus \cdots \oplus (K_0 \otimes
      L_{p-1}) \oplus L_p.
  \]
  On notera $\vide$ le
  générateur positif de $K_{-1}$ et~$L_{-1}$. Si $x$ est dans $K_i$ et $y$ est
  dans $L_j$ avec $i + 1 + j = p$, on notera \nnot{$x \joint y$} l'élément de~$(K
  \joint L)_p$ correspondant.

  La différentielle du complexe $K \joint L$ est donnée par
  \[
    d(x \joint y) = dx \joint y + (-1)^{|x| + 1}x \joint dy,
  \]
  où on a posé
  \[
    dz = e(z)\vide
  \]
  pour $z$ de degré $0$ et
  \[
    d(\vide) = 0 \quad\text{et}\quad |\vide| = -1.
  \]
  En particulier, on a
  \[
    d(x \joint \vide) = dx \joint \vide
    \quad\text{et}\quad
    d(\vide \joint y) = \vide \joint dy.
  \]

  Enfin, l'augmentation de $K \joint L$ est donnée par
  \[
    \aug(x \joint \vide) = \aug(x)
    \quad\text{et}\quad
    \aug(\vide \joint y) = \aug(y),
  \]
  pour $x$ dans $K_0$ et $y$ dans $L_0$.
\end{paragr}

\begin{paragr}\label{paragr:susp_Cda}
  Soit $K$ un complexe dirigé augmenté. On notera \nnot[$\Sigma K$,
  $\Sigma^{-1}$]{$\Sigma K$}
  \termindex{suspension!d'un complexe dirigé augmenté}%
  le complexe dirigé augmenté défini par la suspension du complexe de
  chaînes augmenté sous-jacent à $K$ (voir le
  paragraphe~\ref{paragr:susp_comp}), les sous-monoïdes de positivité
  \[
    (\Sigma K)^\ast_p =
    \begin{cases}
      \N & \text{si $p = 0$}, \\
      K^\ast_{p-1}   & \text{si $p > 0$},
    \end{cases}
  \]
  et l'augmentation nulle $e = \Z \xto{0} \Z$.

  On obtient ainsi un foncteur $\Sigma : \Cda \to \Cda$. Il est immédiat que
  ce foncteur définit un isomorphisme de la catégorie des complexes dirigés
  augmentés vers la sous-catégorie non pleine \nnot{$\CdaSigma$} de la catégorie
  des complexes dirigés augmentés formée des complexes $K$ tels que $K_0 =
  \Z$, $K^\ast_0 = \N$ et d'augmentation nulle, et des morphismes~$f$ tels
  que $f_0 = \id{\Z}$. On notera $\Sigma^{-1}$ l'inverse du foncteur
  $\Sigma$.
\end{paragr}

\begin{paragr}\label{paragr:def_joint_cda}
  Soient $K$ et $L$ deux complexes dirigés augmentés. On définit un nouveau
  complexe dirigé augmenté \nnot{$K \joint L$}, appelé le \ndef[joint! de complexes
  dirigés augmentés]{joint} de $K$ et $L$, par la formule
  \[
    K \joint L = \Sigma^{-1}(\Sigma K  \otimes \Sigma L).
  \]
  (Cette formule a un sens puisque $(\Sigma K \otimes \Sigma L)_0 = \Z \otimes \Z
  \simeq \Z$, $(\Sigma K \otimes \Sigma L)^\ast_0 = \N \otimes \N \subset
  \Z$ vaut $\N$ et l'augmentation de ce complexe est nulle car produit de
  deux augmentations nulles.)

  Explicitement, le complexe de chaînes augmenté sous-jacent à $K \joint L$
  est le joint des complexes de chaînes sous-jacents à $K$ et $L$ et, si $p
  \ge 0$, le sous-monoïde de positivité~$(K \joint L)^\ast_p$ est le
  sous-monoïde engendré par les éléments de la forme $x \joint y$, avec $x$
  dans $K^\ast_i$ et $y$ dans~$L^\ast_j$, où $i$ et $j$ satisfont $i \ge
  -1$, $j \ge -1$ et $i + 1 + j = p$, en convenant que $K^\ast_{-1} = \N$
  et~$L^\ast_{-1} = \N$.

  On obtient ainsi un foncteur
  \[
    \begin{split}
      \Cda \times \Cda & \to \,\,\,\Cda \\
      (K,L)\,\,\,\, & \mapsto K \joint L
    \end{split}
  \]
  qui définit une structure de catégorie monoïdale sur la catégorie des
  complexes dirigés augmentés. L'unité est le complexe nul muni de l'unique
  augmentation et des uniques sous-monoïdes de positivité possibles. On
  notera $\vide$ ce complexe dirigé augmenté. Cette notation est justifiée
  par le fait que $\vide$ est un objet initial de la catégorie des complexes
  dirigés augmentés. De plus, la \oo-catégorie $\nu(\vide)$ est la
  \oo-catégorie vide. On prendra garde que cette structure monoïdale n'est
  \emph{pas} symétrique.
\end{paragr}

\begin{rem}
  Cette opération joint est le pendant dans le monde des complexes dirigés
  augmentés du joint des complexes de parité défini par Street dans
  \cite[section~5]{StreetParComp}.
\end{rem}

\begin{prop}\label{prop:joint_Cda_limind}
  Soit $K$ un complexe dirigé augmenté. Alors les endofoncteurs de la
  catégorie des complexes dirigés augmentés
  \[
    L \mapsto K \joint L
    \quad\text{et}\quad
    L \mapsto L \joint K
  \]
  commutent aux limites inductives connexes.
\end{prop}

\begin{proof}
  La formule
  \[ K \joint L = \Sigma^{-1}(\Sigma K  \otimes \Sigma L)\]
  montre que, pour établir le résultat, il suffit de vérifier
  que les foncteurs $\Sigma K \otimes \bullet$,  $\bullet \otimes \Sigma L$
  et le foncteur d'inclusion $\CdaSigma \hookto \CdaSSigma$ (voir le
  paragraphe~\ref{paragr:susp_Cda}) commutent aux limites inductives
  connexes. Les deux premiers foncteurs commutent à toutes les limites
  inductives puisqu'ils admettent des adjoints à droite. En ce qui concerne
  le troisième foncteur, il suffit de vérifier que la limite inductive
  dans $\CdaSSigma$ d'un diagramme connexe à valeurs dans $\CdaSigma$
  est dans $\CdaSigma$ ; cela résulte du fait que les limites
  inductives se calculent degré par degré dans $\CdaSSigma$ et du fait que
  la limite inductive d'un diagramme connexe constant de valeur $\Z$ vaut
  $\Z$.
\end{proof}

\begin{prop}\label{prop:dual_joint_cda}
  Soient $K$ et $L$ deux complexes dirigés augmentés. Les applications
  \[
    \begin{split}
      (K \joint L)_p & \to (L \joint K)_p \\
      x \joint y & \mapsto y \joint x \pbox{,}
    \end{split}
  \]
  pour $p \ge 0$, définissent un isomorphisme
  \[
    (K \joint L)^\opp \simeq L^\opp \joint K^\opp.
  \]
\end{prop}

\begin{proof}
  La compatibilité aux augmentations, aux sous-monoïdes de positivité et la
  bijectivité sont évidentes. Il s'agit donc de vérifier la compatibilité
  aux différentielles. Notons $s : x \joint y \mapsto y \joint x$
  l'application de l'énoncé. Pour tout élément homogène de $K \joint L$ de
  la forme $x \joint y$ et de degré au moins $1$, on a, en notant $d'$ les
  différentielles dual impair,
  \[
    \begin{split}
      sd'(x \joint y) & = (-1)^{|x| + 1 + |y|}sd(x \joint y)\\
      & = (-1)^{|x| + 1 + |y|}s(dx \joint y + (-1)^{|x|+1} x \joint dy)\\
      & = (-1)^{|x| + 1 + |y|}y \joint dx + (-1)^{|y|} dy \joint x
    \end{split}
  \]
  et
  \[
    \begin{split}
      ds(x \joint y) & = d(y \joint x)
      = d'y \joint x + (-1)^{|y|+1}y \joint d'x \\
      & = (-1)^{|y|} dy \joint x + (-1)^{|y| + 1 +|x|}y \joint dx \\
      & = (-1)^{|x| + 1 + |y|}y \joint dx + (-1)^{|y|} dy \joint x,
    \end{split}
  \]
  d'où le résultat.
\end{proof}

\begin{rem}
  Outre la dualité triviale, la dualité considérée dans l'énoncé
  précédent est la seule dualité pour laquelle on ait un isomorphisme de ce
  type. En particulier, on n'a \emph{pas} d'isomorphismes
  \[
  (K \joint L)^\co \simeq L^\co \joint K^\co
  \quad\text{ou}\quad
  (K \joint L)^\op \simeq K^\op \joint L^\op.
  \]

  De fait, la preuve précédente ne s'adapte pas aux dualités paire et
  totale. Dans cette preuve, et avec ses notations, on a utilisé le fait que
  l'égalité $d'z = (-1)^{|z|}dz$ reste valable quand $z$ est de degré $0$
  puisque dans cette formule, par convention, les différentielles en degré~$0$
  sont les augmentations, et que les dualités préservent les augmentations
  (voir le paragraphe~\ref{paragr:def_dual_Cda}). Mais si $d''$ et $d'''$
  désignent les différentielles dual pair et dual total, les égalités $d''z
  = (-1)^{|z|+1}dz$ et $d'''z = -dz$ sont fausses lorsque $z$ est de
  degré~$0$.
\end{rem}

\begin{paragr}
  Soient $K$ et $L$ deux complexes dirigés augmentés. Les morphismes
  canoniques $\vide \to K$ et $\vide \to L$ induisent des morphismes
  \[
    K \simeq K \joint \vide \longto K \joint L \longot \vide \joint L \simeq
    L
  \]
  qu'on notera $\iota_1$ et $\iota_2$. On a donc des morphismes
  \[
    K \xto{\iota_1} K \joint L \xot{\iota_2} L.
  \]
  \notindex{$\iota_1 : K \to K \joint L$, $\iota_2 : L \to K \joint L$}%
  (Ce sont les morphismes de la bicoaugmentation locale associée au joint,
  voir l'exemple~\ref{exem:coaug}.)
  Explicitement, pour tout élément homogène $x$ de $K$ et tout élément
  homogène $y$ de $L$, on a
  \[ \iota_1(x) = x \joint \vide \quadet \iota_2(y) = \vide \joint y. \]
\end{paragr}

\begin{paragr}\label{paragr:base_joint}
  Soient $K$ et $L$ des complexes dirigés augmentés admettant des bases $X$
  et~$Y$ respectivement. Il est immédiat que le complexe dirigé augmenté $K
  \joint L$ admet pour base l'ensemble
  %
  % INDEXCHECK
  \notindex{$X \joint Y$}%
  \[ X \joint Y =
    \{ x \joint \vide \mid x \in X \} \cup \{x \joint y \mid x \in X, y \in Y\}
    \cup \{\vide \joint y \mid y \in Y\}.
  \]
\end{paragr}

\begin{lemme}\label{lemme:tab_joint}
  Soient $K$ et $L$ deux complexes dirigés augmentés.
  Pour tout élément homogène de $K \joint L$ de la forme $x \joint y$,
  tout $r \ge 0$ et $\epsilon = 0, 1$, on~a
  \[
    \atom{x \joint y}^\epsilon_r =
    \sum_{\substack{p + 1 + q = r\\-1 \le p \le |x|,\, -1 \le q \le |y|}}
    \atom{x}^\epsilon_p \joint \atom{y}^{p+1+\epsilon \bmod 2}_q,
  \]
  où on a posé $\atom{z}^0_{-1} = (e(\atom{z}^0_0)\vide)_-$
  et $\atom{z}^1_{-1} = (e(\atom{z}^1_0)\vide)_+$ pour $z$ un élément
  homogène de $K$ ou $L$, et $\atom{\vide}^\epsilon_{-1} = \vide$ pour
  $\epsilon = 0, 1$.
\end{lemme}

\begin{proof}
  \newcommand\esp{\!\!\!\!\!\!\!\!\!\!}
  Observons tout d'abord que si $z$ est un élément homogène de~$K$ ou $L$, en
  ajoutant la convention $d_0 = e$ à nos précédentes conventions, on a
  \hbox{$d(\atom{z}^\e_r) = \atom{z}^1_{r-1} - \atom{z}^0_{r-1}$}
  pour tout $r \ge 0$ et $\e = 0, 1$.  (Rappelons qu'on a convenu que
  $\atom{z}^\e_r = 0$ pour $\e = 0,1$ dès que $r > |z|$.) De plus, cette
  égalité reste valable pour $z = \vide$ si l'on convient que, pour $\e = 0,
  1$, en plus de $\atom{\vide}^\e_{-1} = \vide$, on a $\atom{\vide}^\e_r = 0$
  pour~$r \ge 0$.

  Démontrons maintenant le lemme.  Pour $r > |x \joint y| = |x| + 1 + |y|$,
  les deux membres de l'égalité sont nuls. Supposons donc $r \le |x \joint
  y|$. On va démontrer le résultat par récurrence descendante sur~$r$ de~$|x
  \joint y|$ à $0$.  Pour $r = |x \joint y|$, les deux membres de l'égalité
  valent~$x \joint y$. Supposons le résultat vrai pour $r + 1$ et
  montrons-le pour $r$. Par hypothèse de récurrence, on a, en sous-entendant
  les « ${\bmod}\,2$ »,
  \[
    \begin{split}
      d(\atom{x \joint y}_{r+1}^0)
      & = \esp\sum_{\substack{p' + 1 + q' = r+1\\
                            -1 \le p' \le |x|,\, -1 \le q' \le |y|}}\esp
        d\big(\atom{x}^0_{p'} \joint \atom{y}^{p'+1}_{q'}\big) \\
      & = \esp\sum_{\substack{p' + 1 + q' = r+1\\
                            -1 \le p' \le |x|,\, -1 \le q' \le |y|}}\esp
          \big(
            d(\atom{x}^0_{p'}) \joint \atom{y}^{p'+1}_{q'} +
            (-1)^{p'+1}\atom{x}^0_{p'} \joint d(\atom{y}^{p'+1}_{q'})
          \big) \\
      & = \esp\sum_{\substack{p' + 1 + q' = r+1\\
                            -1 \le p' \le |x|,\, -1 \le q' \le |y|}}\esp
           d(\atom{x}^0_{p'}) \joint \atom{y}^{p'+1}_{q'} +
          \esp\sum_{\substack{p' + 1 + q' = r+1\\
                            -1 \le p' \le |x|,\, -1 \le q' \le |y|}}\esp
           (-1)^{p'+1}\atom{x}^0_{p'} \joint d(\atom{y}^{p'+1}_{q'})
              \\
      & = \esp\sum_{\substack{p + 1 + q = r\\
                            -2 \le p \le |x|-1,\, -1 \le q \le |y|}}\esp
            d(\atom{x}^0_{p + 1}) \joint \atom{y}^{p}_q +
          \esp\sum_{\substack{p + 1 + q = r\\
                            -1 \le p \le |x|,\, -2 \le q \le |y|-1}}\esp
           (-1)^{p+1}\atom{x}^0_p \joint d(\atom{y}^{p+1}_{q+1}) \\*
      & \phantom{=1} \text{\lp en posant $p = p' - 1$ et $q = q'$ dans le
       terme de gauche et $p = p'$ }\\*
      & \phantom{=1 \lp} \text{et $q = q' - 1$ dans celui de
         droite\rp} \\
      & = \esp\sum_{\substack{p + 1 + q = r\\
                            -1 \le p \le |x|,\, -1 \le q \le |y|}}\esp
            d(\atom{x}^0_{p + 1}) \joint \atom{y}^{p}_q +
          \esp\sum_{\substack{p + 1 + q = r\\
                            -1 \le p \le |x|,\, -1 \le q \le |y|}}\esp
           (-1)^{p+1}\atom{x}^0_p \joint d(\atom{y}^{p+1}_{q+1}) \\*
      & \phantom{=1} \text{\lp car on a, d'une part,
          $d(\atom{x}^0_{-1}) = 0$ et $d(\atom{x}^0_{|x|+1}) = 0$ et,
          d'autre} \\*
      & \phantom{=1 \lp} \text{part, $d(\atom{y}^{p+1}_{-1}) = 0$ et
          $d(\atom{y}^{p+1}_{|y|+1}) = 0$)} \\
      & = \esp\sum_{\substack{p + 1 + q = r\\
                            -1 \le p \le |x|,\, -1 \le q \le |y|}}\esp
          \big(
            d(\atom{x}^0_{p + 1}) \joint \atom{y}^{p}_q +
            (-1)^{p+1}\atom{x}^0_p \joint d(\atom{y}^{p+1}_{q+1})
          \big).
    \end{split}
  \]
  Or, si $p$ est pair, la quantité sous le signe somme vaut
  \[
    \begin{split}
     \MoveEqLeft
     d(\atom{x}^0_{p + 1}) \joint \atom{y}^{p}_q +
       (-1)^{p+1}\atom{x}^0_p \joint d(\atom{y}^{p+1}_{q+1}) \\
       & = \big(\atom{x}^1_p \joint \atom{y}^0_q
         - \atom{x}^0_p \joint \atom{y}^0_q\big)
         - \big(\atom{x}^0_p \joint \atom{y}^1_q
         - \atom{x}^0_p \joint \atom{y}^0_q\big) \\
       & = \atom{x}^1_p \joint \atom{y}^0_q
         - \atom{x}^0_p \joint \atom{y}^1_q \\
    \end{split}
  \]
  et, si $p$ est impair,
  \[
    \begin{split}
     \MoveEqLeft
     d(\atom{x}^0_{p + 1}) \joint \atom{y}^{p}_q +
       (-1)^{p+1}\atom{x}^0_p \joint d(\atom{y}^{p+1}_{q+1}) \\
       & = \big(\atom{x}^1_p \joint \atom{y}^1_q
         - \atom{x}^0_p \joint \atom{y}^1_q\big)
         + \big(\atom{x}^0_p \joint \atom{y}^1_q
         - \atom{x}^0_p \joint \atom{y}^0_q\big) \\
       & = \atom{x}^1_p \joint \atom{y}^1_q
         - \atom{x}^0_p \joint \atom{y}^0_q. \\
    \end{split}
  \]
  Ainsi, on a
  \[
    d(\atom{x \joint y}_{r+1}^0) =
      \esp\sum_{\substack{p + 1 + q = r\\
      -1 \le p \le |x|,\, -1 \le q \le |y|}}\esp
          \big(
            \atom{x}^1_p \joint \atom{y}^p_q
            - \atom{x}^0_p \joint \atom{y}^{p+1}_q
          \big)
  \]
  et donc
  \[
    \atom{x \joint y}^0_r =
      \esp\sum_{\substack{p + 1 + q = r\\
      -1 \le p \le |x|,\, -1 \le q \le |y|}}\esp
      \atom{x}^0_p \joint \atom{y}^{p+1}_q
    \quadet
    \atom{x \joint y}^1_p =
      \esp\sum_{\substack{p + 1 + q = r\\
      -1 \le p \le |x|,\, -1 \le q \le |y|}}\esp
      \atom{x}^1_p \joint \atom{y}^{p}_q,
  \]
  ce qu'on voulait démontrer.
\end{proof}

\begin{prop}\label{prop:joint_base_unit}
  Si $K$ et $L$ sont des complexes dirigés augmentés à base unitaire, alors
  il en est de même de $K \joint L$.
\end{prop}

\begin{proof}
  Notons $X$ et $Y$ les bases respectives des complexes $K$ et $L$.
  Soient $x$ un élément de $X$ et $y$ un élément de $Y$. En vertu
  du lemme précédent, on a
  {
    \allowdisplaybreaks
    \begin{align*}
      \atom{x \joint y}^0_0
      & = \atom{x}^0_{-1} \joint \atom{y}^0_0 +
            \atom{x}^0_0 \joint \atom{y}^1_{-1} \\
      & = (e(x^0_0)\vide)_- \joint \atom{y}^0_0 +
            \atom{x}^0_0 \joint (e(y^1_0)\vide)_+ \\
      & = \vide_- \joint \atom{y}^0_0 + \atom{x}^0_0 \joint \vide_+ \\
      & = 0 \joint \atom{y}^0_0 + \atom{x}^0_0 \joint \vide \\
      & = \atom{x}^0_0 \joint \vide,
    \end{align*}
  }%
  et, par un calcul similaire,
  \[
    \atom{x \joint y}^1_0 = \vide \joint \atom{y}^1_0.
  \]
  On en déduit les égalités
  \[
    e(\atom{x \joint y}^0_0) = e(\atom{x}^0_0) = 1
    \quadet
    e(\atom{x \joint y}^1_0) = e(\atom{y}^1_0) = 1.
  \]
  On a par ailleurs, pour $\e = 0, 1$,
  \[
     \atom{x \joint \vide}^\e_0 = \atom{x}^\e_0 \joint \vide
     \quadet
     \atom{\vide \joint y}^\e_0 = \vide \joint \atom{y}^\e_0,
  \]
  et donc
  \[
    e(\atom{x \joint \vide}^\e_0) = e(\atom{x}^\e_0) = 1
    \quadet
    e(\atom{\vide \joint y}^\e_0) = e(\atom{y}^\e_0) = 1,
  \]
  ce qui achève de prouver que le complexe $K \joint L$ est à base unitaire.
\end{proof}

\begin{lemme}\label{lemme:joint_morph_atom}
  Soient $f : K \to K'$ et $g : L \to L'$ des morphismes entre complexes
  dirigés augmentés à base. Considérons un élément homogène de $K
  \joint L$ de la forme~$x \joint y$. On suppose qu'il existe un élément
  homogène de $K' \joint L'$ de la forme $x' \joint y'$ tel que
  \begin{enumerate}
    \item si $x$ et $y$ sont différents de $\vide$, alors il
      en est de même de $x'$ et $y'$, et on a
      \[
         f(\atom{x}) = \atom{x'} \quadet
         g(\atom{y}) = \atom{y'},
      \]
    \item si $x = \vide$, alors
      \[ x' = \vide \quadet g(\atom{y}) = \atom{y'}, \]
    \item si $y = \vide$, alors
      \[ f(\atom{x}) = \atom{x'} \quadet y' = \vide. \]
  \end{enumerate}
   Alors on a
  \[ (f \joint g)(\atom{x \joint y}) = \atom{x' \joint y'}. \]
\end{lemme}

\begin{proof}
  Pour tout $k \ge 0$ et $\e = 0, 1$, on a
  {
    \allowdisplaybreaks
    \begin{align*}
      (f \joint g)(\atom{x \joint y}^\e_k)
      & =
      (f \joint g)\Big(
      \sum_{\substack{p + 1 + q = k\\
      -1 \le p \le |x|,\, -1 \le q \le |y|}}
      \atom{x}^\epsilon_p \joint \atom{y}^{p+1 + \e \bmod 2}_q\Big) \\*
      & \phantom{= 1} \text{(en vertu du
      lemme~\ref{lemme:tab_joint} et avec ses conventions)} \\
      & =
      \sum_{\substack{p + 1 + q = k\\
      -1 \le p \le |x|,\, -1 \le q \le |y|}}
      f\big(\atom{x}^\epsilon_p\big)
        \joint g\big(\atom{y}^{p+1 + \e \bmod 2}_q\big) \\*
        & \phantom{= 1} \text{(en convenant que $m(\vide) = \vide$ pour $m =
      f, g$)} \\
      & =
      \sum_{\substack{p + 1 + q = k\\
      -1 \le p \le |x'|,\, -1 \le q \le |y'|}}
      \atom{x'}^\epsilon_p \joint \atom{y'}^{p+1 + \e \bmod 2}_q,
    \end{align*}
  }%
  la dernière égalité découlant du fait que, d'une part, on a
  $f(\atom{x}^\e_p) = 0$ si $p > |x'|$ et $g(\atom{y}^{\e'}_q) = 0$ si~$q >
  |y'|$ et, d'autre part, nos hypothèses entraînent les égalités
  $f(\atom{x}^{\e}_{-1}) = \atom{x'}^{\e}_{-1}$ et
  \hbox{$g(\atom{y}^{\e'}_{-1}) = \atom{y'}^{\e'}_{-1}$}
  (toujours avec les conventions du lemme~\ref{lemme:tab_joint}).
  Une seconde application du lemme~\ref{lemme:tab_joint} permet de conclure.
\end{proof}

\begin{prop}\label{prop:joint_rig}
  Si $f : K \to K'$ et $g : L \to L'$ sont deux morphismes rigides
  \noemph{(voir le paragraphe~\ref{paragr:def_rigide})} entre complexes dirigés augmentés à
  base, alors leur joint $f \joint g : K \joint L \to K' \joint L'$ est
  également rigide.
\end{prop}

\begin{proof}
  Cela résulte immédiatement du lemme précédent.
\end{proof}

\begin{prop}\label{prop:joint_morph_atom}
  Soient $f : K \to K'$ et $g : L \to L'$ des morphismes entre complexes
  dirigés augmentés à base unitaire et $x \joint y$ un élément de la
  base de~$K \joint L$. Supposons qu'il existe un élément $x' \joint y'$ de
  la base de $K' \joint L'$ tel que
  \begin{enumerate}
    \item si $x$ et $y$ sont différents de $\vide$, alors il en est de même
      de $x'$ et $y'$, et on a
      \[
        \nu(f)(\atom{x}) = \id{\atom{x'}}
        \quadet
        \nu(g)(\atom{y}) = \id{\atom{y'}},
      \]
    \item si $x = \vide$, alors
    \[ x' = \vide \quadet \nu(g)(\atom{y}) = \id{\atom{y'}}, \]
    \item si $y = \vide$, alors
      \[ \nu(f)(\atom{x}) = \id{\atom{x'}} \quadet y' = \vide, \]
  \end{enumerate}
  où $\id{}$, dans les égalités ci-dessus, désigne une identité itérée
  (éventuellement $0$ fois).
  Alors on a
  \[ \nu(f \joint g)(\atom{x \joint y}) = \id{\atom{x' \joint y'}}. \]
  En particulier, lorsque $x' = f(x)$ et $y' = g(y)$ vérifient les
  hypothèses ci-dessus (en convenant que $m(\vide) = \vide$ pour $m = f,
  g$), on a
  \[ \nu(f \joint g)(\atom{x \joint y}) = \atom{f(x) \joint g(y)}. \]
\end{prop}

\begin{proof}
  Cela résulte immédiatement du lemme~\ref{lemme:joint_morph_atom}.
\end{proof}

\begin{prop}\label{prop:joint_atom_univ}
  Soient $K$ et $L$ deux complexes dirigés augmentés à base unitaire et $x
  \joint y$ un élément de la base de $K \joint L$, avec $x$ de degré $i$ et
  $y$ de degré $j$. 
  \begin{enumerate}
    \item Supposons $i \ge 0$ et $j \ge 0$, et notons $z$ la cellule
      principale de $\Dn{i}$ et $t$ celle de~$\Dn{j}$. Alors le diagramme
  \[
    \xymatrix@C=4pc{
    \nu(\lambda(\Dn{i}) \joint \lambda(\Dn{j}))
    \ar[r]^-{\nu\big(\widetilde{\atom{x}} \joint \widetilde{\atom{y}}\big)}
    & \nu(K \joint L) \\
    \Dn{i+1+j} \ar[u]^{\atom{z \joint t}}
    \ar[ur]_{\atom{x \joint y}}
    & \pbox{,}
    }
  \]
  où $\widetilde{\atom{x}} : \lambda(\Dn{i}) \to K$ et $\widetilde{\atom{y}}
  : \lambda(\Dn{j}) \to L$ désignent les morphismes transposés de $\atom{x}
  : \Dn{i} \to \nu(K)$ et $\atom{y} : \Dn{j} \to \nu(L)$ respectivement, est
  commutatif.
  \item Si $i = -1$, de sorte que $x = \vide$ et $\Dn{i} = \vide$, alors, en
    notant toujours $t$ la cellule principale de $\Dn{j}$, le triangle
    ci-dessus commute pour $z = \vide$ et $\widetilde{\atom{x}}$ l'unique
    morphisme de $\lambda(\Dn{i}) = \vide$ vers $\nu(K)$.
  \item Si $j = -1$, de sorte que $y = \vide$ et $\Dn{j} = \vide$, alors,
    en notant toujours $z$ la cellule principale de $\Dn{i}$, le triangle ci-dessus
    commute pour $t = \vide$ et $\widetilde{\atom{y}}$ l'unique morphisme
    de $\lambda(\Dn{j}) = \vide$ vers $\nu(L)$.
  \end{enumerate}
\end{prop}

\begin{proof}
  Il s'agit de montrer qu'on a
  \[
    \nu\big(\widetilde{\atom{x}} \joint \widetilde{\atom{y}}\big)(\atom{z
    \joint t}) = \atom{x \joint y}.
  \]
  Cela résulte de la proposition précédente puisque, par
  définition, si $i \ge 0$, on a
    $\nu(\widetilde{\atom{x}})(\atom{z}) = \atom{x}$ et, si $j \ge 0$, on a
    $\nu(\widetilde{\atom{y}})(\atom{t}) = \atom{y}$.
\end{proof}

\begin{prop}\label{prop:joint_sans_boucle}
  Si $K$ et $L$ sont des complexes dirigés augmentés décents \noemph{(voir le
  paragraphe~\ref{paragr:def_decent})} à base fortement
 sans boucle, alors il en est de même de $K \joint L$. De plus, les
 morphismes
  \[
    \iota_1 : K \to K \joint L \quadet \iota_2 : L \to K \joint L
  \]
  sont des inclusions rigides ordonnées.
\end{prop}

\begin{proof}
  Soit $M$ un complexe dirigé augmenté décent admettant une base~$Z$.
  On rappelle (voir le paragraphe~\ref{paragr:def_le_N}) qu'on note
  $\leN$ la plus petite relation de préordre sur $Z$ satisfaisant
  \[
    x \leN y \quad\text{si}\quad x \in \supp(d(y)_-) \text{ ou }
    y \in \supp(d(x)_+),
  \]
  où, par convention, $d(z) = 0$ si $z$ est dans $Z_0$. De
  même, on notera $\leNv$ la plus petite relation de préordre sur $Z
  \cup \{\vide\}$ satisfaisant
  \[
    x \leNv y \quad\text{si}\quad x \in \supp(d(y)_-) \text{ ou }
    y \in \supp(d(x)_+),
  \]
  où cette fois-ci on a convenu que $d(z)_+ = e(z)\vide$ et $d(z)_- = 0$ si $z$
  est dans $Z_0$, et que $d(\vide)_+ = 0$ et $d(\vide)_- = 0$. Par
  définition, la relation $\leNv$ est la plus petite relation de
  préordre prolongeant $\leN$ et vérifiant $z \leNv \vide$ pour tout $z$
  dans $Z$ de degré $0$ tel que $e(z) > 0$. On en déduit immédiatement que
  la relation de préordre $\leNv$ est une relation d'ordre si et seulement
  si la relation de préordre $\leN$ en est une.

  Démontrons maintenant la proposition. Il est immédiat que le complexe $K
  \joint L$ est décent et il s'agit de montrer qu'il est à base
  unitaire. Notons $X$ et $Y$ les bases respectives de $K$ et $L$.
  On définit une relation sur la base $X \joint Y$ de $K \joint L$ en posant
  \[
    x \joint y \lec x' \joint y' \quaddefssi
    {
    \setstretch{0}
    \begin{cases}
      & x \lNv x' \\
      \text{ou} \\
      & \text{$x = x'$, $|x|$ est impair et $y \leNv y'$} \\
      \text{ou} \\
      & \text{$x = x'$, $|x|$ est pair et $y \geNv y'$.}
    \end{cases}
    }
  \]
  Par hypothèse, les relations $\leN$ sur $K$ et $L$ sont des relations
  d'ordre et il en est donc de même des relations $\leNv$. On en déduit
  facilement que la relation $\lec$ est également une relation d'ordre.

  Soient $x \joint y$ et $x' \joint y'$ dans $X \joint Y$. Montrons que la
  relation $\lec$ vérifie la propriété
  \[
    x \joint y \in \supp(d(x' \joint y')_-) \text{ ou }
    x' \joint y' \in \supp(d(x \joint y)_+)
    \quadimpl
    x \joint y \lec x' \joint y',
  \]
  ce qui établira que la relation de préordre $\leN$ du complexe $K \joint
  L$ est une relation d'ordre. Supposons tout d'abord que $x \joint y$
  appartienne au support de $d(x' \joint y')_-$. Puisqu'on a
  \[
    \begin{split}
      d(x' \joint y')_-
      & = \big(d(x') \joint y' + (-1)^{|x'|+1} x' \joint d(y')\big)_- \\
      & = d(x')_- \joint y' + x' \joint d(y')_{(-1)^{|x'|}},
    \end{split}
  \]
  la deuxième égalité résultant de la non-simplification de termes pour des
  raisons de bidegrés, cela signifie que $x \joint y$ appartient au support d'un
  des deux termes de cette somme. Si~$x \joint y$ appartient au support
  de~$d(x')_- \joint y'$, cela signifie que $x$ appartient au support
  de~$d(x')_-$ et donc que $x \lNv x'$. On a donc bien $x \joint y \lec x'
  \joint y'$. Si $x \joint y$ appartient au support de~$x' \joint
  d(y')_{(-1)^{|x'|}}$, cela signifie que $x = x'$ et que $y$ appartient au
  support de~$d(y')_{(-1)^{|x'|}}$.  Ainsi, si $|x'|$ est impair, on a $y
  \leNv y'$ et, si $|x'|$ est pair, on a $y' \leNv y$. On a donc dans tous
  les cas $x \joint y \lec x' \joint y'$, ce qu'on voulait montrer. Si
  maintenant~$x' \joint y'$ appartient au support de $d(x \joint y)_+$, on
  conclut de manière analogue en utilisant la formule
  \[
    d(x \joint y)_+ = d(x)_+ \joint y + x \joint d(y)_{(-1)^{|x|+1}}.
  \]
  Ceci achève de montrer que $K \joint L$ est à base unitaire.

  Montrons enfin que le morphisme $\iota_1 : K \to K \joint L$ est une
  inclusion rigide ordonnée, le morphisme $\iota_2 : L \to K \joint L$ se
  traitant de manière analogue. Il est immédiat que ce morphisme est une
  inclusion rigide. Soient $x$ et $x'$ dans la base de $K$. Il s'agit de
  montrer l'implication
  \[ x \joint \vide \leN x' \joint \vide \quadmath{\Rightarrow} x \leN x'. \]
  Or on a
  \[
    x \joint \vide \leN x' \joint \vide \quadmath{\Rightarrow}
      x \joint \vide \lec x' \joint \vide
    \qquad\text{et}\qquad
    x \joint \vide \lec x' \joint \vide \quadmath{\Leftrightarrow}
      x \leN x',
  \]
  d'où le résultat.
\end{proof}

\begin{coro}\label{coro:joint_Steiner}
 Si $K$ et $L$ sont des complexes de Steiner forts, alors il en est de même
 de $K \joint L$. De plus, les morphismes
  \[
    \iota_1 : K \to K \joint L \quadet \iota_2 : L \to K \joint L
  \]
  sont des inclusions rigides ordonnées.
\end{coro}

\begin{proof}
  Cela résulte des propositions~\ref{prop:joint_base_unit}
  et \ref{prop:joint_sans_boucle}.
\end{proof}

\begin{rem}\label{rem:Stf_monoidal}
  En vertu du corollaire précédent (et du fait que le complexe dirigé
  augmenté $\vide$  est de Steiner fort), la catégorie des complexes de
  Steiner forts est une sous-catégorie monoïdale de la catégorie monoïdale
  des complexes dirigés augmentés définie par le joint.
\end{rem}

\begin{prop}
  Soient $K$ un complexe de Steiner fort et $F : I \to \Cda$ un système de
  Steiner fort connexe \noemph{(voir le
  paragraphe~\ref{paragr:def_syst_Steiner})}.
  Alors le foncteur
  \[
    \begin{split}
      K \joint F & : I \to \Cda \\
      &\phantom{{:I}} i \mapsto K \joint F(i)
    \end{split}
  \]
  est un système de Steiner fort.
\end{prop}

\begin{proof}
  Le foncteur $F$ étant un système rigide, il en est de même du foncteur
  \hbox{$K \joint F : i \mapsto K \joint F(i)$} en vertu de la
  proposition~\ref{prop:joint_rig}. Par ailleurs, puisque d'après le
  corollaire~\ref{coro:joint_Steiner} les complexes de Steiner forts sont
  stables par joint, le foncteur $K \joint F$ est à valeurs dans les
  complexes de Steiner forts. Enfin, le foncteur $K \joint \var$ commutant
  aux limites inductives connexes (voir la
  proposition~\ref{prop:joint_Cda_limind}), le morphisme canonique
  \[
    \limind_{i \in I} (K \joint F(i))
    \to
    K \joint \limind_{i \in I} F(i)
  \]
  est un isomorphisme de complexes dirigés augmentés et, pour tout
  objet $i_0$ de~$I$, le morphisme canonique $K \joint F(i_0) \to \limind_{i
  \in I} (K \joint F(i))$ s'identifie à travers cet isomorphisme au joint
  $K \joint F(i_0) \to K \joint \limind_{i \in I} F(i)$ de $K$ et du morphisme
  canonique associé à $F$. On en déduit le résultat en invoquant de nouveau
  la proposition~\ref{prop:joint_rig} et le
  corollaire~\ref{coro:joint_Steiner}.
\end{proof}

\begin{coro}
  Si $K$ est un complexe de Steiner fort, alors le foncteur
  \[
    \begin{aligned}
      \Cda & \to \ooCat \\
      L & \mapsto \nu(K \joint L)
    \end{aligned}
  \]
  commute aux limites inductives des systèmes de Steiner forts connexes.
\end{coro}

\begin{proof}
  Puisque le foncteur $\nu$ commute aux limites inductives des systèmes de
  Steiner forts (théorème~\ref{thm:nu_syst_Steiner}),
  l'assertion résulte de la proposition précédente.
\end{proof}

\begin{coro}
  Soit $K$ un complexe de Steiner fort. Alors le foncteur
  \[
    \begin{split}
      \ThetaAug & \to \ooCat \\
      S  & \mapsto \nu(K \joint \lambda(S))
    \end{split}
  \]
  commute aux sommes globulaires.
\end{coro}

\begin{proof}
  Cela résulte du corollaire précédent puisque, en vertu de
  la proposition~\ref{prop:Theta_Steiner}, les sommes globulaires proviennent
  de systèmes de Steiner forts (qui sont bien sûr connexes).
\end{proof}

\begin{paragr}\label{paragr:def_tr_St}
  Fixons $K$ un complexe de Steiner fort. Soient $C$ une \oo-catégorie
  et \hbox{$u : \nu(K) \to C$} un \oo-foncteur. Nous allons définir une
  \oo-catégorie \nnot{$\cotr{C}{u}$}. Il résulte du corollaire
  précédent que le foncteur
  \[
    \begin{split}
      (\ThetaAug)^\op & \to \Ens \\
      S  & \mapsto \Hom_{\ooCat}(\nu(K \joint \lambda(S)), C)
    \end{split}
  \]
  envoie les sommes globulaires sur des produits globulaires au sens du
  paragraphe~\ref{paragr:pu_ThetaAug}. Ainsi, en vertu de ce même
  paragraphe et avec ces notations, ce foncteur définit une \oo-catégorie
  $\Hom_{\ooCat}(\nu(K \joint \lambda(\Dn{\var})), C)$
  au-dessus de l'ensemble
  \[ \Hom_{\ooCat}(\nu(K \joint \lambda(\vide)), C) = \Hom_{\ooCat}(\nu(K
  \joint \vide), C) \simeq \Hom_{\ooCat}(\nu(K), C). \]
  On définit la \oo-catégorie $\cotr{C}{u}$ comme la fibre en $u$ de cette
  \oo-catégorie au-dessus de~$\Hom_{\ooCat}(\nu(K), C)$. Autrement dit,
  avec les notations du paragraphe~\ref{paragr:pu_ThetaAug}, on
  pose
  \[ \cotr{C}{u} = \Hom_{\ooCat}(\nu(K \joint \lambda(\Dn{\var})), C)_u. \]
  Explicitement, les $i$-flèches de $\cotr{C}{u}$ sont les \oo-foncteurs
  $\nu(K \joint \lambda(\Dn{i})) \to C$ rendant le
  triangle
  \[
    \xymatrix{
      \nu(K \joint \lambda(\Dn{i})) \ar[r] & C \\
      \nu(K) \ar[u]^{\nu(\iota_1)} \ar[ru]_u
    }
  \]
  commutatif.

  Notons qu'un morphisme $f : K \to K'$ entre complexes de Steiner forts
  induit, pour toute \oo-catégorie $C$, une application
  \[
    \Hom_{\ooCat}(\nu(K' \joint \lambda(S)), C)
      \to \Hom_{\ooCat}(\nu(K \joint \lambda(S)), C),
  \]
  naturelle en $S$ dans $\ThetaAug$. Ainsi, toujours en vertu du
  paragraphe~\ref{paragr:pu_ThetaAug}, pour tout triangle commutatif
  \[
    \xymatrix@C=1.5pc{
      \nu(K) \ar[rr]^{\nu(f)} \ar[dr]_u & & \nu(K') \ar[dl]^(.44){u'} \\
      & C & \pbox{,}
    }
  \]
  on obtient un \oo-foncteur $\cotr{C}{u'} \to \cotr{C}{u}$.
\end{paragr}

\begin{prop}\label{prop:pu_cotranche}
  Fixons $K$ un complexe de Steiner fort. Pour tout complexe de Steiner fort
  $L$, toute \oo-catégorie $C$ et tout \oo-foncteur $u : \nu(K) \to C$, on a une
  bijection
  \[
    \Hom_{\cotr{\ooCat}{\nu(K)}}((\nu(K \joint L), \nu(\iota_1)), (C, u))
    \simeq
    \Hom_{\ooCat}(\nu(L), \cotr{C}{u}),
  \]
  naturelle en $L$ et $u$.
\end{prop}

\begin{proof}
  Si $M$ est un complexe de Steiner fort et $t$ est un élément de la base de
  $M$ de degré $i \ge 0$, on notera, pour simplifier,
  \[
    \begin{split}
    t^\epsilon_j & = \atom{t}^\epsilon_j \quad\text{pour $0 \le j \le i$}
    \quad\text{(voir le paragraphe~\ref{paragr:def_atome})}, \\
    t^0_{-1} & = 0 \quadet t^1_{-1} = \vide. \\
    \end{split}
  \]
  On notera par ailleurs $\vide^0_{-1} = \vide$ et $\vide^1_{-1} = \vide$.

  On va produire des fonctions
  \[
    \phi : \Hom_{\cotr{\ooCat}{\nu(K)}}((\nu(K \joint L), \nu(\iota_1)), (C, u))
    \to
    \Hom_{\ooCat}(\nu(L), \cotr{C}{u})
  \]
  et
  \[
    \psi :
    \Hom_{\ooCat}(\nu(L), \cotr{C}{u})
    \to
    \Hom_{\cotr{\ooCat}{\nu(K)}}((\nu(K \joint L), \nu(\iota_1)), (C, u))
  \]
  inverses l'une de l'autre.

  Commençons par $\phi$. Soit $F : \nu(K \joint L) \to C$ un \oo-foncteur
  au-dessous de $\nu(K)$. On définit un \oo-foncteur $\phi(F) : \nu(L) \to
  \cotr{C}{u}$ de la manière suivante. Soit $y : \Dn{i} \to \nu(L)$, pour $i
  \ge 0$, une $i$-flèche de $\nu(L)$. On doit lui associer une $i$-flèche
  $\phi(F)(y)$ de $\cotr{C}{u}$, c'est-à-dire un \oo-foncteur $\nu(K \joint
  \lambda(\Dn{i})) \to C$ au-dessous de $\nu(K)$. On pose
  \[
    \phi(F)(y) = \quad \nu(K \joint \lambda(\Dn{i})) \xto{\nu(K \joint
    \tilde{y})} \nu(K \joint L) \xto{F} C,
  \]
  où on a noté $\tilde{y} : \lambda(\Dn{i}) \to L$ le transposé de $y :
  \Dn{i} \to \nu(L)$. Ce \oo-foncteur étant le composé de deux \oo-foncteurs
  au-dessous de $\nu(K)$, il est bien au-dessous de $\nu(K)$. Le fait que
  $\phi(F)$ soit bien un \oo-foncteur résulte de la naturalité en
  $\Dn{i}$, et plus généralement en $S$ dans $\Theta$, du \oo-foncteur
  $\phi(F)(y)$. En particulier, pour $x$ un élément de la base de $K$ ou $x
  = \vide$, $y$ un élément de la base de~$L$ de degré $i$ et $z$ la cellule
  principale de~$\Dn{i}$, on a, pour $j$ tel que $-1 \le j \le i$ et
  $\epsilon = 0, 1$,
  \[
    \phi(F)(\atom{y})(\atom{x \joint z^\epsilon_j}) = F(\atom{x \joint
    y^\epsilon_j}).
  \]
  De plus, cette formule détermine $\phi(F)$ de manière unique puisque les
  \oo-catégories $\nu(L)$ et~$\nu(K \joint \lambda(\Dn{i}))$ sont engendrées
  librement au sens des polygraphes par leurs atomes en vertu du
  théorème~\ref{thm:Steiner_pol}.

  Définissons maintenant $\psi$. Soit $G : \nu(L) \to \cotr{C}{u}$ un
  \oo-foncteur. Il s'agit de définir un \oo-foncteur $\psi(G) : \nu(K \joint
  L) \to C$ au-dessous de $\nu(K)$. En vertu du
  théorème~\ref{thm:Steiner_pol}, la \oo-catégorie $\nu(K \joint L)$ est
  engendrée librement au sens des polygraphes par ses atomes. Il suffit donc
  de définir $\psi(G)$ sur les atomes de $\nu(K \joint L)$ et de vérifier
  les compatibilités aux sources et aux buts. Soit $x \joint y$ dans la base
  de $K \joint L$ avec $x$ de degré $i$ et $y$ de degré~$j$. On pose
  \[
    \psi(G)(\atom{x \joint y}) = \quad
    \Dn{i+1+j}
    \xto{\atom{x \joint z'_j}}
    \nu(K \joint \lambda(\Dn{j}))
    \xto{G(\atom{y})}
    C,
  \]
  où $z'$ désigne la cellule principale de $\Dn{j}$ si $j \ge 0$, et $z' =
  \vide$ si $j = -1$ et, par convention, $G(\atom{\vide}) : \nu(K) \to C$
  désigne $u$.

  Vérifions maintenant les compatibilités aux sources et aux buts. Fixons
  $m \ge 0$ et supposons que les formules ci-dessus définissent bien un
  $m$-foncteur. Il s'agit de montrer que, pour tous $x$ et $y$ comme
  ci-dessus avec $i + 1 + j = m + 1$, on a
  \[
    \psi(G)(s(\atom{x \joint y})) = s(\psi(G)(\atom{x \joint y}))
    \quad\text{et}\quad
    \psi(G)(t(\atom{x \joint y})) = t(\psi(G)(\atom{x \joint y})).
  \]
  Montrons la première égalité, la seconde se démontrant de manière
  analogue. Notons $z$ et $z'$ les cellules principales respectives de
  $\Dn{i}$ et $\Dn{j}$, en convenant que $z = \vide$ si $i = -1$ et $z' =
  \vide$ si $j = -1$. Puisque la \oo-catégorie \hbox{$\nu(\lambda(\Dn{i})
  \joint \lambda(\Dn{j}))$} est engendrée librement au sens des polygraphes
  par ses atomes, en vertu de la proposition~\ref{prop:eng_pol_comp}, ses
  atomes l'engendrent également par compositions et il existe donc une
  formule $\chi$ exprimant la source de $\atom{z_i \joint z'_j}$ en fonction
  des $\atom{z^\e_k \joint z'^\ep_l}$ avec $-1 \le k \le i$, $-1 \le l \le
  j$, $0 \le k + 1 + l < i + 1 + j$, $\e = 0, 1$ et $\ep = 0, 1$.
  On notera $\chi[\atom{z^\e_k \joint z'^\ep_l}]$ l'évaluation de la
  formule $\chi$ en les éléments $z^\e_k \joint z'^\ep_l$. On a donc
  $s(\atom{z_i \joint z'_j}) = \chi[\atom{z^\e_k \joint z'^\ep_l}]$. Plus
  généralement, il résulte de la proposition~\ref{prop:joint_atom_univ} que
  la même formule $\chi$ permet de calculer la source d'un atome $\atom{m
  \joint n}$, où $m$ est de degré $i$ et $n$ de degré $j$, d'un joint
  quelconque de complexes dirigés augmentés à base unitaire~$M$ et $N$. On
  obtient ainsi
  \[
    \begin{split}
      \psi(G)(s(\atom{x \joint y}))
      & =
      \psi(G)(\chi[\atom{x^\e_k \joint y^\ep_l}]) \\*
      & =
      \chi[\psi(G)(\atom{x^\e_k \joint y^\ep_l})]\\*
      & \phantom{=1} \text{(car $\psi(G)$ est un $m$-foncteur et $k + 1 + l
    < i + 1 + j = m + 1$)} \\
      & = \chi[G(\atom{y^\ep_l})(\atom{x^\e_k \joint z''_l})],
    \end{split}
  \]
  où $z''$ désigne la cellule principale de $\Dn{l}$ (en convenant toujours
  que $z'' = \vide$ si $l = - 1$), cette dernière égalité résultant de la
  définition de $\psi(G)$. Par ailleurs, on a
  \[
    G(\atom{y^\ep_l})(\atom{x^\e_k \joint z''_l})
    = G(\atom{y})(\atom{x^\e_k \joint z'^\ep_l}).
  \]
  En effet, si $l = -1$, c'est une conséquence des égalités
  \[
    G(\vide)(\atom{x^\e_k \joint \vide}) =
      u(\atom{x^\e_k}) = G(\atom{y})(\atom{x^\e_k \joint \vide})
  \]
  et, pour $l \ge 0$ et $\ep = 0$, cela résulte du calcul suivant :
  \[
    \begin{split}
      G(\atom{y^0_l})(\atom{x^\e_k \joint z''_l})
      & =
      G(s_l(\atom{y}))(\atom{x^\e_k \joint z''_l}) \\
      & =
      s_l(G(\atom{y}))(\atom{x^\e_k \joint z''_l}) \\
      & =
      G(\atom{y})
      \big(\nu(K \joint \lambda(\sigma_l^i))(\atom{x^\e_k \joint z''_l})\big)
        \\*
      & \phantom{=1} \text{\lp par définition des sources des
        cellules de $\cotr{C}{u}$\rp}\\
      & =
      G(\atom{y})(\atom{x^\e_k \joint z'^0_l}),
    \end{split}
  \]
  la dernière égalité étant conséquence de la
  proposition~\ref{prop:joint_morph_atom} puisque
  $\sigma_l^i(\atom{z''_l}) = \atom{z'^0_l}$. La démonstration dans le cas
  $\ep = 1$ est analogue. En insérant cette égalité dans notre calcul
  précédent, on obtient
   {
    \begin{align*}
      \psi(G)(s(\atom{x \joint y}))
      & = \chi[G(\atom{y^\ep_l})(\atom{x^\e_k \joint z''_l})] \\
      & = \chi[G(\atom{y})(\atom{x^\e_k \joint z'^\ep_l})] \\*
      & =
      G(\atom{y})(\chi[\atom{x^\e_k \joint z'^\ep_l}]) \\*
      & \phantom{=1} \text{(car $G(\atom{y})$ est un \oo-foncteur)} \\
      & =
      G(\atom{y})(s(\atom{x \joint z'_j})) \\
      & =
      s(G(\atom{y})(\atom{x \joint z'_j})) \\
      & =
      s(\psi(G)(\atom{x \joint y})),
    \end{align*}
  }%
  la dernière égalité résultant de la définition de $\psi(G)$,
  ce qui achève de montrer que~$\psi(G)$ est bien un \oo-foncteur.
  Ce \oo-foncteur $\psi(G)$ est bien au-dessous de $\nu(K)$. En effet, par
  définition, pour $x$ un élément de la base de $K$, on a \hbox{$\psi(G)(\atom{x
  \joint \vide}) = u(\atom{x})$}. Le \oo-foncteur $\psi(G)\iota_1 : \nu(K)
  \to C$ coïncide donc avec le \oo-foncteur $u$ sur les atomes et on obtient
  le résultat puisque $\nu(K)$ est engendrée librement au sens des
  polygraphes par ses atomes en vertu du théorème~\ref{thm:Steiner_pol}.

  Enfin, vérifions que $\phi$ et $\psi$ sont bien des bijections
  inverses l'une de l'autre. Soient $F : \nu(K \joint L) \to C$ au-dessous de
  $\nu(K)$ et $G : \nu(L) \to \cotr{C}{u}$ deux \oo-foncteurs. On a, avec
  les notations précédentes,
  \[
    \psi\phi(F)(\atom{x \joint y}) = \phi(F)(\atom{y})(\atom{x \joint z'_j})
    = F(\atom{x \joint y})
  \]
  et
  \[
    \phi\psi(G)(\atom{y})(\atom{x \joint z^\epsilon_j})
    = \psi(G)(\atom{x \joint y^\epsilon_j})
    = G(\atom{y})(\atom{x \joint z^\epsilon_j}).
  \]
  Les \oo-foncteurs $\psi\phi(F)$ et $F$ (resp. les \oo-foncteurs
  $\phi\psi(G)$ et $G$) coïncident donc sur les atomes et sont donc égaux,
  ce qu'il fallait démontrer.
\end{proof}

\begin{paragr}\label{paragr:pu_tranche}
  Soit $L$ un complexe de Steiner fort. On montre de même que le foncteur
  \[
    \begin{aligned}
      \Cda & \to \ooCat \\
      K & \mapsto \nu(K \joint L)
    \end{aligned}
  \]
  commute aux limites inductives des systèmes de Steiner forts connexes. On en
  déduit que le foncteur
  \[
    \begin{split}
      (\ThetaAug)^\op & \to \Ens \\
      S  & \mapsto \Hom_{\ooCat}(\nu(\lambda(S) \joint L), C)
    \end{split}
  \]
  envoie les sommes globulaires sur des produits globulaires. On définit
  alors, comme dans le paragraphe~\ref{paragr:def_tr_St}, pour toute
  \oo-catégorie $C$ et tout \oo-foncteur $v : \nu(L) \to C$, une
  \oo-catégorie~\nnot{$\trm{C}{v}$} en posant
  \[ \trm{C}{v} = \Hom_{\ooCat}(\nu(\lambda(\Dn{\var}) \joint L), C)_v. \]
  (La décoration « $\mathrm{co}$ » dans cette notation sera expliquée dans le
  paragraphe~\ref{paragr:def_tranche} et surtout dans la
  remarque~\ref{rem:cojoint_ooCat}.) On montre, comme dans la proposition
  précédente, que, pour tout complexe de Steiner fort~$K$, toute
  \oo-catégorie $C$ et tout \oo-foncteur $v : \nu(L) \to C$, on a une
  bijection
  \[
    \Hom_{\cotr{\ooCat}{\nu(L)}}((\nu(K \joint L), \nu(\iota_2)), (C, v)) \simeq
    \Hom_{\ooCat}(\nu(K), \trm{C}{v}),
  \]
  naturelle en $K$ et $v$.
\end{paragr}

\begin{thm}\label{thm:joint}
  Il existe une et une seule (à unique isomorphisme monoïdal près) structure
  de catégorie monoïdale sur $\ooCat$ de produit
  \begin{alignat*}{2}
    \joint & : \ooCat \times \ooCat & \to \ooCat\\
    & \phantom{=1}\quad\,\,(A, B) & \mapsto A \joint B
  \end{alignat*}
  ayant les deux propriétés suivantes :
  \begin{enumerate}
    \item le foncteur $\nu_{\vert \Stf} : \Stf \to \ooCat$, où la catégorie
      des complexes de Steiner forts~$\Stf$ est munie de la structure de
      catégorie monoïdale définie par le joint, s'étend en un foncteur
      monoïdal ;
    \item le foncteur $\joint : \ooCat \times \ooCat \to \ooCat$ commute aux
      petites limites inductives connexes en chaque variable.
  \end{enumerate}
  De plus, cette structure monoïdale est localement bifermée (au sens du
  paragraphe~\ref{paragr:def_loc_ferm}).
\end{thm}

\begin{proof}
  Le théorème résulte du corollaire~\ref{coro:Day_loc} appliqué à $\C =
  \ooCat$ et~$\D$~la catégorie des \oo-catégories de Steiner fortes (voir le
  paragraphe~\ref{paragr:def_ooCat_Stf}) munie du joint (par
  l'identification de cette sous-catégorie à celle des complexes de Steiner
  forts), la petite sous-catégorie dense étant la catégorie $\ThetaAug$. En
  effet, les hypothèses de ce corollaire sont précisément le contenu de la
  proposition~\ref{prop:pu_cotranche} et du
  paragraphe~\ref{paragr:pu_tranche}.
\end{proof}

\begin{paragr}\label{paragr:def_joint}
  On appellera \ndef[joint!$\infty$-catégorique]{joint} le produit monoïdal
  \[
    \joint : \ooCat \times \ooCat \to \ooCat
  \]
  \notindex{$A \joint B$}%
  défini par le théorème précédent. Si $K$ et $L$ sont des complexes de
  Steiner forts, on a, en vertu de ce même théorème, un isomorphisme
  canonique
  \[
    \nu(K) \joint \nu(L) \simeq \nu(K \joint L).
  \]
  En particulier, si $S$ et $T$ sont deux objets de $\ThetaAug$, puisqu'en
  vertu de la proposition~\ref{prop:Theta_Steiner} on a $S \simeq
  \nu\lambda(S)$ et $T \simeq \nu\lambda(T)$, on obtient un isomorphisme
  canonique
  \[ S \joint T \simeq \nu(\lambda(S) \joint \lambda(T)). \]
  Plus généralement, si $A$ et $B$ sont deux \oo-catégories, on a des
  isomorphismes canoniques
  \[
    \begin{split}
    A \joint B
    & \simeq \limind_{\substack{S \to A \in \tr{\ThetaAug}{A}\\
    T \to B \in \tr{\ThetaAug}{B}}} S \joint T \\
    & \simeq \limind_{\substack{S \to A \in \tr{\ThetaAug}{A}\\
    T \to B \in \tr{\ThetaAug}{B}}} \nu(\lambda(S) \joint \lambda(T)).
    \end{split}
  \]
  En effet, puisque la catégorie $\ThetaAug$ est dense dans $\ooCat$ et
  contient la \oo-catégorie vide, toute \oo-catégorie est limite inductive
  canonique connexe d'objets de $\ThetaAug$. La formule résulte alors du
  fait que le joint commute aux limites inductives connexes en chaque
  argument.

  L'unité du joint est l'image par $\nu$ du complexe dirigé
  augmenté $\vide$, c'est-à-dire la \oo-catégorie vide qu'on notera
  également~$\vide$. Si $A$ est une \oo-catégorie, on a donc
  \[
    A \joint \vide \simeq A \simeq \vide \joint A.
  \]

  Si $A$ et $B$ sont deux \oo-catégories, on notera
  \[
    A \xto{\iota_1} A \joint B \xot{\iota_2} B
  \]
  \notindex{$\iota_1 : A \to A \joint B$, $\iota_2 : B \to A \joint B$}%
  les \oo-foncteurs
  \[
    A \simeq A \joint \vide \longto A \joint B
        \longot \vide \joint B \simeq B,
  \]
  où les flèches pointant vers $A \joint B$ sont induites par les
  \oo-foncteurs $\vide \to B$ et \hbox{$\vide \to A$}. (Ce sont les
  \oo-foncteurs de la bicoaugmentation locale associée au joint, voir
  l'exemple~\ref{exem:coaug}.)
\end{paragr}

\begin{paragr}\label{paragr:def_tranche}
  En vertu du théorème~\ref{thm:joint}, la structure de catégorie monoïdale
  définie par le joint est localement bifermée. Cela signifie exactement que
  les foncteurs
  \[
    \begin{split}
      \ooCat & \to \cotr{\ooCat}{A} \\
      B & \mapsto (A \joint B, \iota_1 : A \to A \joint B)
    \end{split}
  \]
  et
  \[
    \begin{split}
      \ooCat & \to \cotr{\ooCat}{B} \\
      A & \mapsto (A \joint B, \iota_2 : B \to A \joint B)
    \end{split}
  \]
  admettent des adjoints à droite. On obtient donc des couples de foncteurs
  adjoints
  \[
    \begin{split}
      \ooCat & \to \cotr{\ooCat}{A}, \\
      B & \mapsto (A \joint B, \iota_1)
    \end{split}
    \qquad
    \qquad
    \begin{split}
      \cotr{\ooCat}{A} & \to \ooCat \\
      (C, A \xto{u} C) & \mapsto \cotr{C}{u} \\
    \end{split}
  \]
  et
  %
  % INDEXCHECK
  \notindex{$\cotr{C}{u}$, $\trm{C}{u}$}%
  \[
    \begin{split}
      \ooCat & \to \cotr{\ooCat}{B}, \\
      A & \mapsto (A \joint B, \iota_2)
    \end{split}
    \qquad
    \qquad
    \begin{split}
      \cotr{\ooCat}{B} & \to \ooCat. \\
      (C, B \xto{v} C) & \mapsto \trm{C}{v} \\
    \end{split}
  \]

  Ainsi, si $A$ et $B$ sont des \oo-catégories et $u : A \to C$ et $v : B
  \to C$ des \oo-foncteurs, on a des bijections naturelles
  \[
    \begin{split}
      \Hom_{\cotr{\ooCat}{A}}((A \joint B, \iota_1), (C, u))
      & \simeq \Hom_{\ooCat}(B, \cotr{C}{u}), \\
      \Hom_{\cotr{\ooCat}{B}}((A \joint B, \iota_2), (C, v))
      & \simeq \Hom_{\ooCat}(A, \trm{C}{v}). \\
    \end{split}
  \]
  Si $C$ est une \oo-catégorie et $u : A \to C$ est un \oo-foncteur, on
  appellera la \oo-catégorie $\cotr{C}{u}$ la
  \ndef[tranche!$\infty$-catégorique!au-dessous]{tranche de $C$ au-dessous
  de $u$}. Si $v : B \to C$ est un \oo-foncteur, on réservera la notation
  $\tr{C}{v}$ et la terminologie « tranche de $C$ au-dessus de $v$ » à une
  variante de~\smash{$\trm{C}{v}$} qu'on introduira dans la
  remarque~\ref{rem:cojoint_ooCat}.

  Notons que, dans le cas où la source de $u$ est de la forme $\nu(K)$ pour
  $K$ un complexe de Steiner fort, la \oo-catégorie $\cotr{C}{u}$ que l'on
  vient d'introduire coïncide, en vertu de la
  proposition~\ref{prop:pu_cotranche}, avec celle définie dans le
  paragraphe~\ref{paragr:def_tr_St}. De même, pour la \oo-catégorie
  \smash{$\trm{C}{v}$} et la \oo-catégorie définie au
  paragraphe~\ref{paragr:pu_tranche}.
\end{paragr}

\begin{rem}
  La tranche $\cotr{C}{u}$ définie au paragraphe précédent est une tranche
  généralisée au sens où on prend la tranche de $C$ au-dessous d'un
  \oo-foncteur. Dans le cas où la source du \oo-foncteur $u$ est la
  \oo-catégorie finale, la donnée de $u$ devient équivalente à celle d'un
  objet de $C$ et on se trouve alors dans le cadre usuel des tranches. Dans
  le chapitre~\ref{sec:desc_expl}, on décrira explicitement ces tranches
  au-dessous d'un objet et on vérifiera que notre définition est compatible
  avec les définitions usuelles quand $C$ est une $1$-catégorie ou une
  $2$-catégorie.
\end{rem}

\begin{paragr}\label{paragr:desc_tr}
  Fixons $A$ une \oo-catégorie. Soient $C$ une \oo-catégorie et $u : A \to C$
  un \oo-foncteur. Par adjonction, pour $i \ge 0$, on a
  \[
    \Hom_{\ooCat}(\Dn{i}, \cotr{C}{u}) \simeq \Hom_{\cotr{\ooCat}{A}}((A \joint
    \Dn{i}, \iota_1), (C, u)).
  \]
  Ainsi, une $i$-flèche de $\cotr{C}{u}$ est donnée par un \oo-foncteur
  $A \joint \Dn{i} \to C$ faisant commuter le triangle
  \[
    \xymatrix{
      A \joint \Dn{i} \ar[r] & C \\
      A \ar[u]^{\iota_1} \ar[ru]_u & \pbox{.}
    }
  \]
  Notons qu'avec les notations du paragraphe~\ref{paragr:pu_ThetaAug} (voir
  également le paragraphe~\ref{paragr:def_tr_St}), on a un isomorphisme
  canonique
  \[ \cotr{C}{u} \simeq \Hom_{\ooCat}(A \joint \cocatD, C)_u. \]
\end{paragr}

\begin{prop}\label{prop:lambda_nu_monoidaux_joint}
  Les foncteurs
  \[
    \lambda : \ooCat \to \Cda
    \quadet
    \nu : \Cda \to \ooCat
  \]
  sont monoïdal et monoïdal lax respectivement, les catégories $\ooCat$ et
  $\Cda$ étant toutes deux munies des structures de catégorie monoïdale
  définies par le joint.
\end{prop}

\begin{proof}
  Par adjonction, il suffit de montrer que le foncteur $\lambda$ est
  monoïdal. On a évidemment $\lambda(\vide) = \vide$. Par ailleurs, si
  $A$ et $B$ sont deux \oo-catégories, on a, en désignant par $S$ et $T$ des
  objets de $\ThetaAug$,
  {
    \allowdisplaybreaks
    \begin{align*}
      \lambda(A \joint B)
      & \simeq
      \lambda\big(\limind_{S \to A, T \to B} \nu(\lambda(S) \joint
      \lambda(T))\big) \\*
      & \phantom{\simeq 1} \text{(en vertu du
      paragraphe~\ref{paragr:def_joint})} \\
      & \simeq
      \limind_{S \to A, T \to B} \lambda\nu(\lambda(S) \joint \lambda(T))
      \\*
      & \phantom{\simeq 1} \text{(car le foncteur $\lambda$ est un adjoint à
      gauche)} \\
      & \simeq
      \limind_{S \to A, T \to B} \lambda(S) \joint \lambda(T) \\*
      & \phantom{\simeq 1} \text{\lp en vertu du théorème~\ref{thm:Steiner}
        puisque $\lambda(S) \joint \lambda(T)$ est de Steiner fort}\\*
      & \phantom{\simeq 1} \text{d'après la proposition~\ref{prop:Theta_Steiner}
        et le corollaire~\ref{coro:joint_Steiner}\rp} \\
      & \simeq
      \big(\limind_{S \to A} \lambda(S)\big) \joint \big(\limind_{T \to
      B} \lambda(T)\big) \\*
      & \phantom{\simeq 1} \text{(en vertu de la
        proposition~\ref{prop:joint_Cda_limind})} \\
      & \simeq
      \lambda\big(\limind_{S \to A} S\big) \joint \lambda\big(\limind_{T \to
      B} T\big) \\
      & \simeq
      \lambda(A) \joint \lambda(B),
    \end{align*}
  }%
  d'où le résultat.
\end{proof}

\begin{prop}\label{prop:dual_joint}
  Soient $A$ et $B$ deux \oo-catégories. On a un isomorphisme naturel
  canonique
  \[
    (A \joint B)^\opp \simeq B^\opp \joint A^\opp.
  \]
\end{prop}

\begin{proof}
  Commençons par observer que si $S$ est un objet de $\ThetaAug$, alors le
  complexe dirigé augmenté $\lambda(S)^\opp$ est un complexe de Steiner fort.
  En effet, d'après les propositions~\ref{prop:dual_lambda_nu}
  et~\ref{prop:dual_Theta}, on a $\lambda(S)^\opp \simeq \lambda(S^\opp)
  \simeq \lambda(T)$, pour un certain $T$ dans~$\ThetaAug$, et on conclut en
  vertu de la proposition~\ref{prop:Theta_Steiner}. En particulier, si $S$
  et $T$ sont deux objets de $\ThetaAug$,
  en vertu du paragraphe~\ref{paragr:def_joint}, on a
  \[
    \nu(\lambda(T)^\opp \joint \lambda(S)^\opp) \simeq
    \nu(\lambda(T)^\opp) \joint \nu(\lambda(S)^\opp).
  \]

  Ceci étant établi, considérons deux \oo-catégories $A$ et $B$.
  On a, en désignant par~$S$ et $T$ des objets de $\ThetaAug$,
  {
    \allowdisplaybreaks
    \begin{align*}
      (A \joint B)^\opp
      & \simeq \big(\limind_{S \to A, T \to B} \nu(\lambda(S) \joint
      \lambda(T))\big)^\opp \\*
      & \phantom{\simeq 1} \text{(en vertu du paragraphe~\ref{paragr:def_joint})} \\
      & \simeq \limind_{S \to A, T \to B} \nu(\lambda(S) \joint
      \lambda(T))^\opp \\*
      & \phantom{\simeq 1} \text{(puisque $C \mapsto C^\opp$ est une
        équivalence de catégories)} \\
      & \simeq \limind_{S \to A, T \to B} \nu((\lambda(S) \joint \lambda(T))^\opp) \\*
      & \phantom{\simeq 1} \text{(en vertu de la
        proposition~\ref{prop:dual_lambda_nu})} \\
      & \simeq \limind_{S \to A, T \to B} \nu(\lambda(T)^\opp \joint \lambda(S)^\opp) \\*
      & \phantom{\simeq 1} \text{(en vertu de la
        proposition~\ref{prop:dual_joint_cda})} \\
      & \simeq \limind_{S \to A, T \to B} \nu(\lambda(T)^\opp) \joint \nu(\lambda(S)^\opp) \\*
      & \phantom{\simeq 1} \text{(en vertu du paragraphe préliminaire à cette
        preuve)} \\
      & \simeq \limind_{T \to B} \nu(\lambda(T)^\opp) \joint
                \limind_{S \to A} \nu(\lambda(S)^\opp) \\*
      & \phantom{\simeq 1} \text{(en vertu du théorème~\ref{thm:joint})} \\
      & \simeq (\limind_{T \to B} \nu\lambda(T))^\opp \joint
                (\limind_{S \to A} \nu\lambda(S))^\opp \\
      & \phantom{\simeq 1} \text{(par une nouvelle application de la
            proposition~\ref{prop:dual_lambda_nu})} \\
      & \simeq (\limind_{T \to B} T)^\opp \joint
                (\limind_{S \to A} S)^\opp \\
      & \phantom{\simeq 1} \text{(en vertu de la proposition~\ref{prop:Theta_Steiner})} \\
      & \simeq B^\opp \joint A^\opp,
    \end{align*}
  }%
  ce qu'il fallait démontrer.
\end{proof}

\begin{prop}\label{prop:dual_tr}
  Soient $C$ une \oo-catégorie et $v : B \to C$ un \oo-foncteur. On a un
  isomorphisme canonique
  \[
    \trm{C}{v} \simeq (\cotr{C^\opp}{v^\opp})^\opp.
  \]
\end{prop}

\begin{proof}
  En effet, pour toute \oo-catégorie $A$, on a des isomorphismes naturels
  {
    \allowdisplaybreaks
    \begin{align*}
      \Hom_{\ooCat}(A, (\cotr{C^\opp}{v^\opp})^\opp)
      & \simeq
      \Hom_{\ooCat}(A^\opp, \cotr{C^\opp}{v^\opp}) \\
      & \simeq
      \Hom_{\cotr{\ooCat}{B^\opp}}((B^\opp \joint A^\opp, \iota_1^\opp), (C^\opp,
        v^\opp)) \\*
      & \phantom{\simeq 1} \text{(par adjonction)} \\
      & \simeq
      \Hom_{\cotr{\ooCat}{B^\opp}}(((A \joint B)^\opp, \iota_2^\opp),
        (C^\opp, v^\opp)) \\*
      & \phantom{\simeq 1} \text{(en vertu de la proposition précédente)} \\
      & \simeq
      \Hom_{\cotr{\ooCat}{B}}((A \joint B, \iota_2), (C, v)) \\
      & \simeq
      \Hom_{\ooCat}(A, \trm{C}{v}),
    \end{align*}
  }%
  d'où le résultat.
\end{proof}

\begin{rem}\label{rem:cojoint_ooCat}
  Le foncteur
  \[
    \begin{split}
      \ooCat \times \ooCat & \to \,\,\,\ooCat \\
      (A,B)\qquad  & \mapsto (B^\op \joint A^\op)^\op,
    \end{split}
  \]
  qu'on appellera le \ndef[joint!$\infty$-catégorique!dual]{joint dual},
  définit une structure de catégorie monoïdale sur $\ooCat$ distincte de
  celle définie par le joint (et de celle définie par $(A, B) \mapsto B
  \joint A$).  Notons $A \joint' B = (B^\op \joint A^\op)^\op$
  \notindex{$A \joint' B$}%
  ce produit tensoriel. Cette structure de catégorie monoïdale est également
  localement bifermée et on dispose donc de couples de foncteurs adjoints
  \[
    \begin{split}
      \ooCat & \to \cotr{\ooCat}{A}, \\
      B & \mapsto (A \joint' B, \iota'_1)
    \end{split}
    \qquad
    \qquad
    \begin{split}
      \cotr{\ooCat}{A} & \to \ooCat \\
      (C, A \xto{u} C) & \mapsto \cotrm{C}{u} \\
    \end{split}
  \]
  \notindex{$\cotrm{C}{u}$, $\tr{C}{u}$}%
  et
  \[
    \begin{split}
      \ooCat & \to \cotr{\ooCat}{B}, \\
      A & \mapsto (A \joint' B, \iota'_2)
    \end{split}
    \qquad
    \qquad
    \begin{split}
      \cotr{\ooCat}{B} & \to \ooCat, \\
      (C, B \xto{v} C) & \mapsto \tr{C}{v} \\
    \end{split}
  \]
  où
  \[
    A \xto{\iota'_1} A \joint' B \xot{\iota'_2} B
  \]
  \notindex{$\iota'_1 : A \to A \joint' B$, $\iota'_2 : B \to A \joint' B$}%
  désigne la bicoaugmentation locale (voir l'exemple \ref{exem:coaug}).
  Explicitement, avec des notations évidentes, on a
  \[
    \iota'_{1, A, B} = \iota^\op_{2, B^\op, A^\op}
    \quadet
    \iota'_{2, A, B} = \iota^\op_{1, B^\op, A^\op}.
  \]
  Ainsi, si $A$ et $B$ sont des \oo-catégories et $u : A \to C$ et $v : B
  \to C$ des \oo-foncteurs, on a des bijections naturelles
  \[
    \begin{split}
      \Hom_{\cotr{\ooCat}{A}}((A \joint' B, \iota'_1), (C, u))
      & \simeq \Hom_{\ooCat}(B, \cotrm{C}{u}), \\
      \Hom_{\cotr{\ooCat}{B}}((A \joint' B, \iota'_2), (C, v))
      & \simeq \Hom_{\ooCat}(A, \tr{C}{v}).
    \end{split}
  \]

  On vérifie immédiatement qu'on a des isomorphismes canoniques
  \[
    \trm{C}{u} \simeq (\tr{C^\co}{u^\co})^\co
    \quadet
    \cotrm{C}{v} \simeq (\cotr{C^\co}{v^\co})^\co,
  \]
  ce qui explique les notations \smash{$\trm{C}{u}$} et
  \smash{$\cotrm{C}{v}$}. On a par ailleurs un isomorphisme canonique
  \[
    \tr{C}{u} \simeq (\cotr{C^\op}{u^\op})^\op.
  \]
  On appellera la \oo-catégorie $\tr{C}{u}$ la
  \ndef[tranche!$\infty$-catégorique!au-dessus]{tranche de $C$
  au-dessus de $u$}. Le choix de privilégier la \oo-catégorie $\tr{C}{u}$
  par rapport à \smash{$\trm{C}{u}$} est dicté par les meilleures propriétés
  formelles dont dispose $\tr{C}{u}$.
\end{rem}

\begin{paragr}\label{paragr:desc_morph_tr}
  Soit
  \[
    \xymatrix@C=1pc{
      A \ar[rr]^u \ar[dr]_{\vphantom{c'}c} & & A' \ar[dl]^{c'} \\
      & C
    }
  \]
  un triangle commutatif de \oo-foncteurs. On va définir un \oo-foncteur
  %
  % INDEXCHECK
  \notindex{$u^\ast : \cotr{C}{c'} \to \cotr{C}{c}$}%
  \[ u^\ast : \cotr{C}{c'} \to \cotr{C}{c}. \]

  Soit $B$ un \oo-catégorie. En vertu du lemme de Yoneda, il suffit de
  définir une application
  \[
  \Hom_{\ooCat}(B, \cotr{C}{c'}) \to \Hom_{\ooCat}(B, \cotr{C}{c}),
  \]
  naturelle en $B$, ou encore, par adjonction, une application
  \[
    \Hom_{\cotr{\ooCat}{A'}}((A' \joint B, \iota_1), (C, c'))
    \to
    \Hom_{\cotr{\ooCat}{A}}((A \joint B, \iota_1), (C, c)).
  \]
  Il est immédiat, par naturalité de $\iota_1$, que la précomposition par
  \[ u \joint B : A \joint B \to A' \joint B \]
  fournit une telle application. La naturalité en $B$ résulte de la
  fonctorialité du joint et on obtient donc bien un \oo-foncteur $u^\ast :
  \cotr{C}{c'} \to \cotr{C}{c}$.

  Ce \oo-foncteur $u^\ast : \cotr{C}{c'} \to \cotr{C}{c}$ peut se décrire
  de manière alternative comme suit. Par fonctorialité du joint,
  le \oo-foncteur $u$ induit une application
  \[
    \Hom_{\ooCat}(A' \joint S, C) \to \Hom_{\ooCat}(A \joint S, C),
  \]
  naturelle en $S$ dans $\ooCat$. En vertu du
  paragraphe~\ref{paragr:pu_ThetaAug}, cette transformation naturelle induit
  un \oo-foncteur
  \[
    \cotr{C}{c'} \simeq \Hom_{\ooCat}(A' \joint \cocatD, C)_{c'} \to
    \Hom_{\ooCat}(A \joint \cocatD, C)_{c}
    \simeq \cotr{C}{c}
  \]
  qui n'est autre que $u^\ast$.

  Puisque $\vide \joint B$, pour $B$ une \oo-catégorie, est canoniquement isomorphe
  à $B$, il résulte de l'adjonction définissant la \oo-catégorie
  $\cotr{C}{c}$ que celle-ci est canoniquement isomorphe à $C$ lorsque
  $c : \vide \to A$ est l'unique morphisme de source la \oo-catégorie vide
  et de but~$A$. On obtient donc un \oo-foncteur
  \[
    \cotr{C}{c'} \to C
  \]
  qu'on appellera \ndef[$\infty$-foncteur d'oubli associé à une
  tranche]{\oo-foncteur d'oubli}.

  En revenant au cas général, notons que le \oo-foncteur
  \hbox{$u^\ast : \cotr{C}{c'} \to \cotr{C}{c}$} est au-dessus de $C$.
  Autrement dit, le triangle
  \[
    \xymatrix@C=1pc{
      \cotr{C}{c'} \ar[rr]^{u^\ast} \ar[dr] & & \cotr{C}{c} \ar[dl] \\
      & C & \pbox{,}
    }
  \]
  où les flèches obliques sont les \oo-foncteurs d'oubli, est commutatif.
\end{paragr}

\chapter[Une application : construction du nerf de Street]{Une application
  :\\ construction du nerf de Street}

Dans ce chapitre, on montre comment le joint permet de définir
facilement les orientaux et le nerf de Street \cite{StreetOrient}.

\begin{paragr}\label{paragr:def_Delta}
  On notera \nnot[$\cDelta$, $\cDeltaAug$]{$\cDelta$} la \ndef{catégorie des
  simplexes}. Rappelons que ses objets sont les ensembles ordonnés
  \[ \Deltan{n} = \{0 \le 1 \le \cdots \le n\}, \quad\text{pour $n \ge 0$}, \]
  \notindex{$\Deltan{n}$}%
  et que ses morphismes sont les applications croissantes (au sens large)
  entre ceux-ci. La \ndef{catégorie des simplexes augmentée} $\cDeltaAug$ se
  définit de la même manière en ajoutant l'ensemble ordonné $\Deltan{-1} =
  \vide$.  On considérera souvent $\cDelta$ et $\cDeltaAug$ comme des
  sous-catégories pleines de la catégorie $\Cat$ des petites catégories et
  donc comme des sous-catégories pleines de $\ooCat$. Avec ces conventions,
  la catégorie $\cDelta$ devient une sous-catégorie pleine de la catégorie
  $\Theta$ de Joyal et la catégorie $\cDeltaAug$ une sous-catégorie pleine
  de $\ThetaAug$.

  La catégorie $\cDeltaAug$ admet $\Deltan{-1}$ comme objet initial et la somme
  disjointe ensembliste
  \[
    (\Deltan{m},\Deltan{n}) \mapsto \Deltan{m} \amalg \Deltan{n} = \Deltan{m+1+n}
  \]
  induit une structure de catégorie monoïdale sur $\cDeltaAug$ d'objet unité
  $\Deltan{-1}$. On munira souvent implicitement la catégorie $\cDeltaAug$
  de la structure de catégorie monoïdale définie par la somme.

  Enfin, on rappelle que la catégorie des \ndef[ensemble
  simplicial]{ensembles simpliciaux} est la catégorie $\pref{\cDelta}$ des
  préfaisceaux sur la catégorie $\cDelta$.
\end{paragr}

\begin{paragr}\label{paragr:def_Street}
  Considérons la \oo-catégorie $\Deltan{0}$, qui n'est autre que la
  \oo-catégorie finale. Cette \oo-catégorie est munie d'une et une seule
  structure de monoïde dans la catégorie monoïdale $(\ooCat, \joint,
  \vide)$. En effet, on a des uniques \oo-foncteurs
  \[
     \Deltan{0} \joint \Deltan{0} \to \Deltan{0}
     \quadet
     \vide \to \Deltan{0}
  \]
  et ceux-ci vérifient trivialement les axiomes des monoïdes. Or, la donnée
  d'une structure de monoïde sur $\Deltan{0}$ est équivalente à celle d'un
  foncteur monoïdal $\DeltaAug \to \ooCat$ envoyant $\Deltan{0}$ sur
  $\Deltan{0}$ (voir par exemple \cite[chapitre VII, section 5]{MacLane}),
  défini à unique isomorphisme monoïdal près. Il existe donc un unique
  (à unique isomorphisme monoïdal près) foncteur monoïdal
  \[
    \OAug : \cDeltaAug \to \ooCat
  \]
  \notindex{$\On{} : \cDelta \to \ooCat$, $\OAug : \cDeltaAug \to \ooCat$}%
  envoyant $\Deltan{0}$ sur $\Deltan{0}$.
  Pour $n \ge -1$, on notera \nnot{$\On{n}$} l'image de $\Deltan{n}$ par ce
  foncteur. Explicitement, on a
  \[
    \On{n} = \Deltan{0} \joint \cdots \joint \Deltan{0},
  \]
  où $\Deltan{0}$ apparait $n+1$ fois dans le membre de droite.
  Par restriction, on obtient un foncteur
  \[
    \On{} : \cDelta \to \ooCat
  \]
  et donc un foncteur
  \[
       N_\infty : \ooCat \to \pref{\cDelta}
  \]
  \notindex{$N_\infty : \ooCat \to \pref{\cDelta}$}%
  défini par
  \[
    C \mapsto (\Deltan{n} \mapsto \Hom_{\ooCat}(\On{n}, C)).
  \]
\end{paragr}

Le but de la suite de ce chapitre est de démontrer que le foncteur
\hbox{$\On{} : \cDelta \to \ooCat$} n'est autre que l'objet
cosimplicial de Street défini dans \cite{StreetOrient} et que le foncteur
\hbox{$N_\infty : \ooCat \to \pref{\cDelta}$} est donc le nerf de Street. En
particulier, on obtiendra que $\On{n}$, pour $n \ge 0$, est le $n$-ième
oriental de Street. Voici une représentation graphique des premiers
orientaux:
  \[
    \shorthandoff{;}
    \On{0} = \Dn{0} = \xymatrix{\{0\}}, \qquad
    \On{1} = \Dn{1} = \xymatrix{0 \ar[r] & 1},
    \qquad
    \On{2} =
    \raisebox{1.5pc}{
    $\xymatrix{
      & 2
      \\
      0 \ar[r] \ar[ur]_{}="s" & 1 \ar[u]
      \ar@{}"s";[]_(0.05){}="ss"_(0.85){}="tt"
      \ar@2"ss";"tt"
    }$
    }
    \text{,}
  \]
  \[
    \shorthandoff{;}
    \On{3} =
    \raisebox{1.5pc}{
    $\xymatrix{
      0 \ar[r]_(0.60){}="03" \ar[d] \ar[dr]_{}="02"_(0.60){}="02'" &
      3
      &
      0 \ar[r]_(0.40){}="03'" \ar[d]_{}="t3"
        &
      3
      \\
      1 \ar[r] & 2 \ar[u]_{}="s3"
      &
      1 \ar[r] \ar[ur]_{}="13"_(0.40){}="13'" & 2 \ar[u]
      \ar@{}"s3";"t3"_(0.20){}="ss3"_(0.80){}="tt3"
      \ar@3"ss3";"tt3"
      \ar@{}"13";[]_(0.05){}="s123"_(0.85){}="t123"
      \ar@2"s123";"t123"
      \ar@{}"02";[lll]_(0.05){}="s012"_(0.85){}="t012"
      \ar@2"s012";"t012"
      \ar@{}"03'";"13'"_(0.05){}="s013"_(0.85){}="t013"
      \ar@2"s013";"t013"
      \ar@{}"03";"02'"_(0.05){}="s023"_(0.85){}="t023"
      \ar@2"s023";"t023"
      \\
    }$
    }
    \text{.}
  \]
Notons qu'une description analogue du $n$-ième oriental
en termes du joint des complexes de parités apparaît déjà chez
Street~\cite[section~6]{StreetParComp}.

\begin{paragr}\label{paragr:def_c}
  On va définir, selon Steiner \cite{Steiner}, un foncteur $c : \DeltaAug
  \to \Cda$. Fixons $m \ge -1$ et décrivons \nnot{$c(\Deltan{m})$}. Le complexe
  de chaînes sous-jacent à $c(\Deltan{m})$ est le complexe normalisé associé
  à l'ensemble simplicial $\Deltan{m}$. Explicitement, pour $n \ge 0$, on a
  \[
    c(\Deltan{m})_n = \Z^{(B_n)},
  \]
  où
  \[
    B_n = \{(i_0, \dots, i_n) \mid 0 \le i_0 < \cdots < i_n \le m\}.
  \]
  Pour $n \ge 1$, la différentielle $d_n : \Z^{(B_n)} \to \Z^{(B_{n-1})}$
  est donnée par
  \[
    d_n(i_0, \dots, i_n) = \sum_{k = 0}^n (-1)^k (i_0, \dots, \widehat{i_k},
    \dots, i_n),
  \]
  où on a posé $(i_0, \dots, \widehat{i_k}, \dots, i_n) =
  (i_0, \dots, i_{k-1}, i_{k+1}, \dots, i_n)$. Enfin, pour $n \ge 0$, les
  sous-monoïdes de positivité sont les
  \[
    c(\Deltan{m})^\ast_n = \N^{(B_n)}
  \]
  et l'augmentation $e : \Z^{(B_0)} \to \Z$ est donnée par la somme des
  coefficients.

  On vérifie immédiatement qu'on obtient bien ainsi un foncteur $\DeltaAug
  \to \Cda$.
\end{paragr}

\begin{rem}
  Le foncteur du paragraphe précédent est la restriction à $\cDeltaAug$ d'un
  foncteur $c : \pref{\cDelta} \to \Cda$ qui est étudié dans
  \cite{AraMaltsiCondE} (voir notamment la section 5).
\end{rem}

\begin{prop}[Steiner]\label{prop:cDelta_Steiner}
  Pour tout $m \ge -1$, le complexe dirigé augmenté $c(\Deltan{m})$ est un
  complexe de Steiner fort.
\end{prop}

\begin{proof}
  La preuve est esquissée dans \cite[exemple 3.8]{Steiner}. Pour une preuve
  détaillée (et s'appliquant non seulement à $\Deltan{m}$ mais également à
  n'importe quel complexe simplicial), voir
  \cite[théorème~8.6]{AraMaltsiCondE}.
\end{proof}

\begin{lemme}
  Pour tous $m, n \ge -1$, on a un isomorphisme canonique
  \[
  c(\Deltan{m}) \joint c(\Deltan{n}) \simeq c(\Deltan{m+1+n}),
  \]
  naturel en $\Deltan{m}$ et $\Deltan{n}$ dans $\cDeltaAug$.
\end{lemme}

\begin{proof}
  Par récurrence, il suffit de traiter le cas $m = 0$. Dans ce cas, on définit
  un isomorphisme $\chi : c(\Deltan{0}) \joint c(\Deltan{n}) \to
  c(\Deltan{1+n})$ par
  \[
    \begin{split}
      \vide \joint (i_0, \dots, i_k) & \mapsto (i_0+1, \dots, i_k+1) \\
      (0) \joint \vide & \mapsto (0) \\
      (0) \joint (i_0, \dots, i_k) & \mapsto (0, i_0+1, \dots, i_k+1). \\
    \end{split}
  \]
  Il est immédiat que cette application définit une bijection de la base du
  complexe $c(\Deltan{0}) \joint c(\Deltan{n})$ sur celle de
  $c(\Deltan{1+n})$ et qu'elle est compatible aux augmentations et aux
  sous-monoïdes de positivité. Pour conclure, il suffit donc de vérifier la
  compatibilité à la différentielle. Soit $(i_0, \dots, i_k)$ un élément de
  la base de $c(\Deltan{n})$. La compatibilité à la différentielle pour les
  éléments de la forme $\vide \joint (i_0, \dots, i_k)$ est évidente. Pour
  ceux de la forme $(0) \joint (i_0, \dots, i_k)$, on a
  {
    \allowdisplaybreaks
    \begin{align*}
      \MoveEqLeft
      \chi(d((0) \joint (i_0, \dots, i_k))) \\
      & = \chi(\vide \joint (i_0, \dots, i_k) - (0) \joint d(i_0, \dots, i_k)) \\
      & = (i_0+1, \dots, i_k+1) - \sum_{l=0}^k (-1)^l \chi((0) \joint (i_0,
      \dots, \widehat{i_l}, \dots, i_k)) \\
      & = (i_0+1, \dots, i_k+1) + \sum_{l=0}^k (-1)^{l+1} (0, i_0+1, \dots,
      \widehat{i_l+1}, \dots, i_k+1) \\
      & = (i_0+1, \dots, i_k+1) + \sum_{l=1}^{k+1} (-1)^l (0, i_0+1, \dots,
      \widehat{i_{l-1}+1}, \dots, i_k+1) \\
      & = d(0, i_0+1, \dots, i_k+1)\\
      & = d(\chi((0) \joint (i_0, \dots, i_k))),
    \end{align*}
  }
  ce qui achève la démonstration.
\end{proof}

\begin{rem}
  En particulier, pour $m \ge -1$, on a
  \[
    c(\Deltan{m}) \simeq c(\Deltan{0}) \joint \cdots \joint c(\Deltan{0}),
  \]
  où $c(\Deltan{0})$ apparaît $m+1$ fois dans le membre de droite. Les
  complexes $c(\Deltan{-1})$ et~$c(\Deltan{0})$ étant trivialement des
  complexes de Steiner forts, le fait que les $c(\Deltan{m})$ sont des
  complexes de Steiner forts (proposition~\ref{prop:cDelta_Steiner}) résulte
  donc aussi de la stabilité des complexes de Steiner forts par le joint
  (corollaire~\ref{coro:joint_Steiner}).
\end{rem}

\begin{prop}
  Pour tous $m, n \ge -1$, on a un isomorphisme canonique
  \[
    \nu(c(\Deltan{m})) \joint \nu(c(\Deltan{n})) \simeq
    \nu(c(\Deltan{m+1+n})),
  \]
  naturel en $\Deltan{m}$ et $\Deltan{n}$ dans $\cDeltaAug$.
\end{prop}

\begin{proof}
  Cela résulte immédiatement de la proposition précédente, du fait que les
  $c(\Deltan{p})$ sont des complexes de Steiner forts
  (proposition~\ref{prop:cDelta_Steiner}) et du caractère monoïdal du joint
  sur les complexes de Steiner forts
  (voir le théorème~\ref{thm:joint}).
\end{proof}

\begin{thm}\label{thm:desc_orient}
  Le foncteur $\OAug : \cDelta \to \ooCat$ est canoniquement isomorphe au
  composé
  \[ \cDeltaAug \xto{c} \Cda \xto{\nu} \ooCat. \]
  En particulier, pour tout $n \ge -1$, on a $\On{n} \simeq \nu(c(\Deltan{n}))$.
\end{thm}

\begin{proof}
  Le foncteur $\OAug : \cDeltaAug \to \ooCat$ étant caractérisé par le fait
  qu'il est monoïdal et envoie $\Deltan{0}$ sur $\Deltan{0}$, il
  suffit de vérifier que le foncteur $\nu c : \cDeltaAug \to \ooCat$ vérifie
  ces deux propriétés.  Or, il est immédiat qu'on a bien $\nu c(\Deltan{0})
  \simeq \Deltan{0}$ et le caractère monoïdal du foncteur $\nu c$ est donné
  par la proposition précédente, d'où le résultat.
\end{proof}

\begin{coro}
  Les foncteurs
  \[ \On{} : \cDelta \to \ooCat \quadet N_\infty : \ooCat \to \pref{\cDelta}
  \]
  sont isomorphes à l'objet cosimplicial de Street et au nerf de Street
  définis dans \cite{StreetOrient}.
\end{coro}

\begin{proof}
  En vertu du théorème précédent, le foncteur $\On{} : \cDelta \to
  \ooCat$ est isomorphe au foncteur $\Deltan{n} \mapsto \nu(c(\Deltan{n}))$
  qui est l'objet cosimplicial de Street d'après~\cite[théorème
  3.2]{SteinerOrient}, d'où le résultat.
\end{proof}

\begin{rem}\label{rem:orient_lax}
  Dans ce chapitre, on a produit l'objet cosimplicial de Street et le nerf
  de Street à partir de la structure de catégorie monoïdale sur $\ooCat$
  définie par le joint. Si on avait utilisé la structure de catégorie
  monoïdale de la remarque~\ref{rem:cojoint_ooCat}, on aurait produit l'objet
  cosimplicial $\Deltan{n} \mapsto (\On{n})^\co$ et le nerf \hbox{$C \mapsto
  N_\infty(C^\co)$}. Quant à la convention de signe exposée dans la
  remarque~\ref{rem:conv_susp}, elle aurait mené à l'objet cosimplicial
  $\Deltan{n} \mapsto (\On{n})^\op$ et au nerf $C \mapsto
  N_\infty(C^\op)$.
\end{rem}

\begin{paragr}
  Rappelons brièvement la définition du joint simplicial (voir par
  exemple~\cite[section 3]{JoyalQuasiKan}). La structure de catégorie
  monoïdale sur $\DeltaAug$ définie au paragraphe~\ref{paragr:def_Delta}
  s'étend par limites inductives canoniques en une structure de catégorie
  monoïdale sur la catégorie des ensembles simpliciaux augmentés
  $\pref{\DeltaAug}$. On vérifie que la catégorie des ensembles simpliciaux
  forme une sous-catégorie monoïdale de cette structure, où on a considéré
  qu'un ensemble simplicial est un ensemble simplicial augmenté en
  l'augmentant au-dessus de $\Deltan{0}$ de l'unique manière possible. On
  obtient ainsi un foncteur
  \[
    \begin{split}
    \pref{\cDelta} \times \pref{\cDelta} & \to \pref{\cDelta} \\
      (X, Y) \,\, & \mapsto X \joint Y
    \end{split}
  \]
  \notindex{$X \joint Y$}%
  qu'on appelle le \ndef[joint!simplicial]{joint simplicial}. On vérifie
  immédiatement que l'unité de cette structure est l'ensemble simplicial
  vide. Par ailleurs, on montre que le joint simplicial commute aux limites
  inductives connexes en chaque variable. Ainsi, en vertu de la
  remarque~\ref{rem:loc_ferm_th_adj}, le joint simplicial définit une
  structure monoïdale localement bifermée (voir le
  paragraphe~\ref{paragr:def_loc_ferm}).
\end{paragr}

\begin{prop}
  Les foncteurs
  \[
    c_\infty : \pref{\cDelta} \to \ooCat
    \quadet
    N_\infty : \ooCat \to \pref{\cDelta},
  \]
  où $c_\infty$ désigne l'adjoint à gauche de $N_\infty$,
  sont monoïdal et monoïdal lax respectivement, les catégories
  $\pref{\cDelta}$ et $\ooCat$ étant toutes deux munies des structures de
  catégorie monoïdale définies par le joint.
\end{prop}

\begin{proof}
  Par adjonction, il suffit de montrer que le foncteur $c_\infty$ est
  monoïdal. On a évidemment $c_\infty(\vide) = \vide$. Produisons maintenant
  la contrainte de compatibilité aux produits monoïdaux. Puisque $c_\infty$
  commute aux limites inductives et que les deux foncteurs joint commutent
  aux limites inductives connexes en chaque variable (l'un en vertu du
  théorème~\ref{thm:joint} et l'autre en vertu du paragraphe précédent), les
  foncteurs
  \[
     (X, Y) \mapsto c_\infty(X) \joint c_\infty(Y)
     \quadet
     (X, Y) \mapsto c_\infty(X \joint Y)
  \]
  de $\pref{\cDelta} \times \pref{\cDelta}$ dans $\ooCat$ commutent tout
  deux aux limites inductives connexes en chaque variable. Par ailleurs,
  pour $m \ge -1$ et $n \ge -1$, on a, en vertu du
  théorème~\ref{thm:desc_orient}, des isomorphismes naturels
  \[
     c_\infty(\Deltan{m}) \joint c_\infty(\Deltan{n})
     \simeq
     \On{m} \joint \On{n}
     \simeq
     \On{m+1+n}
     \simeq
     c_\infty(\Deltan{m+1+n})
     \simeq
     c_\infty(\Deltan{m} \joint \Deltan{n}),
  \]
  où $\Deltan{-1}$ désigne l'ensemble simplicial vide. On obtient le
  résultat puisque tout ensemble simplicial est limite inductive canonique
  de $\Deltan{p}$ avec $p \ge -1$ et que cette limite inductive est connexe.
\end{proof}

\begin{rem}
  Un ingrédient clé à la démonstration de notre théorème A de Quillen
  \oo-catégorique dans \cite{AraMaltsiThmAI} et \cite{AraMaltsiThmAII}
  est un résultat de compatibilité (partielle) du nerf de Street aux
  tranches \oo-catégoriques et simpliciales, ces dernières étant définies
  par les adjoints à droite associés à la structure de catégorie monoïdale
  localement bifermée définie par le joint simplicial. On renvoie à
  \cite[proposition 2.9]{AraMaltsiThmAII} pour un énoncé précis et une
  preuve de ce résultat, ainsi qu'à \cite[proposition 4.6]{AraMaltsiThmAI}
  pour une preuve qui n'utilise pas le joint dans le cas particulier des
  tranches au-dessous d'un objet.
\end{rem}

\chapter{Joint et tranches \pdfn-catégoriques}

Le but de ce chapitre est de montrer que le joint \oo-catégorique induit
par troncation un joint $n$-catégorique, et de montrer que pour $n = 1$ on
obtient le joint catégorique usuel.

\medskip

\emph{Dans tout le chapitre, on fixe un entier $n \ge 0$.}

\begin{paragr}
  On notera \nnot[$\Theta_n$, $\ThetanAug{n}$]{$\Theta_n$} la sous-catégorie
  pleine de la catégorie $\nCat{n}$ des $n$-catégories formée des objets de
  $\Theta$ qui sont des $n$-catégories. De même, on notera $\ThetanAug{n}$
  la sous-catégorie pleine de la catégorie $\nCat{n}$ formée des objets de
  $\ThetaAug$ qui sont des $n$\nbd-catégories.

  On vérifie facilement que, si $S$ est un objet de $\Theta$, alors
  $\ti{n}(S)$ (voir le paragraphe~\ref{paragr:tronque}) est un objet de
  $\Theta_n$. En particulier, le foncteur $\ti{n} : \ooCat \to \nCat{n}$
  induit un foncteur $\Theta \to \Theta_n$ qui fournit un adjoint à gauche
  au foncteur d'inclusion~$\Theta_n \hookto \Theta$. La catégorie $\Theta_n$
  est donc une sous-catégorie réflexive de $\Theta$. De même, la catégorie
  $\ThetanAug{n}$ est une sous-catégorie réflexive de $\ThetaAug$.
\end{paragr}

\begin{prop}\label{prop:Theta_n_dense}
  La catégorie $\Theta_n$ (et donc la catégorie $\ThetanAug{n}$) est une
  sous-catégorie dense de $\nCat{n}$. Autrement dit, si $C$ est une
  $n$-catégorie, le morphisme canonique
  \[
    \limind_{(S, S \to C) \in \tr{\Theta_n}{C}} S \longto C
  \]
  est un isomorphisme de $n$-catégories.
\end{prop}

\begin{proof}
  En vertu de la proposition~\ref{prop:Theta_dense}, on obtient un
  isomorphisme en remplaçant $\Theta_n$ par $\Theta$ dans la formule
  ci-dessus et il suffit donc de montrer que l'inclusion $\tr{\Theta_n}{C}
  \hookto \tr{\Theta}{C}$ est un foncteur cofinal. Or, il résulte
  immédiatement du fait que $\Theta_n$ est un sous-catégorie réflexive de
  $\Theta$ que ce foncteur est un adjoint à droite, d'où le résultat.
\end{proof}

\begin{lemme}
  Soient $p, q \ge 0$ deux entiers. Si $K$ est un complexe dirigé augmenté
  de dimension $p$ et $L$ un complexe dirigé augmenté de dimension $q$,
  alors leur joint~\hbox{$K \joint L$} est un complexe dirigé augmenté de
  dimension $p + 1 + q$.
\end{lemme}

\begin{proof}
  Cela résulte immédiatement de la formule explicite définissant le joint
  des complexes dirigés augmentés (voir les
  paragraphes~\ref{paragr:def_joint_aug} et~\ref{paragr:def_joint_cda}).
\end{proof}

\begin{lemme}
  Soient $S$ un objet de $\ThetanAug{p}$ et $T$ un objet de $\ThetanAug{q}$,
  où $p$ et $q$ sont des entiers positifs. Alors le joint de $S$ et $T$
  est une $(p + 1 + q)$-catégorie.
\end{lemme}

\begin{proof}
  En vertu du paragraphe~\ref{paragr:def_joint}, on a $S \joint T \simeq
  \nu(\lambda(S) \joint \lambda(T))$.  Puisque les foncteurs $\lambda$ et
  $\nu$ se restreignent en des foncteurs entre $r$-catégories et complexes
  dirigés augmentés de dimension au plus~$r$, le résultat est conséquence du
  lemme précédent.
\end{proof}

\begin{prop}\label{prop:dim_joint}
  Soient $p, q \ge 0$ deux entiers. Si $C$ est une $p$-catégorie et $D$
  une $q$-catégorie, alors on a un isomorphisme canonique
  \[
    C \joint D \simeq
    \limind_{\substack{
        (S, S \to C) \in \tr{\ThetanAug{p}}{C}\\
        (T, T \to D) \in \tr{\ThetanAug{q}}{D}}}
    S \joint T.
  \]
  En particulier, le joint de $C$ et $D$ est une $(p + 1 +
  q)$-catégorie.
\end{prop}

\begin{proof}
  En vertu de la proposition~\ref{prop:Theta_n_dense}, pour $r = p, q$, la
  catégorie $\ThetanAug{r}$ est dense dans $\nCat{r}$ et on a donc
  \[
    C \simeq \limind_{\substack{(S, S \to C) \in \tr{\ThetanAug{p}}{C}}} S
    \quadet
    D \simeq \limind_{\substack{(T, T \to D) \in \tr{\ThetanAug{q}}{D}}} T.
  \]
  On obtient alors la formule de l'énoncé en utilisant la commutation du
  joint aux limites inductives connexes en chaque variable (voir le
  théorème~\ref{thm:joint}). Ainsi, en vertu du lemme précédent, le joint de
  $C$ et $D$ est limite inductive de $(p+1+q)$-catégories et est donc
  une $(p+1+q)$-catégorie.
\end{proof}

\begin{paragr}\label{paragr:def_joint_n}
  Soient $C$ et $D$ deux $n$-catégories. On appelle
  \ndef[joint!$n$-catégorique]{joint $n$-catégorique} de $C$ et $D$ la
  $n$-catégorie
  \[ C \joint_n D = \ti{n}(C \joint D). \]
  \notindex{$A \joint_n B$}%
  Cette opération définit un foncteur
  \[
    \begin{split}
      \nCat{n} \times \nCat{n} & \to \nCat{n} \\
      (C, D) \quad\,\,\, & \mapsto C \joint_n D.
    \end{split}
  \]
  Si $C$ est une $n$-catégorie, on a
  \[ \vide \joint_n C = \ti{n}(\vide \joint C) \simeq \ti{n}(C) = C \]
  et, de même,
  \[ C \joint_n \vide \simeq C. \]
  On notera, comme dans le cas de $\ooCat$,
  \[
    C \xto{\iota_1} C \joint_n D \xot{\iota_2} D
  \]
  \notindex{$\iota_1 : A \to A \joint_n B$, $\iota_2 : B \to A \joint_n B$}%
  les \oo-foncteurs canoniques.
\end{paragr}

Le but des énoncés suivants est d'établir que le joint $n$-catégorique
définit une structure de catégorie monoïdale sur $\nCat{n}$.

\begin{lemme}\label{lemme:joint_n}
  Soient $K$ et $L$ deux complexes dirigés augmentés.
  Les morphismes canoniques $K \to \ti{n}(K)$ et $L \to \ti{n}(L)$ induisent
  un isomorphisme
  \[
    \ti{n}(K \joint L) \simeq \ti{n}(\ti{n}(K) \joint \ti{n}(L)).
  \]
\end{lemme}

\begin{proof}
  Il s'agit de montrer que pour tout $r$ tel que $0 \le r \le n$,
  l'application canonique
  \[
    \alpha_r : \ti{n}(K \joint L)_r \to \ti{n}(\ti{n}(K) \joint \ti{n}(L))_r
  \]
  est une bijection. Si $r < n$, cette application n'est autre que
  l'application canonique
  \[
    (K \joint L)_r \to (\ti{n}(K) \joint \ti{n}(L))_r.
  \]
  Or, la description explicite du joint (voir les
  paragraphes~\ref{paragr:def_joint_aug} et~\ref{paragr:def_joint_cda}) montre
  que cette application est l'identité puisque $\ti{n}(K)_p = K_p$ et
  $\ti{n}(L)_q = L_q$ pour $p, q \le r < n$. Traitons maintenant le cas $r =
  n$. Par définition, $\ti{n}(K \joint L)_n$ est le quotient du
  groupe abélien
  \[
       \Big(
       \bigoplus_{\substack{p+1+q = n\\n > p \ge -1,\, n > q \ge -1}}
       (K_p \otimes L_q)
       \Big) \oplus K_n \oplus L_n
  \]
  par le sous-groupe
  \[
       \Big(
       \sum_{\substack{p'+1+q' = n+1\\n > p' \ge -1,\, n > q' \ge -1}}
       d(K_{p'} \otimes L_{q'})
       \Big)
       + d(K_n \otimes L_0) + d(K_0 \otimes L_n) + d(K_{n+1})
       + d(L_{n+1})
  \]
  et $\ti{n}(\ti{n}(K) \joint \ti{n}(L))_n$ le quotient du groupe abélien
  \[
       \Big(
       \bigoplus_{\substack{p+1+q = n\\n > p \ge -1,\, n > q \ge -1}}
       (K_p \otimes L_q)
       \Big) \oplus K_n/d(K_{n+1}) \oplus L_n/d(L_{n+1})
  \]
  par le sous-groupe
  \[
    \Big(
        \sum_{\substack{p'+1+q' = n+1\\n > p' \ge -1,\, n > q' \ge -1}}
        d(K_{p'} \otimes L_{q'})
    \Big)
    + d(K_n/d(K_{n+1}) \otimes L_0) + d(K_0 \otimes L_n/d(L_{n+1}))
  \]
  (modulo un ajustement mineur pour traiter le cas $n = 0$). Puisque
  \[
    d(K_n/d(K_{n+1}) \otimes L_0) = d(K_n \otimes L_0)
    \quadet
    d(K_0 \otimes L_n/d(L_{n+1})) = d(K_0 \otimes L_n),
  \]
  le morphisme canonique du premier quotient vers le second est un
  isomorphisme, ce qu'il fallait démontrer.
\end{proof}

\begin{prop}\label{prop:tronque_joint}
  Soient $C$ et $D$ deux \oo-catégories. Les \oo-foncteurs canoniques $C \to
  \ti{n}(C)$ et $D \to \ti{n}(D)$ induisent un isomorphisme
  \[
    \ti{n}(C \joint D) \simeq \ti{n}(\ti{n}(C) \joint \ti{n}(D)).
  \]
\end{prop}

\begin{proof}
  Commençons par démontrer le résultat dans le cas où $C = S$ et $D = T$
  sont des objets de $\ThetaAug$. Observons tout d'abord que si $U$ est un
  objet de~$\ThetaAug$, alors $\ti{n}(\lambda(U))$ est un complexe de
  Steiner fort. En effet, en vertu de la
  proposition~\ref{prop:tr_lambda}, on a $\ti{n}(\lambda(U)) =
  \lambda(\ti{n}(U))$ et, la $n$-catégorie $\ti{n}(U)$ étant un objet de
  $\ThetaAug$, on conclut en vertu de la
  proposition~\ref{prop:Theta_Steiner}. On a donc
  { \allowdisplaybreaks
    \begin{align*}
      \ti{n}(S \joint T)
      & \simeq
      \ti{n}\nu(\lambda(S) \joint \lambda(T)) \\
      & \phantom{\simeq 1} \text{(en vertu du
      paragraphe~\ref{paragr:def_joint})} \\
      & \simeq
      \nu\ti{n}(\lambda(S) \joint \lambda(T)) \\
      & \phantom{\simeq 1} \text{(en vertu de la
        proposition~\ref{prop:tr_nu})} \\
      & \simeq
      \nu\ti{n}(\ti{n}(\lambda(S)) \joint \ti{n}(\lambda(T))) \\
      & \phantom{\simeq 1} \text{(en vertu du lemme précédent)} \\
      & \simeq
      \ti{n}\nu(\ti{n}(\lambda(S)) \joint \ti{n}(\lambda(T))) \\
      & \simeq
      \ti{n}(\nu\ti{n}(\lambda(S)) \joint \nu\ti{n}(\lambda(T))) \\*
      & \phantom{\simeq 1} \text{(puisque $\ti{n}(\lambda(U))$ pour
          $U = S, T$ est de Steiner fort)} \\
      & \simeq
      \ti{n}(\ti{n}\nu(\lambda(S)) \joint \ti{n}\nu(\lambda(T))) \\
      & \simeq
      \ti{n}(\ti{n}(S) \joint \ti{n}(T)),
    \end{align*}
  }%
  où le dernier isomorphisme résulte de la
  proposition~\ref{prop:Theta_Steiner}.

  Passons au cas général. On a, en désignant par $S$ et $T$ des objets de
  $\ThetaAug$,
  { \allowdisplaybreaks
    \begin{align*}
      \ti{n}(C \joint D)
      & \simeq
      \ti{n}(\limind\limits_{S \to C, T \to D} S \joint T) \\*
      & \phantom{\simeq 1} \text{(en vertu du
      paragraphe~\ref{paragr:def_joint})} \\
      & \simeq
      \limind\limits_{S \to C, T \to D} \ti{n}(S \joint T) \\
      & \phantom{\simeq 1} \text{(puisque $\ti{n}$ est un adjoint à gauche)} \\
      & \simeq
      \limind\limits_{S \to C, T \to D} \ti{n}(\ti{n}(S) \joint \ti{n}(T))
      \\*
      & \phantom{\simeq 1} \text{(en vertu du cas précédent)} \\
      & \simeq
      \ti{n}(\ti{n}(\limind_{S \to C} S) \joint \ti{n}(\limind_{T \to D} T))
      \\*
      & \phantom{\simeq 1} \text{(en vertu du théorème~\ref{thm:joint})} \\
      & \simeq
      \ti{n}(\ti{n}(C) \joint \ti{n}(D)),
    \end{align*}
  }%
  ce qu'il fallait démontrer.
\end{proof}

\begin{prop}\label{prop:joint_n_monoidal}
  La structure de catégorie monoïdale sur $\ooCat$ définie par le joint
  induit une structure de catégorie monoïdale sur $\nCat{n}$ pour le joint
  $n$-catégorique.
\end{prop}

\begin{proof}
  Cela résulte formellement de la proposition précédente. En effet,
  si $A$, $B$ et $C$ sont trois $n$-catégories, on a
  {
    % \allowdisplaybreaks
    \begin{align*}
      A \joint_n (B \joint_n C)
      & =
      \ti{n}(A \joint \ti{n}(B \joint C)) \\
      & \simeq
      \ti{n}(\ti{n}(A) \joint \ti{n}(B \joint C)) \\
      & \simeq
      \ti{n}(A \joint (B \joint C)),
    \end{align*}
  }
  les deux isomorphismes résultant de la proposition précédente. De même, on
  a
  \[
    (A \joint_n B) \joint_n C \simeq \ti{n}((A \joint B) \joint C),
  \]
  et la contrainte d'associativité du joint induit donc une contrainte
  d'associativité pour le joint $n$-catégorique. Par ailleurs, on a
  vérifié au paragraphe~\ref{paragr:def_joint_n} qu'on a
  \[ \vide \joint_n C \simeq C \quadet C \joint_n \vide \simeq C, \]
  d'où le résultat.
\end{proof}

\begin{lemme}
  Pour toute \oo-catégorie $A$ et pour tout $i \ge n$, le \oo-foncteur
  \hbox{$\kappa_i^n : \Dn{i} \to \Dn{n}$} coreprésentant l'identité \noemph{(voir le
  paragraphe~\ref{paragr:def_kappa_nabla})} induit un isomorphisme
  \[
    \ti{n}(A \joint \Dn{i}) \simeq \ti{n}(A \joint \Dn{n}).
  \]
\end{lemme}

\begin{proof}
  Le $n$-tronqué intelligent du \oo-foncteur $\kappa_i^n : \Dn{i} \to
  \Dn{n}$ étant un isomorphisme, l'assertion résulte de la
  proposition~\ref{prop:tronque_joint}.
\end{proof}

\begin{prop}\label{prop:tr_nCat}
  Si $C$ est une $n$-catégorie et $u : A \to C$ est un \oo-foncteur, alors
  les \oo-catégories $\cotr{C}{u}$ et \smash{$\trm{C}{u}$} sont des $n$-catégories.
\end{prop}

\begin{proof}
  Par dualité, il suffit de montrer l'assertion pour la
  \oo-caté\-gorie~$\cotr{C}{u}$. Par définition, les $i$-flèches de cette
  \oo-catégorie sont les \oo-foncteurs $A \joint \Dn{i} \to C$ au-dessous
  de~$A$. Puisque $C$ est une $n$-catégorie, de tels \oo-foncteurs ne
  dépendent que du $n$\nbd-tronqué intelligent $\ti{n}(A \joint \Dn{i})$.
  Fixons $i \ge n$.  En vertu du lemme précédent, le \oo-foncteur $\Dn{i+1}
  \to \Dn{i}$ coreprésentant l'identité induit un isomorphisme sur ces
  $n$\nbd-tronqués. L'application identité $(\cotr{C}{u})_i \to
  (\cotr{C}{u})_{i+1}$, induite par $\kappa_i$, est donc une bijection, d'où
  le résultat.
\end{proof}

\begin{prop}\label{prop:tr_nCat_univ}
  Soient $A$ et $B$ des $n$-catégories.
  On a des couples de foncteurs adjoints
  \[
    \begin{split}
      \nCat{n} & \to \cotr{\nCat{n}}{A}, \\
      B & \mapsto (A \joint_n B, \iota_1)
    \end{split}
    \qquad
    \qquad
    \begin{split}
      \cotr{\nCat{n}}{A} & \to \nCat{n} \\
      (C, A \xto{u} C) & \mapsto \cotr{C}{u} \\
    \end{split}
  \]
  et
  \[
    \begin{split}
      \nCat{n} & \to \cotr{\nCat{n}}{B}, \\
      A & \mapsto (A \joint_n B, \iota_2)
    \end{split}
    \qquad
    \qquad
    \begin{split}
      \cotr{\nCat{n}}{B} & \to \nCat{n}. \\
      (C, B \xto{v} C) & \mapsto \trm{C}{v} \\
    \end{split}
  \]
\end{prop}

\begin{proof}
  Notons tout d'abord que l'énoncé est bien défini puisque, en vertu de la
  proposition précédente, les \oo-catégories $\cotr{C}{u}$ et
  \smash{$\trm{C}{v}$} sont des $n$-catégories. Par ailleurs, on a
  {
    \allowdisplaybreaks
    \begin{align*}
      \Hom_{\nCat{n}}(B, \cotr{C}{u})
      & \simeq
      \Hom_{\ooCat}(B, \cotr{C}{u}) \\*
      & \simeq
      \Hom_{\cotr{\ooCat}{A}}((A \joint B, \iota_1), (C, u)) \\*
      & \phantom{\simeq1} \text{(par adjonction définissant les tranches)} \\
      & \simeq
      \Hom_{\cotr{\ooCat}{\ti{n}(A)}}((\ti{n}(A \joint B), \ti{n}(\iota_1)),
      (C, u)) \\*
      & \phantom{\simeq1} \text{(par adjonction)} \\
      & \simeq
      \Hom_{\cotr{\ooCat}{A}}((A \joint_n B, \iota_1), (C, u)) \\
      & \simeq
      \Hom_{\cotr{\nCat{n}}{A}}((A \joint_n B, \iota_1), (C, u)),
    \end{align*}
  }
  ce qui établit la première adjonction. La deuxième s'en déduit par
  dualité.
\end{proof}

\begin{coro}\label{coro:joint_n_biferme}
  La structure de catégorie monoïdale sur $\nCat{n}$ définie par le joint
  $n$-catégorique est localement bifermée. En particulier, le joint
  $n$-catégorique commute aux limites inductives connexes en chaque
  variable.
\end{coro}

\begin{proof}
  Cela découle immédiatement de la proposition précédente.
\end{proof}

\begin{rem}
  On peut déduire par dualité des résultats analogues pour le joint dual
  (voir la remarque \ref{rem:cojoint_ooCat}). En particulier, le joint dual
  induit une structure de catégorie monoïdale localement bifermée sur
  $\nCat{n}$.
\end{rem}

Dans la suite du chapitre, nous allons comparer le joint $n$-catégorique
dans le cas $n = 1$ avec le joint catégorique classique tel qu'exposé, par
exemple, dans la section~3.1 de~\cite{JoyalQCatAppl}.

\begin{paragr}\label{paragr:def_joint_cl}
  Soient $C$ et $D$ deux catégories. Le \ndef[joint!catégorique classique]{joint
  catégorique classique} est la catégorie \nnot[$A \jointc_1 B$]{$C
\jointc_1 D$} définie de la manière suivante :
  \begin{itemize}
    \item on pose $\Ob(C \jointc_1 D) = \Ob(C) \coprod \Ob(D)$;
    \item pour tous $x$ et $y$ dans $\Ob(C \jointc_1 D)$, on pose
      \[
        \Hom^{}_{C \jointc_1 D}(x, y) =
        \begin{cases}
          \Hom_C(x, y) & \text{si $x$ et $y$ sont dans $\Ob(C)$,} \\
          \Hom_D(x, y) & \text{si $x$ et $y$ sont dans $\Ob(D)$,} \\
          \ast & \text{si $x$ est dans $\Ob(C)$ et $y$ dans
          $\Ob(D)$,} \\
          \vide & \text{si $x$ est dans $\Ob(D)$ et $y$ dans $\Ob(C)$;}
        \end{cases}
      \]
    \item la composition et les identités sont définies de la manière évidente.
  \end{itemize}
  On obtient ainsi un foncteur $\jointc_1 : \Cat \times \Cat \to \Cat$. On
  vérifie que ce foncteur définit une structure de catégorie monoïdale sur
  $\Cat$ d'unité la catégorie initiale $\vide$. Fixons maintenant $A$ et $B$
  deux catégories. En utilisant le fait que $\vide$ est l'unité de la
  structure monoïdale, on obtient des foncteurs
  \[
    A \xto{\iota_1} A \jointc_1 B \xot{\iota_2} B,
  \]
  \notindex{$\iota_1 : A \to A \jointc_1 B$, $\iota_2 : B \to A \jointc_1 B$}%
  et donc des foncteurs
  \[
    \begin{split}
      \Cat & \to \cotr{\Cat}{A} \\
      B & \mapsto (A \jointc_1 B, \iota_1 : A \to A \jointc_1 B)
    \end{split}
  \]
  et
  \[
    \begin{split}
      \Cat & \to \cotr{\Cat}{B} \\
      A & \mapsto (A \jointc_1 B, \iota_2 : B \to A \jointc_1 B).
    \end{split}
  \]
  Ces foncteurs admettent des adjoints à droite qu'on appellera les
  \ndef[tranche!catégorique classique]{tranches catégoriques classiques}
  au-dessous et au-dessus. En particulier, le foncteur joint catégorique
  classique commute aux limites inductives connexes en chaque variable. Dans
  le cas où $A$ et $B$ sont la catégorie finale, on retrouve les tranches
  usuelles $\cotr{C}{c}$ et~$\tr{C}{c}$ au-dessous ou au-dessus d'un objet
  $c$ de $C$.
\end{paragr}

\begin{lemme}
  Soit $p \ge -1$. On a un isomorphisme canonique $\ti{1}(\On{p}) \simeq
  \Deltan{p}$.
\end{lemme}

\begin{proof}
  En vertu du théorème~\ref{thm:desc_orient}, on a $\On{p} \simeq \nu
  c(\Deltan{p})$ et donc, d'après la proposition~\ref{prop:tr_nu},
  $\ti{1}(\On{p}) \simeq \ti{1}\nu c(\Deltan{p}) \simeq
  \nu\ti{1}c(\Deltan{p})$. Par ailleurs, puisque $\Deltan{p}$ appartient à
  $\ThetaAug$, on a $\Deltan{p} \simeq \nu \lambda(\Deltan{p})$ en vertu de
  la proposition~\ref{prop:Theta_Steiner}. Pour conclure, il suffit donc de
  définir un isomorphisme $\ti{1}c(\Deltan{p}) \to \lambda(\Deltan{p})$,
  ce qui est immédiat.
\end{proof}

\begin{lemme}
  Pour tous $p, q \ge -1$, on a un isomorphisme canonique
  \[
    \Deltan{p} \joint_1 \Deltan{q} \simeq \Deltan{p} \jointc_1
    \Deltan{q},
  \]
  naturel en $\Deltan{p}$ et $\Deltan{q}$ dans $\cDeltaAug$.
\end{lemme}

\begin{proof}
  On vérifie immédiatement qu'on a $\Deltan{p} \jointc_1 \Deltan{q} \simeq
  \Deltan{p + 1 + q}$. Par ailleurs, on a
  \[
    \begin{split}
      \Deltan{p} \joint_1 \Deltan{q}
      & =
      \ti{1}(\Deltan{p} \joint \Deltan{q}) \\
      & \simeq
      \ti{1}(\ti{1}(\On{p}) \joint \ti{1}(\On{q})) \\*
      & \phantom{\simeq1} \text{(en vertu du lemme précédent)} \\
      & \simeq
      \ti{1}(\On{p} \joint \On{q}) \\
      & \phantom{\simeq1} \text{(en vertu de la
        proposition~\ref{prop:tronque_joint})} \\
      & \simeq
      \ti{1}(\On{p + 1 + q}) \\
      & \simeq
      \Deltan{p + 1 + q}, \\
    \end{split}
  \]
  d'où le résultat.
\end{proof}

\begin{prop}
  Le joint $1$-catégorique et le joint catégorique classique sont
  canoniquement isomorphes.
  Autrement dit, on a un isomorphisme canonique
  \[ C \joint_1 D \simeq C \jointc_1 D, \]
  naturel en $C$ et $D$ dans $\Cat$.
\end{prop}

\begin{proof}
  Le lemme précédent donne un isomorphisme canonique
  \[ C \joint_1 D \to C \jointc_1 D \]
  pour $C$ et $D$ dans $\cDeltaAug$. Or toute catégorie est limite inductive
  connexe d'objets de $\DeltaAug$ et les deux foncteurs commutent aux
  limites inductives connexes en chaque variable, d'où le résultat.
\end{proof}

\begin{coro}\label{coro:comp_tr_1}
  Si $C$ est une catégorie et $u : A \to C$ est un foncteur, alors
  les catégories $\cotr{C}{u}$ et \smash{$\trm{C}{u}$} sont les tranches
  catégoriques classiques.
\end{coro}

\begin{proof}
  En vertu du cas $n = 1$ de la proposition \ref{prop:tr_nCat_univ} et de la
  proposition précédente, les tranches $\cotr{C}{u}$ et \smash{$\trm{C}{u}$}
  vérifient les mêmes propriétés universelles que les tranches catégoriques
  classiques (voir le paragraphe~\ref{paragr:def_joint_cl}), d'où le
  résultat.
\end{proof}

\begin{rem}
  Si $C$ est une $1$-catégorie et $u : A \to C$ est un $1$-foncteur, alors
  la catégorie \smash{$\trm{C}{u}$} coïncide avec la \oo-catégorie $\tr{C}{u}$ de la
  remarque~\ref{rem:cojoint_ooCat} (qui est donc une catégorie). Cela résulte
  immédiatement du fait que $D^\co = D$ si $D$ est une $1$-catégorie.
\end{rem}

\chapter[Description explicite des tranches au-dessous d'un
objet]{Description explicite des\\ tranches au-dessous d'un objet}
\label{sec:desc_expl}

Soient $C$ une \oo-catégorie et $c$ un objet de $C$. En considérant $c$
comme un \oo-foncteur $\Dn{0} \to C$, le paragraphe~\ref{paragr:def_tranche}
permet de définir une \oo-catégorie $\cotr{C}{c}$. Le but de ce chapitre
est de décrire explicitement cette \oo-catégorie.

\medbreak

\emph{Dans ce chapitre, on utilisera librement les conventions sur les formules
de composition de cellules dans une \oo-catégorie qu'on a fixées au
paragraphe~\ref{paragr:conv_ooCat}, ainsi que les notations relatives aux
atomes des complexes dirigés augmentés à base unitaire introduites au
paragraphe~\ref{paragr:def_atome}.}

\begin{paragr}
  Fixons $i \ge 0$. Nous allons commencer par décrire la
  \hbox{$(i+1)$}\nbd-catégorie $\Dn{0} \joint \Dn{i}$. On notera $a$
  l'unique objet de $\Dn{0}$ et $x$ la cellule principale de $\Dn{i}$. En
  vertu de la proposition~\ref{prop:Theta_Steiner} et du
  paragraphe~\ref{paragr:def_joint}, on a
  \[
    \Dn{0} \joint \Dn{i} \simeq \nu(\lambda(\Dn{0}) \joint \lambda(\Dn{i})).
  \]
  Par ailleurs, les paragraphes~\ref{paragr:desc_lambda_Dn}
  et~\ref{paragr:base_joint} montrent que le complexe $\lambda(\Dn{0})
  \joint \lambda(\Dn{i})$ a pour base l'ensemble formé des
  \[ a \joint \vide, \quad \vide \joint x^\e_k, \quad a \joint x^\e_k, \]
  où $k$ varie entre $0$ et $i$, et $\e = 0, 1$ (en se souvenant que
  $x^0_i = x^1_i$).
\end{paragr}

\begin{prop}\label{prop:s_t_cone_vide}
  Pour tout $k$ tel que $0 < k \le i$ et $\e = 0, 1$, on a
  \[
    s(\atom{\vide \joint x^\e_k}) = \atom{\vide \joint x^0_{k-1}}
    \quadet
    t(\atom{\vide \joint x^\e_k}) = \atom{\vide \joint x^1_{k-1}}.
  \]
\end{prop}

\begin{proof}
  Cela résulte immédiatement de l'égalité
  \[
    d(\vide \joint x^\e_k) = \vide \joint d(x^\e_k)
      = \vide \joint x^1_{k-1} - \vide \joint x^0_{k-1}.
    \qedhere
  \]
\end{proof}

\begin{lemme}\label{lemme:tab_cone}
  Pour tout $k$ tel que $0 \le k \le i$, tout $l$ tel que $0 \le l < k + 1$
  et $\e = 0, 1$, on a
  \[
    \atom{a \joint x^\e_k}^0_l = a \joint x^1_{l-1}
  \]
  et
  \[
    \atom{a \joint x^\e_k}^1_l = \vide \joint x^\eta_l + a \joint x^0_{l-1},
  \]
  où $\eta$ vaut $\e$ si $k = l$ et $1$ sinon. En particulier, on a
  \[
    \atom{a \joint x^\e_k}^0_0 = a \joint \vide
    \quadet
    \atom{a \joint x^\e_k}^1_0 = \vide \joint x^\eta_0.
  \]
\end{lemme}

\begin{proof}
  En vertu du lemme~\ref{lemme:tab_joint} (et avec ses conventions), on a
  \[
    \atom{a \joint x^\e_k}^0_l
    =
    \atom{a}^0_{-1} \joint \atom{x^\e_k}^0_l +
    \atom{a}^0_0 \joint \atom{x^\e_k}^1_{l-1}
    =
    a \joint x^1_{l-1}
  \]
  et
  \[
    \atom{a \joint x^\e_k}^1_l
    =
    \atom{a}^1_{-1} \joint \atom{x^\e_k}^1_l +
    \atom{a}^1_0 \joint \atom{x^\e_k}^0_{l-1}
    =
    \vide \joint x^\eta_l + a \joint x^0_{l-1},
  \]
  ce qu'on voulait démontrer.
\end{proof}

\begin{prop}\label{prop:s_t_cone}
  Pour tout $k$ tel que $0 \le k \le i$ et $\e = 0, 1$,
  on a
  \[
    s(\atom{a \joint x^\e_k}) = \atom{a \joint x^1_{k-1}}
  \]
  et
  \[
    t(\atom{a \joint x^\e_k}) =
    \atom{\vide \joint x^\e_k} \ast_0 \atom{a \joint x^0_0} \ast_1 \cdots
    \ast_{k-1} \atom{a \joint x^0_{k-1}}.
  \]
\end{prop}

\begin{proof}
  C'est le cas $l = k$ du lemme plus général suivant.
\end{proof}

\begin{lemme}\label{lemme:s_t_cone_iter}
  Pour tous $k, l$ tels que $0 \le l \le k \le i$ et $\e = 0, 1$,
  on a
  \[
    s_l(\atom{a \joint x^\e_k}) = \atom{a \joint x^1_{l-1}}
  \]
  et
  \[
    t_l(\atom{a \joint x^\e_k}) =
    \atom{\vide \joint x^\eta_l} \ast_0 \atom{a \joint x^0_0} \ast_1 \cdots
    \ast_{l-1} \atom{a \joint x^0_{l-1}},
  \]
  où $\eta$ vaut $\e$ si $k = l$ et $1$ sinon.
\end{lemme}

\begin{proof}
  En vertu du lemme~\ref{lemme:tab_cone}, on a
  $\atom{a \joint x^\e_k}^0_l = a \joint x^1_{l-1}$ et donc
  $s_l(\atom{a \joint x^\e_k}) = \atom{a \joint x^1_{l-1}}$ par définition
  des atomes et de leurs sources.

  Démontrons la deuxième égalité par récurrence sur $l$. Pour $l =
  0$, en vertu du lemme~\ref{lemme:tab_cone}, on a
  $\atom{a \joint x^\e_k}^1_0 = \vide \joint x^\eta_0$
  et donc
  $t_0(\atom{a \joint x^\e_k}) = \atom{\vide \joint x^\eta_0}$.
  Supposons maintenant l'égalité démontrée au rang $l - 1$ et
  montrons-la au rang $l \le k$. Par la proposition~\ref{prop:s_t_cone_vide}
  et l'hypothèse de récurrence, on a
  \[
    \begin{split}
      \MoveEqLeft
      s_{l-1}\big(\atom{\vide \joint x^\eta_l} \ast_0 \atom{a \joint x^0_0}
      \ast_1 \cdots \ast_{l-2} \atom{a \joint x^0_{l-2}}\big) \\
      & =
      \atom{\vide \joint x^0_{l-1}} \ast_0 \atom{a \joint x^0_0}
      \ast_1 \cdots \ast_{l-2} \atom{a \joint x^0_{l-2}} \\
      & =
      t_{l-1}(\atom{a \joint x^0_{l-1}})
    \end{split}
  \]
  et la cellule
  \[
    u = \big(\atom{\vide \joint x^\eta_l} \ast_0 \atom{a \joint x^0_0}
    \ast_1 \cdots \ast_{l-2} \atom{a \joint x^0_{l-2}}\big)
    \ast_{l-1} \atom{a \joint x^0_{l-1}}
  \]
  est donc bien définie. Par ailleurs, en utilisant le
  lemme~\ref{lemme:tab_cone}, on obtient
  \[
    u^{}_l = \vide \joint x^\eta_l + a \joint x^0_{l-1} = \atom{a \joint
    x^\e_k}^1_l = t_l(\atom{a \joint x^\e_k})^{}_l.
  \]
  Pour conclure, il suffit donc de montrer qu'on a
  \[
    s(t_l(\atom{a \joint x^\e_k})) = s(u)
    \quadet
    t(t_l(\atom{a \joint x^\e_k})) = t(u).
  \]
  Or, en utilisant la première égalité de l'énoncé, on a
  \[
    \begin{split}
      s(t_l(\atom{a \joint x^\e_k}))
      =
      s_{l-1}(\atom{a \joint x^\e_k})
      =
      \atom{a \joint x^1_{l - 2}}
      =
      s_{l-1}(\atom{a \joint x^0_{l-1}})
      =
      s(u)
    \end{split}
  \]
  et, en utilisant l'hypothèse de récurrence, on a
  {
    \allowdisplaybreaks
    \begin{align*}
      t(t_l(\atom{a \joint x^\e_k}))
      & =
      t_{l-1}(\atom{a \joint x^\e_k}) \\*
      & =
      \atom{\vide \joint x^1_{l-1}} \ast_0 \atom{a \joint x^0_0} \ast_1 \cdots
      \ast_{l-2} \atom{a \joint x^0_{l-2}} \\
      & =
      t_{l-1}\big(
      \atom{\vide \joint x^\eta_l} \ast_0 \atom{a \joint x^0_0} \ast_1 \cdots
      \ast_{l-2} \atom{a \joint x^0_{l-2}} \joint_{l-1} \atom{a \joint
      x^0_{l-1}}\big) \\*
      & =
      t(u),
    \end{align*}
  }
  d'où le résultat.
\end{proof}

\begin{prop}\label{prop:desc_tr_pol}
  Soit $C$ une \oo-catégorie. Fixons
  \begin{itemize}
    \item $c$ un objet de $C$ ;
    \item $d$ une $i$-flèche de $C$ ;
  \item pour tout $k$ tel que $0 \le k \le i$ et $\e = 0, 1$, $\alpha^\e_k$
    une $(k+1)$-flèche de $C$, avec~$\alpha^0_i = \alpha^1_i$,
  \end{itemize}
  vérifiant
  les égalités
  \begin{equation*}
    \label{eq:s_t_cone}
    s(\alpha^\e_k) =
      \begin{cases}
         c & \text{si $k = 0$,} \\
        \alpha^1_{k-1} & \text{si $k > 0$,} \\
      \end{cases}
    \quadet
    t(\alpha^\e_k) = d^\e_k \ast_0 \alpha^0_0 \ast_1 \dots \ast_{k-1}
    \alpha^0_{k-1},
  \end{equation*}
  où on a posé
  \[
    d^\e_k =
      \begin{cases}
        s_k(d) & \text{si $\e = 0$,} \\
        t_k(d) & \text{si $\e = 1$.}
      \end{cases}
  \]
  Alors il existe un et un seul \oo-foncteur $h : \Dn{0} \joint \Dn{i} \to C$
  tel que
  \[
    c = h(\atom{a \joint \vide}),
    \quad
    d = h(\atom{\vide \joint x_i})
    \quadet
    \alpha^\e_k = h(\atom{a \joint x^\e_k}),
  \]
  pour tout $k$ tel que $0 \le k \le i$ et $\e = 0, 1$.
\end{prop}

\begin{proof}
  En vertu du théorème~\ref{thm:Steiner_pol}, la \oo-catégorie
  $\nu(\lambda(\Dn{0}) \joint \lambda(\Dn{i}))$, isomorphe à la
  \oo-catégorie $\Dn{0} \joint \Dn{i}$, est engendrée librement par ses
  atomes au sens des polygraphes. Cela signifie que la donnée d'un
  \oo-foncteur $h : \Dn{0} \joint \Dn{i} \to C$ est équivalente à donnée de
  \[
    c = h(\atom{a \joint \vide}), \quad
    d^\e_k = h(\atom{\vide \joint x^\e_k}) \quadet
    \alpha^\e_k = h(\atom{a \joint x^\e_k}),
  \]
  pour $0 \le k \le i$ et $\e = 0, 1$, avec $d^0_i = d^1_i$ et $\alpha^0_i =
  \alpha^1_i$, compatibles aux sources et aux buts. En vertu de la
  proposition~\ref{prop:s_t_cone}, cette compatibilité pour les atomes de la
  forme~$\atom{a \joint x^\e_k}$ s'exprime par les égalités de l'énoncé.
  Pour les atomes de la forme $\vide \joint x^\e_k$ avec $0 < k \le i$, en
  vertu de la proposition~\ref{prop:s_t_cone_vide}, ces compatibilités sont
  données par les égalités
  \[ s(d^\e_k) = d^0_{k-1} \quadet t(d^\e_k) = d^1_{k-1} \]
  Or, la donnée de $d^\e_k$ satisfaisant à ces égalités est équivalente à
  celle de \hbox{$d = d^0_i = d^1_i$}, d'où le résultat.
\end{proof}

\begin{paragr}\label{paragr:desc_tr_pol}
  Soient $C$ une \oo-catégorie et $c$ un objet de $C$. Par définition (voir
  le paragraphe~\ref{paragr:desc_tr}), les $i$-flèches de $\cotr{C}{c}$
  correspondent aux \oo-foncteurs $h : \Dn{0} \joint \Dn{i} \to C$ rendant
  le triangle
  \[
    \xymatrix{
    \Dn{0} \joint \Dn{i} \ar[r]^h & C \\
    \Dn{0} \ar[u]^{\iota_1} \ar[ru]_c
    }
  \]
  commutatif. En vertu de la proposition précédente, un tel
  \oo-foncteur est déterminé par un couple $(d, \alpha)$, où $d$ est une
  $i$-flèche de $C$ et $\alpha$ est une famille de
  cellules~$\alpha^\e_k$ de~$C$, pour $0 \le k \le i$ et $\e = 0,
  1$, avec $\alpha^0_i = \alpha^1_i$,
\[
    \begin{split}
      \alpha^\e_0 & : c \to d^\e_0, \quad\text{$1$-flèche}, \\
      \alpha^\e_k & : \alpha^1_{k-1} \to d^\e_k \comp_0 \alpha^0_0 \comp_1 \cdots
      \comp_{k-1} \alpha^0_{k-1}, \quad\text{$(k+1)$-flèche},
      \quad\text{pour $0 < k \le i$},
    \end{split}
  \]
  où on a posé
  \[
    d^\e_k =
      \begin{cases}
        s_k(d) & \text{si $\e = 0$,} \\
        t_k(d) & \text{si $\e = 1$.}
      \end{cases}
  \]
  Dans la suite de ce chapitre, on identifiera les $i$-flèches de
  $\cotr{C}{c}$ avec de tels couples~$(d, \alpha)$. On notera $\alpha_i$
  pour $\alpha^0_i = \alpha^1_i$.

  Voici une représentation graphique des objets, des $1$-flèches et des
  $2$-flèches de $\cotr{C}{c}$ :
 \[
   \xymatrix@R=3pc{
   c \ar[d]_{\alpha_0} \\
   d \pbox{,}
   }
   \qquad
   \qquad
    \shorthandoff{;}
    \xymatrix@C=1.5pc@R=3pc{
      & c
    \ar[dl]_{\alpha^0_0}_{}="f" \ar[dr]^{\alpha^1_0}_{}="s" \\
      d^0_0 \ar[rr]_d & & d^1_0
      \ar@{}"s";[ll]_(.15){}="ss"
      \ar@{}"s";[ll]_(.55){}="tt"
      \ar@<0.0ex>@2"ss";"tt"_{\alpha_1} \pbox{,}
    }
   \qquad
   \qquad
    \shorthandoff{:;}
    \xymatrix@C=1.5pc@R=3pc{
      & c
    \ar[dl]_{\alpha^0_0}_{}="f" \ar[dr]^{\alpha^1_0}_{}="s" \\
      d^0_0
      \ar@/^2ex/@{.>}[rr]_(.30){d^0_1}^{}="0"
      \ar@/_2ex/[rr]_(.33){d^1_1}^{}="1"
      \ar@{:>}"0";"1"^{\,d}
      & & d^1_0
      \ar@{}"s";[ll]_(.15){}="ss"
      \ar@{}"s";[ll]_(.55){}="tt"
      \ar@<-1.5ex>@/^-1ex/@{:>}"ss";"tt"_(.30){\alpha^0_1}_{}="11"
      \ar@<-0ex>@/^1ex/@2"ss";"tt"^(.20){\!\!\alpha^1_1}^{}="00"
      \ar@3"00";"11"_{\alpha_2}
      \pbox{.}
    }
  \]
\end{paragr}

Le but de la suite de ce chapitre est de décrire la structure de
\oo-catégorie de~$\cotr{C}{c}$ en termes des $(d, \alpha)$.

\begin{paragr}
  En vertu de l'adjonction définissant la \oo-catégorie $\cotr{C}{c}$,
  les sources, buts, identités et compositions de cette \oo-catégorie sont
  induits par les \oo-foncteurs
  \begin{alignat*}{3}{}
    \Dn{0} \joint \sigma_i & : \Dn{0} \joint \Dn{i-1} \to \Dn{0} \joint \Dn{i}
             && \qquad\text{pour $i \ge 1$,} \\
    \Dn{0} \joint \tau_i & : \Dn{0} \joint \Dn{i-1} \to \Dn{0} \joint \Dn{i}
             && \qquad\text{pour $i \ge 1$,} \\
    \Dn{0} \joint \kappa_i & : \Dn{0} \joint \Dn{i+1} \to \Dn{0} \joint \Dn{i}
             && \qquad\text{pour $i \ge 0$,} \\
    \Dn{0} \joint \nabla^i_j & : \Dn{0} \joint \Dn{i} \to (\Dn{0} \joint
    \Dn{i}) \amalg_{\Dn{0} \joint \Dn{j}} (\Dn{0} \joint \Dn{i})
             && \qquad\text{pour $i > j \ge 0$,}
  \end{alignat*}
  où  $\sigma_i$, $\tau_i$, $\kappa_i$ et $\nabla^i_j$ désignent les
  \oo-foncteurs des paragraphes~\ref{paragr:def_disque} et~\ref{paragr:def_kappa_nabla}, et
  où on a identifié $(\Dn{0} \joint \Dn{i}) \amalg_{\Dn{0} \joint \Dn{j}}
  (\Dn{0} \joint \Dn{i}$) et $\Dn{0} \joint (\Dn{i} \amalg_{\Dn{j}} \Dn{i})$.
\end{paragr}

Nous allons commencer par décrire concrètement les morphismes $\Dn{0} \joint \sigma_i$,
$\Dn{0} \joint \tau_i$ et $\Dn{0} \joint \kappa_i$. On note toujours $a$
l'unique objet de $\Dn{0}$.

\begin{prop}\label{prop:desc_tr_s_t}
  Fixons $i \ge 1$ et notons $x$ la cellule principale de $\Dn{i-1}$ et $y$
  celle de $\Dn{i}$. Alors le \oo-foncteur $\Dn{0} \joint \sigma_i : \Dn{0}
  \joint \Dn{i-1} \to \Dn{0} \joint \Dn{i}$ est donné par
  \begin{alignat*}{3}{}
    \atom{a \joint \vide} & \mapsto \atom{a \joint \vide}, \\
    \atom{\vide \joint x^\e_k} & \mapsto \atom{\vide \joint y^\e_k} &&
      \qquad\text{pour $0 \le k < i - 1$ et $\e = 0, 1$,}\\
    \atom{\vide \joint x_{i-1}} & \mapsto \atom{\vide \joint
      y^0_{i-1}}, \\
    \atom{a \joint x^\e_k} & \mapsto \atom{a \joint y^\e_k} &&
      \qquad\text{pour $0 \le k < i - 1$ et $\e = 0, 1$,}\\
    \atom{a \joint x_{i-1}} & \mapsto \atom{a \joint y^0_{i-1}}.
  \end{alignat*}
  De même, le \oo-foncteur $\Dn{0} \joint \tau_i : \Dn{0} \joint \Dn{i-1} \to \Dn{0} \joint
  \Dn{i}$ est donné par
  \begin{alignat*}{3}{}
    \atom{a \joint \vide} & \mapsto \atom{a \joint \vide}, \\
    \atom{\vide \joint x^\e_k} & \mapsto \atom{\vide \joint y^\e_k} &&
      \qquad\text{pour $0 \le k < i - 1$ et $\e = 0, 1$,}\\
    \atom{\vide \joint x_{i-1}} & \mapsto \atom{\vide \joint
      y^1_{i-1}}, \\
    \atom{a \joint x^\e_k} & \mapsto \atom{a \joint y^\e_k} &&
      \qquad\text{pour $0 \le k < i - 1$ et $\e = 0, 1$,}\\
    \atom{a \joint x_{i-1}} & \mapsto \atom{a \joint y^1_{i-1}}.
  \end{alignat*}
\end{prop}

\begin{proof}
  Cela résulte immédiatement de la proposition~\ref{prop:joint_morph_atom}
  et des formules
  \[
     \sigma_i(\atom{x^\e_k}) =
     \begin{cases}
       \atom{y^\e_k} & \text{si $0 \le k < i - 1$,} \\
       \atom{y^0_{i-1}} & \text{si $k = i - 1$,}
     \end{cases}
     \quadet
     \tau_i(\atom{x^\e_k}) =
     \begin{cases}
       \atom{y^\e_k} & \text{si $0 \le k < i - 1$,} \\
       \atom{y^1_{i-1}} & \text{si $k = i - 1$,}
     \end{cases}
   \]
   pour $\e = 0,1$.
\end{proof}

\begin{prop}\label{prop:desc_tr_k}
  Fixons $i \ge 0$ et notons $x$ la cellule principale de $\Dn{i+1}$ et $y$
  celle de $\Dn{i}$. Alors le \oo-foncteur
  $\Dn{0} \joint \kappa_i : \Dn{0} \joint \Dn{i+1} \to \Dn{0} \joint \Dn{i}$
  est donné par
  \begin{alignat*}{3}{}
    \atom{a \joint \vide} & \mapsto \atom{a \joint \vide}, \\
    \atom{\vide \joint x^\e_k} & \mapsto \atom{\vide \joint y^\e_k} &&
      \qquad\text{pour $0 \le k < i$ et $\e = 0, 1$,}\\
    \atom{\vide \joint x^\e_i} & \mapsto \atom{\vide \joint y_i} &&
      \qquad\text{pour $\e = 0, 1$,}\\
    \atom{\vide \joint x_{i+1}} & \mapsto \id{\atom{\vide \joint
      y_i}}, \\
    \atom{a \joint x^\e_k} & \mapsto \atom{a \joint y^\e_k} &&
      \qquad\text{pour $0 \le k < i$ et $\e = 0, 1$,}\\
    \atom{a \joint x^\e_i} & \mapsto \atom{a \joint y_i} &&
      \qquad\text{pour $\e = 0, 1$,}\\
    \atom{a \joint x_{i+1}} & \mapsto \id{\atom{a \joint y_i}}.
  \end{alignat*}
\end{prop}

\begin{proof}
  Cela résulte immédiatement de la proposition~\ref{prop:joint_morph_atom}
  et de la formule
  \[
    \kappa_i(\atom{x^\e_k}) =
      \begin{cases}
        \atom{y^\e_k} & \text{si $0 \le k < i$,} \\
        \atom{y_i} & \text{si $k = i$,} \\
        \id{\atom{y_i}} & \text{si $k = i+1$,}
      \end{cases}
  \]
  pour $\e = 0,1$.
\end{proof}

\begin{paragr}
  Fixons maintenant $j$ tel que $0 \le j < i$. Nous allons décrire
  explicitement le \oo-foncteur
  \[
    \Dn{0} \joint \nabla^i_j : \Dn{0} \joint \Dn{i} \to \Dn{0} \joint \Dn{i}
    \amalg_{\Dn{0} \joint \Dn{j}} \Dn{0} \joint \Dn{i}.
  \]
  Nous noterons $x$ la cellule principale de l'objet $\Dn{i}$
  apparaissant dans la source de~$\Dn{0} \joint \nabla^i_j$, et $y$ et $z$
  les cellules principales des objets $\Dn{i}$ apparaissant de gauche à
  droite dans le but de $\Dn{0} \joint \nabla^i_j$.

  En vertu du paragraphe \ref{paragr:desc_lambda_cocat} et avec ses
  notations, une base du complexe dirigé augmenté
  \[
    \lambda\big(\Dn{0} \joint \Dn{i}
    \amalg_{\Dn{0} \joint \Dn{j}} \Dn{0} \joint \Dn{i}\big)
    \simeq
    \lambda(\Dn{0}) \joint \lambda(\Dn{i} \amalgDn{j} \Dn{i})
  \]
  est donnée par les
  \[
    \atom{a \joint \vide}, \quad
    \atom{\vide \joint y^\e_i}, \quad
    \atom{\vide \joint z^\e_i}, \quad
    \atom{a \joint y^\e_i}, \quad
    \atom{a \joint z^\e_i},
  \]
  pour $0 \le k \le i$ et $\e = 0, 1$, modulo les identifications
  \[
      y^0_j = z^1_j, \quad
      y^\e_k = z^\e_k \quad\text{pour $0 \le k < j$ et $\e = 0, 1$,} \\
  \]
  ainsi que les identifications triviales $y^0_i = y^1_i$ et
  $z^0_i = z^1_i$.
\end{paragr}

\begin{lemme}
  Pour tout $k$ tel que $j < k \le i$ et $\e = 0, 1$, la $(k+1)$-flèche
  \[
    \atom{\vide \joint y^\eta_{j+1}} \ast_0 \atom{a \joint z^0_0} \ast_1 \cdots
    \ast_{j-1} \atom{a \joint z^0_{j-1}} \ast_j \atom{a \joint z^\e_k}
    \ast_{j+1} \atom{a \joint y^\e_k}
  \]
  de $\Dn{0} \joint \Dn{i} \amalg_{\Dn{0} \joint \Dn{j}} \Dn{0} \joint
  \Dn{i}$, où $\eta$ vaut $\e$ si $k = j + 1$ et $1$ sinon, est bien
  définie.
\end{lemme}

\begin{proof}
  En vertu de la proposition~\ref{prop:s_t_cone}, la $j$-flèche
  \[
    \atom{\vide \joint z^1_j} \ast_0 \atom{a \joint z^0_0} \ast_1
    \cdots \ast_{j-1} \atom{a \joint z^0_{j-1}}
  \]
  est bien définie.
  Puisque, pour $l \le j - 1$, on a
  \[
    s_l(\atom{\vide \joint y^\eta_{j+1}})
    =
    \atom{\vide \joint y^0_l}
    =
    \atom{\vide \joint z^0_l}
    =
    s_l(\atom{\vide \joint z^1_j}),
  \]
  on en déduit que la cellule
  \[
  \atom{\vide \joint y^\eta_{j+1}} \ast_0 \atom{a \joint z^0_0} \ast_1
  \cdots \ast_{j-1} \atom{a \joint z^0_{j-1}}
  \]
  est également bien définie. De plus, on a
  \[
    \begin{split}
     \MoveEqLeft
      s_j\big(
        \atom{\vide \joint y^\eta_{j+1}} \ast_0 \atom{a \joint z^0_0} \ast_1
      \cdots \ast_{j-1} \atom{a \joint z^0_{j-1}}\big) \\
      & =
        \atom{\vide \joint y^0_j} \ast_0 \atom{a \joint z^0_0} \ast_1
      \cdots \ast_{j-1} \atom{a \joint z^0_{j-1}} \\
      & =
        \atom{\vide \joint z^1_j} \ast_0 \atom{a \joint z^0_0} \ast_1
      \cdots \ast_{j-1} \atom{a \joint z^0_{j-1}} \\
      & =
      t_j(\atom{a \joint z^\e_k}),
    \end{split}
  \]
  la dernière égalité résultant du lemme~\ref{lemme:s_t_cone_iter},
  ce qui montre que la cellule
  \[
    \atom{\vide \joint y^\eta_{j+1}} \ast_0 \atom{a \joint z^0_0} \ast_1 \cdots
    \ast_{j-1} \atom{a \joint z^0_{j-1}} \ast_j \atom{a \joint z^\e_k}
  \]
  est bien définie. Enfin, on a
  {
    \allowdisplaybreaks
    \begin{align*}
     \MoveEqLeft
     s_{j+1}\big(
        \atom{\vide \joint y^\eta_{j+1}} \ast_0 \atom{a \joint z^0_0} \ast_1
        \cdots \ast_{j-1} \atom{a \joint z^0_{j-1}} \ast_j \atom{a \joint
        z^\e_k}\big) \\*
        & =
        \atom{\vide \joint y^\eta_{j+1}} \ast_0 \atom{a \joint z^0_0} \ast_1
        \cdots \ast_{j-1} \atom{a \joint z^0_{j-1}}
        \ast_j s_{j+1}(\atom{a \joint z^\e_k}) \\
      & =
        \atom{\vide \joint y^\eta_{j+1}} \ast_0 \atom{a \joint z^0_0} \ast_1
        \cdots \ast_{j-1} \atom{a \joint z^0_{j-1}} \ast_j \atom{a \joint
        z^1_j} \\*
      & \phantom{=1} \text{(en vertu du lemme~\ref{lemme:s_t_cone_iter})} \\
      & =
        \atom{\vide \joint y^\eta_{j+1}} \ast_0 \atom{a \joint y^0_0} \ast_1
        \cdots \ast_{j-1} \atom{a \joint y^0_{j-1}} \ast_j \atom{a \joint
        y^0_j} \\
      & =
        t_{j+1}(\atom{a \joint y^\e_k}),
    \end{align*}
  }%
  la dernière égalité résultant de nouveau du
  lemme~\ref{lemme:s_t_cone_iter}, ce qui achève de montrer que la cellule de
  l'énoncé est bien définie.
\end{proof}

\begin{lemme}\label{lemme:s_t_nabla}
  Pour tout $k$ tel que $j < k \le i$, $\e = 0, 1$ et pour tout $l$ tel
  que $0 \le l \le k + 1$, on a
  \[
    (\Dn{0} \joint \nabla^i_j)(\atom{a \joint x^\e_k})^0_l =
      \begin{cases}
        a \joint y^1_{l-1} & \text{si $0 \le l \le j+1$,} \\
        a \joint y^1_{l-1} + a \joint z^1_{l-1}
          & \text{si $j + 1 < l < k + 1$}, \\
        a \joint y^\e_{l-1} + a \joint z^\e_{l-1}
          & \text{si $l = k + 1$},
      \end{cases}
  \]
  et
  \[
    (\Dn{0} \joint \nabla^i_j)(\atom{a \joint x^\e_k})^1_l =
      \begin{cases}
        \vide \joint y^1_l + a \joint z^0_{l-1} &
          \text{si $0 \le l \le j$,} \\
        \vide \joint y^\eta_l + \vide \joint z^\eta_l + a \joint z^0_{l-1} &
          \text{si $l = j+1$,} \\
        \vide \joint y^\eta_l + \vide \joint z^\eta_l +
        a \joint y^0_{l-1} + a \joint z^0_{l-1} &
          \text{si $j+1 < l < k+1$,} \\
        a \joint y^\e_{l-1} + a \joint z^\e_{l-1} &
          \text{si $l = k+1$,} \\
      \end{cases}
  \]
  où $\eta$ vaut $\e$ si $l = k$ et $1$ sinon.
\end{lemme}

\begin{proof}
  Montrons la première égalité. On a, en utilisant le
  lemme~\ref{lemme:tab_cone},
  \[
    \begin{split}
    (\Dn{0} \joint \nabla^i_j)(\atom{a \joint x^\e_k})^0_l
    & =
    \lambda(\Dn{0} \joint \nabla^i_j)(\atom{a \joint x^\e_k}^0_l) \\
    & =
    \lambda(\Dn{0} \joint \nabla^i_j)(a \joint x^\delta_{l-1}) \\
    & =
    a \joint \lambda(\nabla^i_j)(x^\delta_{l-1}),
    \end{split}
  \]
  où $\delta$ vaut $\e$ si $l = k + 1$ et $1$ sinon, et on obtient l'égalité
  par la description de $\lambda(\nabla^i_j)$ donnée au
  paragraphe~\ref{paragr:desc_lambda_cocat}.

  Montrons la seconde. Le cas $l = k + 1$ s'obtient comme ci-dessus. Pour $l
  < k + 1$, en utilisant de nouveau le lemme~\ref{lemme:tab_cone}, on a
  \[
    \begin{split}
    (\Dn{0} \joint \nabla^i_j)(\atom{a \joint x^\e_k})^1_l
    & =
    \lambda(\Dn{0} \joint \nabla^i_j)(\atom{a \joint x^\e_k}^1_l) \\
    & =
    \lambda(\Dn{0} \joint \nabla^i_j)(\vide \joint x^\eta_l + a \joint x^0_{l-1}) \\
    & =
    \vide \joint \lambda(\nabla^i_j)(x^\eta_l) + a \joint
      \lambda(\nabla^i_j)(x^0_{l-1})
    \end{split}
  \]
  et on obtient le résultat en utilisant de nouveau la description de
  $\lambda(\nabla^i_j)$ donnée au paragraphe~\ref{paragr:desc_lambda_cocat}.
\end{proof}

\begin{prop}\label{prop:desc_tr_nabla}
  Le \oo-foncteur $\Dn{0} \joint \nabla^i_j : \Dn{0} \joint \Dn{i} \to \Dn{0} \joint \Dn{i}
  \amalg_{\Dn{0} \joint \Dn{j}} \Dn{0} \joint \Dn{i}$ est donné par
  {
    \allowdisplaybreaks
  \begin{alignat*}{3}{}
    \atom{a \joint \vide} & \mapsto \atom{a \joint \vide}, \\
    \atom{\vide \joint x^0_k} & \mapsto \atom{\vide \joint z^0_k}
                                   && \qquad\text{pour $0 \le k \le j$,} \\
    \atom{\vide \joint x^1_k} & \mapsto \atom{\vide \joint y^1_k}
                                   && \qquad\text{pour $0 \le k \le j$,} \\
    \atom{\vide \joint x^\e_k} & \mapsto \atom{\vide \joint y^\e_k} \ast_j
                                   \atom{\vide \joint z^\e_k}
                                   && \qquad\text{pour $j < k \le i$ et
                                      $\e = 0, 1$,} \\
    \atom{a \joint x^0_k} & \mapsto \atom{a \joint z^0_k}
                                   && \qquad\text{pour $0 \le k \le j$,} \\
    \atom{a \joint x^1_k} & \mapsto \atom{a \joint y^1_k}
                                   && \qquad\text{pour $0 \le k \le j$,} \\
    \atom{a \joint x^\e_k} & \mapsto u^\e_k && \qquad\text{pour $j < k \le i$
                                          et $\e = 0,1$},
  \end{alignat*}
  }
  où
  \[
    u^\e_k =
    \atom{\vide \joint y^\eta_{j+1}} \ast_0 \atom{a \joint z^0_0} \ast_1 \cdots
    \ast_{j-1} \atom{a \joint z^0_{j-1}} \ast_j \atom{a \joint z^\e_k}
    \ast_{j+1} \atom{a \joint y^\e_k},
  \]
  avec $\eta$ valant $\e$ si $k = j+1$ et $1$ sinon.
\end{prop}

\begin{proof}
  Le cas des atomes de la forme $\atom{\vide \joint x^\e_k}$ résulte de la
  naturalité de~$\iota_2$ et, plus précisément, de la commutativité du
  carré
  \[
    \xymatrix@C=3.5pc{
      \Dn{0} \joint \Dn{i} \ar[r]^-{\Dn{0} \joint \nabla^i_j} &
    \Dn{0} \joint \big(\Dn{i} \amalg_{\Dn{j}} \Dn{i}\big) \\
      \Dn{i} \ar[u]^{\iota_2} \ar[r]_-{\nabla^i_j} &
    \Dn{i} \amalgDn{j} \Dn{i} \ar[u]_{\iota_2} \pbox{,}
    }
  \]
  ainsi que de la description explicite du \oo-foncteur $\nabla^i_j : \Dn{i}
  \to \Dn{i} \amalgDn{j} \Dn{i}$ (voir le
  paragraphe~\ref{paragr:def_kappa_nabla}).

  Le cas de l'atome $\atom{a \joint \vide}$ et des atomes de la forme
  $\atom{a \joint x^\e_k}$ avec $0 \le k \le j$ est conséquence de la
  proposition~\ref{prop:joint_morph_atom} et des formules
  \[
    \nabla^i_j(\atom{x^0_k}) = \atom{z^0_k}
      \quadet
    \nabla^i_j(\atom{x^1_k}) = \atom{y^1_k},
  \]
  pour $0 \le k \le j$.

  Enfin, traitons le cas des atomes de la forme $\atom{a \joint x^\e_k}$
  avec $k > j$.  Soient $l$ tel que $0 \le l \le k+1$ et $\ep = 0,1$. Il
  s'agit de montrer l'égalité $(\Dn{0} \joint \nabla^i_j)(\atom{a \joint x^\e_k})^\ep_l =
  (u^\e_k)^\ep_l$.  Le membre de gauche a été calculé dans le
  lemme~\ref{lemme:s_t_nabla}.  Calculons celui de droite.

  Si $l \le j+1$ (et donc $l \le k$), on a, en utilisant le
  lemme~\ref{lemme:tab_cone},
   \[
      (u^\e_k)^0_l = (s_l(u^\e_k))_l = \atom{a \joint y^\e_k}^0_l
      = a \joint y^1_{l-1}.
  \]
  Si $j + 1 <  l \le k + 1$, on a
  \[
    (u^\e_k)^\ep_l = \atom{a \joint z^\e_k}^\ep_l + \atom{a \joint
    y^\e_k}^\ep_l
  \]
  et donc, pour $l = k + 1$,
  \[
    (u^\e_k)^\ep_{l} = a \joint z^\e_k + a \joint y^\e_k
  \]
  et, pour $j + 1 < l < k + 1$, en vertu du lemme~\ref{lemme:tab_cone},
  \[
    (u^\e_k)^0_l = a \joint z^1_{l-1} + a \joint y^1_{l-1}
    \quadet
    (u^\e_k)^1_l = \vide \joint z^\eta_l + a \joint z^0_{l-1} +
    \vide \joint y^\eta_l + a \joint y^0_{l-1}.
  \]
  Si $l = j+1$, on a
  \[
    t_l(u^\e_k) = t_l\big(\atom{\vide \joint y^\eta_{j+1}}
      \ast_0 \atom{a \joint z^0_0} \ast_1 \cdots \ast_{j-1}
      \atom{a \joint z^0_{j-1}} \ast_j \atom{a \joint z^\e_k}\big)
  \]
  et donc, toujours en utilisant le lemme~\ref{lemme:tab_cone},
  \[
    (u^\e_k)^1_l
    = \atom{\vide \joint y^\eta_{j+1}}^1_l + \atom{a \joint z^\e_k}^1_l
    = \vide \joint y^\eta_l + \vide \joint z^\eta_l + a \joint z^0_{l-1}.
  \]
  Enfin, si $l < j + 1$, on a
  \[
    t_l(u^\e_k) = t_l\big(\atom{\vide \joint y^\eta_{j+1}}
      \ast_0 \atom{a \joint z^0_0} \ast_1 \cdots \ast_{l-1}
      \atom{a \joint z^0_{l-1}}\big)
  \]
  et donc
  \[
    (u^\e_k)^1_l
    = \atom{\vide \joint y^\eta_{j+1}}^1_l + \atom{a \joint z^0_{l-1}}^1_l
    = \vide \joint y^1_l + a \joint z^0_{l-1}.
  \]
  On a bien retrouvé dans tous les cas la valeur de $(\Dn{0} \joint
  \nabla^i_j)(\atom{a \joint x^\e_k})^\ep_l$ obtenue dans le
  lemme~\ref{lemme:s_t_nabla}, ce qui achève la démonstration.
\end{proof}

\begin{prop}\label{prop:desc_tr}
  Soient $C$ une \oo-catégorie et $c$ un objet de $C$. Fixons
  une $i$-flèche $(d, \alpha)$ de $\cotr{C}{c}$.
  \begin{enumerate}
    \item Si $i \ge 1$, on a $s(d, \alpha) = (s(d), \gamma)$, où
      \begin{alignat*}{3}{}
        \gamma^\e_k & = \alpha^\e_k
                   && \qquad\text{pour $0 \le k < i - 1$ et $\e = 0,1$,} \\
        \gamma_{i-1} & = \alpha^0_{i-1}.
      \end{alignat*}
    \item Si $i \ge 1$, on a $t(d, \alpha) = (t(d), \gamma)$, où
      \begin{alignat*}{3}{}
        \gamma^\e_k & = \alpha^\e_k
                   && \qquad\text{pour $0 \le k < i - 1$ et $\e = 0,1$,} \\
        \gamma_{i-1} & = \alpha^1_{i-1}.
      \end{alignat*}
    \item Si $i \ge 0$, on a $\id{(d, \alpha)} = (\id{d}, \gamma)$, où
      \begin{alignat*}{3}{}
        \gamma^\e_k & = \alpha^\e_k
                   && \qquad\text{pour $0 \le k < i$ et $\e = 0,1$,} \\
        \gamma^\e_i & = \alpha_i
                   && \qquad\text{pour $\e = 0, 1$,} \\
        \gamma_{i+1} & = \id{\alpha_i}.
      \end{alignat*}
  \end{enumerate}
  Soit $(e, \beta)$ une seconde $i$-flèche de $\cotr{C}{c}$.
  \begin{enumerate}[resume]
    \item Si $(d, \alpha)$ et $(e, \beta)$ sont $j$-composables pour un $j$
      tel que $0 \le j < i$, alors on a $(d, \alpha) \comp_j (e, \beta) = (d
      \comp_j e, \gamma)$, où
      \begin{alignat*}{3}{}
        \gamma^0_k & = \beta^0_k
                   && \qquad\text{pour $0 \le k \le j$,} \\
        \gamma^1_k & = \alpha^1_k
                   && \qquad\text{pour $0 \le k \le j$}
      \end{alignat*}
      et
      \[
        \gamma^\e_k = d^\eta_{j+1} \ast_0 \beta^0_0 \ast_1 \cdots
        \ast_{j-1} \beta^0_{j-1} \ast_j \beta^\e_k
        \ast_{j+1} \alpha^\e_k
      \]
      pour $j < k \le i$ et $\e = 0, 1$, où $\eta$ vaut $\e$ si $k = j + 1$
      et $1$ sinon.
  \end{enumerate}
\end{prop}

\begin{proof}
  Ces formules sont la traduction, à travers la bijection de la
  proposition~\ref{prop:desc_tr_pol} et du
  paragraphe~\ref{paragr:desc_tr_pol}, des formules obtenues dans les
  propositions~\ref{prop:desc_tr_s_t}, \ref{prop:desc_tr_k}
  et~\ref{prop:desc_tr_nabla}.
\end{proof}

\begin{rem}
  Il résulte de la description de $\cotr{C}{c}$ obtenue dans la proposition
  précédente que cette \oo-catégorie est isomorphe à une sous-\oo-catégorie
  pleine de la \oo-catégorie~$HC$ des cylindres dans $C$ introduite par
  Métayer dans~\cite{MetPolRes} (voir notre remarque~\ref{rem:HC} pour une
  définition abstraite de $HC$). Plus précisément, cette \oo-catégorie est
  la fibre en $c$ du \oo-foncteur $HC \to C$ qui envoie un cylindre sur sa
  « cellule source ». La \oo-catégorie $\cotr{C}{c}$ apparaît également dans
  \cite{ALM} où elle est notée~$\Lambda(C, c)$.
\end{rem}

\begin{rem}
  Si $C$ est une $n$-catégorie, il résulte immédiatement de la
  proposition~\ref{prop:desc_tr} que la \oo-catégorie $\cotr{C}{c}$ est une
  $n$-catégorie. On retrouve ainsi un cas particulier de la
  proposition~\ref{prop:tr_nCat}.
\end{rem}

\begin{paragr}
  Soient $C$ une $2$-catégorie et $c$ un objet de $C$. En vertu des résultats
  de ce chapitre, en utilisant des notations adaptées au cas de la
  dimension $2$, on obtient la description suivante de la $2$-catégorie
  $\cotr{C}{c}$. Les objets de $\cotr{C}{c}$ sont les couples $(d, f)$, où
  $d$ est un objet de $C$ et $f : c \to d$ est une $1$-flèche de $C$ :
   \[
     \xymatrix{
     c \ar[d]_f \\
     d \pbox{.}
     }
   \]
   Si $(d, f)$ et $(d', f')$ sont deux objets de $\cotr{C}{c}$, une
   $1$-flèche du premier vers le second est un couple $(g, \alpha)$, où $g :
   d \to d'$ est une $1$-flèche de $C$ et $\alpha : f' \to gf$ une
   $2$-flèche de~$C$ :
  \[
    \shorthandoff{;}
    \xymatrix@C=1.5pc{
      & c
    \ar[dl]_{f^{}}_{}="f" \ar[dr]^{f'}_{}="s" \\
      d \ar[rr]_g & & d'
      \ar@{}"s";[ll]_(.15){}="ss"
      \ar@{}"s";[ll]_(.55){}="tt"
      \ar@<0.0ex>@2"ss";"tt"_{\alpha} \pbox{.}
    }
  \]
  Enfin, si $(g, \alpha)$ et $(g', \alpha')$ sont deux telles $1$-flèches,
  une $2$-flèche de la première vers la seconde est une $2$-flèche $\gamma :
  g \to g'$ satisfaisant $(\gamma \comp_0 f) \comp_1 \alpha = \alpha'$ :
  \[
    \shorthandoff{:;}
    \xymatrix@C=1.5pc@R=3pc{
      & c
    \ar[dl]_{f^{}}_{}="f" \ar[dr]^{f'}_{}="s" \\
      d
      \ar@/^2ex/@{.>}[rr]_(.30){g}^{}="0"
      \ar@/_2ex/[rr]_(.33){g'}^{}="1"
      \ar@{:>}"0";"1"^{\,\gamma}
      & & d'
      \ar@{}"s";[ll]_(.15){}="ss"
      \ar@{}"s";[ll]_(.55){}="tt"
      \ar@<-1.5ex>@/^-1ex/@{:>}"ss";"tt"_(.30){\alpha}_{}="11"
      \ar@<-0ex>@/^1ex/@2"ss";"tt"^(.20){\!\!\alpha'}^{}="00"
      \ar@{}"00";"11"|-{=}
      \pbox{.}
    }
  \]
  Par ailleurs, si $(d, f)$ est un objet et $(g, \alpha)$ est une
  $1$-flèche de $\cotr{C}{c}$, on a
  \[
    \id{(d, f)} = (\id{d}, \id{f})
    \quadet
    \id{(g, \alpha)} = \id{g}.
  \]
  Si $(g, \alpha) : (d, f) \to (d', f')$ et $(g', \alpha') : (d', f') \to
  (d'', f'')$ sont deux $1$-flèches composables de $\cotr{C}{c}$, on a
  \[
    (g', \alpha') \comp_0 (g, \alpha)
      = (g'g, (g' \comp_0 \alpha) \comp_1 \alpha').
  \]
  Enfin, les compositions horizontales et verticales des $2$-flèches de
  $\cotr{C}{c}$ sont héritées de celles de $C$, ce qui achève de décrire la
  $2$-catégorie $\cotr{C}{c}$.
\end{paragr}

\begin{rem}
  La description de $\cotr{C}{c}$ donnée dans le paragraphe précédent montre
  que cette $2$-catégorie est la $2$-catégorie $\cotrCeg{C}{c}$ introduite
  par Bullejos et Cegarra dans~\cite[section 2.1]{BullCegGeom2Cat},
  $2$-catégorie qui est, à une dualité près, un cas particulier de la
  construction des $2$-catégories « comma » définie par Gray dans
  \cite[paragraphe~I.2.5]{GrayFCT}. C'est par contre notre $2$-catégorie
  \smash{$\cotrm{C}{c}$} qui est noté $\cotrCeg{C}{c}$ par Cegarra
  dans~\cite{CegThmB}. Dans~\cite{ChicheThese}, Chiche appelle les
  $2$-catégories $\cotr{C}{c}$ et \smash{$\cotrm{C}{c}$} la optranche lax de $C$
  au-dessous de $c$ (voir sa définition~1.4.5) et la optranche colax de $C$
  au-dessous de $c$ (voir sa définition~1.4.7).
\end{rem}

\begin{rem}
  Lorsque $C$ est une $1$-catégorie, la description explicite
  de~$\cotr{C}{c}$ qu'on a donnée dans le cas où $C$ est une $2$-catégorie
  montre que $\cotr{C}{c}$ est la tranche $1$\nbd-catégorique usuelle. On
  retrouve ainsi un cas particulier du corollaire~\ref{coro:comp_tr_1}.
\end{rem}

\chapter[Fonctorialités des tranches : résultats pour les
complexes][Fonctorialités des tranches pour les complexes]{Fonctorialités
des tranches :\\ résultats pour les complexes}

Dans ce chapitre, on définit des tranches pour les complexes dirigés
augmentés. Ces tranches sont définies par des formules \forlang{ad hoc} dont
on montre ensuite qu'elles sont imposées par une relation d'adjonction.
Elles sont par ailleurs compatibles aux tranches \oo-catégoriques
introduites dans le chapitre~\ref{sec:joint} lorsqu'on se restreint aux
complexes de Steiner forts. Une partie importante du chapitre est consacrée
à l'étude des propriétés de fonctorialité et de $2$-fonctorialité de ces
tranches pour les complexes dirigés augmentés.

\section{Tranches pour les complexes dirigés augmentés}

\begin{paragr}
  Fixons $K$ un complexe dirigé augmenté. Soient $M$ un complexe dirigé
  augmenté et $g : K \to M$ un morphisme. On va définir un complexe dirigé
  augmenté~\nnot{$\cotr{M}{g}$}.
  \termindex{tranche!pour les complexes dirigés augmentés}%
  On conviendra des égalités suivantes :
  \[
    K_{-1} = \Z, \quad M_{-1} = \Z, \quad K^\ast_{-1} = \N, \quad
    M^\ast_{-1} = \N, \quad d_0 = e \quadet g_{-1} = \id{\Z}
  \]
  et, pour $j \le -2$,
  \[
    K_j = 0, \quad M_j = 0, \quad d_{j+1} = 0 \quadet g_j = 0. \]
  On définit
  $(\cotr{M}{g})_0$ comme le sous-ensemble
  \[
    (\cotr{M}{g})_0 \subset \prod_{j \ge -1} \Hom_\Ab(K_j, M_{j+1})
  \]
  constitué des éléments $(u_j)_{j \ge -1}$ vérifiant, pour $j \ge 0$,
  \[
    (-1)^{j+1}(d_{j+1}u_j - u_{j-1}d_j) = e(u_{-1}(1)).g_j.
  \]
  Notons qu'en convenant que $u_j = 0$ pour tout $j \le -2$, l'égalité
  ci-dessus reste vraie pour tout $j$ dans $\Z$ : pour $j = -1$, on obtient
  l'égalité $eu_{-1} = e(u_{-1}(1)).\id{\Z}$ et, pour $j \le -2$, l'égalité
  $0 = 0$. Pour $i \ge 1$, on pose
  \[
    (\cotr{M}{g})_i = \prod_{j \ge -1} \Hom_\Ab(K_j, M_{i+j+1}).
  \]

  Pour $i > 0$, on définit la différentielle
  \[ d_i : (\cotr{M}{g})_i \to (\cotr{M}{g})_{i-1} \]
  en envoyant $u = (u_j)_{j \ge -1}$ sur $d_i(u) = (d_i(u)_j)_{j \ge -1}$
  défini par, pour $j \ge -1$,
  \[
    d_i(u)_j = (-1)^{j+1}(d_{i+j+1}u_j - u_{j-1}d_j).
  \]
  Notons que pour $j = -1$, la formule se simplifie, en vertu de nos
  conventions, en $d_i(u)_{-1} = d_iu_{-1}$. Par ailleurs, en étendant cette
  formule à tout $j$ dans $\Z$, en convenant à nouveau que $u_j = 0$ pour
  $j \le -2$, on obtient bien que $d_i(u)_j = 0$ pour $j \le -2$.

  On définit l'augmentation
  \[ e : (\cotr{M}{g})_0 \to \Z \]
  en envoyant $(u_j)_{j \ge -1}$ sur $e(u_{-1}(1))$.

  Enfin, les sous-monoïdes de positivité, pour $i \ge 0$, sont les
  sous-ensembles
  \[
    (\cotr{M}{g})^\ast_i \subset (\cotr{M}{g})_i
  \]
  constitués des $(u_j)_{j \ge -1}$ tels que, pour tout $j \ge -1$, on ait
  l'inclusion
  \[ u_j(K^\ast_j) \subset M^\ast_{i+j+1}. \]
\end{paragr}

\begin{prop}
  Fixons $K$ un complexe dirigé augmenté. Soient $M$ un complexe dirigé
  augmenté et $g : K \to M$ un morphisme. Alors $\cotr{M}{g}$ est bien un
  complexe dirigé augmenté.
\end{prop}

\begin{proof}
  Commençons par montrer que $d_1$ est bien à valeurs dans
  $(\cotr{M}{g})_0$. Soit $u = (u_j)_{j \ge -1}$ un élément de
  $(\cotr{M}{g})_1$. Il s'agit de montrer que, pour tout $j \ge 0$, on a
  \[ (-1)^{j+1}(d_{j+1}d_1(u)_j - d_1(u)_{j-1}d_j) = e(d_1(u)_{-1}(1)).g_j. \]
  Plus généralement, pour tout $i \ge 1$, tout $u = (u_j)_{j \ge -1}$ dans
  $(\cotr{M}{g})_i$ et tout $j \ge -1$, on a
  \[
    \begin{split}
     \MoveEqLeft
      (-1)^{j+1}(d_{i+j}d_i(u)_j - d_i(u)_{j-1}d_j) \\
      & = (-1)^{j+1}\Big[(d_{i+j}(-1)^{j+1}\big(d_{i+j+1}u_j - u_{j-1}d_j\big)
      - (-1)^j \big(d_{i+j}u_{j-1} - u_{j-2}d_{j-1}\big)d_j\Big] \\
      & = -d_{i+j}u_{j-1}d_j + d_{i+j}u_{j-1}d_j \\
      & = 0.
    \end{split}
  \]
  En particulier, pour $i = 1$, on a
  \[ (-1)^{j+1}(d_{j+1}d_1(u)_j - d_1(u)_{j-1}d_j) = 0. \]
  Pour $j = -1$, on trouve
  \[ ed_1(u)_{-1} = 0, \]
  et on a donc bien, pour tout $j \ge 0$,
  \[ (-1)^{j+1}(d_{j+1}d_1(u)_j - d_1(u)_{j-1}d_j) = e(d_1(u)_{-1}(1)).g_j. \]

  Montrons maintenant que, pour tout $i \ge 2$, on a
  \hbox{$d_{i-1}d_i = 0$}. Soit donc $u = (u_j)_{j \ge -1}$ un élément
  de $(\cotr{M}{g})_i$. On a, pour tout $j \ge -1$,
  \[
     (d_{i-1}d_i(u))_j
       = (-1)^{j+1}(d_{i+j}d_i(u)_j - d_i(u)_{j-1}d_j) \\
       = 0,
  \]
  la dernière égalité résultant du calcul ci-dessus, d'où l'assertion.

  Enfin, la relation $ed_1(u)_{-1} = 0$ trouvée précédemment montre qu'on a
  bien $ed_1 = 0$, d'où le résultat.
\end{proof}

\begin{paragr}
  Fixons $K$ un complexe dirigé augmenté. Soit $f : M \to M'$ un morphisme
  de complexes dirigés augmentés au-dessous de $K$, c'est-à-dire un triangle
  commutatif
  \[
    \xymatrix@C=1pc{
    & K \ar[dr]^{g'} \ar[dl]_{\vphantom{g'}g} \\
      M \ar[rr]_f  & & M'
    }
  \]
  de morphismes de complexes dirigés augmentés. On vérifie immédiatement
  qu'on associe à $f$ un morphisme de complexes dirigés augmentés
  $\cotr{M}{g} \to \cotr{M'}{g'}$ en envoyant, pour tout $i \ge 0$, un élément
  $(u_j)_{j \ge -1}$ de $(\cotr{M}{g})_i$ sur l'élément $(fu_j)_{j \ge -1}$
  de $(\cotr{M'}{g'})_i$. On définit ainsi un foncteur
  \[
    \begin{split}
      \cotr{\Cda}{K} & \to \Cda \\
      (M, K \xto{g} M) & \mapsto \cotr{M}{g}. \\
    \end{split}
  \]
\end{paragr}

\begin{prop}\label{prop:pu_cotr_Cda}
  Fixons $K$ un complexe dirigé augmenté. Les foncteurs
  \[
    \begin{split}
      \Cda & \to \cotr{\Cda}{K} \\
      L & \mapsto (K \joint L, \iota_1)
    \end{split}
    \qquad
    \qquad
    \begin{split}
      \cotr{\Cda}{K} & \to \Cda \\
      (M, K \xto{g} M) & \mapsto \cotr{M}{g} \\
    \end{split}
  \]
  forment un couple de foncteurs adjoints.
\end{prop}

\begin{proof}
  Soient $L$ un complexe dirigé augmenté et $M$ un complexe dirigé augmenté
  muni d'un morphisme $g : K \to M$. On va produire des fonctions
  \[
    \phi : \Hom_{\cotr{\Cda}{K}}((K \joint L, \iota_1), (M, g))
    \to
    \Hom_{\Cda}(L, \cotr{M}{g}),
  \]
  \[
    \psi :
    \Hom_{\Cda}(L, \cotr{M}{g})
    \to
    \Hom_{\cotr{\Cda}{K}}((K \joint L, \iota_1), (M, g)),
  \]
  naturelles en $L$ et $(M, g)$,
  inverses l'une de l'autre.

  Commençons par $\phi$. Soit $F : K \joint L \to M$ un morphisme
  au-dessous de $K$. On définit un morphisme $\phi(F) : L \to \cotr{M}{g}$
  de la manière suivante. Soient $j \ge 0$ et $y$ dans~$L_j$. On doit
  définir un élément $\phi(F)_j(y)$ de $(\cotr{M}{g})_j$ et donc,
  pour tout $i \ge -1$, un morphisme $\phi(F)_j(y)_i : K_i \to M_{i+j+1}$.
  Pour $x$ dans $K_i$, on pose
  \[ \phi(F)_j(y)_i(x) = F_{i+1+j}(x \joint y), \]
  en convenant, pour $i = -1$, d'identifier $1$ dans
  $K_{-1} = \Z$ avec $\vide$. Vérifions tout d'abord que, pour $j = 0$,
  $\phi(F)_0(y)$ est bien dans $(\cotr{M}{g})_0$. Il s'agit de voir que,
  pour tout $i \ge 0$, on a
  \[
      (-1)^{i+1}(d_{i+1}\phi(F)_0(y)_i - \phi(F)_0(y)_{i-1}d_i) =
      e(\phi(F)_0(y)_{-1}(1)).g_i.
  \]
  Plus généralement, pour tout $j \ge 0$ et tout $i \ge -1$, on a
  {
    \allowdisplaybreaks
    \begin{align*}
      \MoveEqLeft
      (-1)^{i+1}(d_{i+j+1}\phi(F)_j(y)_i - \phi(F)_j(y)_{i-1}d_i)(x) \\
      & =
      (-1)^{i+1}(d_{i+j+1}F_{i+j+1}(x \joint y)
        - F_{i+j}(d_ix \joint y)) \\
      & =
        (-1)^{i+1}(F_{i+j}d_{i+j+1}(x \joint y)
        - F_{i+j}(d_ix \joint y)) \\
      & =
        (-1)^{i+1}F_{i+j}(d_{i+j+1}(x \joint y) - d_ix \joint y) \\
      & =
        (-1)^{i+1}F_{i+j}((-1)^{i+1} x \joint d_jy) \\
      & =
      F_{i+j}(x \joint d_j y).
    \end{align*}
  }
  En revenant au cas $j = 0$, on a donc
  {
    \allowdisplaybreaks
    \begin{align*}
      \MoveEqLeft
      (-1)^{i+1}(d_{i+1}\phi(F)_0(y)_i - \phi(F)_0(y)_{i-1}d_i)(x) \\*
      & =
      F_i(x \joint e(y)\vide) \\
      & =
      e(y)F_i(x \joint \vide) \\*
      & =
      e(y)g_i(x),
    \end{align*}
  }
  puisque $F$ est au-dessous de $K$. D'où l'assertion puisque
  \[ e(\phi(F)_0(y)_{-1}(1)) = eF_j(\vide \joint y) = e(y). \]
  Vérifions maintenant la compatibilité aux différentielles. Pour tout $j
  \ge 1$ et tout $i \ge -1$, on a
  \[
    \begin{split}
      (d_j\phi(F)_j(y))_i(x)
      & =
      (-1)^{i+1}(d_{i+j+1}\phi(F)_j(y)_i - \phi(F)_j(y)_{i-1}d_i)(x) \\
      & =
      F_{i+j}(x \joint d_j y) \\*
      & \phantom{=1} \text{(par le calcul du début de la preuve)} \\
      & = \phi(F)_{j-1}(d_jy)_i(x),
    \end{split}
  \]
  d'où la compatibilité aux différentielles. La compatibilité aux
  augmentations résulte d'un calcul précédent :
  \[
    e(\phi(F)_0(y)) = e(\phi(F)_0(y)_{-1}(1)) = e(y).
  \]
  La compatibilité aux sous-monoïdes de positivité étant évidente, on a bien
  établi que $\phi(F) : L \to \cotr{M}{g}$ est un morphisme.

  Définissons maintenant $\psi$. Soit $G : L \to \cotr{M}{g}$ un morphisme.
  Il s'agit de définir un morphisme $\psi(G) : K \joint L \to M$ au-dessous
  de $K$. Soient $x$ un élément de $K_i$ et $y$ un élément de $L_j$ avec $i
  \ge -1$, $j \ge -1$ et $i + 1 + j \ge 0$. On pose
  \[ \psi(G)_{i+1+j}(x \joint y) = G_j(y)_i(x), \]
  en convenant que $G_{-1}(\vide)_i = g_i$.
  Vérifions que $\psi(G)$ est bien un morphisme au-dessous de $K$. La
  compatibilité aux sous-monoïdes de positivité est évidente. Il est
  immédiat que, pour tout $i \ge 0$, le morphisme $\psi(G)_i$ est bien
  au-dessous de $K_i$ ; en effet, pour $x$ dans $K_i$, on a
  \[
    \psi(G)_i(x \joint \vide) = G_{-1}(\vide)_i(x) = g_i(x).
  \]
  Ainsi, pour $i = 0$, on a
  \[
    e\psi(G)_0(x \joint \vide) = eg_0(x) = e(x).
  \]
  Par ailleurs, pour $y$ dans $L_0$, on a
  \[
    e(\psi(G)_0(\vide \joint y)) = e(G_0(y)_{-1}(1)) = eG_0(y) = e(y),
  \]
  ce qui montre la compatibilité de $\psi(G)_0$ aux augmentations.
  Montrons la compatibilité aux différentielles. Soient donc $x$ dans $K_i$
  et $y$ dans $L_j$ avec $i \ge -1$, $j \ge -1$ et~\hbox{$i + 1 + j \ge 1$}.
  Si $j = -1$, de sorte qu'on peut supposer $y = \vide$, on a
  \[
    d_i\psi(G)_i(x \joint \vide) = d_ig_i(x) = g_{i-1}d_i(x)
    = \psi(G)_{i-1}(d_i x \joint \vide) = \psi(G)_{i-1}d_i(x \joint \vide).
  \]
  Si $j = 0$, on a
  \[
    \begin{split}
      d_{i+1}\psi(G)_{i+1}(x \joint y)
      & = d_{i+1}G_0(y)_i(x) \\
      & = \big(G_0(y)_{i-1}d_i + (-1)^{i+1}e(G_0(y)_{-1}(1))g_i\big)(x) \\
      & \phantom{=1} \text{(par définition de $(\cotr{M}{g})_0$)} \\
      & = G_0(y)_{i-1}(d_i(x)) + (-1)^{i+1}e(y)g_i(x) \\*
      & \phantom{=1} \text{(par compatibilité de $G$ aux
          augmentations)} \\
      & = \psi(G)_i(d_i x \joint y) + (-1)^{i+1}\psi(G)_i(x \joint
        e(y)\vide) \\*
      & \phantom{=1} \text{(car $\psi(G)_i$ est au-dessous de $K$)} \\
      & = \psi(G)_i\big(d_i x \joint y + (-1)^{i+1}x \joint
      e(y)\vide\big) \\
      & = \psi(G)_i d_{i+1}(x \joint y).
    \end{split}
  \]
  Enfin, si $j \ge 1$, on a
  \[
    \begin{split}
      d_{i+j+1}\psi(G)_{i+j+1}(x \joint y)
      & = d_{i+j+1}G_j(y)_i(x) \\
      & = \big(G_j(y)_{i-1}d_i + (-1)^{i+1}(d_j(G_j(y)))_i\big)(x) \\
      & \phantom{=1} \text{(par définition de $d_j(G_j(y))$)} \\
      & = G_j(y)_{i-1}(d_ix) + (-1)^{i+1}(G_{j-1}(d_jy))_i(x) \\
      & = \psi(G)_{i+j}(d_i x \joint y) + (-1)^{i+1}\psi(G)_{i+j}(x \joint
      d_j y) \\
      & = \psi(G)_{i+j}\big(d_i x \joint y + (-1)^{i+1}x \joint
      d_j y\big) \\
      & = \psi(G)_{i+j}d_{i+j+1}(x \joint y),
    \end{split}
  \]
  ce qui achève de montrer que $\psi(G) : K \joint L \to M$ est bien un
  morphisme au-dessous de~$K$.

  Pour conclure, il suffit de vérifier que les applications $\phi$ et $\psi$
  sont bien inverses l'une de l'autre, ce qui est immédiat.
\end{proof}

\begin{paragr}
  Fixons $L$ un complexe dirigé augmenté. Soient $M$ un complexe dirigé
  augmenté et $g : L \to M$ un morphisme. On définit un complexe dirigé
  augmenté~\nnot{$\trm{M}{g}$} en posant
  \[
    \trm{M}{g} = (\cotr{M^\opp}{g^\opp})^\opp.
  \]
  On laisse le soin au lecteur de décrire explicitement ce complexe dirigé
  augmenté. On définit ainsi un foncteur
  \[
    \begin{split}
      \cotr{\Cda}{L} & \to \Cda \\
      (M, L \xto{g} M) & \mapsto \trm{M}{g}\,.
    \end{split}
  \]
\end{paragr}

\begin{prop}\label{prop:pu_tr_Cda}
  Fixons $L$ un complexe dirigé augmenté. Les foncteurs
  \[
    \begin{split}
      \Cda & \to \cotr{\Cda}{L} \\
      K & \mapsto (K \joint L, \iota_2)
    \end{split}
    \qquad
    \qquad
    \begin{split}
      \cotr{\Cda}{L} & \to \Cda \\
      (M, L \xto{g} M) & \mapsto \trm{M}{g} \\
    \end{split}
  \]
  forment un couple de foncteurs adjoints.
\end{prop}

\begin{proof}
  Le résultat se déduit par dualité des propositions~\ref{prop:pu_cotr_Cda}
  et~\ref{prop:dual_joint_cda} par un argument semblable à celui de la
  preuve de l'assertion analogue pour les \oo-catégories
  (proposition~\ref{prop:dual_tr}).
\end{proof}

\begin{coro}
  La structure de catégorie monoïdale sur $\Cda$ définie par le joint est
  localement bifermée.
\end{coro}

\begin{proof}
  Cela résulte des propositions~\ref{prop:pu_cotr_Cda}
  et~\ref{prop:pu_tr_Cda}.
\end{proof}

\begin{rem}
  On aurait pu montrer le corollaire précédent sans décrire explicitement
  les adjoints à droite. En effet, celui-ci résulte de la
  remarque~\ref{rem:loc_ferm_th_adj} puisque la catégorie $\Cda$ est
  localement présentable et que le joint commute aux limites inductives
  connexes en chaque variable (voir la
  proposition~\ref{prop:joint_Cda_limind}).
\end{rem}

\begin{prop}\label{prop:tr_Cda_ooCat}
  Fixons $K$ un complexe dirigé augmenté. Soient $M$ un complexe dirigé
  augmenté et $g : K \to M$ un morphisme. Alors, on dispose d'un morphisme
  canonique
  \[ \nu(\cotr{M}{g}) \to \cotr{\nu(M)}{\nu(g)}, \]
  où $\cotr{\nu(M)}{\nu(g)}$ désigne la tranche \oo-catégorique introduite
  au paragraphe~\ref{paragr:def_tranche}. De plus, lorsque $K$ est un
  complexe de Steiner fort, ce morphisme est un isomorphisme.
\end{prop}

\begin{proof}
  La sous-catégorie $\Theta$ étant dense dans $\ooCat$
  (proposition~\ref{prop:Theta_dense}), il suffit de
  définir une application
  \[
    \Hom_{\ooCat}(S, \nu(\cotr{M}{g}))
    \to
    \Hom_{\ooCat}(S, \cotr{\nu(M)}{\nu(g)}),
  \]
  naturelle en $S$ dans $\Theta$, et de montrer que celle-ci est bijective
  lorsque $K$ est un complexe de Steiner fort.
  Soit donc $S$ dans $\Theta$. En vertu de la
  proposition~\ref{prop:Theta_Steiner}, on a un isomorphisme naturel $S
  \simeq \nu(\lambda(S))$. En notant $L$ le complexe de Steiner fort
  $\lambda(S)$, on a des applications naturelles
  {
    \allowdisplaybreaks
    \begin{align*}
      \Hom_{\ooCat}(S, \nu(\cotr{M}{g}))
      & \simeq
      \Hom_{\ooCat}(\nu(L), \nu(\cotr{M}{g})) \\
      & \simeq
      \Hom_{\Cda}(\lambda\nu(L), \cotr{M}{g}) \\*
      & \phantom{\simeq1} \text{(par adjonction)} \\
      & \simeq
      \Hom_{\Cda}(L, \cotr{M}{g}) \\*
      & \phantom{\simeq1} \text{(en vertu du théorème~\ref{thm:Steiner})} \\
      & \simeq
      \Hom_{\cotr{\Cda}{K}}((K \joint L, \iota_1), (M, g)) \\*
      & \phantom{\simeq1} \text{(par adjonction)} \\
      & \xto{\alpha}
      \Hom_{\cotr{\ooCat}{\nu(K)}}((\nu(K \joint L), \nu(\iota_1)), (\nu(M),
      \nu(g))) \\*
      & \xto{\beta}
      \Hom_{\cotr{\ooCat}{\nu(K)}}((\nu(K) \joint \nu(L), \iota_1), (\nu(M),
      \nu(g))) \\*
      & \simeq
      \Hom_{\ooCat}(\nu(L), \cotr{\nu(M)}{\nu(g)}) \\*
      & \phantom{\simeq1} \text{(par adjonction)} \\*
      & \simeq
      \Hom_{\ooCat}(S, \cotr{\nu(M)}{\nu(g)}),
    \end{align*}
  }%
  où l'application $\alpha$ est induite par le foncteur $\nu$ et
  l'application $\beta$ par la contrainte du
  foncteur monoïdal lax $\nu$ (voir la
  proposition~\ref{prop:lambda_nu_monoidaux_joint}), d'où le morphisme
  recherché. Notons que l'application $\alpha$ se factorise en
  \[
    \begin{split}
      \Hom_{\cotr{\Cda}{K}}((K \joint L, \iota_1), (M, g))
      & \to \Hom_{\cotr{\Cda}{\lambda\nu(K)}}((\lambda\nu(K \joint L),
      \lambda\nu(\iota_1)), (M, g)) \\
      & \simeq \Hom_{\cotr{\ooCat}{\nu(K)}}((\nu(K \joint L), \nu(\iota_1)),
      (\nu(M), \nu(g))).
    \end{split}
  \]
  Ainsi, si $K$ est un complexe de Steiner fort, cette application est une
  bijection en vertu du théorème~\ref{thm:Steiner} et du
  corollaire~\ref{coro:joint_Steiner}. Par ailleurs, toujours sous
  l'hypothèse que $K$ est un complexe de Steiner fort, en vertu du
  théorème~\ref{thm:joint}, l'application $\beta$ est également une
  bijection, d'où le résultat.
\end{proof}

\section{Morphisme associé à un triangle}
\label{sec:img_tri_Cda}

\begin{paragr}\label{paragr:img_tri_Cda}
  Dans cette section, on fixe
  \[
    \shorthandoff{;}
    \xymatrix@C=1.5pc{
      K \ar[rr]^f \ar[dr]_{g}_{}="f" & & K' \ar[dl]^(0.42){g'} \\
      & L
      \ar@{}"f";[ur]_(.15){}="ff"
      \ar@{}"f";[ur]_(.55){}="oo"
      \ar@<-0.5ex>@2"ff";"oo"^{h}
    }
  \]
  un triangle dans la catégorie des complexes dirigés augmentés commutant à
  une anti\-homotopie $h$ de $g$ vers $g'f$ près. On suppose de plus que le
  complexe $L$ est décent (voir le paragraphe~\ref{paragr:def_decent}). On va
  définir un morphisme de complexes dirigés augmentés \nnot{$(f, h)^\ast :
  \cotr{L}{g'} \to \cotr{L}{g}$}.

  On conviendra que, pour tout $i \le -1$, on a
  \[ h_i = 0. \]
  Pour tout $i \ge 0$, on définit un morphisme $(f, h)^\ast_i :
  (\cotr{L}{g'})_i \to (\cotr{L}{g})_i$ en envoyant
  $u' = (u'_j : K'_j \to L_{i+1+j})_{j \ge -1}$ sur $(f, h)^\ast_i(u') =
  ((f, h)^\ast_i(u')_j)_{j \ge -1}$ défini par, pour tout $j \ge -1$,
  \[
    (f, h)^\ast_i(u')_j =
    \begin{cases}
     u'_jf_j + e(u'_{-1}(1)).h_j & \text{si $i = 0$,} \\
     u'_jf_j & \text{sinon}.
    \end{cases}
  \]
  (On vérifiera dans la preuve de la proposition suivante que $(f,
  h)^\ast_0$ est bien à valeurs dans $(\cotr{L}{g})_0$.) Afin d'avoir une
  formule uniforme en $i$, on étendra parfois l'augmentation $e$ de $L$ à
  tout élément homogène de $L$ en posant $e(z) = 0$ si $z$ est un élément
  homogène de degré non nul. Avec cette convention, pour tout $i \ge 0$, on
  a
  \[
    (f, h)^\ast_i(u')_j = u'_jf_j + e(u'_{-1}(1)).h_j.
  \]
  Notons qu'en vertu de nos précédentes conventions cette égalité reste
  vraie pour tout~$j$ dans $\Z$. Enfin, remarquons qu'on a
  $(f, h)^\ast_i(u')_{-1} = u'_{-1}$.
\end{paragr}

\begin{prop}\label{prop:img_tri_Cda}
  Les morphismes $(f, h)^\ast_i$, pour $i \ge 0$, définissent un morphisme
  de complexes dirigés augmentés $(f, h)^\ast : \cotr{L}{g'} \to
  \cotr{L}{g}$.
\end{prop}

\begin{proof}
  Commençons par vérifier que le morphisme $(f, h)^\ast_0$ est bien à valeurs
  dans~$(\cotr{L}{g})_0$. Soit donc $u' = (u'_j)_{j \ge -1}$ un élément de
  $(\cotr{L}{g'})_0$. Par définition, en posant $z = u'_{-1}(1) = (f,
  h)^\ast_0(u')_{-1}(1)$, on a, pour tout $j$ dans $\Z$,
   \[
     d_{j+1}u'_j - u'_{j-1}d_j = (-1)^{j+1} e(z)g'_j.
   \]
  Pour tout $j \ge -1$, on a
  \[
    \begin{split}
      d_{j+1}(f, h)^\ast_0(u')_j - (f, h)^\ast_0(u')_{j-1}d_j
      & =
      d_{j+1}(u'_jf_j + e(z)h_j) - (u'_{j-1}f_{j-1} + e(z)h_{j-1})d_j \\
      & =
      (d_{j+1}u'_j - u'_{j-1}d_j)f_j + e(z)(d_{j+1}h_j - h_{j-1}d_j) \\
      & =
      (-1)^{j+1}e(z)g'_jf_j + (-1)^je(z)(g'_jf_j - g_j) \\
      & =
      (-1)^{j+1}e(z)g_j,
    \end{split}
  \]
  ce qui montre que $(f, h)^\ast_0(u')$ appartient à $(\cotr{L}{g})_0$.

  La compatibilité de $(f, h)^\ast$ à l'augmentation est évidente et la
  compatibilité aux sous-monoïdes de positivité résulte du fait que $L$ est
  décent. Il nous reste à montrer la compatibilité aux différentielles.
  Soient donc $i \ge 1$ et $u' = (u'_j)_{j \ge -1}$ un élément
  de~$(\cotr{L}{g'})_i$. Pour tout $j \ge -1$, on a
  \[
    \begin{split}
    d_i(f, h)^\ast_i(u')_j
    & =
    (-1)^{j+1}(d_{i+j+1}(f, h)^\ast_i(u')_j - (f, h)^\ast_i(u')_{j-1}d_j) \\
    & =
    (-1)^{j+1}(d_{i+j+1}u'_jf_j - u'_{j-1}f_{j-1}d_j) \\
    & =
    (-1)^{j+1}(d_{i+j+1}u'_j - u'_{j-1}d_j)f_j \\
    & = d_i(u')_jf_j \\
    & = (f, h)^\ast_{i-1}(d_i(u'))_j,
    \end{split}
  \]
  la dernière égalité étant également valable quand $i = 1$ puisqu'on a
  \[
    e(d_1(u')_{-1}(1)) = ed_1(u') = 0,
  \]
  d'où le résultat.
\end{proof}

\begin{paragr}\label{paragr:fonc_Cda_comm}
  Un cas particulier important est celui où le triangle
  \[
    \shorthandoff{;}
    \xymatrix@C=1.5pc{
      K \ar[rr]^f \ar[dr]_{g}_{}="f" & & K' \ar[dl]^(0.42){g'} \\
      & L
      \ar@{}"f";[ur]_(.15){}="ff"
      \ar@{}"f";[ur]_(.55){}="oo"
      \ar@<-0.5ex>@2"ff";"oo"^{\id{}}
    }
  \]
  est commutatif et l'antihomotopie $h$ est l'identité de $g'f = g$ (voir le
  paragraphe~\ref{paragr:def_antih_id}). Dans ce cas, en vertu de la
  proposition précédente, on obtient un morphisme
  \hbox{$(f, \id{})^\ast : \cotr{L}{g'} \to \cotr{L}{g}$} qu'on notera plus
  simplement
  %
  % INDEXCHECK
  \notindex{$f^\ast : \cotr{L}{g'} \to \cotr{L}{g}$}%
  \[
    f^\ast : \cotr{L}{g'} \to \cotr{L}{g}.
  \]
\end{paragr}

\begin{prop}\label{prop:morph_tr_Cda_ooCat}
  Soit
  \[
    \xymatrix@C=1.5pc{
      K \ar[rr]^f \ar[dr]_{g}_{}="f" & & K' \ar[dl]^(0.42){g'} \\
                                     & L
    }
  \]
  un triangle commutatif dans la catégorie des complexes dirigés augmentés.
  Alors le carré
  \[
    \xymatrix{
    \nu(\cotr{L}{g'}) \ar[r]^{\nu(f^\ast)} \ar[d]
    & \nu(\cotr{L}{g}) \ar[d] \\
    \cotr{\nu(L)}{\nu(g')} \ar[r]_{\nu(f)^\ast} &
    \cotr{\nu(L)}{\nu(g)} \pbox{,}
    }
  \]
  où les morphismes verticaux sont ceux de la
  proposition~\ref{prop:tr_Cda_ooCat} et le foncteur $\nu(f)^\ast$ celui du
  paragraphe~\ref{paragr:desc_morph_tr}, est commutatif.
\end{prop}

\begin{proof}
  Il s'agit de démontrer que les deux \oo-foncteurs allant du coin supérieur
  gauche au coin inférieur droit coïncident sur les cellules ou, autrement
  dit, que pour tout $l \ge 0$, le carré
  \[
    \xymatrix{
    \Hom_{\ooCat}(\Dn{l}, \nu(\cotr{L}{g'})) \ar[r]^{\nu(f^\ast)} \ar[d]
    & \Hom_{\ooCat}(\Dn{l}, \nu(\cotr{L}{g})) \ar[d] \\
    \Hom_{\ooCat}(\Dn{l}, \cotr{\nu(L)}{\nu(g')}) \ar[r]_{\nu(f)^\ast} &
    \Hom_{\ooCat}(\Dn{l}, \cotr{\nu(L)}{\nu(g)})
    }
  \]
  est commutatif. Par définition de $\nu(f)^\ast$ et des morphismes
  canoniques (voir la preuve de la proposition~\ref{prop:tr_Cda_ooCat}),
  cela revient à montrer que le carré
  \[
    \xymatrix@C=3pc{
    \Hom_{\Cda}(\lambda(\Dn{l}), \cotr{L}{g'})
    \ar[r]^{f^\ast} \ar[d]_{\psi}
    & \Hom_{\Cda}(\lambda(\Dn{l}), \cotr{L}{g}) \ar[d]^{\psi} \\
    \Hom_{\cotr{\Cda}{K'}}(K' \joint \lambda(\Dn{l}), L)
    \ar[r]_{f \joint \lambda(\Dn{l})}
    & \Hom_{\cotr{\Cda}{K}}(K \joint \lambda(\Dn{l}), L) \pbox{,} \\
    }
  \]
  où $f^\ast$ désigne la postcomposition par $f^\ast$, $f \joint
  \lambda(\Dn{l})$ la précomposition par $f \joint \lambda(\Dn{l})$ et
  $\psi$ les isomorphismes de (la preuve de) la
  proposition~\ref{prop:pu_cotr_Cda}, est commutatif. Soit donc $G :
  \lambda(\Dn{l}) \to \cotr{L}{g'}$ un morphisme. Soient $x$ dans $K_i$
  et $y$ dans $\lambda(\Dn{l})_j$ avec $i \ge -1$, $j \ge -1$ et $i + 1 + j
  \ge 0$. Si $j \ge 0$, on a
  \[
    \psi(f^\ast G)(x \joint y) = (f^\ast G)(y)_i(x) = G(y)_i(f(x)) =
    \psi(G)(f(x) \joint y),
  \]
  la première et la dernière égalité résultant de la définition de $\psi$ et
  la deuxième de la définition de $f^\ast$. Si $j = -1$, on peut supposer $y
  = \vide$ et on a
  \[
    \psi(f^\ast G)(x \joint \vide) = g(x) = g'f(x) =
    \psi(G)(f(x) \joint \vide),
  \]
  la première et la dernière égalité résultant encore une fois de la
  définition de $\psi$, d'où le résultat.
\end{proof}

\begin{paragr}
  Un cas particulier du paragraphe~\ref{paragr:fonc_Cda_comm} est celui
  d'un triangle commutatif
  \[
    \shorthandoff{;}
    \xymatrix@C=1.5pc{
      \vide \ar[rr]^{\vide_K} \ar[dr]_{\vide_L}_{}="f" & & K \ar[dl]^g \\
      & L
      \ar@{}"f";[ur]_(.15){}="ff"
      \ar@{}"f";[ur]_(.55){}="oo"
      \ar@<-1.0ex>@2"ff";"oo"^{\id{}} & \pbox{,}
    }
  \]
  où $\vide_K$ et $\vide_L$ désignent les uniques morphismes de source
  $\vide$ et de buts respectifs~$K$ et~$L$, et $1$ désigne l'antihomotopie
  identité de $\vide_L$. En vertu de ce même paragraphe, on obtient
  un morphisme
  \[ \vide_K^\ast : \cotr{L}{g} \to \cotr{L}{\vide_L}. \]
  On vérifie immédiatement que le morphisme de $\cotr{L}{\vide_L}$ vers $L$
  qui, pour $i \ge 0$, envoie un élément $(u_j)_{j \ge -1}$ de
  $(\cotr{L}{\vide_L})_i$ sur $u_{-1}(1)$ dans $L_i$ est un isomorphisme.
  On obtient donc un morphisme
  \[ \cotr{L}{g} \to L \]
  qu'on appellera \ndef[morphisme d'oubli associé à une tranche]{morphisme
  d'oubli}. Explicitement, ce morphisme envoie, pour $i \ge 0$,
  un élément $(u_j)_{j \ge -1}$ de $(\cotr{L}{g})_i$ sur $u_{-1}(1)$ dans
  $L_i$.
\end{paragr}

\begin{prop}\label{prop:img_tri_Cda_au-dessus}
  Le morphisme $(f, h)^\ast : \cotr{L}{g'} \to \cotr{L}{g}$ est
  au-dessus de $L$. Autrement dit, le triangle
  \[
    \xymatrix@C=1.5pc{
      \cotr{L}{g'} \ar[rr]^{(f, h)^\ast} \ar[rd]_{U'}
      & & \cotr{L}{g} \ar[dl]^{U} \\
      & L & \pbox{,}
    }
  \]
  où $U$ et $U'$ désignent les morphismes d'oubli, est commutatif.
\end{prop}

\begin{proof}
  Pour tout $i \ge 0$ et tout $u' = (u'_j)_{j \ge -1}$ dans
  $(\cotr{L}{g'})_i$, on a
  \[ U(f, h)^\ast(u') = u'_{-1}(1) = U'(u'), \]
  d'où l'assertion.
\end{proof}

\section{Fonctorialité des morphismes associés aux triangles}

\begin{paragr}\label{paragr:fonct_tri_id_Cda}
  Soit $g : K \to L$ un morphisme de complexes dirigés augmentés, où le
  complexe $L$ est décent. Formons le triangle
  \[
    \shorthandoff{;}
    \xymatrix@C=1.5pc{
    K \ar[rr]^{\id{K}} \ar[dr]_(0.40){g}_{}="f" & & K \ar[dl]^(0.40){g} \\
      & L
      \ar@{}"f";[ur]_(.15){}="ff"
      \ar@{}"f";[ur]_(.55){}="oo"
      \ar@<-0.5ex>@2"ff";"oo"^{\id{g}}
      & \pbox{,}
    }
  \]
  où $\id{g}$ désigne l'antihomotopie identité de $g$. À partir de ce
  triangle, en vertu de la proposition~\ref{prop:img_tri_Cda}, on obtient un
  morphisme $(\id{K}, \id{g})^\ast : \cotr{L}{g} \to \cotr{L}{g}$.
\end{paragr}

\begin{prop}
  On a $(\id{K}, \id{g})^\ast = \id{\cotr{L}{g}}$.
\end{prop}

\begin{proof}
  Cela résulte immédiatement de la définition de $(\id{K}, \id{g})^\ast$.
\end{proof}

\begin{paragr}
  Considérons maintenant un diagramme
  \[
    \shorthandoff{;}
    \xymatrix@C=2.5pc@R=2.5pc{
      K \ar[r]^f \ar[dr]_{}="g"_(.40){g}
      & K' \ar[r]^{f'}_(.75){}="fp" \ar[d]_(.70){}="gp"_(.50){g'} & K''
      \ar[dl]_{}="gpp"^(.33){g''} \\
      & L
      \ar@{}"g";[u]_(0.10){}="x"
      \ar@{}"g";[u]_(.85){}="y"
      \ar@<-0.1ex>@2"x";"y"^(.30)h
      \ar@{}"gp";"fp"_(.25){}="x2"
      \ar@{}"gp";"fp"_(.75){}="y2"
      \ar@<0.4ex>@2"x2";"y2"^(0.40){h'\!}
    }
  \]
  de complexes dirigés augmentés, où $h$ et $h'$ sont des antihomotopies de
  $g$ vers $g'f$ et de $g'$ vers $g''f'$ respectivement. On suppose toujours
  le complexe $L$ décent.

  En composant ce diagramme, on obtient un triangle
  \[
    \shorthandoff{;}
    \xymatrix@C=1.5pc{
    K \ar[rr]^{f'f} \ar[dr]_{g}_{}="f" & & K' \ar[dl]^(0.42){g''} \\
    & L
    \ar@{}"f";[ur]_(.15){}="ff"
    \ar@{}"f";[ur]_(.55){}="oo"
    \ar@<-0.5ex>@2"ff";"oo"^{h''}
    & \pbox{,}
    }
  \]
  où $h''$ est l'antihomotopie  $h'f + h$ (voir les
  paragraphes~\ref{paragr:def_antih_sesqui} et~\ref{paragr:def_antih_comp})
  En vertu de la proposition~\ref{prop:img_tri_Cda}, on obtient donc un
  triangle
  \[
    \xymatrix@C=1.5pc{
      \cotr{L}{g''} \ar[rr]^{(f', h')^\ast} \ar[rd]_{(f'f, h'f + h)^\ast}
      & & \cotr{L}{g'} \ar[dl]^{(f, h)^\ast} \\
      & \cotr{L}{g}
    }
  \]
  de complexes dirigés augmentés.
\end{paragr}

La proposition suivante affirme que ce triangle est commutatif.

\begin{prop}\label{prop:fonct_tri_Cda}
  On a $(f, h)^\ast(f', h')^\ast = (f'f, h'f + h)^\ast$.
\end{prop}

\begin{proof}
  Soient $i \ge 0$ et $u'' = (u''_j)_{j \ge -1}$ un élément de
  $(\cotr{L}{g''})_i$. En posant $z = u''_{-1}(1)$ et en utilisant le fait que
  chacun des morphismes du triangle de l'énoncé est au-dessus de $L$, pour
  tout $j \ge -1$, on obtient
  \[
    \begin{split}
    (f, h)^\ast_i (f', h')^\ast_i(u'')_j
      & =
      (f', h')^\ast_i(u'')_jf_j + e(z)h_j  \\
      & =
      (u''_jf'_j + e(z)h'_j)f_j + e(z)h_j  \\
      & =
      u''_jf'_jf_j + e(z)(h'_jf_j + h_j)  \\
      & =
      (f'f, h'f + h)_i^\ast(u'')_j,
    \end{split}
  \]
  d'où le résultat.
\end{proof}

\section{Homotopie associée à un cône}

\begin{paragr}\label{paragr:img_cone_Cda}
  Dans cette section, on fixe un diagramme
  \[
    \shorthandoff{;:}
    \xymatrix@C=1.5pc@R=3pc{
    K \ar@/^2ex/[rr]^(.33){f'}_{}="1" \ar@/_2ex/[rr]^(.30)f_{}="0"
    \ar[dr]_{}="f"_{\phantom{g'}g}
    \ar@2"0";"1"_k
    & & K' \ar[dl]^{g'} \\
    & L
    \ar@{}"f";[ur]_(.15){}="ff"
    \ar@{}"f";[ur]_(.55){}="oo"
    \ar@<-0.5ex>@/^1ex/@{:>}"ff";"oo"^(.18){h'\!\!}_(.30){}="h'"
    \ar@<-2.0ex>@/^-1ex/@2"ff";"oo"_(.36){h}_(.80){}="h"
    \ar@3"h";"h'"_(.20){H_{}}
    }
  \]
  de complexes dirigés augmentés, où $h$ et $h'$ sont des antihomotopies de
  source $g$ et de buts respectifs $g'f$ et $g'f'$, $k$ est une
  antihomotopie de $f$ vers $f'$ et $H$ est une $2$\nbd-antihomotopie (voir
  le paragraphe~\ref{paragr:def_n-homot}) de $g'k + h$ vers $h'$. On suppose
  le complexe~$L$ décent.

  À partir de ces données, en vertu de la
  proposition~\ref{prop:img_tri_Cda}, on obtient deux morphismes
  \[
    (f, h)^\ast, (f', h')^\ast : \cotr{L}{g'} \to \cotr{L}{g}.
  \]
  On va définir une homotopie $(k, H)^\ast$ de $(f', h')^\ast$ vers $(f,
  h)^\ast$.
  \notindex{$(k, H)^\ast : (f', h')^\ast \Rightarrow (f, h)^\ast$}%
  Pour tout $i \ge 0$, on définit un morphisme $(k, H)^\ast_i
  : (\cotr{L}{g'})_i \to (\cotr{L}{g})_{i+1}$ en posant, pour tout $u' =
  (u_j)_{j \ge -1}$ dans~$(\cotr{L}{g'})_i$ et tout $j \ge -1$,
  \[
    (k, H)^\ast_i(u')_j =
    \begin{cases}
      u'_{j+1}k_j + e(u'_{-1}(1)).H_j & \text{si $i = 0$,} \\
      u'_{j+1}k_j & \text{sinon}.
    \end{cases}
  \]
  Comme dans la section~\ref{sec:img_tri_Cda}, on étendra parfois
  l'augmentation $e$ de $L$ à tout élément homogène de $L$ en posant~$e(z) =
  0$ si $z$ est un élément homogène de degré non nul. Avec cette convention,
  pour $i \ge 0$, on a
  \[
    (k, H)^\ast_i(u')_j = u'_{j+1}k_j + e(u'_{-1}(1)).H_j.
  \]
  On conviendra que, pour tout $i \le -1$, on a
  \[ H_i = 0. \]
  Pour $j = -1$, en vertu de cette convention ainsi que de nos précédentes
  conventions, on trouve $(k, H)^\ast_i(u')_{-1} = 0$. Plus généralement,
  on a $(k, H)^\ast_i(u')_j = 0$ pour tout~$j \le -1$ et l'égalité
  $(k, H)^\ast_i(u')_j = u'_{j+1}k_j + e(u'_{-1}(1)).H_j$ est donc
  valable pour tout $j$ dans $\Z$.
\end{paragr}

\begin{prop}\label{prop:img_cone_Cda}
  Les morphismes $(k, H)^\ast_i : (\cotr{L}{g'})_i \to (\cotr{L}{g})_{i+1}$,
  pour $i \ge 0$, définissent une homotopie de $(f', h')^\ast$ vers $(f,
  h)^\ast$.
\end{prop}

\begin{proof}
   La compatibilité aux sous-monoïdes de positivité est évidente.
   Vérifions que $(k, H)^\ast$ est bien une homotopie de complexes de chaînes
   de $(f', h')^\ast$ vers~$(f, h)^\ast$.  Soient $i \ge 0$ et $u' =
   (u'_j)_{j \ge -1}$ dans $(\cotr{L}{g'})_i$. Si $i \ge 1$, pour tout $j
   \ge -1$, on a
   {
     \allowdisplaybreaks
     \begin{align*}
       (d_{i+1}(k, H)^\ast_i + (k, H)^\ast_{i-1}d_i)(u')_j
       & =
       (-1)^{j+1}(d_{i+j+2}(k, H)^\ast_i(u')_j - (k, H)^\ast_i(u')_{j-1}d_j)
       \\*
       & \phantom{=1}
       \quad + (d_i(u')_{j+1}k_j + e(d_i(u')_{-1}(1))H_j) \\
       & =
       (-1)^{j+1}(d_{i+j+2}(k, H)^\ast_i(u')_j - (k, H)^\ast_i(u')_{j-1}d_j)
       \\*
       & \phantom{=1}
       \quad + d_i(u')_{j+1}k_j \\*
       & \phantom{=1} \text{(puisque $ed_i = 0$ pour tout $i \ge 1$)} \\
       & =
       (-1)^{j+1}(d_{i+j+2}u'_{j+1}k_j - u'_jk_{j-1}d_j) \\*
       & \phantom{=1}
       \quad + (-1)^{j+2}(d_{i+j+2}u'_{j+1} - u'_jd_{j+1})k_j \\
       & =
       (-1)^{j+1}u'_j(d_{j+1}k_j - k_{j-1}d_j) \\
       & =
       -u'_j(f'_j - f_j) \\
       & =
       u'_jf_j - u'_jf'_j \\
       & =
       ((f, h)^\ast_i - (f', h')^\ast_i)(u')_j.
     \end{align*}
   }%
   Si maintenant $i = 0$, de sorte qu'on a, pour tout $j$ dans $\Z$,
   \[
     d_{j+1}u'_j - u'_{j-1}d_j = (-1)^{j+1} e(z)g'_j,
   \]
   où on a posé $z = u_{-1}(1)$, alors, pour tout $j \ge -1$, on a
   {
     \allowdisplaybreaks
     \begin{align*}
       (d_1(k, H)^\ast_0)(u')_j
       & =
       (-1)^{j+1}(d_{j+2}(k, H)^\ast_0(u')_j - (k, H)^\ast_0(u')_{j-1}d_j) \\
       & =
       (-1)^{j+1}\big[d_{j+2}(u'_{j+1}k_j + e(z)H_j) - (u'_jk_{j-1} +
         e(z)H_{j-1})d_j\big] \\
       & =
         (-1)^{j+1}\big[d_{j+2}u'_{j+1}k_j + e(z)(d_{j+2}H_j - H_{j-1}d_j)
           - u'_jk_{j-1}d_j\big] \\
       & =
       (-1)^{j+1}\big[(u'_jd_{j+1} + (-1)^{j+2}e(z)g'_{j+1})k_j \\*
       & \phantom{=1} \quad
       + (-1)^{j} e(z)(h'_j - (g'_{j+1}k_j + h_j)) - u'_jk_{j-1}d_j\big] \\
       & =
       (-1)^{j+1}\big[(u'_jd_{j+1}k_j - u'_jk_{j-1}d_j + (-1)^j e(z)(h'_j -
       h_j)\big] \\
       & =
       (-1)^{j+1}u'_j(d_{j+1}k_j - k_{j-1}d_j) - e(z)(h'_j - h_j) \\
       & =
       -u'_j(f'_j - f_j) - e(z)(h'_j - h_j) \\
       & =
       (u'_jf_j + e(z)h_j) - (u'_jf'_j + e(z)h'_j) \\
       & =
       ((f, h)^\ast_0 - (f', h')^\ast_0)(u')_j,
     \end{align*}
   }%
   d'où le résultat.
\end{proof}

\begin{prop}\label{prop:img_cone_Cda_au-dessus}
  L'homotopie $(k, H)^\ast$ de $(f', h')^\ast$ vers $(f, h)^\ast$, qui sont
  deux morphismes de $\cotr{L}{g'}$ vers $\cotr{L}{g}$ au-dessus de $L$, est
  au-dessus de $L$ au sens où, en notant $U : \cotr{L}{g} \to L$ et $U' :
  \cotr{L}{g'} \to L$ les morphismes d'oubli, on a $U(k, H)^\ast = \id{U'}$.
\end{prop}

\begin{proof}
  Pour tout $i \ge 0$ et $u' = (u'_j)_{j \ge -1}$ dans
  $(\cotr{L}{g'})_i$, on a
  \[
    U(k, H)^\ast(u') = ((k, H)^\ast(u')_{-1})(1) = 0 = \id{U'}(u'),
  \]
  d'où l'assertion.
\end{proof}

\section{Fonctorialités des morphismes associés aux cônes}

\begin{paragr}\label{paragr:fonct_cone_id_Cda}
  Considérons un diagramme
  \[
    \shorthandoff{;}
    
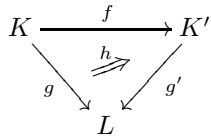
\xymatrix@C=1.5pc{
      K \ar[rr]^f \ar[dr]_{g}_{}="f" & & K' \ar[dl]^(0.42){g'} \\
      & L
      \ar@{}"f";[ur]_(.15){}="ff"
      \ar@{}"f";[ur]_(.55){}="oo"
      \ar@<-0.5ex>@2"ff";"oo"^{h}
    }
  \]
  de complexes dirigés augmentés, où $h$ est une antihomotopie de $g$ vers
  $g'f$. On suppose le complexe $L$ décent.

  À partir de ce diagramme, on peut former un diagramme
  \[
    \shorthandoff{;:}
    
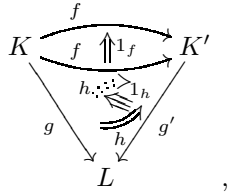
\xymatrix@C=1.5pc@R=3pc{
    K \ar@/^2ex/[rr]^(.33){f}_{}="1" \ar@/_2ex/[rr]^(.30)f_{}="0"
    \ar[dr]_{}="f"_{\phantom{g'}g}
    \ar@2"0";"1"_{\id{f}}
    & & K' \ar[dl]^{g'} \\
    & L
    \ar@{}"f";[ur]_(.15){}="ff"
    \ar@{}"f";[ur]_(.55){}="oo"
    \ar@<-0.5ex>@/^1ex/@{:>}"ff";"oo"^(.18){h\!\!}_(.30){}="h'"
    \ar@<-2.0ex>@/^-1ex/@2"ff";"oo"_(.36){h}_(.80){}="h"
    \ar@3"h";"h'"_(.20){\,\id{h}}
    & \pbox{,}
    }
  \]
  où $\id{f}$ désigne l'antihomotopie identité de $f$ et $\id{h}$ la
  $2$-antihomotopie identité de $h$ (voir le
  paragraphe~\ref{paragr:def_antih_id}). On obtient
  donc, en vertu de la proposition~\ref{prop:img_cone_Cda}, une
  homotopie~$(\id{f}, \id{h})^\ast$ de $(f, h)^\ast$ vers $(f, h)^\ast$.
\end{paragr}

\begin{prop}\label{prop:fonct_cone_id_Cda}
  On a $(\id{f}, \id{h})^\ast = \id{(f, h)^\ast}$.
\end{prop}

\begin{proof}
  Cela résulte immédiatement de la définition de $(\id{f}, \id{h})^\ast$.
\end{proof}

\begin{paragr}\label{paragr:fonct_cone_sg_Cda}
  Considérons maintenant un diagramme
    \[
      \shorthandoff{;:}
      \xymatrix@C=3.5pc@R=3.5pc{
      K  \ar[r]^f \ar[dr]_{}="g"_(.4){\phantom{g''}g\!} &
      K' \ar@/^2ex/[r]^(.33){f''}_{}="1"
      \ar@/_2ex/[r]^(.30){f'}_{}="0"_(.70){}="fp"
      \ar[d]_(.50){}="gp2"_(.20){}="gp"_(0.52){g'} &
      K'' \ar[dl]^(.4){\!g''} &
      \ar@2"0";"1"_{k} \\
        & L
      \ar@{}"g";[u]_(0.10){}="x"
      \ar@{}"g";[u]_(.75){}="y"
      \ar@<-0.1ex>@2"x";"y"^(.30)h
      \ar@{}"gp2";"fp"_(.10){}="ff2"
      \ar@{}"gp2";"fp"_(.55){}="oo2"
      \ar@<+0.5ex>@/^1ex/@{:>}"ff2";"oo2"^{\!\!h''}_(.30){}="h'''"
      \ar@<-0.5ex>@/^-1.5ex/@2"ff2";"oo2"_(.47){\!\!\!h'}_(.80){}="h''"
      \ar@3"h''";"h'''"_(.20){H}
      }
    \]
    de complexes dirigés augmentés, où $h$, $h'$ et $h''$
    sont des antihomotopies de $g$ vers~$g'f$, de $g'$ vers $g''f'$ et de
    $g'$ vers $g''f''$ respectivement, $k$ est une antihomotopie de $f'$
    vers~$f''$ et $H$ est une $2$-antihomotopie de $g''k + h'$ vers $h''$.
    On suppose toujours le complexe $L$ décent.

    À partir de ces données, en vertu des
    propositions~\ref{prop:img_tri_Cda} et~\ref{prop:img_cone_Cda}, on
    obtient un diagramme
    \[
      \shorthandoff{;:}
      \xymatrix@C=4.5pc{
        \cotr{L}{g''}
        \ar@/^2.5ex/[r]^{(f'', h'')^\ast}_{}="1"
        \ar@/_2.5ex/[r]_{(f', h')^\ast}_{}="0"
        \ar@2"1";"0"^(.55){\,(k, H)^\ast}
        &
       \cotr{L}{g'}
       \ar[r]^{(f, h)^\ast}
       & \cotr{L}{g} \pbox{.}
      }
    \]
    En vertu de la proposition~\ref{prop:fonct_tri_Cda}, on a
    \[
      (f, h)^\ast(f'', h'')^\ast = (f''f, h''f + h)^\ast
      \quadet
      (f, h)^\ast(f', h')^\ast = (f'f, h'f + h)^\ast,
    \]
    et en composant ce diagramme on obtient donc une homotopie
    $(f, h)^\ast(k, H)^\ast$ de~$(f''f, h''f+h)^\ast$
    vers $(f'f, h'f+h)^\ast$.

    Par ailleurs, en composant le diagramme de
    départ, on obtient un diagramme
    \[
      \shorthandoff{;:}
      \xymatrix@C=1.5pc@R=3pc{
        K \ar@/^2ex/[rr]^(.33){f''f}_{}="1" \ar@/_2ex/[rr]^(.30){f'f}_{}="0"
        \ar[dr]_{}="f"_{\phantom{g'}g}
        \ar@2"0";"1"_{kf}
        & & K'' \ar[dl]^{g''} \\
        & L
        \ar@{}"f";[ur]_(.15){}="ff"
        \ar@{}"f";[ur]_(.55){}="oo"
        \ar@<-0.5ex>@/^1ex/@{:>}"ff";"oo"^(.18){\!\!}_(.30){}="h'"
        \ar@<-2.0ex>@/^-1ex/@2"ff";"oo"_(.36){}_(.80){}="h"
        \ar@3"h";"h'"_(.20){} & \pbox{,}
        }
    \]
    où la $2$-flèche courbée de devant est l'antihomotopie $h'f + h$,
    la $2$-flèche courbée de derrière, en pointillé, l'antihomotopie $h''f + h$ et la
    $3$-flèche est la $2$-antihomotopie $Hf + \id{h}$ (voir les
    paragraphes~\ref{paragr:def_antih_sesqui}
    et~\ref{paragr:def_homot_cod2}) de $(g''k + h')f + h = g''(kf) + (h'f +
    h)$ vers $h''f + h$. Pour tout $n \ge 0$, on a~$(Hf + \id{h})_n =
    (Hf)_n$ et on désignera plus simplement cette $2$\nbd-antihomotopie par
    $Hf$. On obtient donc, toujours en vertu de la
    proposition~\ref{prop:img_cone_Cda}, une homotopie~$(kf, Hf)^\ast$ de
    $(f''f, h''f+h)^\ast$ vers $(f'f, h'f+h)^\ast$.
\end{paragr}

\begin{prop}\label{prop:fonct_cone_sg_Cda}
  On a $(f, h)^\ast (k, H)^\ast = (kf, Hf)^\ast$.
\end{prop}

\begin{proof}
  Soient $i \ge 0$ et $u'' = (u''_j)_{j \ge -1}$ un élément de
  $(\cotr{L}{g''})_i$. Pour tout $j \ge -1$, en posant $z = u''_{-1}(1)$, on
  a
  \[
    \begin{split}
      ((f, h)^\ast (k, H)^\ast)_i (u'')_j
      & =
      (f, h)^\ast_{i+1} (k, H)^\ast_i (u'')_j \\
      & =
      (k, H)^\ast_i(u'')_jf_j \\
      & =
      u''_{j+1}k_jf_j + e(z)H_jf_j\\
      & =
      (kf, Hf)^\ast_i(u'')_j,
    \end{split}
  \]
  d'où le résultat.
\end{proof}

\begin{paragr}\label{paragr:fonct_cone_sd_Cda}
  De même, considérons un diagramme
  \[
      \shorthandoff{;:}
      \xymatrix@C=3.5pc@R=3.5pc{
      K
      \ar@/^2ex/[r]^(.33){f'}_{}="1"
      \ar@/_2ex/[r]^(.30){f}_{}="0"_(.70){}="f"
      \ar[dr]_{}="g"_(.40){\phantom{g''}g\!}
      \ar@2"0";"1"_{k}
      &
      K' \ar[r]^{f''}_(.75){}="fp"
         \ar[d]_(.70){}="gp2"_(.20){}="gp"^(0.52){g'}
      &
      K'' \ar[dl]^(.40){\!g''}
      \\
      & L
      \ar@{}"g";"gp"_(.15){}="ff1"
      \ar@{}"g";"gp"_(.80){}="oo1"
      \ar@<-0.0ex>@/^1ex/@{:>}"ff1";"oo1"^(.35){h'\!\!}_(.30){}="h'"
      \ar@<-1.0ex>@/^-1.5ex/@2"ff1";"oo1"_(.50){\!h}_(.80){}="h"
      \ar@3"h";"h'"_(.20){H_{}}
      \ar@{}"gp2";"fp"_(.25){}="x2"
      \ar@{}"gp2";"fp"_(.75){}="y2"
      \ar@<0.4ex>@2"x2";"y2"^{h''}
      }
  \]
  de complexes dirigés augmentés, où $h$, $h'$ et $h''$ sont des
  antihomotopies de $g$ vers~$g'f$, de $g$ vers $g'f'$ et de
  $g'$ vers $g''f''$ respectivement, $k$ est une antihomotopie de $f$
  vers~$f'$ et $H$ est une $2$-antihomotopie de $g'k + h$ vers $h'$. On
  suppose toujours le complexe~$L$ décent.

  À partir de ces données, on obtient un diagramme
  \[
    \shorthandoff{;:}
    \xymatrix@C=4.5pc{
    \cotr{L}{g''}
    \ar[r]^{(f'', h'')^\ast}
    &
    \cotr{L}{g'}
    \ar@/^2.5ex/[r]^{(f'\!, h')^\ast}_{}="1"
    \ar@/_2.5ex/[r]_{(f, h)^\ast}_{}="0"
    \ar@2"1";"0"^(.55){\,(k, H)^\ast}
    & \cotr{L}{g} \pbox{.}
    }
  \]
  En vertu de la proposition~\ref{prop:fonct_tri_Cda}, on a
  \[
      (f', h')^\ast(f'', h'')^\ast = (f''f', h''f' + h')^\ast
      \quadet
      (f, h)^\ast(f'', h'')^\ast = (f''f, h''f + h)^\ast,
  \]
  et en composant ce diagramme on obtient donc une homotopie
  $(k, H)^\ast(f'', h'')^\ast$ de~$(f''f', h''f'+h')^\ast$
  vers $(f''f, h''f+h)^\ast$.

  Par ailleurs, en composant le diagramme de
  départ, on obtient un diagramme
  \[
    \shorthandoff{;:}
    \xymatrix@C=1.5pc@R=3pc{
      K \ar@/^2ex/[rr]^(.33){f''\!f'}_{}="1" \ar@/_2ex/[rr]^(.30){f''\!f}_{}="0"
      \ar[dr]_{}="f"_{\phantom{g'}g}
      \ar@2"0";"1"_{f''\!k}
      & & K'' \ar[dl]^{g''} \\
      & L
      \ar@{}"f";[ur]_(.15){}="ff"
      \ar@{}"f";[ur]_(.55){}="oo"
      \ar@<-0.5ex>@/^1ex/@{:>}"ff";"oo"^(.18){\!\!}_(.30){}="h'"
      \ar@<-2.0ex>@/^-1ex/@2"ff";"oo"_(.36){}_(.80){}="h"
      \ar@3"h";"h'"_(.20){\,H'} & \pbox{,}
      }
  \]
  où la $2$-flèche courbée de devant est l'antihomotopie $h''f
  + h$, la $2$\nbd-flèche courbée de derrière, en pointillé, l'antihomotopie
  $h''f' + h'$ et la $3$-flèche est la $2$-antihomotopie \hbox{$H' =
  (\id{h''f'} + H) + (h''k + \id{h})$}
  (où les « + » de gauche et de droite désignent l'opération définie au
  paragraphe~\ref{paragr:def_homot_cod2}, celui du milieu celle définie au
  paragraphe~\ref{paragr:def_antih_comp} et~$h''k'$ la $2$-antihomotopie du
  paragraphe~\ref{paragr:def_homot_Gray}). Ceci a bien un sens puisque,
  $h''k + \id{h}$ étant une $2$\nbd-antihomotopie de $((g''f'')k + h''f) +
  h$ vers $(h''f + g'k) + h$ et $\id{h''f'} + H$ une $2$-antihomotopie de
  $h''f' + (g'k + h)$ vers $h''f' + h'$, $H'$ est une $2$\nbd-antihomotopie
  de $g''(f''k) + (h''f + h)$ vers $h''f' + h'$. On désignera plus
  simplement~$H'$ par $H + h''k$.  On obtient donc une homotopie $(f''k, H +
  h''k)^\ast$ de~$(f''f', h''f'+h')^\ast$ vers $(f''f, h''f+h)^\ast$.
\end{paragr}

\begin{prop}\label{prop:fonct_cone_sd_Cda}
  On a $(k, H)^\ast (f'', h'')^\ast = (f''k, H + h''k)^\ast$.
\end{prop}

\begin{proof}
  Soient $i \ge 0$ et $u'' = (u''_j)_{j \ge -1}$ un élément de
  $(\cotr{L}{g''})_i$. Pour tout $j \ge -1$, en posant $z = u''_{-1}(1)$, on
  a
  {
    \allowdisplaybreaks
    \begin{align*}
      ((k, H)^\ast(f'',h'')^\ast)_i (u'')_j
      & =
      (k, H)^\ast_i(f'',h'')^\ast_i (u'')_j \\
      & =
      (f'', h'')^\ast_i(u'')_{j+1}k_j + e(z)H_j\\
      & =
      (u''_{j+1}f''_{j+1} + e(z)h''_{j+1})k_j + e(z)H_j\\
      & =
      u''_{j+1}f''_{j+1}k_j + e(z)(H_j + h''_{j+1}k_j)\\
      & =
      (f''k, H + h''k)^\ast_i(u'')_j,
    \end{align*}
  }%
  d'où le résultat.
\end{proof}

\begin{paragr}\label{paragr:fonct_cone_Cda}
  Enfin, considérons un diagramme
  \[
      \shorthandoff{;:!}
      
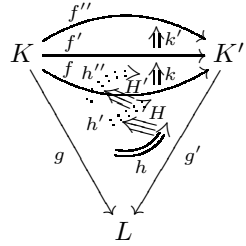
\xymatrix@C=2pc@R=4.5pc{
        K \ar@/^3.5ex/[rr]^(.30)*+<-.6em>{\labelstyle f''}_(.65){}="2"
        \ar[rr]^(.25)*+<-.3em>{\labelstyle f'}_(.65){}="1"
        \ar@/_3.5ex/[rr]^(.20)*+<-.3em>{\labelstyle f}_(.65){}="0"
        \ar[dr]_{}="f"_{\phantom{g'}g}
        \ar@2"0";"1"_k
        \ar@2"1";"2"_{k'}
        & & K' \ar[dl]^{g'} \\
        & L
        \ar@{}"f";[ur]_(.15){}="ff"
        \ar@{}"f";[ur]_(.55){}="oo"
        \ar@<2.0ex>@/^1.5ex/@{:>}"ff";"oo"^(.40)*+<-.3em>{\labelstyle h''\!\!}_(.30){}="h''"
        \ar@<0ex>@/^0ex/@{:>}"ff";"oo"^(.0)*+<-.5em>{\labelstyle{h'\!\!}}_(.30){}="h'"_(.70){}="h'2"
        \ar@<-2.0ex>@/^-1.5ex/@2"ff";"oo"_(.36){h}_(.80){}="h"
        \ar@3"h'2";"h''"_(.28){H'}
        \ar@3"h";"h'"_(.20){H_{}}
        }
  \]
  de complexes dirigés augmentés, où $h$, $h'$ et $h''$ sont des
  antihomotopies de $g$ vers~$g'f$, $g'f'$ et $g'f''$ respectivement, $k$ et
  $k'$ sont des antihomotopies de $f$ vers $f'$ et $f'$ vers~$f''$
  respectivement et $H$ et $H'$ sont des $2$-antihomotopies de $g'k + h$ vers
  $h'$ et de~$g'k' + h'$ vers~$h''$ respectivement. On suppose toujours
  le complexe $L$ décent.

  À partir de ces données, en vertu de la
  proposition~\ref{prop:img_cone_Cda}, on obtient un diagramme
  \[
    \shorthandoff{;:}
    \xymatrix@C=5pc{
    \cotr{L}{g'}
    \ar@/^5.5ex/[r]^(.38){(f''\!, h'')^\ast}_{}="0"
    \ar[r]_(.30){(f'\!, h')^\ast}_{}="1"
    \ar@/_5.5ex/[r]_(.33){(f, h)^\ast}_{}="2"
    \ar@2"1";"2"^(.40){(k, H)^\ast}
    \ar@2"0";"1"^(.75){(k'\!, H')^\ast}
    &
    \cotr{L}{g} \pbox{.}
    }
  \]
  En composant ce diagramme on obtient une homotopie $(k, H)^\ast + (k',
  H')^\ast$ de $(f'', h'')^\ast$ vers $(f, h)^\ast$.

  Par ailleurs, en composant le diagramme de
  départ, on obtient un diagramme
  \[
    \shorthandoff{;:}
    
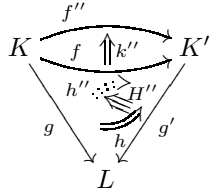
\xymatrix@C=1.5pc@R=3pc{
    K \ar@/^2ex/[rr]^(.33){f''}_{}="1" \ar@/_2ex/[rr]^(.30)f_{}="0"
    \ar[dr]_{}="f"_{\phantom{g'}g}
    \ar@2"0";"1"_{k''}
    & & K' \ar[dl]^{g'} \\
    & L
    \ar@{}"f";[ur]_(.15){}="ff"
    \ar@{}"f";[ur]_(.55){}="oo"
    \ar@<-0.5ex>@/^1ex/@{:>}"ff";"oo"^(.18){h''\!\!}_(.30){}="h'"
    \ar@<-2.0ex>@/^-1ex/@2"ff";"oo"_(.36){h}_(.80){}="h"
    \ar@3"h";"h'"_(.20){\,H''} & \pbox{,}
    }
  \]
  où $k''$ désigne l'antihomotopie $k' + k$ et $H''$ la $2$-antihomotopie
  $H' + (\id{g'k'} + H)$ (où le «~+~» de gauche désigne l'opération
  définie au paragraphe~\ref{paragr:def_antih_comp} et celui de droite celle
  définie au paragraphe~\ref{paragr:def_homot_cod2}).
  Ceci a un sens puisque, $\id{g'k'} + H$ étant une
  $2$\nbd-antihomotopie de $g'k' + (g'k + h)$ vers~\hbox{$g'k' + h'$} et $H'$
  une $2$-antihomotopie de $g'k' + h'$ vers~$h''$, $H''$ est une
  $2$-antihomotopie de~\hbox{$g'(k'+k) + h$} vers~$h''$.  On désignera plus
  simplement $H''$ par $H' + H$.  On obtient donc, toujours en vertu de la
  proposition~\ref{prop:img_cone_Cda}, une homotopie $(H' + H, k' + k)^\ast$
  de~$(f'', h'')^\ast$ vers $(f, h)^\ast$.
\end{paragr}

\begin{prop}\label{prop:fonct_cone_Cda}
  On a $(k, H)^\ast + (k', H')^\ast = (k' + k, H' + H)^\ast$.
\end{prop}

\begin{proof}
  Soient $i \ge 0$ et $u' = (u'_j)_{j \ge -1}$ un élément de
  $(\cotr{L}{g'})_i$. Pour tout $j \ge -1$, en posant $z = u'_{-1}(1)$, on
  a
  {
    \allowdisplaybreaks
    \begin{align*}
      ((k, H)^\ast + (k', H')^\ast)_i (u')_j
      & =
      (k, H)^\ast_i(u')_j + (k', H')^\ast_i(u')_j \\
      & =
      (u'_{j+1}k_j + e(z)H_j) + (u'_{j+1}k'_j + e(z)H'_j) \\
      & =
      u'_{j+1}(k'_j + k_j) + e(z)(H'_j + H_j) \\
      & =
      (k' + k, H' + H)^\ast_i (u')_j
    \end{align*}
  }
  d'où le résultat.
\end{proof}

\chapter[Fonctorialités des tranches : résultats pour les
\pdfoo-catégories][Fonctorialités des tranches pour les
\pdfoo-catégories]{Fonctorialités des tranches :\\ résultats pour les
\pdfoo-catégories}
\label{sec:fonct_tr}

% TOCHECK
\kern-5pt

Fixons $C$ une \oo-catégorie. Les tranches définies dans le
chapitre~\ref{sec:joint} permettent d'associer à tout \oo-foncteur $u : A \to
C$ une \oo-catégorie $\cotr{C}{u}$. Le but de ce chapitre est de montrer,
sous des hypothèses techniques que l'on conjecture être superflues,
que cette correspondance s'étend en une correspondance
% \[
{
  \shorthandoff{;:}
  \allowdisplaybreaks
  \begin{align*}
  \raisebox{20pt}{
  \xymatrix{
    A \ar[d]_u \\
    C
  }
  }
  \hskip4.2em
  & \mapsto
  \hskip4.7em
  \cotr{C}{u}\pbox{,} \\
  \raisebox{20pt}{
    \xymatrix@C=1.5pc{
      A \ar[rr]^v \ar[dr]_{u}_{}="f" & & A' \ar[dl]^(0.42){u'} \\
      & C
      \ar@{}"f";[ur]_(.15){}="ff"
      \ar@{}"f";[ur]_(.55){}="oo"
      \ar@<-0.5ex>@2"ff";"oo"^{\alpha}
    }
  }
  \quad
  & \mapsto
  \hskip2.2em
  \xymatrix{\cotr{C}{u'} \ar[r]^{(v, \alpha)^\ast} & \cotr{C}{u} \pbox{,}} \\
  \raisebox{20pt}{
    \xymatrix@C=1.5pc@R=3pc{
    A \ar@/^2ex/[rr]^(.33){v'}_{}="1" \ar@/_2ex/[rr]^(.30)v_{}="0"
    \ar[dr]_{}="f"_{\phantom{u'}u}
    \ar@2"0";"1"_{\beta}
    & & A' \ar[dl]^{u'} \\
    & C
    \ar@{}"f";[ur]_(.15){}="ff"
    \ar@{}"f";[ur]_(.55){}="oo"
    \ar@<-0.5ex>@/^1ex/@{:>}"ff";"oo"^(.18){\alpha'\!\!}_(.30){}="h'"
    \ar@<-2.0ex>@/^-1ex/@2"ff";"oo"_(.36){\alpha}_(.80){}="h"
    \ar@3"h";"h'"_(.20){\,\,\Lambda}
    }
  }
  \quad
  & \mapsto
  \quad
      \xymatrix@C=4.5pc{
        \cotr{C}{u'}
        \ar@/^2.5ex/[r]^{(v', \alpha')^\ast}_{}="1"
        \ar@/_2.5ex/[r]_{(v, \alpha)^\ast}_{}="0"
        \ar@2"1";"0"^(.55){\,(\beta, \Lambda)^\ast}
        &
       \cotr{C}{u'} \pbox{,}
      }
  \end{align*}
}%
% \]
où les $2$-flèches et $3$-flèches à gauche du signe <<~$\mapsto$~>>
sont des transformations lax et des $2$-transformations lax respectivement,
et la $2$-flèche à droite de ce signe est une transformation oplax. Plus
généralement, on conjecture que cette correspondance s'étend en un
\oo-foncteur de Gray (voir la conjecture~\ref{conj:fond}).

\section{Un lemme pour se ramener aux complexes dirigés augmentés}

Le but de cette section est de démontrer un lemme technique qui nous
permettra de ramener les fonctorialités des tranches pour les \oo-catégories
aux fonctorialités analogues pour les complexes dirigés augmentés qu'on a
établies dans le chapitre précédent.

\begin{paragr}
  Dans cette section, et uniquement dans celle-ci, nous utiliserons la
  théorie du produit tensoriel de Gray telle qu'elle est rappelée dans
  l'appendice~\ref{sec:Gray}. Si $C$ et~$D$ sont deux \oo-catégories, nous
  noterons $C \otimes D$ leur produit tensoriel (de Gray). Rappelons (voir
  l'appendice~\ref{sec:trans_oplax} et plus particulièrement le corollaire
  \ref{coro:trans_oplax_abs}) que la donnée d'une transformation oplax entre
  deux \oo-foncteurs de $C$ vers~$D$ correspond à celle qu'un \oo-foncteur
  $\Dn{1} \otimes C \to D$, où $\Dn{1}$ désigne la \oo-catégorie du
  paragraphe~\ref{paragr:def_disque}.
\end{paragr}

\begin{paragr}
  Dans cette section, on fixe
  \[ \xymatrix{K \ar[r]^g & L & \ar[l]_{g'} K'} \]
  deux morphismes de complexes dirigés augmentés et $S$ un complexe dirigé
  augmenté.

  À partir de ces données, on définit deux foncteurs de $\cotr{\Cda}{L}$
  vers $\Cda$ :
  \[
      (M, L \xto{a} M) \mapsto S \otimes \cotr{M}{ag'}
    \quadet
      (M, L \xto{a} M) \mapsto \cotr{M}{ag}\,;
  \]
  ainsi que deux foncteurs de $\cotr{\ooCat}{\nu(L)}$ vers $\ooCat$ :
  \[
      (C, \nu(L) \xto{b} C) \mapsto \nu(S) \otimes \cotr{C}{b\nu(g')}
    \quadet
      (C, \nu(L) \xto{b} C) \mapsto \cotr{C}{b\nu(g)}.
  \]
\end{paragr}

\begin{lemme}\label{lemme:fonct_Cda_ooCat}
  On suppose que $K$, $K'$, $L$ et $S$ sont des complexes de Steiner forts et
  que $g'$ est une inclusion rigide ordonnée \noemph{(voir le
  paragraphe~\ref{paragr:def_rigide_ordonnee})}. Alors, pour tout morphisme
  \[ \alpha_{(M, a)} : S \otimes \cotr{M}{ag'} \to \cotr{M}{ag}, \]
  naturel en $(M, a)$ dans $\cotr{\Cda}{L}$, il existe un et un seul
  \oo-foncteur
  \[ \beta_{(C, b)} : \nu(S) \otimes \cotr{C}{b\nu(g')} \to \cotr{C}{b\nu(g)}, \]
  naturel en $(C, b)$ dans $\cotr{\ooCat}{\nu(L)}$, tel que, pour tout
  complexe de Steiner fort $M$ muni d'un morphisme $a : L \to M$, le
  diagramme
  \[
    \xymatrix@C=3pc{
      \nu(S) \otimes \nu(\cotr{M}{ag'}) \ar[r]^\mu \ar[d]_\simeq
      &
      \nu(S \otimes \cotr{M}{ag'})
      \ar[r]^-{\nu(\alpha_{(M, a)})}
      &
      \nu(\cotr{M}{ag}) \ar[d]^\simeq
      \\
      \nu(S) \otimes (\cotr{\nu(M)}{\nu(ag')})
      \ar[rr]_{\beta_{(\nu(M), \nu(a))}}
      & &
      \cotr{\nu(M)}{\nu(ag)} \pbox{,}
    }
  \]
  où les isomorphismes verticaux proviennent de ceux de la
  proposition~\ref{prop:tr_Cda_ooCat} et $\mu$ désigne
  la contrainte du foncteur monoïdal lax $\nu$ \noemph{(voir la
  proposition~\ref{prop:lambda_nu_mon_tens})}, soit commutatif.
\end{lemme}

\begin{proof}
  Par adjonction, la donnée d'un \oo-foncteur
  \[ \beta_{(C, b)} : \nu(S) \otimes \cotr{C}{b\nu(g')} \to \cotr{C}{b\nu(g)} \]
  est équivalente à celle d'un \oo-foncteur
  \[ \cotr{C}{b\nu(g')} \to \HomLax(\nu(S), \cotr{C}{b\nu(g)}), \]
  où $\HomLax$ désigne le $\Hom$ interne à gauche du produit tensoriel,
  introduit au paragraphe~\ref{paragr:def_HomOpLax}. Par ailleurs,
  la sous-catégorie pleine de $\ooCat$ formée des \oo-catégories de Steiner
  fortes étant dense dans $\ooCat$ (car elle contient $\Theta$, elle-même
  dense dans $\ooCat$ d'après la proposition~\ref{prop:Theta_dense}), la
  donnée d'un tel \oo-foncteur est équivalente à celle d'une application
  \[
    \Hom_{\ooCat}(\nu(T), \cotr{C}{b\nu(g')})
    \to
    \Hom_{\ooCat}(\nu(T), \HomLax(\nu(S), \cotr{C}{b\nu(g)})),
  \]
  naturelle en $T$ dans la sous-catégorie pleine $\Stf$ de $\Cda$ formée des
  complexes de Steiner forts (en vertu de la pleine fidélité du foncteur
  $\nu$ restreint aux complexes de Steiner forts, voir le
  théorème~\ref{thm:Steiner}). Par adjonction, la donnée d'une telle
  application correspond à celle d'une application
  \[
    \Hom_{\ooCat}(\nu(T), \cotr{C}{b\nu(g')})
    \to
    \Hom_{\ooCat}(\nu(S) \otimes \nu(T), \cotr{C}{b\nu(g)}),
  \]
  ou, de nouveau par adjonction (voir le
  paragraphe~\ref{paragr:def_tranche}), à celle d'une application
  \[
    \begin{split}
      \MoveEqLeft
      \Hom_{\cotr{\ooCat}{\nu(K')}}((\nu(K') \joint \nu(T), \iota_1), (C,
      b\nu(g')))
      \\
      & \qquad\qquad \to
      \Hom_{\cotr{\ooCat}{\nu(K)}}((\nu(K) \joint (\nu(S)
      \otimes \nu(T)), \iota_1), (C, b\nu(g))),
    \end{split}
  \]
  ou encore, en vertu de la propriété universelle de la somme amalgamée, à
  celle d'une application
  \[
    \begin{split}
    \MoveEqLeft
    \Hom_{\cotr{\ooCat}{\nu(L)}}((\nu(L) \amalg_{\nu(K')} (\nu(K') \joint
    \nu(T)), \e_1), (C, b)) \\
    & \qquad\qquad \to
    \Hom_{\cotr{\ooCat}{\nu(L)}}((\nu(L) \amalg_{\nu(K)} (\nu(K) \joint
    (\nu(S) \otimes \nu(T))), \e_1), (C, b)),
    \end{split}
  \]
  où $\e_1$ désigne la première inclusion canonique dans une somme
  amalgamée. Puisqu'on demande que le \oo-foncteur $\beta_{(C, b)}$ soit
  naturel en $(C, b)$ dans $\cotr{\ooCat}{\nu(L)}$, en vertu du lemme de
  Yoneda, la donnée de l'application ci-dessus pour tout $(C, b)$ est
  équivalente à celle d'un \oo-foncteur
  \[
    \nu(L) \amalg_{\nu(K)} (\nu(K) \joint (\nu(S) \otimes \nu(T)))
    \to
    \nu(L) \amalg_{\nu(K')} (\nu(K') \joint \nu(T))
  \]
  au-dessous de $\nu(L)$, naturel en $T$ dans $\Stf$, ou, ce qui revient au
  même, d'un \oo-foncteur
  \[
    \nu(K) \joint (\nu(S) \otimes \nu(T))
    \to
    \nu(L) \amalg_{\nu(K')} (\nu(K') \joint \nu(T))
  \]
  faisant commuter le carré
  \[
    \xymatrix@C=2.5pc{
      \nu(K) \joint (\nu(S) \otimes \nu(T)) \ar[r] & \nu(L) \amalg_{\nu(K')}
      (\nu(K') \joint \nu(T)) \\
      \nu(K) \ar[u]^{\iota_1} \ar[r]_{\nu(g)} & \nu(L) \ar[u]_{\e_1} \pbox{,}
    }
  \]
  toujours naturel en $T$ dans $\Stf$.

  En vertu de la stabilité des complexes de Steiner forts par produit
  tensoriel et joint (proposition~\ref{prop:tens_Steiner} et
  corollaire~\ref{coro:joint_Steiner}) et de la compatibilité du foncteur $\nu$ à ces
  opérations (théorèmes~\ref{thm:produit_tens} et \ref{thm:joint}), le
  complexe $K \joint (S \otimes T)$ est un complexe de Steiner fort et on a
  un isomorphisme canonique
  \[
    \nu(K) \joint (\nu(S) \otimes \nu(T))
    \simeq
    \nu(K \joint (S \otimes T)).
  \]
  De même, le complexe $K' \joint T$ est un complexe de Steiner fort et on a
  un isomorphisme canonique
  \[
    \nu(K') \joint \nu(T)
    \simeq
    \nu(K' \joint T).
  \]
  De plus, puisque, d'une part, le morphisme $g' : K' \to L$ est une
  inclusion rigide ordonnée par hypothèse et, d'autre part, le morphisme
  $\iota_1 : K' \to K' \joint T$ est une inclusion rigide ordonnée en vertu
  du corollaire~\ref{coro:joint_Steiner}, le
  théorème~\ref{thm:nu_somme_amalg} entraîne que la somme amalgamée $L
  \amalg_{K'} (K' \joint T)$ est un complexe de Steiner fort et qu'on a un
  isomorphisme canonique
  \[
   \nu(L) \amalg_{\nu(K')} (\nu(K') \joint \nu(T))
   \simeq
   \nu(L \amalg_{K'} (K' \joint T)).
  \]
  Ainsi, puisque le foncteur $\nu$ restreint aux complexes de Steiner forts
  est pleinement fidèle, la donnée du \oo-foncteur $\beta_{(C, b)}$, naturel
  en~$(C, b)$ dans~$\cotr{\ooCat}{\nu(L)}$, est équivalente à celle d'un
  morphisme
  \[
    \alpha'_T : K \joint (S \otimes T) \to L \amalg_{K'} (K' \joint T),
  \]
  naturel en $T$ dans $\Stf$, faisant commuter le carré
  \[
    \xymatrix@C=2.5pc{
      K \joint (S \otimes T) \ar[r]^-{\alpha'_T} & L \amalg_{K'} (K' \joint T) \\
      K \ar[u]^{\iota_1} \ar[r]_g & L \ar[u]_{\e_1} \pbox{.}
    }
  \]

  Nous allons montrer qu'un morphisme
  \[ \alpha_{(M, a)} : S \otimes \cotr{M}{ag'} \to \cotr{M}{ag}, \]
  naturel en $(M, a)$ dans $\cotr{\Cda}{L}$, permet de définir un morphisme
  $\alpha'_T$ comme ci-dessus. En effet, par adjonction, le morphisme 
  $\alpha_{(M, a)}$ correspond à un morphisme
  \[ \cotr{M}{ag'} \to \Homig_{\Cda}(S, \cotr{M}{ag}), \]
  où $\Homig_{\Cda}$ désigne le $\Hom$ interne à gauche du
  produit tensoriel de $\Cda$ (qui existe en vertu du
  paragraphe~\ref{paragr:def_produit_cda}). On dispose donc d'une
  application
  \[
    \Hom_{\Cda}(T, \cotr{M}{ag'}) \to \Hom_{\Cda}(T,
    \Homig_{\Cda}(S, \cotr{M}{ag})),
  \]
  naturelle en $T$ dans $\Cda$, ainsi qu'en $(M, a)$ dans $\cotr{\Cda}{L}$.
  Cette application correspond, par adjonction, à une application
  \[
    \Hom_{\Cda}(T, \cotr{M}{ag'}) \to \Hom_{\Cda}(S \otimes T,
    \cotr{M}{ag}),
  \]
  ou, de nouveau par adjonction (voir la
  proposition~\ref{prop:pu_cotr_Cda}), à une application
  \[ \Hom_{\cotr{\Cda}{K'}}((K' \joint T, \iota_1), (M, ag'))
  \to \Hom_{\cotr{\Cda}{K}}((K \joint (S \otimes T), \iota_1), (M, ag)), \]
  ou encore, en vertu de la propriété universelle de la somme amalgamée, à
  celle d'une application
  \[
     \begin{split}
      \MoveEqLeft
      \Hom_{\cotr{\Cda}{L}}((L \amalg_{K'} (K' \joint T), \e_1), (M, a)) \\
      & \qquad \qquad 
      \to \Hom_{\cotr{\Cda}{L}}((L \amalg_K (K \joint (S
        \otimes T)), \e_1), (M, a)).
     \end{split}
  \]
  Ainsi, le lemme de Yoneda fournit un morphisme
  \[
    L \amalg_K (K \joint (S \otimes T)) \to L \amalg_{K'} (K' \joint T)
  \]
  au-dessous de $L$ ou, ce qui revient au même, un morphisme
  \[
    \alpha'_T : K \joint (S \otimes T) \to L \amalg_{K'} (K' \joint T)
  \]
  faisant commuter le carré
  \[
    \xymatrix@C=2.5pc{
      K \joint (S \otimes T) \ar[r]^-{\alpha'_T} & L \amalg_{K'} (K' \joint T) \\
      K \ar[u]^{\iota_1} \ar[r]_g & L \ar[u]_{\e_1} \pbox{.}
    }
  \]
  Ce morphisme est naturel en $T$ dans $\Cda$.

  Ainsi, on a associé à tout morphisme $\alpha_{(M, a)}$, naturel en $(M,
  a)$ dans $\cotr{\Cda}{L}$, un morphisme $\alpha'_T$ naturel en $T$ dans
  $\Stf$ (et même dans $\Cda$). Or, on a montré que la donnée d'un tel
  $\alpha'_T$ est équivalente à celle d'un \oo-foncteur $\beta_{(C, b)}$,
  naturel en $(C, b)$ dans $\cotr{\ooCat}{\nu(L)}$. On vérifie que la
  manière dont est défini $\beta_{(C, b)}$ à partir de $\alpha'_T$, lui-même
  défini à partir de $\alpha_{(M, a)}$, s'exprime par la commutativité du
  diagramme de l'énoncé, d'où le résultat.
\end{proof}

\begin{rem}
  La construction du lemme précédent est naturelle en $S$ au sens suivant.
  Si
  \[
    \alpha_1 : S_1 \otimes \cotr{M}{ag'} \to \cotr{M}{ag}
    \quadet
    \alpha_2 : S_2 \otimes \cotr{M}{ag'} \to \cotr{M}{ag}
  \]
  sont deux morphismes comme dans l'énoncé du lemme (en particulier,
  $S_1$ et $S_2$ sont des complexes de Steiner forts) et que $f : S_1 \to
  S_2$ est un morphisme de complexes dirigés augmentés rendant le diagramme
  \[
    \xymatrix@R=1pc{
     S_1 \otimes \cotr{M}{ag'} \ar[dr]^{\alpha_1} \ar[dd]_{f \otimes
     \cotr{M}{ag'}} \\
     & \cotr{M}{ag} \\
     S_2 \otimes \cotr{M}{ag'} \ar[ur]_{\alpha_2}
    }
  \]
  commutatif, alors le
  diagramme
  \[
    \xymatrix@R=1pc{
     \nu(S_1) \otimes \cotr{C}{b\nu(g')} \ar[dr]^{\beta_1} \ar[dd]_{\nu(f) \otimes
     \cotr{C}{b\nu(g')}} \\
     & \cotr{C}{b\nu(g)} \\
     \nu(S_2) \otimes \cotr{C}{b\nu(g')} \ar[ur]_{\beta_2} &
      \hbox{\phantom{$\cotr{C}{b\nu(g)}$}}\pbox{,}
    }
  \]
  où $\beta_1$ et $\beta_2$ sont les \oo-foncteurs associés par le lemme à
  $\alpha_1$ et $\alpha_2$ respectivement, est également commutatif. Cela
  résulte immédiatement de la propriété d'unicité du \oo-foncteur $\beta_1$ donnée par le
  lemme.
\end{rem}

Les deux énoncés qui suivent sont des reformulations de ce lemme dans les
cas \hbox{$S = \lambda(\Dn{0})$} et $S = \lambda(\Dn{1})$,
où $\Dn{0}$ et $\Dn{1}$ désignent les \oo-catégories du
paragraphe~\ref{paragr:def_disque}, cas qui seront les seuls que nous
utiliserons dans ce texte. \emph{Nous supposons, comme dans le lemme, que
les complexes $K$, $K'$ et $L$ sont des complexes de Steiner forts et que le
morphisme~$g'$ est une inclusion rigide ordonnée.}

\begin{coro}\label{coro:lemme_D0}
  Pour tout morphisme
  \[ \alpha_{(M, a)} : \cotr{M}{ag'} \to \cotr{M}{ag}, \]
  naturel en $(M, a)$ dans $\cotr{\Cda}{L}$, il existe un et un seul
  \oo-foncteur
  \[ \beta_{(C, b)} : \cotr{C}{b\nu(g')} \to \cotr{C}{b\nu(g)}, \]
  naturel en $(C, b)$ dans $\cotr{\ooCat}{\nu(L)}$, tel que, pour tout
  complexe de Steiner fort $M$ muni d'un morphisme $a : L \to M$, le carré
  \[
    \xymatrix@C=3.5pc{
      \nu(\cotr{M}{ag'})
      \ar[r]^-{\nu(\alpha_{(M, a)})}
      \ar[d]_\simeq
      &
      \nu(\cotr{M}{ag})
      \ar[d]^\simeq
      \\
      \cotr{\nu(M)}{\nu(ag')}
      \ar[r]_-{\beta_{(\nu(M), \nu(a))}}
      &
      \cotr{\nu(M)}{\nu(ag)} \pbox{,}
    }
  \]
  où les isomorphismes verticaux sont ceux de la
  proposition~\ref{prop:tr_Cda_ooCat}, soit commutatif.
\end{coro}

\begin{proof}
  C'est ce qu'on obtient en reformulant le cas $S = \lambda(\Dn{0})$ du
  lemme précédent en utilisant le fait que $\lambda(\Dn{0})$ et
  $\nu(\lambda(\Dn{0})) \simeq \Dn{0}$ sont les unités des produits
  tensoriels sur $\Cda$ et $\ooCat$ respectivement (voir les paragraphes
  \ref{paragr:def_produit_comp} et \ref{paragr:def_tens}).
\end{proof}

\begin{coro}\label{coro:lemme_D1}
  Pour toute homotopie
  \[
    \shorthandoff{;:}
    \xymatrix@C=5pc@R=3.5pc{
      \cotr{M}{ag'}
      \ar@/^3ex/[r]^(.50){\alpha'_{(M, a)}}_{}="1"
      \ar@/_3ex/[r]_(.50){\alpha_{(M, a)}}_{}="0"
      \ar@2"1";"0"|{\Gamma_{(M, a)}}
      &
      \cotr{M}{ag}
      \pbox{,}
    }
  \]
  naturelle en $(M, a)$ dans $\cotr{\Cda}{L}$, il existe une et une seule
  transformation oplax
  \[
    \shorthandoff{;:}
    \xymatrix@C=5pc@R=3.5pc{
      \cotr{C}{b\nu(g')}
      \ar@/^3ex/[r]^(.50){\beta'_{(C, b)}}_{}="1"
      \ar@/_3ex/[r]_(.50){\beta_{(C, b)}}_{}="0"
      \ar@2"1";"0"|{\Delta_{(C, b)}}
      &
      \cotr{C}{b\nu(g)}
      \pbox{,}
    }
  \]
  naturelle en $(C, b)$ dans $\cotr{\ooCat}{\nu(L)}$, telle que, pour tout
  complexe de Steiner fort $M$ muni d'un morphisme $a : L \to M$,
  le diagramme
  \[
    \shorthandoff{;:}
    \xymatrix@C=5pc@R=3.5pc{
      \nu(\cotr{L}{ag'})
      \ar[d]_{\simeq}
      \ar@/^3ex/[r]^(.50){\nu(\alpha'_{(M, a)})}_{}="1"
      \ar@/_3ex/[r]_(.50){\nu(\alpha_{(M, a)})}_{}="0"
      \ar@2"1";"0"|{\nu(\Gamma_{(M, a)})}
      &
      \nu(\cotr{L}{ag})
      \ar[d]^{\simeq}
      \\
      \cotr{\nu(L)}{\nu(ag')}
      \ar@/^3ex/[r]^(.50){\beta'_{(\nu(M), \nu(a))}}_{}="2"
      \ar@/_3ex/[r]_(.50){\beta_{(\nu(M), \nu(a))}}_{}="1"
      \ar@2"2";"1"|{\Delta_{(\nu(M), \nu(a))}}
      &
      \cotr{\nu(L)}{\nu(ag)} \pbox{,}
      }
  \]
  où les isomorphismes verticaux sont ceux de la
  proposition~\ref{prop:tr_Cda_ooCat} et $\nu(\Gamma_{(M, a)})$ désigne la
  transformation oplax associée à l'homotopie $\Gamma_{(M, a)}$ \noemph{(voir le
  paragraphe~\ref{paragr:def_nu_homot})}, soit commutatif.
\end{coro}

\begin{proof}
  C'est ce qu'on obtient en reformulant le cas $S = \lambda(\Dn{1})$ du
  lemme précédent en utilisant la description des homotopies en termes de
  produit tensoriel par $\lambda(\Dn{1})$ (voir le
  paragraphe~\ref{paragr:homot_abs}), la description des
  transformations oplax et de leurs composés avec un \oo-foncteur en termes
  de produit tensoriel par $\nu(\lambda(\Dn{1})) \simeq \Dn{1}$ (voir le
  corollaire~\ref{coro:trans_oplax_abs} et la
  proposition~\ref{prop:trans_sesqui_abs}) et la description de la
  transformation oplax associée à une homotopie (voir le
  paragraphe~\ref{paragr:def_nu_homot}).
\end{proof}

\begin{rem}\label{rem:nat_trans_oplax}
  La naturalité en $(C, b)$ dans $\cotr{\ooCat}{\nu(L)}$ de la
  transformation oplax $\Delta_{(C, b)}$ du corollaire précédent peut
  s'exprimer de la manière suivante. Pour tout triangle commutatif
  \[
    \xymatrix@C=1.5pc{
      & \nu(L) \ar[dl]_{b} \ar[dr]^{b'} \\
      C \ar[rr]_u & & C'
    }
  \]
  de $\ooCat$, le diagramme
  \[
    \shorthandoff{;:}
    \xymatrix@C=5pc@R=3.5pc{
      \cotr{C}{b\nu(g')}
      \ar[d]_{u_\ast}
      \ar@/^3ex/[r]^(.50){\beta'_{(C, b)}}_{}="1"
      \ar@/_3ex/[r]_(.50){\beta_{(C, b)}}_{}="0"
      \ar@2"1";"0"|{\Delta_{(C, b)}}
      &
      \cotr{C}{b\nu(g)}
      \ar[d]^{u_\ast}
      \\
      \cotr{C'}{b'\nu(g')}
      \ar@/^3ex/[r]^(.50){\beta'_{(C', b')}}_{}="2"
      \ar@/_3ex/[r]_(.50){\beta_{(C', b')}}_{}="1"
      \ar@2"2";"1"|{\Delta_{(C', b')}}
      &
      \cotr{C'}{b'\nu(g)} \pbox{,}
      }
  \]
  où les \oo-foncteurs verticaux sont ceux induits par $u$ vu comme
  morphisme de $\cotr{\ooCat}{\nu(K')}$ et de $\cotr{\ooCat}{\nu(K)}$
  respectivement, est commutatif au sens où on a
  \[ u_\ast \comp \Delta_{(C, b)} = \Delta_{(C', b')} \comp u_\ast. \]
  La naturalité de l'homotopie $\Gamma_{(M, a)}$ en $(M, a)$ dans
  $\cotr{\Cda}{L}$ peut s'exprimer de manière analogue.
\end{rem}

\section{\pdfoo-foncteur associé à un triangle}
\label{sec:img_tri}

\begin{paragr}\label{paragr:ann_img_tri}
  Dans cette section, on fixe
  \[
    \shorthandoff{;}
    \xymatrix@C=1.5pc{
      K \ar[rr]^f \ar[dr]_{g}_{}="f" & & K' \ar[dl]^(0.42){g'} \\
      & L
      \ar@{}"f";[ur]_(.15){}="ff"
      \ar@{}"f";[ur]_(.55){}="oo"
      \ar@<-0.5ex>@2"ff";"oo"^{h}
    }
  \]
  un triangle dans la catégorie des complexes dirigés augmentés commutant à
  une antihomotopie $h$ de~$g$ vers $g'f$ près.

  Soit $C$ une \oo-catégorie munie d'un \oo-foncteur
  \hbox{$b : \nu(L) \to C$}.
  On souhaite associer à ces données un \oo-foncteur
  \[
    (f, h, b)^\ast : \cotr{C}{b\nu(g')} \to \cotr{C}{b\nu(g)}
  \]
  \notindex{$(f, h, b)^\ast : \cotr{C}{b\nu(g')} \to \cotr{C}{b\nu(g)}$}%
  au-dessus de $C$. Nous y parviendrons sous l'hypothèse que les complexes
  $K$, $K'$ et $L$ sont des complexes de Steiner forts et que le morphisme
  $g'$ est une inclusion rigide ordonnée (voir le
  paragraphe~\ref{paragr:def_rigide_ordonnee}).

  Notons que lorsque $C = \nu(M)$ pour $M$ un complexe dirigé augmenté et
  $b = \nu(a)$ pour $a : L \to M$, on dispose bien d'un tel \oo-foncteur.
  En effet, on a un triangle
  \[
    \shorthandoff{;}
    \xymatrix@C=1.5pc{
      K \ar[rr]^f \ar[dr]_{ag}_{}="f" & & K' \ar[dl]^(0.42){ag'} \\
      & M
      \ar@{}"f";[ur]_(.15){}="ff"
      \ar@{}"f";[ur]_(.55){}="oo"
      \ar@<-0.5ex>@2"ff";"oo"^{ah}
    }
  \]
  et donc, en vertu de la proposition~\ref{prop:img_tri_Cda},
  un \oo-foncteur
  \[
    \nu((f, ah)^\ast) : \nu(\cotr{M}{ag'}) \to \nu(\cotr{M}{ag})
  \]
  qui, à travers les isomorphismes de la
  proposition~\ref{prop:tr_Cda_ooCat}, définit un \oo-foncteur
  comme souhaité.
\end{paragr}

\begin{thm}\label{thm:img_tri}
  On suppose que $K$, $K'$ et $L$ sont des complexes de Steiner forts et que
  $g$ est une inclusion rigide ordonnée. Soit $C$ une \oo-catégorie
  munie d'un \oo-foncteur $b : \nu(L) \to C$. Il existe un et un seul
  \oo-foncteur
  \[ (f, h, b)^\ast : \cotr{C}{b\nu(g')} \to \cotr{C}{b\nu(g)}, \]
  naturel en $(C, b)$ dans $\cotr{\ooCat}{\nu(L)}$, tel que, pour tout
  complexe de Steiner fort $M$ muni d'un morphisme $a : L \to M$, le
  carré
  \[
    \xymatrix@C=3.5pc{
      \nu(\cotr{M}{ag'})
      \ar[r]^-{\nu((f, ah)^\ast)}
      \ar[d]_\simeq
      &
      \nu(\cotr{M}{ag})
      \ar[d]^\simeq
      \\
      \cotr{\nu(M)}{\nu(ag')}
      \ar[r]_-{(f, h, \nu(a))^\ast}
      &
      \cotr{\nu(M)}{\nu(ag)}
      \pbox{,}
    }
  \]
  où les isomorphismes verticaux sont ceux de la
  proposition~\ref{prop:tr_Cda_ooCat}, soit commutatif.
\end{thm}

\begin{proof}
  C'est ce qu'affirme le corollaire~\ref{coro:lemme_D0} pour
  \[
    \alpha_{(M, a)} = (f, ah)^\ast : \cotr{M}{ag'} \to \cotr{M}{ag},
  \]
  dont la naturalité résulte immédiatement des formules du
  paragraphe~\ref{paragr:img_tri_Cda}.
\end{proof}

\begin{rem}\label{rem:img_tri_expl}
  En déroulant la preuve du théorème précédent (et donc celle du
  lemme~\ref{lemme:fonct_Cda_ooCat}) et en utilisant les formules du
  paragraphe~\ref{paragr:img_tri_Cda} définissant le morphisme $(f, ah)^\ast
  : \cotr{M}{ag'} \to \cotr{M}{ag}$, on peut décrire le
  \oo-foncteur~$(f, h, b)^\ast$ de la manière suivante. Soit $T$ un complexe
  de Steiner fort. Il suffit de décrire l'application
  \[
    \Hom_{\ooCat}(\nu(T), \cotr{C}{b\nu(g')}) \to \Hom_{\ooCat}(\nu(T),
    \cotr{C}{b\nu(g)})
  \]
  induite par $(f, h, b)^\ast$. Or, on a des bijections naturelles
  \[
     \begin{split}
       \Hom_{\ooCat}(\nu(T), \cotr{C}{b\nu(g')})
     & \simeq \Hom_{\cotr{\ooCat}{\nu(L)}}((\nu(L \amalg_{K'} (K' \joint T)),
     \nu(\e_1)), (C, b)) \\
     & \subset \Hom_{\ooCat}(\nu(L \amalg_{K'} (K' \joint T)), C), \\
     \Hom_{\ooCat}(\nu(T), \cotr{C}{b\nu(g)})
     & \simeq \Hom_{\cotr{\ooCat}{\nu(K)}}((\nu(K \joint T), \nu(\iota_1)),
     (C, b\nu(g))) \\
     & \subset \Hom_{\ooCat}(\nu(K \joint T), C),
     \end{split}
   \]
  où $\e_1$ désigne la première inclusion canonique,
  et l'application
  \[
    \Hom_{\ooCat}(\nu(T), \cotr{C}{b\nu(g')}) \to \Hom_{\ooCat}(\nu(T),
    \cotr{C}{b\nu(g)})
  \]
  est induite par le morphisme
  \[
    (f, h)_\ast :
    K \joint T
    \to
    L \amalg_{K'} (K' \joint T)
  \]
  défini de la manière suivante :
  \[
      (f, h)_\ast(x \joint y) =
      \begin{cases}
        g(x) & \text{si $y = \vide$,} \\
        f(x) \joint y + e(y)h(x) & \text{sinon,}
      \end{cases}
  \]
  où on convient que $f(\vide) = \vide$, $h(\vide) = 0$ et $e(y) = 0$ si $y$
  n'est pas de degré $0$. Nous n'utiliserons pas cette description de $(f,
  h, b)^\ast$ dans la suite de ce texte.
\end{rem}

\begin{paragr}
  Considérons le cas particulier d'un triangle commutatif
  \[
    \shorthandoff{;}
    \xymatrix@C=1.5pc{
      K \ar[rr]^f \ar[dr]_{g}_{}="f" & & K' \ar[dl]^(0.42){g'} \\
                                     & L
      \ar@{}"f";[ur]_(.15){}="ff"
      \ar@{}"f";[ur]_(.55){}="oo"
      \ar@<-0.5ex>@2"ff";"oo"^{\id{g}}
      & \pbox{,}
    }
  \]
  de complexes dirigés augmentés, c'est-à-dire le cas où $h$ est
  l'antihomotopie identité de $g = g'f$. On suppose que les complexes $K$,
  $K'$ et $L$ sont des complexes de Steiner forts et que le morphisme $g'$
  est une inclusion rigide ordonnée.

  Soit $C$ une \oo-catégorie munie d'un \oo-foncteur
  \hbox{$b : \nu(L) \to C$}. On pose
  \[
    c = b\nu(g) \quadet c' = b\nu(g').
  \]

  En vertu du théorème~\ref{thm:img_tri}, on dispose d'un \oo-foncteur
  \[ (f, \id{g}, b)^\ast : \cotr{C}{c'} \to \cotr{C}{c}. \]
  Par ailleurs, en considérant le triangle commutatif
  \[
    \shorthandoff{;}
    \xymatrix@C=1.5pc{
    \nu(K) \ar[rr]^{\nu(f)} \ar[dr]_{b\nu(g)}_{}="f" & &
    \nu(K') \ar[dl]^{b\nu(g')} \\
    & C & \pbox{,}
    }
  \]
  obtenu un appliquant le foncteur $\nu$ et en composant par $b : \nu(L)
  \to C$, on obtient, en vertu du paragraphe~\ref{paragr:desc_morph_tr}, un
  \oo-foncteur
  \[ \nu(f)^\ast : \cotr{C}{c'} \to \cotr{C}{c}. \]
\end{paragr}

\begin{prop}\label{prop:fonct_img_tri_comm}
  Les \oo-foncteurs $(f, \id{g}, b)^\ast$ et $\nu(f)^\ast$ du paragraphe
  ci-dessus coïncident.
\end{prop}

\begin{proof}
  Il résulte de la définition de $\nu(f)^\ast : \cotr{C}{c'} \to
  \cotr{C}{c}$ (paragraphe~\ref{paragr:desc_morph_tr}) que ce \oo-foncteur
  est naturel en $(C, b)$ dans $\cotr{\ooCat}{\nu(L)}$. Ainsi, en vertu de
  la caractérisation du \oo-foncteur $(f, 1, b)^\ast$ donnée par le
  théorème~\ref{thm:img_tri}, il suffit de vérifier, pour tout complexe
  de Steiner fort $M$ muni d'un morphisme $a : L \to M$, la commutativité du
  carré
  \[
    \xymatrix@C=3.5pc{
      \nu(\cotr{M}{ag'})
      \ar[r]^-{\nu((f, \id{ag})^\ast)}
      \ar[d]_\simeq
      &
      \nu(\cotr{M}{ag})
      \ar[d]^\simeq
      \\
      \cotr{\nu(M)}{\nu(ag')}
      \ar[r]_-{\nu(f)^\ast}
      &
      \cotr{\nu(M)}{\nu(ag)}
      \pbox{,}
    }
  \]
  où les isomorphismes verticaux sont ceux de la
  proposition~\ref{prop:tr_Cda_ooCat}, ce qui résulte de
  la proposition~\ref{prop:morph_tr_Cda_ooCat}.
\end{proof}

\begin{rem}\label{rem:conv_pas_rig_ord}
  Si l'on ne suppose pas que le morphisme $g'$ est une inclusion rigide
  ordonnée, le théorème~\ref{thm:img_tri} ne fournit par de \oo-foncteur
  $(f, \id{g}, b)^\ast$. Dans ce cas, il sera parfois commode de définir ce
  \oo-foncteur par l'égalité de la proposition précédente :
  \[ (f, \id{g}, b)^\ast = \nu(f)^\ast. \]
\end{rem}

\begin{rem}\label{rem:img_tri_oubli}
  Il résulte de la proposition~\ref{prop:fonct_img_tri_comm} et de la
  convention de la remarque précédente que, si $g : K
  \to L$ est un morphisme entre complexes de Steiner forts,
  en considérant le triangle
  \[
    \shorthandoff{;}
    \xymatrix@C=1.5pc{
    \vide \ar[rr]^{\vide_K} \ar[dr]_{\vide_L}_{}="f" & & K \ar[dl]^{g} \\
      & L
      \ar@{}"f";[ur]_(.15){}="ff"
      \ar@{}"f";[ur]_(.55){}="oo"
      \ar@<-1.0ex>@2"ff";"oo"^(0.50){\id{\vide_{\!L}}\,} & \pbox{,}
    }
  \]
  où $\vide_K$ et $\vide_L$ désignent les uniques morphismes de source
  $\vide$ et de buts respectifs $K$ et $L$, le \oo-foncteur
  \[
    (\vide_K, \id{\vide_{\!L}}, b)^\ast : \cotr{C}{c} \to \cotr{C}{\vide} \simeq C
  \]
  est le \oo-foncteur d'oubli du paragraphe~\ref{paragr:desc_morph_tr}.
\end{rem}

\section{Fonctorialité des \pdfoo-foncteurs associés aux triangles}
\label{sec:fonct_tri}

\begin{paragr}\label{paragr:fonct_tri_id}
  Soient $g : K \to L$ une inclusion rigide ordonnée entre complexes de
  Steiner forts et $C$ une \oo-catégorie munie d'un \oo-foncteur \hbox{$b :
  \nu(L) \to C$}. On pose~$c = b\nu(g)$.

  Formons, comme dans le paragraphe~\ref{paragr:fonct_tri_id_Cda}, le triangle commutatif
  \[
    \shorthandoff{;}
    \xymatrix@C=1.5pc{
    K \ar[rr]^{\id{K}} \ar[dr]_(0.40){g}_{}="f" & & K \ar[dl]^(0.40){g} \\
      & L
      \ar@{}"f";[ur]_(.15){}="ff"
      \ar@{}"f";[ur]_(.55){}="oo"
      \ar@<-0.5ex>@2"ff";"oo"^{\id{g}}
      & \pbox{.}
    }
  \]
  À partir de ce triangle, en vertu du théorème~\ref{thm:img_tri}, on obtient un
  \oo-foncteur $(\id{K}, \id{g}, b)^\ast : \cotr{C}{c} \to \cotr{C}{c}$.
\end{paragr}

\begin{prop}
  On a $(\id{K}, \id{g}, b)^\ast = \id{\cotr{C}{c}}$.
\end{prop}

\begin{proof}
  En vertu de la proposition~\ref{prop:fonct_img_tri_comm}, on a
  $(\id{K}, \id{g}, b)^\ast = \nu(\id{K})^\ast$, où $\nu(\id{K})^\ast$ est
  le \oo-foncteur du paragraphe~\ref{paragr:desc_morph_tr}. Or, il est
  immédiat que $\nu(\id{K})^\ast$ est le \oo-foncteur identité.
\end{proof}

\begin{paragr}\label{paragr:ann_fonct_tri}
  Considérons maintenant un diagramme
  \[
    \shorthandoff{;}
    \xymatrix@C=2.5pc@R=2.5pc{
      K \ar[r]^f \ar[dr]_{}="g"_(.40){g}
      & K' \ar[r]^{f'}_(.75){}="fp" \ar[d]_(.70){}="gp"_(.50){g'} & K''
      \ar[dl]_{}="gpp"^(.33){g''} \\
      & L
      \ar@{}"g";[u]_(0.10){}="x"
      \ar@{}"g";[u]_(.85){}="y"
      \ar@<-0.1ex>@2"x";"y"^(.30)h
      \ar@{}"gp";"fp"_(.25){}="x2"
      \ar@{}"gp";"fp"_(.75){}="y2"
      \ar@<0.4ex>@2"x2";"y2"^(0.40){h'\!}
    }
  \]
  de complexes dirigés augmentés, où $h$ et $h'$ sont des antihomotopies
  de~$g$ vers $g'f$ et de~$g'$ vers $g''f'$ respectivement. On suppose que
  les complexes $K$, $K'$, $K''$ et $L$ sont des complexes de Steiner forts
  et que les morphismes $g'$ et $g''$ sont des inclusions rigides ordonnées.

  Soit $C$ une \oo-catégorie munie d'un \oo-foncteur
  \hbox{$b : \nu(L) \to C$}. On pose
  \[
    c = b\nu(g), \quad c' = b\nu(g') \quadet c'' = b\nu(g'').
  \]

  En composant le diagramme ci-dessus, on obtient un triangle
  \[
    \shorthandoff{;}
    \xymatrix@C=1.5pc{
    K \ar[rr]^{f'f} \ar[dr]_{g}_{}="f" & & K'' \ar[dl]^(0.42){g''} \\
    & L
    \ar@{}"f";[ur]_(.15){}="ff"
    \ar@{}"f";[ur]_(.55){}="oo"
    \ar@<-0.5ex>@2"ff";"oo"^{h''}
    & \pbox{,}
    }
  \]
  où $h''$ est l'antihomotopie  $h'f + h$. En vertu du
  théorème~\ref{thm:img_tri}, on obtient donc un triangle
  \[
    \xymatrix@C=1.5pc{
      \cotr{C}{c''} \ar[dr]_{(f', h', b)^\ast} \ar[rr]^{(f'f, h'f + h, b)^\ast}
      & & \cotr{C}{c} \\
      & \cotr{C}{c'} \ar[ur]_{(f, h, b)^\ast} 
    }
  \]
  de \oo-foncteurs.
\end{paragr}

La proposition suivante affirme que ce triangle est commutatif.

\begin{prop}\label{prop:fonct_tri}
  On a $(f, h, b)^\ast(f', h', b)^\ast = (f'f, h'f + h, b)^\ast$.
\end{prop}

\begin{proof}
  Par définition, les \oo-foncteurs $(f, h, b)^\ast$ et $(f', h', b)^\ast$
  sont naturels en $(C, b)$ et il en est donc de même de leur composé $(f,
  h, b)^\ast(f', h', b)^\ast$. Ainsi, en vertu de la caractérisation du
  \oo-foncteur $(f'f, h'f + h, b)^\ast$ donnée par le
  théorème~\ref{thm:img_tri}, il suffit de vérifier l'égalité recherchée
  lorsque $C = \nu(M)$ pour $M$ un complexe de Steiner fort et $b = \nu(a)$
  pour $a : L \to M$. Or, cette égalité résulte de la
  proposition~\ref{prop:fonct_tri_Cda} appliquée au diagramme
  \[
    \shorthandoff{;}
    \begin{gathered}[b]
    \xymatrix@C=2.5pc@R=2.5pc{
      K \ar[r]^f \ar[dr]_{}="g"_(.40){ag}
      & K' \ar[r]^{f'}_(.75){}="fp" \ar[d]_(.70){}="gp"_(.50){ag'} & K''
      \ar[dl]_{}="gpp"^(.33){ag''} \\
      & M
      \ar@{}"g";[u]_(0.10){}="x"
      \ar@{}"g";[u]_(.85){}="y"
      \ar@<-0.1ex>@2"x";"y"^(.30){ah}
      \ar@{}"gp";"fp"_(.25){}="x2"
      \ar@{}"gp";"fp"_(.75){}="y2"
      \ar@<0.4ex>@2"x2";"y2"^{ah'}
      &
      \pbox{.}
    }
    \\[-\dp\strutbox]
    \end{gathered}
    \qedhere
  \]
\end{proof}

\begin{rem}\label{rem:conv_fonct_tri}
  Si $h$ est l'antihomotopie $\id{g}$, alors, avec le convention énoncée
  dans la remarque~\ref{rem:conv_pas_rig_ord}, la conclusion de la
  proposition précédente reste valable sans supposer que $g'$ est une
  inclusion rigide ordonnée. En effet, dans la preuve de la proposition,
  cette hypothèse sur $g'$ n'est utilisée que pour justifier l'existence de
  $(f, h, b)^\ast$ et sa naturalité en $(C, b)$.
\end{rem}

\begin{prop}\label{prop:img_tri_au-dessus}
  Le \oo-foncteur
  \[
    (f, h, b)^\ast : \cotr{C}{c'} \to \cotr{C}{c}
  \]
  du théorème~\ref{thm:img_tri} est au-dessus de $C$. Autrement dit, le
  triangle
  \[
    \xymatrix@C=1.5pc{
      \cotr{C}{c'} \ar[rr]^{(f, h, b)^\ast} \ar[rd]_{U'}
      & & \cotr{C}{c} \ar[dl]^{U} \\
      & C & \pbox{,}
    }
  \]
  où $U$ et $U'$ désignent les \oo-foncteurs d'oubli, est commutatif.
\end{prop}

\begin{proof}
  En appliquant la proposition~\ref{prop:fonct_tri} sous la forme donnée par
  la remarque précédente au diagramme
  \[
    \shorthandoff{;}
    \xymatrix@C=2.5pc@R=2.5pc{
    \vide \ar[r]^{\vide_K} \ar[dr]_{}="g"_(.40){\vide_L}
      & K \ar[r]^{f}_(.75){}="fp" \ar[d]_(.70){}="gp"_(.54){g} & K'
      \ar[dl]_{}="gpp"^(.33){g'} \\
      & L
      \ar@{}"g";[u]_(0.10){}="x"
      \ar@{}"g";[u]_(.85){}="y"
      \ar@<-0.1ex>@2"x";"y"^(.30){\id{\vide_{\!L}}\!\!}
      \ar@{}"gp";"fp"_(.25){}="x2"
      \ar@{}"gp";"fp"_(.75){}="y2"
      \ar@<0.4ex>@2"x2";"y2"^{h\,}
      & \pbox{,}
    }
  \]
  où $\vide_M$, pour $M$ un complexe dirigé augmenté, désigne l'unique
  morphisme de $\vide$ vers~$M$, on obtient l'égalité
  \[
    (\vide_K, \id{\vide_{\!L}}, b)^\ast (f, h, b)^\ast
    = (\vide_{K'}, \id{\vide_{\!L}}, b)^\ast.
  \]
  Or, en vertu de la remarque~\ref{rem:img_tri_oubli}, on a
  \[
    (\vide_K, \id{\vide_{\!L}}, b)^\ast = U
    \quadet
    (\vide_{K'}, \id{\vide_{\!L}}, b)^\ast = U',
  \]
  d'où le résultat.
\end{proof}

\section{Transformation oplax associée à un cône}

\begin{paragr}\label{paragr:ann_img_cone}
  Dans cette section, on fixe un diagramme
    \[
      \shorthandoff{;:}
      \xymatrix@C=1.5pc@R=3pc{
        K \ar@/^2ex/[rr]^(.33){f'}_{}="1" \ar@/_2ex/[rr]^(.30)f_{}="0"
        \ar[dr]_{}="f"_{\phantom{g'}g}
        \ar@2"0";"1"_k
        & & K' \ar[dl]^{g'} \\
        & L
        \ar@{}"f";[ur]_(.15){}="ff"
        \ar@{}"f";[ur]_(.55){}="oo"
        \ar@<-0.5ex>@/^1ex/@{:>}"ff";"oo"^(.18){h'\!\!}_(.30){}="h'"
        \ar@<-2.0ex>@/^-1ex/@2"ff";"oo"_(.36){h}_(.80){}="h"
        \ar@3"h";"h'"_(.20){H_{}}
        }
    \]
  de complexes dirigés augmentés, où $h$ et $h'$ sont des antihomotopies
  de source $g$ et de buts respectifs $g'f$ et $g'f'$, $k$ est une
  antihomotopie de $f$ vers $f'$ et $H$ est une $2$\nbd-antihomotopie de $g'k +
  h$ vers $h'$. On suppose que les complexes $K$, $K'$ et $L$ sont des
  complexes de Steiner forts et que $g'$ est une inclusion rigide ordonnée.

  Soit $C$ une \oo-catégorie munie d'un \oo-foncteur \hbox{$b : \nu(L) \to
  C$}. En vertu du théorème~\ref{thm:img_tri} et de la
  proposition~\ref{prop:img_tri_au-dessus}, on peut associer à ces données
  des \oo-foncteurs
  \[
    (f, h, b)^\ast, (f', h', b)^\ast : \cotr{C}{b\nu(g')} \to
    \cotr{C}{b\nu(g)}
  \]
  au-dessus de $C$. Nous allons définir une transformation oplax
  \[
    \shorthandoff{;:}
    \xymatrix@C=5pc@R=3.5pc{
      \cotr{C}{b\nu(g')}
      \ar@/^3ex/[r]^(.50){(f', h', b)^\ast}_{}="1"
      \ar@/_3ex/[r]_(.50){(f, h, b)^\ast}_{}="0"
      \ar@2"1";"0"|{(k, H, b)^\ast}
      &
      \cotr{C}{b\nu(g)}
    }
  \]
  \notindex{$(k, H, b)^\ast : (f', h', b)^\ast \Rightarrow (f, h, b)^\ast$}%
  au-dessus de $C$.

  Notons que lorsque $C = \nu(M)$ pour $M$ un complexe dirigé augmenté et
  $b = \nu(a)$ pour $a : L \to M$, on dispose bien d'une telle
  transformation oplax. En effet, on a un diagramme
  \[
      \shorthandoff{;:}
      \xymatrix@C=1.5pc@R=3pc{
        K \ar@/^2ex/[rr]^(.33){f'}_{}="1" \ar@/_2ex/[rr]^(.30)f_{}="0"
        \ar[dr]_{}="f"_{\phantom{g'}ag}
        \ar@2"0";"1"_k
        & & K' \ar[dl]^{ag'} \\
        & M
        \ar@{}"f";[ur]_(.15){}="ff"
        \ar@{}"f";[ur]_(.55){}="oo"
        \ar@<-0.5ex>@/^1ex/@{:>}"ff";"oo"^(.18){ah'\!\!}_(.30){}="h'"
        \ar@<-2.0ex>@/^-1ex/@2"ff";"oo"_(.30){ah\,\,}_(.80){}="h"
        \ar@3"h";"h'"_(.20){aH_{}}
        }
  \]
  et donc, en vertu de la proposition~\ref{prop:img_cone_Cda},
  une homotopie
  \[
    \shorthandoff{;:}
    \xymatrix@C=5pc@R=3.5pc{
      \cotr{M}{ag'}
      \ar@/^3ex/[r]^(.50){(f', ah')^\ast}_{}="1"
      \ar@/_3ex/[r]_(.50){(f, ah)^\ast}_{}="0"
      \ar@2"1";"0"|{(k, aH)^\ast}
      &
      \cotr{M}{ag} \pbox{.}
    }
  \]
  Or, en vertu du paragraphe~\ref{paragr:def_nu_homot}, cette homotopie
  induit une transformation oplax
  \[
    \shorthandoff{;:}
    \xymatrix@C=5pc@R=3.5pc{
      \nu(\cotr{M}{ag'})
      \ar@/^3ex/[r]^(.50){\nu((f', ah')^\ast)}_{}="1"
      \ar@/_3ex/[r]_(.50){\nu((f, ah)^\ast)}_{}="0"
      \ar@2"1";"0"|{\nu((k, aH)^\ast)}
      &
      \nu(\cotr{M}{ag}) \pbox{,}
    }
  \]
  transformation oplax qui, à travers les isomorphismes de la
  proposition~\ref{prop:tr_Cda_ooCat}, définit une transformation oplax
  comme souhaitée.
\end{paragr}

\begin{thm}\label{thm:img_cone}
  Soit $C$ une \oo-catégorie munie d'un \oo-foncteur \hbox{$b : \nu(L) \to
  C$}. Il existe une et une seule transformation oplax
  \[
    \shorthandoff{;:}
    \xymatrix@C=5pc@R=3.5pc{
      \cotr{C}{b\nu(g')}
      \ar@/^3ex/[r]^(.50){(f', h', b)^\ast}_{}="1"
      \ar@/_3ex/[r]_(.50){(f, h, b)^\ast}_{}="0"
      \ar@2"1";"0"|{(k, H, b)^\ast}
      &
      \cotr{C}{b\nu(g)}
    }
  \]
  naturelle en $(C, b)$ dans $\cotr{\ooCat}{\nu(L)}$ \noemph{(voir la
  remarque~\ref{rem:nat_trans_oplax})}, telle que, pour tout complexe de
  Steiner fort $M$ muni d'un morphisme $a : L \to M$, le diagramme
  \[
    \shorthandoff{;:}
    \xymatrix@C=5pc@R=3.5pc{
      \nu(\cotr{M}{ag'})
      \ar[d]_{\simeq}
      \ar@/^3ex/[r]^(.50){\nu((f', ah')^\ast)}_{}="1"
      \ar@/_3ex/[r]_(.50){\nu((f, ah)^\ast)}_{}="0"
      \ar@2"1";"0"|{\nu((k, aH)^\ast)}
      &
      \nu(\cotr{M}{ag})
      \ar[d]^{\simeq}
      \\
      \cotr{\nu(M)}{\nu(ag')}
      \ar@/^3ex/[r]^(.50){(f', h', \nu(a))^\ast}_{}="2"
      \ar@/_3ex/[r]_(.50){(f, h, \nu(a))^\ast}_{}="1"
      \ar@2"2";"1"|{(k, H, \nu(a))^\ast}
      &
      \cotr{\nu(M)}{\nu(ag)} \pbox{,}
      }
  \]
  où les isomorphismes verticaux sont ceux de la
  proposition~\ref{prop:tr_Cda_ooCat} et $\nu((k, aH)^\ast)$ désigne la
  transformation oplax associée à l'homotopie $(k, aH)^\ast$ par la
  construction du paragraphe~\ref{paragr:def_nu_homot}, soit
  commutatif.
\end{thm}

\begin{proof}
  C'est ce qu'affirme le corollaire~\ref{coro:lemme_D1} pour
  \[
    \Gamma_{(M, a)} =
    \shorthandoff{;:}
    \xymatrix@C=5pc@R=3.5pc{
      \cotr{M}{ag'}
      \ar@/^3ex/[r]^(.50){(f', ah')^\ast}_{}="1"
      \ar@/_3ex/[r]_(.50){(f, ah)^\ast}_{}="0"
      \ar@2"1";"0"|{(k, aH)^\ast}
      &
      \cotr{M}{ag}
      \pbox{,}
    }
  \]
  dont la naturalité résulte immédiatement des formules du
  paragraphe~\ref{paragr:img_cone_Cda}.
\end{proof}

\begin{rem}
  En déroulant la preuve du théorème précédent (et donc celle du
  lemme~\ref{lemme:fonct_Cda_ooCat}) et en utilisant les formules du
  paragraphe~\ref{paragr:img_cone_Cda} définissant l'homotopie $(k,
  aH)^\ast$, on peut décrire la transformation oplax $(k, H,
  b)^\ast$, ou plus précisément le \oo-foncteur $\cotr{C}{b\nu(g')} \to
  \HomLax(\Dn{1}, \cotr{C}{b\nu(g)})$ qui lui est associé (voir la
  proposition~\ref{prop:cotrans_abs}), de la manière suivante. Soit $T$ un
  complexe de Steiner fort. Il suffit de décrire l'application
  \[
    \Hom_{\ooCat}(\nu(T), \cotr{C}{b\nu(g')}) \to \Hom_{\ooCat}(\nu(T),
    \HomLax(\Dn{1}, \cotr{C}{b\nu(g)}))
  \]
  induite par ce \oo-foncteur. Or, on a des bijections naturelles
  \[
    \Hom_{\ooCat}(\nu(T), \cotr{C}{b\nu(g')})
     \simeq \Hom_{\cotr{\ooCat}{\nu(L)}}((\nu(L \amalg_{K'} (K' \joint T)),
     \nu(\e_1)), (C, b)),
  \]
  où $\e_1$ désigne la première inclusion canonique, et
   \[
     \begin{split}
       \MoveEqLeft \Hom_{\ooCat}(\nu(T), \HomLax(\Dn{1}, \cotr{C}{b\nu(g)})) \\
     & \simeq \Hom_{\cotr{\ooCat}{\nu(K)}}((\nu(K \joint (\lambda(\Dn{1})
       \otimes T)), \nu(\iota_1)), (C, b\nu(g))),
     \end{split}
   \]
  et l'application
  \[
    \Hom_{\ooCat}(\nu(T), \cotr{C}{b\nu(g')}) \to \Hom_{\ooCat}(\nu(T),
    \HomLax(\Dn{1}, \cotr{C}{b\nu(g)}))
  \]
  est induite par le morphisme
  \[
    (k, H)_\ast : K \joint (\lambda(\Dn{1}) \otimes T) \to L \amalg_{K'} (K' \joint T)
  \]
  défini de la manière suivante :
  \[
    \begin{split}
      (k, H)_\ast(x \joint \vide) & = g(x), \\
      (k, H)_\ast(x \joint (d^0_0 \otimes y)) & = f'(x) \joint y + e(y)h'(x), \\
      (k, H)_\ast(x \joint (d^1_0 \otimes y)) & = f(x) \joint y + e(y)h(x), \\
      (k, H)_\ast(x \joint (d_1 \otimes y)) & = k(x) \joint y + e(y)H(x),
   \end{split}
  \]
  où $d$ désigne la cellule principale de $\Dn{1}$ et où, en plus des
  conventions du paragraphe~\ref{rem:img_tri_expl}, on convient que
  $k(\vide) = 0$ et que~$H(\vide) = 0$. Nous n'utiliserons pas
  cette description de $(k, H, b)^\ast$ dans la suite de ce texte.
\end{rem}

\section{Fonctorialités des transformations oplax associées aux cônes}
\label{sec:fonct_tr_cone}

\begin{paragr}
  Considérons un diagramme
  \[
    \shorthandoff{;}
    
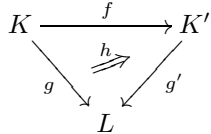
\xymatrix@C=1.5pc{
      K \ar[rr]^f \ar[dr]_{g}_{}="f" & & K' \ar[dl]^(0.42){g'} \\
      & L
      \ar@{}"f";[ur]_(.15){}="ff"
      \ar@{}"f";[ur]_(.55){}="oo"
      \ar@<-0.5ex>@2"ff";"oo"^{h}
    }
  \]
  de complexes dirigés augmentés, où $h$ est une antihomotopie de $g$ vers
  $g'f$. On suppose que les complexes $K$, $K'$ et $L$ sont des complexes de
  Steiner forts et que le morphisme~$g'$ est une inclusion rigide ordonnée.

  Soit $C$ une \oo-catégorie munie d'un \oo-foncteur
  \hbox{$b : \nu(L) \to C$}. On pose
  \[
    c = b\nu(g) \quadet c' = b\nu(g').
  \]

  À partir de ce diagramme, on peut former, comme dans le
  paragraphe~\ref{paragr:fonct_cone_id_Cda}, le diagramme
  \[
    \shorthandoff{;:}
    
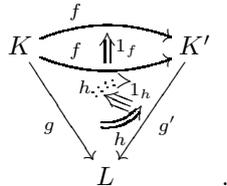
\xymatrix@C=1.5pc@R=3pc{
    K \ar@/^2ex/[rr]^(.33){f}_{}="1" \ar@/_2ex/[rr]^(.30)f_{}="0"
    \ar[dr]_{}="f"_{\phantom{g'}g}
    \ar@2"0";"1"_{\id{f}}
    & & K' \ar[dl]^{g'} \\
    & L
    \ar@{}"f";[ur]_(.15){}="ff"
    \ar@{}"f";[ur]_(.55){}="oo"
    \ar@<-0.5ex>@/^1ex/@{:>}"ff";"oo"^(.18){h\!\!}_(.30){}="h'"
    \ar@<-2.0ex>@/^-1ex/@2"ff";"oo"_(.36){h}_(.80){}="h"
    \ar@3"h";"h'"_(.20){\,\id{h}}
    & \pbox{.}
    }
  \]
  On obtient donc, en vertu du théorème~\ref{thm:img_cone}, une
  transformation oplax~$(\id{f}, \id{h}, b)^\ast$ de $(f, h, b)^\ast$ vers $(f, h,
  b)^\ast$.
\end{paragr}

\begin{prop}\label{prop:fonct_cone_id}
  On a $(\id{f}, \id{h}, b)^\ast = \id{(f, h, b)^\ast}$.
\end{prop}

\begin{proof}
  Le \oo-foncteur $(f, h, b)^\ast$ est naturel en $(C, b)$ par définition et
  il en est donc de même de sa transformation oplax identité $\id{(f, h,
  b)^\ast}$ (cela résulte formellement du fait que les transformations oplax
  identité sont des identités pour la composition horizontale par un
  \oo-foncteur).  Ainsi, en vertu de la caractérisation de la transformation
  oplax $(\id{f}, \id{h}, b)^\ast$ donnée par le
  théorème~\ref{thm:img_cone}, il suffit de démontrer l'égalité recherchée
  lorsque \hbox{$C = \nu(M)$} pour $M$ un complexe de Steiner fort et $b =
  \nu(a)$ pour~$a : L \to M$. Or, cela résulte de la
  proposition~\ref{prop:fonct_cone_id_Cda} appliquée au diagramme
  \[
    \shorthandoff{;}
    \xymatrix@C=1.5pc{
      K \ar[rr]^f \ar[dr]_{ag}_{}="f" & & K' \ar[dl]^(0.42){ag'} \\
      & M
      \ar@{}"f";[ur]_(.15){}="ff"
      \ar@{}"f";[ur]_(.55){}="oo"
      \ar@<-0.5ex>@2"ff";"oo"^{ah\,}
      & \pbox{,}
    }
  \]
  ainsi que de la proposition~\ref{prop:compat_mu_id} qui
  affirme que $\nu$ envoie une homotopie identité sur une transformation
  oplax identité.
\end{proof}

\begin{paragr}\label{paragr:fonct_cone_sg}
  Considérons maintenant un diagramme
    \[
      \shorthandoff{;:}
      \xymatrix@C=3.5pc@R=3.5pc{
      K  \ar[r]^f \ar[dr]_{}="g"_(.4){\phantom{g''}g\!} &
      K' \ar@/^2ex/[r]^(.33){f''}_{}="1"
      \ar@/_2ex/[r]^(.30){f'}_{}="0"_(.70){}="fp"
      \ar[d]_(.50){}="gp2"_(.20){}="gp"_(0.52){g'} &
      K'' \ar[dl]^(.4){\!g''} &
      \ar@2"0";"1"_{k} \\
        & L
      \ar@{}"g";[u]_(0.10){}="x"
      \ar@{}"g";[u]_(.75){}="y"
      \ar@<-0.1ex>@2"x";"y"^(.30)h
      \ar@{}"gp2";"fp"_(.10){}="ff2"
      \ar@{}"gp2";"fp"_(.55){}="oo2"
      \ar@<+0.5ex>@/^1ex/@{:>}"ff2";"oo2"^{\!\!h''}_(.30){}="h'''"
      \ar@<-0.5ex>@/^-1.5ex/@2"ff2";"oo2"_(.47){\!\!\!h'}_(.80){}="h''"
      \ar@3"h''";"h'''"_(.20){H}
      }
    \]
    de complexes dirigés augmentés, où $h$, $h'$ et $h''$
    sont des antihomotopies de $g$ vers~$g'f$, de $g'$ vers $g''f'$ et de
    $g'$ vers $g''f''$ respectivement, $k$ est une antihomotopie de $f'$
    vers~$f''$ et $H$ est une $2$-antihomotopie de $g''k + h'$ vers
    $h''$. On suppose que les complexes $K$, $K'$, $K''$ et $L$ sont
    des complexes de Steiner forts et que les morphismes $g'$ et $g''$ sont
    des inclusions rigides ordonnées.

    Soit $C$ une \oo-catégorie munie d'un \oo-foncteur
    \hbox{$b : \nu(L) \to C$}. On pose
    \[
      c = b\nu(g),\quad c' = b\nu(g') \quadet c'' = b\nu(g'').
    \]

    À partir de ces données, comme dans le
    paragraphe~\ref{paragr:fonct_cone_sg_Cda} et avec les mêmes notations,
    en utilisant les théorèmes~\ref{thm:img_tri} et~\ref{thm:img_cone}, ainsi
    que la proposition~\ref{prop:fonct_tri}, on obtient deux transformations
    oplax
    \[ (f, h, b)^\ast \comp (k, H, b)^\ast \quadet (kf, Hf, b)^\ast \]
    de $(f''f, h''f+h, b)^\ast$ vers $(f'f, h'f+h, b)^\ast$, qui sont deux
    \oo-foncteurs de $\cotr{C}{c''}$ vers~$\cotr{C}{c}$.
\end{paragr}

\begin{prop}\label{prop:fonct_cone_sg}
  On a $(f, h, b)^\ast \comp (k, H, b)^\ast = (kf, Hf, b)^\ast$.
\end{prop}

\begin{proof}
  Le \oo-foncteur $(f, h, b)^\ast$ et la transformation oplax $(k, H,
  b)^\ast$ sont naturels en $(C, b)$ par définition et il en est donc de
  même de la transformation oplax composée $(f, h, b)^\ast \comp (k, H,
  b)^\ast$ (cela résulte formellement de l'associativité de la composition
  horizontale d'un \oo-foncteur et d'une transformation oplax). Ainsi, en
  raisonnant comme dans la preuve de la proposition précédente, on se ramène
  au cas où $C = \nu(M)$ pour $M$ un complexe de Steiner fort et $b =
  \nu(a)$ pour $a : L \to M$. Or, ce cas résulte de la
  proposition~\ref{prop:fonct_cone_sg_Cda} appliquée au diagramme
  \[
    \shorthandoff{;:}
    \xymatrix@C=3.5pc@R=3.7pc{
      K  \ar[r]^f \ar[dr]_{}="g"_(.40){\phantom{g''}ag\!} &
      K' \ar@/^2ex/[r]^(.33){f''}_{}="1"
      \ar@/_2ex/[r]^(.30){f'}_{}="0"_(.70){}="fp"
      \ar[d]_(.50){}="gp2"_(.20){}="gp"_(0.52){ag'} &
      K'' \ar[dl]^(.40){\!ag''} &
      \ar@2"0";"1"_{k} \\
      & M
      \ar@{}"g";[u]_(0.10){}="x"
      \ar@{}"g";[u]_(.75){}="y"
      \ar@<-0.1ex>@2"x";"y"^(.30){ah}
      \ar@{}"gp2";"fp"_(.10){}="ff2"
      \ar@{}"gp2";"fp"_(.55){}="oo2"
      \ar@<+0.5ex>@/^1ex/@{:>}"ff2";"oo2"^{\quad ah''}_(.30){}="h'''"
      \ar@<-0.5ex>@/^-1.5ex/@2"ff2";"oo2"_(.60){\nquad\!\! ah'}_(.80){}="h''"
      \ar@3"h''";"h'''"_(.20){\,\,aH}
      & \pbox{,}
    }
  \]
  ainsi que de la proposition~\ref{prop:compat_mu_sesqui} affirmant que
  $\nu$ envoie le composé horizontal d'un morphisme et d'une homotopie
  sur le composé horizontal du \oo-foncteur et de la transformation oplax
  associés.
\end{proof}

\begin{rem}
  Pour les mêmes raisons que celles données dans la
  remarque~\ref{rem:conv_fonct_tri}, si $h$ est l'antihomotopie $\id{g}$,
  la conclusion de la proposition précédente reste valable sans supposer que
  $g'$ est une inclusion rigide ordonnée. 
\end{rem}

\begin{prop}
  La transformation oplax $(k, H, b)^\ast$ du théorème~\ref{thm:img_cone} de
  $(f', h', b)^\ast$ vers $(f, h, b)^\ast$ est au-dessus de $C$. Autrement
  dit, on a l'égalité \hbox{$U \comp (k, H, b)^\ast = \id{U'}$}, où $U :
  \cotr{C}{c} \to C$ et $U' : \cotr{C}{c'} \to C$ désignent les
  \oo-foncteurs d'oubli.
\end{prop}

\begin{proof}
  En appliquant la proposition~\ref{prop:fonct_cone_sg} sous la forme
  donnée par la remarque précédente au diagramme
  \[
    \shorthandoff{;:}
    \xymatrix@C=3.5pc@R=3.5pc{
    \vide  \ar[r]^{\vide_K} \ar[dr]_{}="g"_(.40){\vide_L\!} &
    K \ar@/^2ex/[r]^(.33){f'}_{}="1"
    \ar@/_2ex/[r]^(.30){f}_{}="0"_(.70){}="fp"
    \ar[d]_(.50){}="gp2"_(.20){}="gp"_(0.52){g} &
    K' \ar[dl]^(.40){\!g'} &
    \ar@2"0";"1"_{k} \\
    & L
    \ar@{}"g";[u]_(0.10){}="x"
    \ar@{}"g";[u]_(.75){}="y"
    \ar@<-0.1ex>@2"x";"y"^(.30){\id{\vide_{\!L}}^{}\!}
    \ar@{}"gp2";"fp"_(.10){}="ff2"
    \ar@{}"gp2";"fp"_(.55){}="oo2"
    \ar@<+0.5ex>@/^1ex/@{:>}"ff2";"oo2"^{\!\!h'}_(.30){}="h''"
    \ar@<-0.5ex>@/^-1.5ex/@2"ff2";"oo2"_(.47){\!\!\!h}_(.80){}="h'"
    \ar@3"h'";"h''"_(.20){H_{}}
    & \pbox{,}
    }
  \]
  où $\vide_M$, pour $M$ un complexe dirigé augmenté, désigne l'unique
  morphisme de $\vide$ vers~$M$, on obtient l'égalité
  \[
    (\vide_K, \id{\vide_{\!L}}, b)^\ast \comp (k, H, b)^\ast =
    (\id{\vide_{\!K'}}, \id{\id{\vide_{\!L}}^{}}, b)^\ast.
  \]
  Or, en vertu de la remarque~\ref{rem:img_tri_oubli} et de la
  proposition~\ref{prop:fonct_cone_id}, on a
  \[
    (\vide_K, \id{\vide_{\!L}}, b)^\ast = U
    \quadet
    (\id{\vide_{\!K'}}, \id{\id{\vide_{\!L}}^{}}, b)^\ast =
    \id{(\vide_{K'}, \id{\vide_{\!L}}, b)^\ast} = \id{U'},
  \]
  d'où le résultat.
\end{proof}

\begin{paragr}
  De même, considérons un diagramme
    \[
      \shorthandoff{;:}
      \xymatrix@C=3.5pc@R=3.5pc{
      K
      \ar@/^2ex/[r]^(.33){f'}_{}="1"
      \ar@/_2ex/[r]^(.30){f}_{}="0"_(.70){}="f"
      \ar[dr]_{}="g"_(.40){\phantom{g''}g\!}
      \ar@2"0";"1"_{k}
      &
      K' \ar[r]^{f''}_(.75){}="fp"
         \ar[d]_(.70){}="gp2"_(.20){}="gp"^(0.52){g'}
      &
      K'' \ar[dl]^(.40){\!g''}
      \\
      & L
      \ar@{}"g";"gp"_(.15){}="ff1"
      \ar@{}"g";"gp"_(.80){}="oo1"
      \ar@<-0.0ex>@/^1ex/@{:>}"ff1";"oo1"^(.35){h'\!\!}_(.30){}="h'"
      \ar@<-1.0ex>@/^-1.5ex/@2"ff1";"oo1"_(.50){\!h}_(.80){}="h"
      \ar@3"h";"h'"_(.20){H_{}}
      \ar@{}"gp2";"fp"_(.25){}="x2"
      \ar@{}"gp2";"fp"_(.75){}="y2"
      \ar@<0.4ex>@2"x2";"y2"^{h''}
      }
    \]
  de complexes dirigés augmentés, où $h$, $h'$ et $h''$ sont des
  antihomotopies de $g$ vers~$g'f$, de $g$ vers $g'f'$ et de
  $g'$ vers $g''f''$ respectivement, $k$ est une antihomotopie de $f$
  vers~$f'$ et $H$ est une $2$-antihomotopie de $g'k + h$ vers $h'$. On
  suppose toujours que les complexes $K$, $K'$, $K''$ et $L$ sont des
  complexes de Steiner forts et que les morphismes $g'$ et $g''$ sont des
  inclusions rigides ordonnées.

  Soit $C$ une \oo-catégorie munie d'un \oo-foncteur
  \hbox{$b : \nu(L) \to C$}. On pose
  \[
    c = b\nu(g),\quad c' = b\nu(g') \quadet c'' = b\nu(g'').
  \]

  À partir de ces données, comme dans le
  paragraphe~\ref{paragr:fonct_cone_sd_Cda} et avec les mêmes notations,
  en utilisant les théorèmes~\ref{thm:img_tri} et~\ref{thm:img_cone}, ainsi
  que la proposition~\ref{prop:fonct_tri}, on obtient deux transformations
  oplax
  \[
    (k, H, b)^\ast \comp (f'', h'', b)^\ast
    \quadet
    (f''k, H + h''k, b)^\ast
  \]
  de~$(f''f', h''f'+h', b)^\ast$ vers $(f''f, h''f+h, b)^\ast$.
\end{paragr}

\begin{prop}
  On a $(k, H, b)^\ast \comp (f'', h'', b)^\ast = (f''k, H + h''k, b)^\ast$.
\end{prop}

\begin{proof}
  La démonstration est une adaptation immédiate de la démonstration de la
  proposition~\ref{prop:fonct_cone_sg}, l'usage de la
  proposition~\ref{prop:fonct_cone_sg_Cda} étant remplacé par celui de la
  proposition~\ref{prop:fonct_cone_sd_Cda}.
\end{proof}

\begin{paragr}
  Enfin, considérons un diagramme
  \[
      \shorthandoff{;:!}
      \xymatrix@C=2pc@R=4.5pc{
        K \ar@/^3.5ex/[rr]^(.30)*+<-.6em>{\labelstyle f''}_(.65){}="2"
        \ar[rr]^(.25)*+<-.3em>{\labelstyle f'}_(.65){}="1"
        \ar@/_3.5ex/[rr]^(.20)*+<-.3em>{\labelstyle f}_(.65){}="0"
        \ar[dr]_{}="f"_{\phantom{g'}g}
        \ar@2"0";"1"_k
        \ar@2"1";"2"_{k'}
        & & K' \ar[dl]^{g'} \\
        & L
        \ar@{}"f";[ur]_(.15){}="ff"
        \ar@{}"f";[ur]_(.55){}="oo"
        \ar@<2.0ex>@/^1.5ex/@{:>}"ff";"oo"^(.40)*+<-.3em>{\labelstyle h''\!\!}_(.30){}="h''"
        \ar@<0ex>@/^0ex/@{:>}"ff";"oo"^(.0)*+<-.5em>{\labelstyle{h'\!\!}}_(.30){}="h'"_(.70){}="h'2"
        \ar@<-2.0ex>@/^-1.5ex/@2"ff";"oo"_(.36){h}_(.80){}="h"
        \ar@3"h'2";"h''"_(.28){H'}
        \ar@3"h";"h'"_(.20){H_{}}
        }
  \]
  de complexes dirigés augmentés, où $h$, $h'$ et $h''$ sont des
  antihomotopies de $g$ vers~$g'f$, $g'f'$ et $g'f''$ respectivement, $k$ et
  $k'$ sont des antihomotopies de $f$ vers $f'$ et de $f'$ vers~$f''$
  respectivement et $H$ et $H'$ sont des $2$-antihomotopies de $g'k + h$ vers
  $h'$ et de~$g'k' + h'$ vers~$h''$ respectivement. On suppose que les
  complexes $K$, $K'$ et $L$ sont des complexes de Steiner forts et
  que le morphisme $g'$ est une inclusion rigide ordonnée.

  Soit $C$ une \oo-catégorie munie d'un \oo-foncteur
  \hbox{$b : \nu(L) \to C$}. On pose
  \[
    c = b\nu(g) \quadet c' = b\nu(g').
  \]

  À partir de ces données, comme dans le
  paragraphe~\ref{paragr:fonct_cone_Cda} et avec les mêmes notations, en
  utilisant le théorème~\ref{thm:img_cone}, on obtient deux transformations
  oplax
  \[
    (k, H, b)^\ast \circ (k', H', b)^\ast
    \quadet
    (k' + k, H' + H, b)^\ast
  \]
  de~$(f'', h'', b)^\ast$ vers $(f, h, b)^\ast$ (la composition verticale
  des transformations oplax est définie dans le
  paragraphe~\ref{paragr:def_comp_trans}).
\end{paragr}

\begin{prop}
  On a $(k, H, b)^\ast \circ (k', H', b)^\ast = (k' + k, H' + H, b)^\ast$.
\end{prop}

\begin{proof}
  Les transformations oplax $(k, H, b)^\ast$ et $(k', H', b)^\ast$ sont
  naturelles en $(C, b)$ par définition et il en est donc de même de la transformation
  oplax composée $(k, H, b)^\ast \circ (k', H', b)^\ast$ (cela résulte de la
  compatibilité entre la composition horizontale par un \oo-foncteur et la
  composition verticale des transformations oplax, voir
  l'appendice~\ref{sec:conj} qui montre que les \oo-catégories,
  \oo-foncteurs et transformations oplax forment une sesquicatégorie).
  Ainsi, en raisonnant comme dans les démonstrations
  précédentes, on se ramène au cas où $C = \nu(M)$ pour $M$ un complexe de
  Steiner fort et $b = \nu(a)$ pour $a : L \to M$. Or, ce cas résulte de la
  proposition~\ref{prop:fonct_cone_Cda} appliquée au diagramme
  \[
    \shorthandoff{;:!}
    \xymatrix@C=2pc@R=4.5pc{
      K \ar@/^3.5ex/[rr]^(.30)*+<-.6em>{\labelstyle f''}_(.65){}="2"
      \ar[rr]^(.25)*+<-.3em>{\labelstyle f'}_(.65){}="1"
      \ar@/_3.5ex/[rr]^(.20)*+<-.3em>{\labelstyle f}_(.65){}="0"
      \ar[dr]_{}="f"_{\phantom{g'}ag}
      \ar@2"0";"1"_k
      \ar@2"1";"2"_{k'}
      & & K' \ar[dl]^{ag'} \\
      & M
      \ar@{}"f";[ur]_(.15){}="ff"
      \ar@{}"f";[ur]_(.55){}="oo"
      \ar@<2.0ex>@/^1.5ex/@{:>}"ff";"oo"^(.40)*+<-.3em>{\labelstyle ah''\!\!\!}_(.30){}="h''"
      \ar@<0ex>@/^0ex/@{:>}"ff";"oo"^(.0)*+<-.5em>{\labelstyle{ah'\!\!}}_(.30){}="h'"_(.70){}="h'2"
      \ar@<-2.0ex>@/^-1.5ex/@2"ff";"oo"_(.36){ah}_(.80){}="h"
      \ar@3"h'2";"h''"
      \ar@3"h";"h'"
      & \pbox{,}
      }
  \]
  où les $3$-flèches sont, du bas vers le haut, les $2$-antihomotopies $aH$
  et $aH'$, ainsi que de la proposition~\ref{prop:compat_mu_comp} affirmant
  que $\nu$ envoie le composé horizontal de deux homotopies sur le composé
  horizontal des deux transformations oplax associées.
\end{proof}

\appendix

\chapter{Produit tensoriel \pdfoo-catégorique}
\label{sec:Gray}

Le but de cet appendice est d'introduire le produit tensoriel de Gray et les
$\Hom$ internes à gauche et à droite correspondants, et de démontrer leurs
principales propriétés. La définition du produit tensoriel que nous adoptons
est inspirée d'idées de Steiner \cite[Section 7]{Steiner}.

\begin{paragr}\label{paragr:base_tens}
  On rappelle qu'on a défini au paragraphe \ref{paragr:def_produit_cda},
  selon Steiner, le produit tensoriel de deux complexes dirigés augmentés
  et qu'on obtient ainsi une structure de catégorie monoïdale bifermée sur
  la catégorie des complexes dirigés augmentés.

  Si $K$ et $L$ sont deux complexes dirigés augmentés admettant des bases
  $X$ et~$Y$ respectivement, on vérifie immédiatement que le complexe dirigé
  augmenté $K \otimes L$ admet pour base l'ensemble
  %
  % INDEXCHECK
  \notindex{$X \otimes Y$}%
  \[
  X \otimes Y = \{ x \otimes y \mid x \in X, y \in Y \}.
  \]
\end{paragr}

\begin{lemme}[Steiner]\label{lemme:tab_tens}
  Soient $K$ et $L$ des complexes dirigés augmentés à base. Pour tout
  élément homogène $x$ de $K$ et tout élément homogène $y$ de $L$, tout $r
  \ge 0$ et $\epsilon = 0, 1$, on~a
  \[
    \atom{x \otimes y}^\epsilon_r =
    \sum_{\substack{p + q = r\\0 \le p \le |x|,\, 0 \le q \le |y|}}
    \atom{x}^\epsilon_p \otimes \atom{y}^{p+\epsilon \bmod 2}_q.
  \]
\end{lemme}

\begin{proof}
  La formule est donnée dans \cite[exemple 3.10]{Steiner} dans le cas où $x$
  et $y$ sont dans la base de $K$ et $L$ respectivement. Elle se démontre
  par récurrence, essentiellement comme la formule de notre
  lemme~\ref{lemme:tab_joint}
\end{proof}

\begin{prop}[Steiner]\label{prop:tens_Steiner}
  Si $K$ et $L$ sont des complexes dirigés augmentés à base unitaire
  \emph{\lp}\emph{resp.} à base fortement sans boucle\emph{\rp}, alors il en
  est de même de $K \otimes L$. En particulier, si $K$ et $L$ sont des
  complexes de Steiner forts, alors il en est de même de $K \otimes L$.
\end{prop}

\begin{proof}
  Voir \cite[exemple 3.10]{Steiner}.
\end{proof}

\begin{rem}
  En vertu de la proposition précédente (et du fait que le complexe dirigé
  augmenté $\Zdec$ du paragraphe~\ref{paragr:def_produit_cda} est de Steiner
  fort), la catégorie des complexes de Steiner forts est une sous-catégorie
  monoïdale de la catégorie monoïdale des complexes dirigés augmentés définie
  par le produit tensoriel.
\end{rem}

\begin{lemme}\label{lemme:tens_morph_atom}
  Soient $f : K \to K'$ et $g : L \to L'$ des morphismes entre complexes
  dirigés augmentés à base et soient $x$ un élément homogène de $K$ et $y$
  un élément homogène de $L$. On suppose qu'il existe un élément homogène
  $x'$ de $K$ et un élément homogène $y'$ de $L$ tels qu'on ait
  \[
     f(\atom{x}) = \atom{x'}
     \quadet
     g(\atom{y}) = \atom{y'}.
  \]
  Alors on a
  \[ (f \otimes g)(\atom{x \otimes y}) = \atom{x' \otimes y'}. \]
\end{lemme}

\begin{proof}
  La preuve est une adaptation immédiate de la preuve de l'assertion
  analogue pour le joint (lemme~\ref{lemme:joint_morph_atom}) en remplaçant
  l'usage de la formule du lemme~\ref{lemme:tab_joint} par celui de la
  formule du lemme~\ref{lemme:tab_tens}.
\end{proof}

\begin{prop}\label{prop:tens_rig}
  Si $f : K \to K'$ et $g : L \to L'$ sont deux morphismes rigides \noemph{(voir le
  paragraphe~\ref{paragr:def_rigide})} entre complexes dirigés augmentés à
  base, alors leur produit tensoriel $f \otimes g : K \otimes L \to
  K' \otimes L'$ est également rigide.
\end{prop}

\begin{proof}
  Cela résulte immédiatement du lemme précédent.
\end{proof}

\begin{prop}\label{prop:tens_morph_atom}
  Soient $f : K \to K'$ et $g : L \to L'$ des morphismes entre complexes
  dirigés augmentés à base unitaire et soient $x$ un élément de la base de $K$ et $y$
  un élément de la base de $L$. On suppose qu'il existe un élément $x'$ de
  la base de~$K'$ et un élément $y'$ de la base de $L'$ tels qu'on ait
  \[
    \nu(f)(\atom{x}) = \id{\atom{x'}}
    \quadet
    \nu(g)(\atom{y}) = \id{\atom{y'}},
  \]
  où $\id{}$ désigne une identité itérée (éventuellement $0$ fois).
  Alors on a
  \[ \nu(f \otimes g)(\atom{x \otimes y}) = \id{\atom{x' \otimes y'}}. \]
  En particulier, lorsque $x' = f(x)$ et $y' = g(y)$ vérifient les
  hypothèses ci-dessus, on a
  \[ \nu(f \otimes g)(\atom{x \otimes y}) = \atom{f(x) \otimes g(y)}. \]
\end{prop}

\begin{proof}
  Cela résulte immédiatement du lemme~\ref{lemme:tens_morph_atom}.
\end{proof}

\begin{prop}\label{prop:tens_atom_univ}
  Soient $K$ et $L$ deux complexes dirigés augmentés à base unitaire, $x$
  un élément de la base de $K$ de degré $i$ et $y$ un élément de la base de
  $L$ de degré $j$. Notons $z$ la cellule principale de $\Dn{i}$ et $t$
  celle de $\Dn{j}$. Alors le diagramme
  \[
    \xymatrix@C=4pc{
    \nu(\lambda(\Dn{i}) \otimes \lambda(\Dn{j}))
    \ar[r]^-{\nu\big(\widetilde{\atom{x}} \otimes \widetilde{\atom{y}}\big)}
    & \nu(K \otimes L) \\
    \Dn{i+j} \ar[u]^{\atom{z \otimes t}}
    \ar[ur]_{\atom{x \otimes y}}
    & \pbox{,}
    }
  \]
  où $\widetilde{\atom{x}} : \lambda(\Dn{i}) \to K$ et $\widetilde{\atom{y}}
  : \lambda(\Dn{j}) \to L$ désignent les transposés de $\atom{x} : \Dn{i}
  \to \nu(K)$ et $\atom{y} : \Dn{j} \to \nu(L)$ respectivement,
  est commutatif.
\end{prop}

\begin{proof}
  Il s'agit de montrer qu'on a
  \[
    \nu\big(\widetilde{\atom{x}} \otimes \widetilde{\atom{y}}\big)(\atom{z
    \otimes t}) = \atom{x \otimes y}.
  \]
  Cela résulte de la proposition précédente puisque, par définition, on a
  \[
    \nu(\widetilde{\atom{x}})(\atom{z}) = \atom{x}
    \quadet
    \nu(\widetilde{\atom{y}})(\atom{t}) = \atom{y}.
    \qedhere
  \]
\end{proof}

\begin{prop}
  Soient $K$ un complexe de Steiner fort et $F : I \to \Cda$ un système de
  Steiner fort \noemph{(voir le paragraphe~\ref{paragr:def_syst_Steiner})}. Alors le
  foncteur
  \[
    \begin{split}
      K \otimes F & : I \to \Cda \\
      &\phantom{{:I}} i \mapsto K \otimes F(i)
    \end{split}
  \]
  est un système de Steiner fort.
\end{prop}

\begin{proof}
  Le foncteur $F$ étant un système rigide, il en est de même du foncteur
  \hbox{$K \otimes F : i \mapsto K \otimes F(i)$} en vertu de la
  proposition~\ref{prop:tens_rig}. Par ailleurs, puisque d'après la
  proposition~\ref{prop:tens_Steiner} les complexes de Steiner forts sont
  stables par produit tensoriel, le foncteur $K \otimes F$ est à valeurs
  dans les complexes de Steiner forts. Enfin, le foncteur $K \otimes \var$
  admettant un adjoint à droite, il commute aux limites inductives. Le
  morphisme canonique
  \[
    \limind_{i \in I} (K \otimes F(i))
    \to
    K \otimes \limind_{i \in I} F(i)
  \]
  est donc un isomorphisme de complexes dirigés augmentés et, pour tout
  objet $i_0$ de~$I$, le morphisme canonique $K \otimes F(i_0) \to \limind_{i
  \in I} (K \otimes F(i))$ s'identifie à travers cet isomorphisme au produit
  tensoriel $K \otimes F(i_0) \to K \otimes \limind_{i \in I} F(i)$ de $K$
  et du morphisme canonique associé à $F$. On en déduit le résultat en
  invoquant de nouveau les propositions~\ref{prop:tens_Steiner} et
  \ref{prop:tens_rig}.
\end{proof}

\begin{coro}
  Si $K$ est un complexe de Steiner fort, alors le foncteur
  \[
    \begin{aligned}
      \Cda & \to \ooCat \\
      L & \mapsto \nu(K \otimes L)
    \end{aligned}
  \]
  commute aux limites inductives des systèmes de Steiner forts.
\end{coro}

\begin{proof}
  Puisque le foncteur $\nu$ commute aux limites inductives des systèmes de
  Steiner forts (théorème~\ref{thm:nu_syst_Steiner}),
  l'assertion résulte de la proposition précédente.
\end{proof}

\begin{coro}
  Soit $K$ un complexe de Steiner fort. Alors le foncteur
  \[
    \begin{split}
      \Theta & \to \ooCat \\
      S  & \mapsto \nu(K \otimes \lambda(S))
    \end{split}
  \]
  commute aux sommes globulaires.
\end{coro}

\begin{proof}
  Cela résulte du corollaire précédent puisqu'en vertu de la
  proposition~\ref{prop:Theta_Steiner}, les sommes globulaires proviennent
  de systèmes de Steiner forts.
\end{proof}

\begin{paragr}\label{paragr:Hom_lax}
  Fixons $K$ un complexe de Steiner fort. Soit $C$ une \oo-catégorie. Nous
  allons définir une \oo-catégorie \nnot{$\HomLax(\nu(K), C)$}. Il résulte du
  corollaire précédent que le foncteur
  \[
    \begin{split}
      \Theta^\op & \to \Ens \\
      S  & \mapsto \Hom_{\ooCat}(\nu(K \otimes \lambda(S)), C)
    \end{split}
  \]
  envoie les sommes globulaires sur des produits globulaires, au sens du
  paragraphe~\ref{paragr:pu_Theta}. Ainsi, en vertu de ce même
  paragraphe, ce foncteur définit une \oo-catégorie et c'est cette
  \oo-catégorie qu'on notera $\HomLax(\nu(K), C)$. Autrement dit, avec les
  notations du paragraphe~\ref{paragr:pu_Theta}, on pose
  \[ \HomLax(\nu(K), C) = \Hom_{\ooCat}(\nu(K \otimes \lambda(\Dn{\var})), C). \]
  Explicitement, les $i$-flèches de $\HomLax(\nu(K), C)$ sont les
  \oo-foncteurs de $\nu(K \otimes \lambda(\Dn{i}))$ vers $C$.

  Notons qu'un morphisme $f : K \to K'$ entre complexes de Steiner forts
  induit, pour tout \oo-catégorie $C$, une application
  \[
    \Hom_{\ooCat}(\nu(K' \otimes \lambda(S)), C)
      \to \Hom_{\ooCat}(\nu(K \otimes \lambda(S)), C),
  \]
  naturelle en $S$ dans $\Theta$. Ainsi, toujours en vertu du
  paragraphe~\ref{paragr:pu_Theta}, un tel morphisme induit un \oo-foncteur
  $\HomLax(\nu(K'), C) \to \HomLax(\nu(K), C)$.
\end{paragr}

\begin{prop}\label{prop:pu_lax}
  Fixons $K$ un complexe de Steiner fort. Pour tout complexe de Steiner fort
  $L$ et toute \oo-catégorie $C$, on a une bijection
  \[
    \Hom_{\ooCat}(\nu(K \otimes L), C) \simeq
    \Hom_{\ooCat}(\nu(L), \HomLax(\nu(K), C)),
  \]
  naturelle en $L$ et $C$.
\end{prop}

\begin{proof}
  Si $M$ est un complexe de Steiner fort et $t$ est un élément de la base de
  $M$ de degré $i \ge 0$, on notera, pour simplifier,
  $t^\epsilon_j = \atom{t}^\epsilon_j$, pour $0 \le j \le i$
  (voir le paragraphe~\ref{paragr:def_atome} pour la notation
  $\atom{t}^\e_j$).

  On va produire des fonctions
  \[
    \phi :
    \Hom_{\ooCat}(\nu(K \otimes L), C)
    \to
    \Hom_{\ooCat}(\nu(L), \HomLax(\nu(K), C))
  \]
  et
  \[
    \psi :
    \Hom_{\ooCat}(\nu(L), \HomLax(\nu(K), C))
    \to
    \Hom_{\ooCat}(\nu(K \otimes L), C)
  \]
  inverses l'une de l'autre.

  Commençons par $\phi$. Soit $F : \nu(K \otimes L) \to C$ un \oo-foncteur.
  On définit un \oo-foncteur $\phi(F) : \nu(L) \to \HomLax(\nu(K), C)$ de la
  manière suivante. Soit \hbox{$y : \Dn{i} \to \nu(L)$}, pour $i \ge 0$, une
  $i$-flèche de $\nu(L)$.  On doit lui associer une $i$\nbd-flèche~$\phi(F)(y)$
  de $\HomLax(\nu(K), C)$, c'est-à-dire un \oo-foncteur $\nu(K \otimes
  \lambda(\Dn{i})) \to C$. On pose
  \[
    \phi(F)(y) = \quad \nu(K \otimes \lambda(\Dn{i})) \xto{\nu(K \otimes
    \tilde{y})} \nu(K \otimes L) \xto{F} C,
  \]
  où on a noté $\tilde{y} : \lambda(\Dn{i}) \to L$ le transposé de $y : \Dn{i}
  \to \nu(L)$. Le fait qu'on obtienne bien ainsi un \oo-foncteur résulte
  de la naturalité en $\Dn{i}$, et plus généralement en~$S$ dans~$\Theta$,
  du \oo-foncteur $\phi(F)(y)$. En particulier, pour $x$ un élément de degré
  $i$ de la base de~$K$, $y$ un élément de la base de $L$ et $z$ la cellule
  principale de $\Dn{i}$, on a, pour $j$ tel que $0 \le j \le i$ et
  $\epsilon = 0, 1$,
  \[
    \phi(F)(\atom{y})(\atom{x \otimes z^\epsilon_j}) = F(\atom{x \otimes
    y^\epsilon_j}).
  \]
  De plus, cette formule détermine $\phi(F)$ de manière unique puisque les
  \oo-catégories $\nu(L)$ et~$\nu(K \otimes \lambda(\Dn{i}))$ sont engendrées
  librement au sens des polygraphes par leurs atomes en vertu du
  théorème~\ref{thm:Steiner_pol}.

  Définissons maintenant $\psi$. Soit $G : \nu(L) \to \HomLax(\nu(K), C)$ un
  \oo-foncteur. Il s'agit de définir un \oo-foncteur $\psi(G) : \nu(K \otimes L)
  \to C$. En vertu du théorème~\ref{thm:Steiner_pol}, la \oo-catégorie
  $\nu(K \otimes L)$ est engendrée librement au sens des polygraphes par ses
  atomes.  Il suffit donc de définir $\psi(G)$ sur les atomes de $\nu(K
  \otimes L)$ et de vérifier les compatibilités aux sources et aux buts.
  Soient $x$ de degré $i$ dans la base de $K$ et $y$ de degré $j$ dans la
  base de $L$. On pose
  \[
    \psi(G)(\atom{x \otimes y}) = \quad
    \Dn{i+j}
    \xto{\atom{x \otimes z'_j}}
    \nu(K \otimes \lambda(\Dn{j}))
    \xto{G(\atom{y})}
    C,
  \]
  où $z'$ désigne la cellule principale de $\Dn{j}$.

  Vérifions maintenant les compatibilités aux sources et aux buts. Fixons
  $m \ge 0$ et supposons que les formules ci-dessus définissent bien un
  $m$-foncteur. Il s'agit de montrer que, pour tous $x$ et $y$ comme
  ci-dessus avec $i + j = m + 1$, on a
  \[
    \psi(G)(s(\atom{x \otimes y})) = s(\psi(G)(\atom{x \otimes y}))
    \quad\text{et}\quad
    \psi(G)(t(\atom{x \otimes y})) = t(\psi(G)(\atom{x \otimes y})).
  \]
  Montrons la première égalité, la seconde se démontrant de manière
  analogue. Notons $z$ et $z'$ les cellules principales respectives de
  $\Dn{i}$ et $\Dn{j}$. Puisque la \oo-catégorie \hbox{$\nu(\lambda(\Dn{i})
  \otimes \lambda(\Dn{j}))$} est engendrée librement au sens des polygraphes
  par ses atomes, en vertu de la proposition~\ref{prop:eng_pol_comp}, ses
  atomes l'engendrent également par compositions et il existe donc une
  formule $\chi$ exprimant la source de $\atom{z_i \otimes z'_j}$ en
  fonction des $\atom{z^\e_k \otimes z'^\ep_l}$ avec $0 \le k \le i$, $0 \le
  l \le j$, $k + l < i + j$, $\e = 0, 1$ et $\ep = 0, 1$. On notera
  $\chi[\atom{z^\e_k \otimes z'^\ep_l}]$ l'évaluation de la
  formule $\chi$ en les éléments $z^\e_k \otimes z'^\ep_l$. On a donc
  $s(\atom{z_i \otimes z'_j}) = \chi[\atom{z^\e_k \otimes z'^\ep_l}]$. Plus
  généralement, il résulte de la proposition~\ref{prop:tens_atom_univ} que
  la même formule $\chi$ permet de calculer la source d'un atome $\atom{m
  \otimes n}$, où $m$ est de degré~$i$ et $n$ de degré $j$, d'un produit
  tensoriel quelconque de complexes dirigés augmentés à base unitaire~$M$ et
  $N$. On obtient ainsi
  \[
    \begin{split}
      \psi(G)(s(\atom{x \otimes y}))
      & =
      \psi(G)(\chi[\atom{x^\e_k \otimes y^\ep_l}]) \\*
      & =
      \chi[\psi(G)(\atom{x^\e_k \otimes y^\ep_l})]\\*
      & \phantom{=1} \text{(car $\psi(G)$ est un $m$-foncteur et $k + l
    < i + j = m + 1$)} \\
      & = \chi[G(\atom{y^\ep_l})(\atom{x^\e_k \otimes z''_l})],
    \end{split}
  \]
  où $z''$ désigne la cellule principale de $\Dn{l}$, cette dernière
  égalité résultant de la définition de $\psi(G)$. Par ailleurs, on a
  \[
    G(\atom{y^\ep_l})(\atom{x^\e_k \otimes z''_l})
    = G(\atom{y})(\atom{x^\e_k \otimes z'^\ep_l}).
  \]
  En effet, pour $\ep = 0$, on a
  \[
    \begin{split}
      G(\atom{y^0_l})(\atom{x^\e_k \otimes z''_l})
      & =
      G(s_l(\atom{y}))(\atom{x^\e_k \otimes z''_l}) \\
      & =
      s_l(G(\atom{y}))(\atom{x^\e_k \otimes z''_l}) \\
      & =
      G(\atom{y})
      \big(\nu(K \otimes \lambda(\sigma_l^i))(\atom{x^\e_k \otimes z''_l})\big)
        \\*
      & \phantom{=1} \text{(par définition des sources des
      cellules de $\HomLax(\nu(K), C)$)} \\
      & =
      G(\atom{y})(\atom{x^\e_k \otimes z'^0_l}),
    \end{split}
  \]
  la dernière égalité étant conséquence de la
  proposition~\ref{prop:tens_morph_atom} puisque $\sigma_l^i(\atom{z''_l})
  = \atom{z'^0_l}$. La démonstration dans le cas $\ep = 1$ est analogue. En
  insérant cette égalité dans notre calcul précédent, on obtient
   {
    \allowdisplaybreaks
    \begin{align*}
      \psi(G)(s(\atom{x \otimes y}))
      & = \chi[G(\atom{y^\ep_l})(\atom{x^\e_k \otimes z''_l})] \\
      & = \chi[G(\atom{y})(\atom{x^\e_k \otimes z'^\ep_l})] \\*
      & =
      G(\atom{y})(\chi[\atom{x^\e_k \otimes z'^\ep_l}]) \\*
      & \phantom{=1} \text{(car $G(\atom{y})$ est un \oo-foncteur)} \\
      & =
      G(\atom{y})(s(\atom{x \otimes z'_j})) \\
      & =
      s(G(\atom{y})(\atom{x \otimes z'_j})) \\
      & =
      s(\psi(G)(\atom{x \otimes y})),
    \end{align*}
  }%
  la dernière égalité résultant de la définition de $\psi(G)$,
  ce qui achève de montrer que~$\psi(G)$ est bien un \oo-foncteur.

  Enfin, vérifions que $\phi$ et $\psi$ sont bien des bijections
  inverses l'une de l'autre. Soient $F : \nu(K \otimes L) \to C$ et $G :
  \nu(L) \to \HomLax(\nu(K), C)$ deux \oo-foncteurs. On a, avec les notations
  précédentes,
  \[
    \psi\phi(F)(\atom{x \otimes y}) = \phi(F)(\atom{y})(\atom{x \otimes z'_j})
    = F(\atom{x \otimes y})
  \]
  et
  \[
    \phi\psi(G)(\atom{y})(\atom{x \otimes z^\epsilon_j}) = \psi(G)(\atom{x
    \otimes y^\epsilon_j}) = G(\atom{y})(\atom{x \otimes z^\epsilon_j}).
  \]
  Les \oo-foncteurs $\psi\phi(F)$ et $F$ (resp. les \oo-foncteurs
  $\phi\psi(G)$ et $G$) coïncident donc sur les atomes et sont donc égaux, ce
  qu'il fallait démontrer.
\end{proof}

\begin{paragr}\label{paragr:pu_oplax}
  Soit $L$ un complexe de Steiner fort. On montre de même que le foncteur
  \[
    \begin{aligned}
      \Cda & \to \ooCat \\
      K & \mapsto \nu(K \otimes L)
    \end{aligned}
  \]
  commute aux limites inductives des systèmes de Steiner forts. On en déduit
  que le foncteur
  \[
    \begin{split}
      \Theta^\op & \to \Ens \\
      S  & \mapsto \Hom_{\ooCat}(\nu(\lambda(S) \otimes L), C)
    \end{split}
  \]
  envoie les sommes globulaires sur des produits globulaires.  On définit
  alors, comme dans le paragraphe~\ref{paragr:Hom_lax}, pour toute
  \oo-catégorie $C$, une \oo-catégorie
  \nnot[$\HomOpLax(\nu(K), C)$]{$\HomOpLax(\nu(L), C)$} en posant
  \[
     \HomOpLax(\nu(L), C) =
     \Hom_{\ooCat}(\nu(\lambda(\Dn{\var}) \otimes L), C).
  \]
  On montre, comme dans la proposition précédente, que, pour tout complexe
  de Steiner fort $K$ et toute \oo-catégorie $C$, on a une bijection
  \[
    \Hom_{\ooCat}(\nu(K \otimes L), C) \simeq
    \Hom_{\ooCat}(\nu(K), \HomOpLax(\nu(L), C)),
  \]
  naturelle en $K$ et $C$.
\end{paragr}

\begin{thm}\label{thm:produit_tens}
  Il existe une et une seule (à unique isomorphisme monoïdal près) structure
  de catégorie monoïdale sur $\ooCat$ de produit
  \begin{alignat*}{2}
    \otimes & : \ooCat \times \ooCat & \to \ooCat\\
    & \phantom{=1}\quad\,\,(A, B) & \mapsto A \otimes B
  \end{alignat*}
  ayant les deux propriétés suivantes :
  \begin{enumerate}
    \item le foncteur $\nu_{\vert \Stf} : \Stf \to \ooCat$, où la catégorie
      des complexes de Steiner forts~$\Stf$ est munie de la structure de
      catégorie monoïdale définie par le produit tensoriel, s'étend en un
      foncteur monoïdal ;
    \item le foncteur $\otimes : \ooCat \times \ooCat \to \ooCat$ commute
      aux petites limites inductives en chaque variable.
  \end{enumerate}
  De plus, cette structure monoïdale est bifermée.
\end{thm}

\begin{proof}
  L'assertion résulte du théorème~\ref{thm:Day} appliqué à $\C = \ooCat$
  et~$\D$~la catégorie des \oo-catégories de Steiner fortes (voir le
  paragraphe~\ref{paragr:def_ooCat_Stf}) munie du produit tensoriel (par
  l'identification de cette sous-catégorie à celle des complexes de Steiner
  forts), la petite sous-catégorie dense de l'énoncé étant la catégorie
  $\Theta$. En effet, les hypothèses de ce théorème sont précisément le
  contenu de la proposition~\ref{prop:pu_lax} et du
  paragraphe~\ref{paragr:pu_oplax}.
\end{proof}

\begin{paragr}\label{paragr:def_tens}
  On appellera \ndef[produit tensoriel!de Gray]{produit tensoriel de Gray}
  ou plus simplement \ndef[produit tensoriel!$\infty$-catégorique]{produit
  tensoriel} le produit monoïdal
  \[
    \otimes : \ooCat \times \ooCat \to \ooCat
  \]
  \notindex{$A \otimes B$}%
  défini dans le théorème précédent. Si $K$ et $L$ sont des complexes de
  Steiner forts, on a, en vertu de ce même théorème, un isomorphisme
  canonique
  \[
    \nu(K) \otimes \nu(L) \simeq \nu(K \otimes L).
  \]
  En particulier, si $S$ et $T$ sont deux objets de $\Theta$, puisqu'en
  vertu de la proposition~\ref{prop:Theta_Steiner} on a $S \simeq
  \nu\lambda(S)$ et $T \simeq \nu\lambda(T)$, on obtient un isomorphisme
  canonique
  \[ S \otimes T \simeq \nu(\lambda(S) \otimes \lambda(T)). \]
  Plus généralement, si $A$ et $B$ sont deux \oo-catégories, on a des
  isomorphismes canoniques
  \[
    A \otimes B
    \simeq \limind_{\substack{S \to A \in \tr{\Theta}{A}\\
    T \to B \in \tr{\Theta}{B}}} S \otimes T
    \simeq \limind_{\substack{S \to A \in \tr{\Theta}{A}\\
    T \to B \in \tr{\Theta}{B}}} \nu(\lambda(S) \otimes \lambda(T)).
  \]
  En effet, puisque la catégorie $\Theta$ est dense dans $\ooCat$, toute
  \oo-catégorie est limite inductive canonique d'objets de $\Theta$. La
  formule résulte alors du fait que le produit tensoriel commute aux limites
  inductives en chaque variable.

  Notons que l'unité du produit tensoriel est l'image par $\nu$ du complexe
  dirigé augmenté $\Zdec$ du paragraphe~\ref{paragr:def_produit_cda},
  c'est-à-dire la \oo-catégorie finale $\Dn{0}$. Si $A$ est une
  \oo-catégorie, on a donc
  \[
    A \otimes \Dn{0} \simeq A \simeq \Dn{0} \otimes A.
  \]
\end{paragr}

\begin{rem}
  Le produit tensoriel défini au paragraphe précédent coïncide avec celui
  introduit par Al-Agl et Steiner dans \cite{AlAglSteiner} et
  étudié par Crans dans \cite{CransThese}. En effet, le produit tensoriel
  étudié dans ces textes commutant aux limites inductives en chaque
  variable, il suffit de vérifier que les deux produits tensoriels
  coïncident sur une sous-catégorie dense de $\ooCat$. On peut vérifier que
  c'est le cas pour la sous-catégorie pleine des cubes, c'est-à-dire des
  objets de la forme $\Delta_1 \otimes \cdots \otimes \Delta_1$, où
  $\Delta_1$ apparaît $n$ fois avec $n \ge 0$.
\end{rem}

\begin{paragr}\label{paragr:def_HomOpLax}
  En vertu du théorème~\ref{thm:produit_tens}, la structure de catégorie
  monoïdale définie par le produit tensoriel est bifermée. On obtient donc
  des foncteurs
  \[
    \HomOpLax : \ooCat^\op \times \ooCat \to \ooCat
    \quad\text{et}\quad
    \HomLax : \ooCat^\op \times \ooCat \to \ooCat
  \]
  \notindex{$\HomOpLax(A, B)$, $\HomLax(A, B)$}%
  vérifiant, pour toutes \oo-catégories $A$, $B$ et $C$,
  \[
      \Hom_{\ooCat}(A \otimes B, C)
      \simeq \Hom_{\ooCat}(A, \HomOpLax(B, C))
  \]
  et
  \[
      \Hom_{\ooCat}(A \otimes B, C)
      \simeq \Hom_{\ooCat}(B, \HomLax(A, C)),
  \]
  et ce de manière naturelle en $A$, $B$ et $C$.
  En particulier, pour tout $i \ge 0$, on a
  \[
    \HomOpLax(A, B)_i = \Hom_{\ooCat}(\Dn{i} \otimes A, B)
  \]
  et
  \[
    \HomLax(A, B)_i = \Hom_{\ooCat}(A \otimes \Dn{i}, B).
  \]

  Notons que, dans le cas où $A$ est de la forme $\nu(K)$ pour
  $K$ un complexe de Steiner fort, la \oo-catégorie $\HomLax(\nu(K), C)$ que
  l'on vient d'introduire coïncide, en vertu de la
  proposition~\ref{prop:pu_lax}, avec celle définie dans le
  paragraphe~\ref{paragr:Hom_lax}. De même pour le foncteur $\HomOpLax$
  défini au paragraphe~\ref{paragr:pu_oplax}.
\end{paragr}

\begin{prop}\label{prop:lambda_nu_mon_tens}
  Les foncteurs
  \[
    \lambda : \ooCat \to \Cda
    \quadet
    \nu : \Cda \to \ooCat
  \]
  sont monoïdal et monoïdal lax respectivement, les catégories $\ooCat$ et
  $\Cda$ étant toutes deux munies des structures de catégorie monoïdale
  définies par le produit tensoriel.
\end{prop}

\begin{proof}
  La démonstration est une adaptation immédiate de la preuve de l'assertion
  analogue pour le joint (proposition~\ref{prop:lambda_nu_monoidaux_joint}).
\end{proof}

\begin{prop}\label{prop:dual_tens_cda}
  Soient $K$ et $L$ deux complexes dirigés augmentés. Les applications
  \[
    \begin{split}
      (K \otimes L)_p & \to (L \otimes K)_p \\
      x \otimes y & \mapsto y \otimes x \pbox{,}
    \end{split}
  \]
  pour $p \ge 0$, définissent des isomorphismes
  \[
    (K \otimes L)^\opp \simeq L^\opp \otimes K^\opp
    \quadet
    (K \otimes L)^\co \simeq L^\co \otimes K^\co.
  \]
  En particulier, les applications identité de $(K \otimes L)_p$, pour $p
  \ge 0$, définissent un isomorphisme
  \[
    (K \otimes L)^\op \simeq K^\op \otimes L^\op.
  \]
\end{prop}

\begin{proof}
  Commençons par démontrer le résultat pour le dual impair. La compatibilité
  aux augmentations, aux sous-monoïdes de positivité et la bijectivité sont
  évidentes. Il s'agit donc de vérifier la compatibilité aux
  différentielles. Notons $s : x \otimes y \mapsto y \otimes x$ l'application
  de l'énoncé. Pour tout élément homogène de $K \otimes L$ de
  la forme $x \otimes y$ et de degré au moins $1$, on a,
  en notant $d'$ les différentielles dual impair,
  \[
    \begin{split}
      sd'(x \otimes y) & = (-1)^{|x| + |y|}sd(x \otimes y)\\
      & = (-1)^{|x| + |y|}s(dx \otimes y + (-1)^{|x|} x \otimes dy)\\
      & = (-1)^{|x| + |y|}y \otimes dx + (-1)^{|y|} dy \otimes x
    \end{split}
  \]
  et
  \[
    \begin{split}
      ds(x \otimes y) & = d(y \otimes x)
      = d'y \otimes x + (-1)^{|y|}y \otimes d'x \\
      & = (-1)^{|y|} dy \otimes x + (-1)^{|y| + |x|}y \otimes dx \\
      & = (-1)^{|x| + |y|}y \otimes dx + (-1)^{|y|} dy \otimes x,
    \end{split}
  \]
  d'où le premier isomorphisme. Le résultat pour le dual pair se démontre de
  manière analogue : en notant $d'$ les différentielles dual pair, on a
  \[
    \begin{split}
      sd'(x \otimes y) & = (-1)^{|x| + |y| + 1}sd(x \otimes y)\\
      & = (-1)^{|x| + |y| + 1}s(dx \otimes y + (-1)^{|x|} x \otimes dy)\\
      & = (-1)^{|x| + |y| + 1}y \otimes dx + (-1)^{|y| + 1} dy \otimes x
    \end{split}
  \]
  et
  \[
    \begin{split}
      ds(x \otimes y) & = d(y \otimes x)
      = d'y \otimes x + (-1)^{|y|}y \otimes d'x \\
      & = (-1)^{|y| + 1} dy \otimes x + (-1)^{|y| + |x| + 1}y \otimes dx \\
      & = (-1)^{|x| + |y| + 1}y \otimes dx + (-1)^{|y| + 1} dy \otimes x,
    \end{split}
  \]
  et on obtient ainsi le deuxième isomorphisme. Le troisième se déduit des
  deux premiers : l'application $s \circ s$, qui n'est autre que l'identité,
  définit un isomorphisme
  \[
    (A \otimes B)^\op = ((A \otimes B)^\opp)^\co \simeq (B^\opp \otimes
    A^\opp)^\co \simeq (A^\opp)^\co \otimes (B^\opp)^\co = A^\op
    \otimes B^\op. \qedhere
  \]
\end{proof}

\begin{rem}
  Outre la dualité triviale, les dualités considérées dans l'énoncé
  précédent sont les seules pour lesquelles on a des isomorphismes de ce
  type.
\end{rem}

\begin{prop}\label{prop:dual_tens}
  Soient $A$ et $B$ deux \oo-catégories. On a des isomorphismes naturels
  canoniques
  \[
    (A \otimes B)^\opp \simeq B^\opp \otimes A^\opp
    \quadet
    (A \otimes B)^\co \simeq B^\co \otimes A^\co.
  \]
  En particulier, on a un isomorphisme naturel canonique
  \[ (A \otimes B)^\op \simeq A^\op \otimes B^\op. \]
\end{prop}

\begin{proof}
  La preuve est une adaptation immédiate de la preuve de l'assertion
  analogue pour le joint (proposition~\ref{prop:dual_joint}) en remplaçant
  l'usage de la proposition~\ref{prop:dual_joint_cda} par celui de la
  proposition~\ref{prop:dual_tens_cda}.
\end{proof}

\begin{prop}\label{prop:dual_HomLax}
  Soient $A$ et $B$ deux \oo-catégories. On a des isomorphismes naturels
  canoniques
  \[
    \begin{split}
      \HomOpLax(A, B)^\opp & \simeq \HomLax(A^\opp, B^\opp), \\
      \HomOpLax(A, B)^\co & \simeq \HomLax(A^\co, B^\co), \\
      \HomOpLax(A, B)^\op & \simeq \HomOpLax(A^\op, B^\op).
    \end{split}
  \]
\end{prop}

\begin{proof}
  Notons $D$ la dualité paire ou impaire. Pour toute \oo-catégorie $C$, on a
  {
    \allowdisplaybreaks
    \begin{align*}
      \Hom_{\ooCat}(C, D(\HomOpLax(A, B)))
      & \simeq
      \Hom_{\ooCat}(D(C), \HomOpLax(A, B)) \\
      & \simeq
      \Hom_{\ooCat}(D(C) \otimes A, B) \\*
      & \phantom{\simeq1} \text{(par adjonction)} \\
      & \simeq
      \Hom_{\ooCat}(D(D(C) \otimes A), D(B)) \\
      & \simeq
      \Hom_{\ooCat}(D(A) \otimes C, D(B)) \\*
      & \phantom{\simeq1} \text{(en vertu de la proposition précédente)} \\
      & \simeq
      \Hom_{\ooCat}(C, \HomLax(D(A), D(B))),
    \end{align*}
  }
  d'où les deux premiers isomorphismes. Le troisième en résulte :
  \begin{align*}
    \HomOpLax(A, B)^\op
    & = (\HomOpLax(A, B)^\opp)^\co \\
    & \simeq \HomLax(A^\opp, B^\opp)^\co \\
    & \simeq \HomOpLax((A^\opp)^\co, (B^\opp)^\co) \\
    & \simeq \HomOpLax(A^\op, B^\op). \qedhere
  \end{align*}
\end{proof}

\begin{prop}
  Soient $p, q \ge 0$ deux entiers. Si $A$ est une $p$-catégorie et $B$
  une $q$-catégorie, alors on a un isomorphisme canonique
  \[
    A \otimes B \simeq
    \limind_{\substack{
        (S, S \to A) \in \tr{\Theta_p}{A}\\
        (T, T \to B) \in \tr{\Theta_q}{B}}}
    S \otimes T.
  \]
  En particulier, le produit tensoriel de $A$ et $B$ est une $(p +
  q)$-catégorie.
\end{prop}

\begin{proof}
  La preuve est une adaptation immédiate de la preuve de l'assertion
  analogue pour le joint (proposition~\ref{prop:dim_joint}).
\end{proof}

% TOCHECK
\goodbreak

\emph{Dans la suite de cet appendice on fixe un entier $n \ge 0$.}

\begin{paragr}
  Soient $A$ et $B$ deux $n$-catégories. On
  appelle \ndef[produit tensoriel!$n$-catégorique]{produit tensoriel
  $n$-catégorique} de $A$ et $B$ la $n$-catégorie
  \[ A \otimes_n B = \ti{n}(A \otimes B). \]
  \notindex{$A \otimes_n B$}%
  Cette opération définit un foncteur
  \[
    \begin{split}
      \nCat{n} \times \nCat{n} & \to \nCat{n} \\
      (A, B) \quad\,\,\, & \mapsto A \otimes_n B.
    \end{split}
  \]
\end{paragr}

\begin{lemme}
  Soient $K$ et $L$ deux complexes dirigés augmentés.  Les morphismes
  canoniques $K \to \ti{n}(K)$ et $L \to \ti{n}(L)$ induisent un
  isomorphisme
  \[
    \ti{n}(K \otimes L) \simeq \ti{n}(\ti{n}(K) \otimes \ti{n}(L)).
  \]
\end{lemme}

\begin{proof}
  La preuve est une adaptation facile de la preuve de l'assertion
  analogue pour le joint (lemme~\ref{lemme:joint_n}).
\end{proof}

\begin{prop}
  Soient $A$ et $B$ deux \oo-catégories. Les \oo-foncteurs canoniques $A \to
  \ti{n}(A)$ et $B \to \ti{n}(B)$ induisent un isomorphisme
  \[
    \ti{n}(A \otimes B) \simeq \ti{n}(\ti{n}(A) \otimes \ti{n}(B)).
  \]
\end{prop}

\begin{proof}
  La preuve est une adaptation immédiate de la preuve de l'assertion
  analogue pour le joint (proposition~\ref{prop:tronque_joint}).
\end{proof}

\begin{prop}
  La structure de catégorie monoïdale sur $\ooCat$ définie par le produit
  tensoriel induit une structure de catégorie monoïdale sur $\nCat{n}$ pour
  le produit tensoriel $n$-catégorique.
\end{prop}

\begin{proof}
  La preuve est une adaptation immédiate de la preuve de l'assertion
  analogue pour le joint (proposition~\ref{prop:joint_n_monoidal}).
\end{proof}

\begin{prop}
  Soit $B$ une $n$-catégorie. Pour toute \oo-catégorie $A$, les
  \oo-catégories $\HomOpLax(A, B)$ et $\HomLax(A, B)$ sont des
  $n$-catégories.
\end{prop}

\begin{proof}
  La preuve est une adaptation immédiate de la preuve de l'assertion
  analogue pour le joint (proposition~\ref{prop:tr_nCat}).
\end{proof}

\begin{prop}
  La structure de catégorie monoïdale sur $\nCat{n}$ définie par le produit
  tensoriel $n$-catégorique est bifermée, ses $\Hom$ internes étant donnés
  par les restrictions de $\HomOpLax$ et $\HomLax$ à $\nCat{n}$.
\end{prop}

\begin{proof}
  La preuve est une adaptation immédiate de la preuve de l'assertion
  analogue pour le joint (corollaire~\ref{coro:joint_n_biferme}).
\end{proof}

\begin{paragr}
  Pour $n = 1$, le produit tensoriel $1$-catégorique n'est autre que le
  produit cartésien. En effet, le produit tensoriel $1$-catégorique comme le
  produit cartésien commutant aux limites inductives en chaque variable, il
  suffit de démontrer que ces deux foncteurs coïncident sur une
  sous-catégorie dense de $\Cat$. On peut le vérifier pour la sous-catégorie
  pleine de $\Cat$ formée des $\Deltan{1} \times \dots \times \Deltan{1}$,
  où $\Deltan{1}$ apparaît $n$ fois avec~$n \ge 0$. Il s'agit alors de
  montrer qu'on a $\ti{1}(\Deltan{1} \otimes \cdots \otimes \Deltan{1}) =
  \Deltan{1} \times \cdots \times \Deltan{1}$, ce qu'on vérifie sans
  difficulté en utilisant les complexes dirigés augmentés.

  De même, si $n = 2$, on peut démontrer que le produit tensoriel
  $2$-catégorique est le produit tensoriel introduit par
  Gray dans \cite{GrayFCT}. Comme ci-dessus, il suffit de démontrer qu'on a
  $\ti{2}(\Deltan{1} \otimes \cdots \otimes \Deltan{1}) \simeq \Deltan{1}
  \otimesG \cdots \otimesG \Deltan{1}$, où $\otimesG$ désigne le
  produit tensoriel de~\cite{GrayFCT}, ce qu'on peut également faire en
  utilisant les complexes dirigés augmentés.
\end{paragr}

\chapter[Compléments sur les transformations oplax]{Compléments sur
les\\ transformations oplax}
\label{sec:trans_oplax}

% TOCHECK
\kern-10pt

Dans cet appendice, on montre que les transformations oplax entre
\oo-foncteurs de $C$ vers $D$, définies par des formules explicites au
paragraphe~\ref{paragr:def_trans}, sont en correspondance canonique avec les
\oo-foncteurs de $\Dn{1} \otimes C$ vers $D$. On étudie par ailleurs le
lien entre $\Dn{1} \otimes C$, $\Dn{0} \joint C$ et une
suspension $\susp{C}$ de $C$.

\section{Description de la \pdfoo-catégorie des cylindres}

Si $C$ est une \oo-catégorie, le paragraphe~\ref{paragr:def_HomOpLax}
permet de définir une \oo-catégorie $\HomLax(\Dn{1}, C)$. Le but de cette
section est de décrire explicitement cette \oo-catégorie.

\begin{paragr}
  Fixons $i \ge 0$. Nous allons commencer par décrire la
  \hbox{$(i+1)$}\nbd-caté\-gorie~\hbox{$\Dn{1} \otimes \Dn{i}$}. On notera $a$
  la cellule principale de $\Dn{1}$ et $x$ celle de $\Dn{i}$. En vertu de la
  proposition~\ref{prop:Theta_Steiner} et du
  paragraphe~\ref{paragr:def_tens}, on a
  \[
    \Dn{1} \otimes \Dn{i} \simeq \nu(\lambda(\Dn{1}) \otimes \lambda(\Dn{i})).
  \]
  Par ailleurs, les paragraphes~\ref{paragr:desc_lambda_Dn}
  et~\ref{paragr:base_tens} montrent que le complexe $\lambda(\Dn{1})
  \otimes \lambda(\Dn{i})$ a pour base l'ensemble formé des
  \[
    a^0_0 \otimes x^\e_k, \quad a^1_0 \otimes x^\e_k,
    \quad a \otimes x^\e_k,
  \]
  où $k$ varie entre $0$ et $i$, et $\e = 0, 1$ (en se souvenant que
  $x^0_i = x^1_i$).
\end{paragr}

\begin{prop}\label{prop:s_t_cyl_vide}
  Pour tout $k$ tel que $0 < k \le i$, $\e = 0, 1$ et $\ep = 0, 1$, on a
  \[
    s(\atom{a^\ep_0 \otimes x^\e_k}) = \atom{a^\ep_0 \otimes x^0_{k-1}}
    \quadet
    t(\atom{a^\ep_0 \otimes x^\e_k}) = \atom{a^\ep_0 \otimes x^1_{k-1}}.
  \]
\end{prop}

\begin{proof}
  Cela résulte immédiatement de l'égalité
  \[
    d(a^\ep_0 \otimes x^\e_k) = a^\ep_0 \otimes d(x^\e_k)
      = a^\ep_0 \otimes x^1_{k-1} - a^\ep_0 \otimes x^0_{k-1}.
    \qedhere
  \]
\end{proof}

\begin{lemme}\label{lemme:tab_cyl}
  Pour tout $k$ tel que $0 \le k \le i$, $\e = 0, 1$, tout $l$ tel que $0 \le l < k + 1$
  et $\ep = 0, 1$, on a
  \[
    \atom{a \otimes x^\e_k}^\ep_l =
    a^\ep_0 \otimes x^\eta_l + a \otimes x^\epd_{l-1},
  \]
  où $\eta$ vaut $\e$ si $k = l$ et $\ep$ sinon, en convenant que
  $x^0_{-1} = 0 = x^1_{-1}$. En particulier, on~a
  \[
    \atom{a \otimes x^\e_k}^\ep_0 = a^\ep_0 \otimes x^\eta_0.
  \]
\end{lemme}

\begin{proof}
  En vertu du lemme~\ref{lemme:tab_tens} et de la convention de l'énoncé
  pour le cas $l = 0$, on a
  \[
    \atom{a \otimes x^\e_k}^\ep_l
    =
    \atom{a}^\ep_0 \otimes \atom{x^\e_k}^\ep_l +
    \atom{a}^\ep_1 \otimes \atom{x^\e_k}^\epd_{l-1}
    =
    a^\ep_0 \otimes x^\eta_l + a \otimes x^\epd_{l-1},
  \]
  ce qu'on voulait démontrer.
\end{proof}

\begin{prop}\label{prop:s_t_cyl}
  Pour tout $k$ tel que $0 \le k \le i$ et $\e = 0, 1$,
  on a
  \[
    s(\atom{a \otimes x^\e_k}) =
      \atom{a \otimes x^1_{k-1}} \ast_{k-1} \cdots \ast_1 \atom{a \otimes x^1_0}
        \ast_0 \atom{a^0_0 \otimes x^\e_k}
  \]
  et
  \[
    t(\atom{a \otimes x^\e_k}) =
    \atom{a^1_0 \otimes x^\e_k} \ast_0 \atom{a \otimes x^0_0} \ast_1 \cdots
    \ast_{k-1} \atom{a \otimes x^0_{k-1}}.
  \]
\end{prop}

\begin{proof}
  C'est le cas $l = k$ du lemme plus général suivant.
\end{proof}

\begin{lemme}\label{lemme:s_t_cyl_iter}
  Pour tous $k, l$ tels que $0 \le l \le k \le i$ et $\e = 0, 1$,
  on a
  \[
    s_l(\atom{a \otimes x^\e_k}) =
      \atom{a \otimes x^1_{l-1}} \ast_{l-1} \cdots \ast_1 \atom{a \otimes x^1_0}
      \ast_0 \atom{a^0_0 \otimes x^{\eta_0}_l}
  \]
  et
  \[
    t_l(\atom{a \otimes x^\e_k}) =
    \atom{a^1_0 \otimes x^{\eta_1}_l} \ast_0 \atom{a \otimes x^0_0} \ast_1
    \cdots
    \ast_{l-1} \atom{a \otimes x^0_{l-1}},
  \]
  où $\eta_0$ et $\eta_1$ valent $\e$ si $k = l$ et $0$ et $1$
  respectivement sinon.
\end{lemme}

\begin{proof}
  On va démontrer les deux égalités simultanément par récurrence sur $l$.
  Pour $l = 0$, en vertu du lemme~\ref{lemme:tab_cyl}, pour $\ep = 0, 1$, on
  a $\atom{a \otimes x^\e_k}^\ep_0 = a^\ep_0 \otimes x^{\eta_\ep}_0$
  et donc
  \[
    s_0(\atom{a \otimes x^\e_k}) = \atom{a^0_0 \otimes x^{\eta_0}_0}
    \quadet
    t_0(\atom{a \otimes x^\e_k}) = \atom{a^1_0 \otimes x^{\eta_1}_0}.
  \]
  Supposons maintenant l'égalité démontrée au rang $l - 1$ et
  montrons-la au rang $l \le k$. Par la proposition~\ref{prop:s_t_cyl_vide}
  et l'hypothèse de récurrence, on a
  \[
    \begin{split}
      \MoveEqLeft
      t_{l-1}\big(\atom{a \otimes x^1_{l-2}} \ast_{l-2} \cdots \ast_1
        \atom{a\otimes x^1_0} \ast_0 \atom{a^0_0 \otimes x^{\eta_0}_l}\big) \\
      & =
      \atom{a \otimes x^1_{l-2}} \ast_{l-2} \cdots \ast_1 \atom{a \otimes x^1_0}
      \ast_0 \atom{a^0_0 \otimes x^1_{l-1}} \\
      & = s_{l-1}(\atom{a \otimes x^1_{l-1}})
    \end{split}
  \]
  et
  \[
    \begin{split}
      \MoveEqLeft
      s_{l-1}\big(\atom{a^1_0 \otimes x^{\eta_1}_l} \ast_0 \atom{a \otimes x^0_0} \ast_1
        \cdots \ast_{l-2} \atom{a \otimes x^0_{l-2}}\big) \\
      & =
        \atom{a^1_0 \otimes x^0_{l-1}} \ast_0 \atom{a \otimes x^0_0} \ast_1
        \cdots \ast_{l-2} \atom{a \otimes x^0_{l-2}} \\
      & =
      t_{l-1}(\atom{a \otimes x^0_{l-1}})
    \end{split}
  \]
  et les cellules
  \[
    \begin{split}
      u & = \atom{a \otimes x^1_{l-1}} \ast_{l-1}
      \big(\atom{a \otimes x^1_{l-2}} \ast_{l-2}
      \cdots \ast_1 \atom{a \otimes x^1_0}
      \ast_0 \atom{a^0_0 \otimes x^{\eta_0}_l}\big), \\
      v & = \big(\atom{a^1_0 \otimes x^{\eta_1}_l} \ast_0 \atom{a \otimes
      x^0_0} \ast_1 \cdots \ast_{l-2} \atom{a \otimes x^0_{l-2}}\big)
      \ast_{l-1} \atom{a \otimes x^0_{l-1}}
    \end{split}
  \]
  sont donc bien définies. Par ailleurs, en utilisant le
  lemme~\ref{lemme:tab_cyl}, on obtient
  \[
    u^{}_l = a \otimes x^1_{l-1} + a^0_0 \otimes x^{\eta_0}_l =
    s_l(\atom{a \otimes x^\e_k})^{}_l
  \]
  et de même
  \[
    v^{}_l = a^1_0 \otimes x^{\eta_1}_l + a \otimes x^0_{l-1} =
    t_l(\atom{a \otimes x^\e_k})^{}_l.
  \]
  Pour conclure, il suffit donc de montrer qu'on a, d'une part,
  \[
    s(s_l(\atom{a \otimes x^\e_k})) = s(u)
    \quadet
    t(s_l(\atom{a \otimes x^\e_k})) = t(u)
  \]
  et, d'autre part,
  \[
    s(t_l(\atom{a \otimes x^\e_k})) = s(v)
    \quadet
    t(t_l(\atom{a \otimes x^\e_k})) = t(v).
  \]
  Or, en utilisant l'hypothèse de récurrence, on a
  \[
    \begin{split}
      s(s_l(\atom{a \otimes x^\e_k}))
      & =
      s_{l-1}(\atom{a \otimes x^\e_k}) \\
      & =
      \atom{a \otimes x^1_{l-2}} \ast_{l-2} \cdots \ast_1 \atom{a \otimes x^1_0}
      \ast_0 \atom{a^0_0 \otimes x^{0}_{l-1}} \\
      & =
      s_{l-1}\big(
        \atom{a \otimes x^1_{l-1}} \ast_{l-1} \atom{a \otimes x^1_{l-2}} \ast_{l-2}
        \cdots \ast_1 \atom{a \otimes x^1_0} \ast_0 \atom{a^0_0 \otimes
        x^{\eta_0}_{l}}
      \big) \\
      & =
      s(u),
    \end{split}
  \]
  \[
    \begin{split}
      t(s_l(\atom{a \otimes x^\e_k}))
      & =
      t_{l-1}(\atom{a \otimes x^\e_k}) \\
      & =
      \atom{a^1_0 \otimes x^{1}_{l-1}} \ast_0 \atom{a \otimes x^0_0} \ast_1
      \cdots \ast_{l-2} \atom{a \otimes x^0_{l-2}} \\
      & =
      t_{l-1}(\atom{a \otimes x^1_{l-1}}) \\
      & =
      t_{l-1}\big(
        \atom{a \otimes x^1_{l-1}} \ast_{l-1} \atom{a \otimes x^1_{l-2}} \ast_{l-2}
        \cdots \ast_1 \atom{a \otimes x^1_0} \ast_0 \atom{a^0_0 \otimes
        x^{\eta_0}_{l}}
      \big) \\
      & =
      t(u),
    \end{split}
  \]
  \[
    \begin{split}
      s(t_l(\atom{a \otimes x^\e_k}))
      & =
      s_{l-1}(\atom{a \otimes x^\e_k}) \\
      & =
      \atom{a \otimes x^1_{l-2}} \ast_{l-2} \cdots \ast_1 \atom{a \otimes x^1_0}
      \ast_0 \atom{a^0_0 \otimes x^{0}_{l-1}} \\
      & =
      s_{l-1}(\atom{a \otimes x^0_{l-1}}) \\
      & =
      s_{l-1}\big(
        \atom{a^1_0 \otimes x^{\eta_1}_{l}} \ast_0 \atom{a \otimes x^0_0} \ast_1
        \cdots \ast_{l-2} \atom{a \otimes x^0_{l-2}} \ast_{l-1} \atom{a \otimes x^0_{l-1}}
      \big) \\
      & =
      s(v)
    \end{split}
  \]
  et
  \[
    \begin{split}
      t(t_l(\atom{a \otimes x^\e_k}))
      & =
      t_{l-1}(\atom{a \otimes x^\e_k}) \\
      & =
        \atom{a^1_0 \otimes x^{1}_{l-1}} \ast_0 \atom{a \otimes x^0_0}
        \ast_1 \cdots \ast_{l-2} \atom{a \otimes x^0_{l-2}} \\
      & =
      t_{l-1}\big(
        \atom{a^1_0 \otimes x^{\eta_1}_{l}} \ast_0 \atom{a \otimes x^0_0} \ast_1
        \cdots \ast_{l-2} \atom{a \otimes x^0_{l-2}} \ast_{l-1} \atom{a \otimes x^0_{l-1}}
      \big) \\
      & =
      t(v),
    \end{split}
  \]
  d'où le résultat.
\end{proof}

\begin{prop}\label{prop:desc_cyl_pol}
  Soit $C$ une \oo-catégorie. Fixons
  \begin{itemize}
    \item $c$ et $d$ deux $i$-flèches de $C$ ;
  \item pour tout $k$ tel que $0 \le k \le i$ et $\e = 0, 1$, $\alpha^\e_k$
    une $(k+1)$-flèche de $C$, avec~$\alpha^0_i = \alpha^1_i$,
  \end{itemize}
  vérifiant
  les égalités
  \[
    s(\alpha^\e_k) =
    \alpha^1_{k-1} \comp_{k-1} \dots \comp_1 \alpha^1_0 \comp_0 c^\e_k
    \quadet
    t(\alpha^\e_k) = d^\e_k \ast_0 \alpha^0_0 \ast_1 \dots \ast_{k-1}
    \alpha^0_{k-1},
  \]
  où, pour $e = c, d$, on a posé
  \[
    e^\e_k =
      \begin{cases}
        s_k(e) & \text{si $\e = 0$,} \\
        t_k(e) & \text{si $\e = 1$.}
      \end{cases}
  \]
  Alors il existe un et un seul \oo-foncteur $h : \Dn{1} \otimes \Dn{i} \to C$
  tel que
  \[
    c = h(\atom{a^0_0 \otimes x_i}),
    \quad
    d = h(\atom{a^1_0 \otimes x_i})
    \quadet
    \alpha^\e_k = h(\atom{a \otimes x^\e_k}),
  \]
  pour tout $k$ tel que $0 \le k \le i$ et $\e = 0, 1$.
\end{prop}

\begin{proof}
  En vertu du théorème~\ref{thm:Steiner_pol}, la \oo-catégorie
  $\nu(\lambda(\Dn{1}) \otimes \lambda(\Dn{i}))$, isomorphe à la
  \oo-catégorie $\Dn{1} \otimes \Dn{i}$, est engendrée librement par ses
  atomes au sens des polygraphes. On conclut alors comme dans la preuve
  de la proposition~\ref{prop:desc_tr_pol} en utilisant les
  propositions~\ref{prop:s_t_cyl_vide} et~\ref{prop:s_t_cyl}.
\end{proof}

\begin{paragr}\label{paragr:desc_cyl_pol}
  Soit $C$ une \oo-catégorie. Par définition (voir le
  paragraphe~\ref{paragr:def_HomOpLax}), les $i$\nbd-flèches de
  $\HomLax(\Dn{1}, C)$ correspondent aux \oo-foncteurs $h : \Dn{1} \otimes
  \Dn{i} \to C$.  En vertu de la proposition précédente, un tel \oo-foncteur
  est déterminé par un triplet~$(c, d, \alpha)$, où $c$ et $d$ sont des
  $i$-flèches de $C$ et $\alpha$ est une famille de cellules~$\alpha^\e_k$
  de~$C$, pour $0 \le k \le i$ et $\e = 0, 1$, avec $\alpha^0_i =
  \alpha^1_i$,
  \[
    \begin{split}
      \alpha^\e_k & :
        \alpha^1_{k-1} \comp_{k-1} \dots \comp_1 \alpha^1_0 \comp_0 c^\e_k
        \to
        d^\e_k \ast_0 \alpha^0_0 \ast_1 \dots \ast_{k-1} \alpha^0_{k-1},
        \quad
        \text{$(k+1)$-flèche},
    \end{split}
  \]
  où, pour $e = c, d$, on a posé
  \[
    e^\e_k =
      \begin{cases}
        s_k(e) & \text{si $\e = 0$,} \\
        t_k(e) & \text{si $\e = 1$.}
      \end{cases}
  \]
  Dans la suite de cette section, on identifiera les $i$-flèches de
  $\HomLax(\Dn{1}, C)$ avec de tels triplets~$(c, d, \alpha)$. On notera
  $\alpha_i$ pour $\alpha^0_i = \alpha^1_i$.

  Voici une représentation graphique des objets, des $1$-flèches et des
  $2$-flèches de~$\HomLax(\Dn{1}, C)$ :
  \[
   \xymatrix@R=3pc{
   c \ar[d]_{\alpha_0} \\
   d \pbox{,}
   }
   \qquad
   \qquad
    \shorthandoff{;}
    \xymatrix@C=3pc@R=3pc{
      c^0_0 \ar[r]^c \ar[d]_{\alpha^0_0} &
      c^1_0 \ar[d]^{\alpha^1_0} \\
      d^0_0 \ar[r]_d & d^1_0
      \ar@{}[u];[l]_(.30){}="s"
      \ar@{}[u];[l]_(.70){}="t"
      \ar@2"s";"t"_{\alpha_1}
      \pbox{,}
    }
   \qquad
   \qquad
    \shorthandoff{:;}
    \xymatrix@C=3pc@R=3pc{
      c^0_0
      \ar@/^2ex/[r]^(0.70){c^0_1}_{}="0"
      \ar@/_2ex/[r]_(0.70){c^1_1}_{}="1"
      \ar[d]_{}="f"_{\alpha^0_0}
      \ar@2"0";"1"_{c\,\,}
      &
      c^1_0
      \ar[d]^{\alpha^1_0} \\
      d^0_0
      \ar@{.>}@/^2ex/[r]^(0.30){d^0_1}_{}="0"
      \ar@/_2ex/[r]_(0.30){d^1_1}_{}="1"
      \ar@{:>}"0";"1"_{d\,\,}
      &
      d^1_0
      \ar@{}[u];[l]_(.40){}="x"
      \ar@{}[u];[l]_(.60){}="y"
      \ar@<-1.5ex>@/_1ex/@{:>}"x";"y"_(0.60){\alpha^0_1\,}_{}="0"
      \ar@<1.5ex>@/^1ex/@2"x";"y"^(0.40){\!\alpha^1_1}_{}="1"
      \ar@{}"1";"0"_(.05){}="z"
      \ar@{}"1";"0"_(.95){}="t"
      \ar@3{>}"z";"t"_{\alpha}
      \pbox{.}
    }
  \]
\end{paragr}

Le but de la suite de cette section est de décrire la structure de
\oo-catégorie de~$\HomLax(\Dn{1}, C)$ en termes des $(c, d, \alpha)$.

\begin{paragr}
  En vertu de l'adjonction définissant la \oo-catégorie $\HomLax(\Dn{1},
  C)$, les sources, buts, identités et compositions de cette \oo-catégorie
  sont induits par les \oo-foncteurs
  \begin{alignat*}{3}{}
    \Dn{1} \otimes \sigma_i & : \Dn{1} \otimes \Dn{i-1} \to \Dn{1} \otimes \Dn{i}
             && \qquad\text{pour $i \ge 1$,} \\
    \Dn{1} \otimes \tau_i & : \Dn{1} \otimes \Dn{i-1} \to \Dn{1} \otimes \Dn{i}
             && \qquad\text{pour $i \ge 1$,} \\
    \Dn{1} \otimes \kappa_i & : \Dn{1} \otimes \Dn{i+1} \to \Dn{1} \otimes \Dn{i}
             && \qquad\text{pour $i \ge 0$,} \\
    \Dn{1} \otimes \nabla^i_j & : \Dn{1} \otimes \Dn{i} \to (\Dn{1} \otimes
    \Dn{i}) \amalg_{\Dn{1} \otimes \Dn{j}} (\Dn{1} \otimes \Dn{i})
             && \qquad\text{pour $i > j \ge 0$,}
  \end{alignat*}
  où  $\sigma_i$, $\tau_i$, $\kappa_i$ et $\nabla^i_j$ désignent les
  \oo-foncteurs des paragraphes~\ref{paragr:def_disque}
  et~\ref{paragr:def_kappa_nabla}, et où on a identifié $(\Dn{1} \otimes
  \Dn{i}) \amalg_{\Dn{1} \otimes \Dn{j}} (\Dn{1} \otimes \Dn{i})$ et $\Dn{1}
  \otimes (\Dn{i} \amalg_{\Dn{j}} \Dn{i})$.
\end{paragr}

Nous allons commencer par décrire concrètement les morphismes $\Dn{1}
\otimes \sigma_i$, $\Dn{1} \otimes \tau_i$ et $\Dn{1} \otimes \kappa_i$.
On note toujours $a$ la cellule principale de $\Dn{1}$.

\begin{prop}\label{prop:desc_cyl_s_t}
  Fixons $i \ge 1$ et notons $x$ la cellule principale de $\Dn{i-1}$ et $y$
  celle de $\Dn{i}$. Alors le \oo-foncteur $\Dn{1} \otimes \sigma_i : \Dn{1}
  \otimes \Dn{i-1} \to \Dn{1} \otimes \Dn{i}$ est donné par
  \begin{alignat*}{3}{}
    \atom{a^\ep_0 \otimes x^\e_k} & \mapsto \atom{a^\ep_0 \otimes y^\e_k} &&
      \qquad\text{pour $0 \le k < i - 1$, $\e = 0, 1$ et $\ep = 0, 1$,}\\
    \atom{a^\ep_0 \otimes x_{i-1}} & \mapsto \atom{a^\ep_0 \otimes
      y^0_{i-1}} &&
      \qquad\text{pour $\ep = 0, 1$,} \\
    \atom{a \otimes x^\e_k} & \mapsto \atom{a \otimes y^\e_k} &&
      \qquad\text{pour $0 \le k < i - 1$ et $\e = 0, 1$,}\\
    \atom{a \otimes x_{i-1}} & \mapsto \atom{a \otimes y^0_{i-1}}.
  \end{alignat*}
  De même, le \oo-foncteur $\Dn{1} \otimes \tau_i : \Dn{1} \otimes \Dn{i-1} \to \Dn{1} \otimes
  \Dn{i}$ est donné par
  \begin{alignat*}{3}{}
    \atom{a^\ep_0 \otimes x^\e_k} & \mapsto \atom{a^\ep_0 \otimes y^\e_k} &&
      \qquad\text{pour $0 \le k < i - 1$, $\e = 0, 1$ et $\ep = 0, 1$,}\\
    \atom{a^\ep_0 \otimes x_{i-1}} & \mapsto \atom{a^\ep_0 \otimes
      y^1_{i-1}} &&
      \qquad\text{pour $\ep = 0, 1$,} \\
    \atom{a \otimes x^\e_k} & \mapsto \atom{a \otimes y^\e_k} &&
      \qquad\text{pour $0 \le k < i - 1$ et $\e = 0, 1$,}\\
    \atom{a \otimes x_{i-1}} & \mapsto \atom{a \otimes y^1_{i-1}}.
  \end{alignat*}
\end{prop}

\begin{proof}
  Cela résulte immédiatement de la proposition~\ref{prop:tens_morph_atom}
  et des formules
  \[
     \sigma_i(\atom{x^\e_k}) =
     \begin{cases}
       \atom{y^\e_k} & \text{si $0 \le k < i - 1$,} \\
       \atom{y^0_{i-1}} & \text{si $k = i - 1$,}
     \end{cases}
     \quadet
     \tau_i(\atom{x^\e_k}) =
     \begin{cases}
       \atom{y^\e_k} & \text{si $0 \le k < i - 1$,} \\
       \atom{y^1_{i-1}} & \text{si $k = i - 1$,}
     \end{cases}
   \]
   pour $\e = 0,1$.
\end{proof}

\begin{prop}\label{prop:desc_cyl_k}
  Fixons $i \ge 0$ et notons $x$ la cellule principale de $\Dn{i+1}$ et $y$
  celle de $\Dn{i}$. Alors le \oo-foncteur
  $\Dn{1} \otimes \kappa_i : \Dn{1} \otimes \Dn{i+1} \to \Dn{1} \otimes \Dn{i}$
  est donné par
  \begin{alignat*}{3}{}
    \atom{a^\ep_0 \otimes x^\e_k} & \mapsto \atom{a^\ep_0 \otimes y^\e_k} &&
      \qquad\text{pour $0 \le k < i$, $\e = 0, 1$ et $\ep = 0, 1$,}\\
    \atom{a^\ep_0 \otimes x^\e_i} & \mapsto \atom{a^\ep_0 \otimes y_i} &&
      \qquad\text{pour $\e = 0, 1$ et $\ep = 0, 1$,}\\
    \atom{a^\ep_0 \otimes x_{i+1}} & \mapsto \id{\atom{a^\ep_0 \otimes
      y_i}} &&
      \qquad\text{pour $\ep = 0, 1$,}\\
    \atom{a \otimes x^\e_k} & \mapsto \atom{a \otimes y^\e_k} &&
      \qquad\text{pour $0 \le k < i$ et $\e = 0, 1$,}\\
    \atom{a \otimes x^\e_i} & \mapsto \atom{a \otimes y_i} &&
      \qquad\text{pour $\e = 0, 1$,}\\
    \atom{a \otimes x_{i+1}} & \mapsto \id{\atom{a \otimes y_i}}.
  \end{alignat*}
\end{prop}

\begin{proof}
  Cela résulte immédiatement de la proposition~\ref{prop:tens_morph_atom}
  et de la formule
  \[
    \kappa_i(\atom{x^\e_k}) =
      \begin{cases}
        \atom{y^\e_k} & \text{si $0 \le k < i$,} \\
        \atom{y_i} & \text{si $k = i$,} \\
        \id{\atom{y_i}} & \text{si $k = i+1$,}
      \end{cases}
  \]
  pour $\e = 0,1$.
\end{proof}

\begin{paragr}
  Fixons maintenant $j$ tel que $0 \le j < i$. Nous allons décrire
  explicitement le \oo-foncteur
  \[
    \Dn{1} \otimes \nabla^i_j : \Dn{1} \otimes \Dn{i} \to \Dn{1} \otimes \Dn{i}
    \amalg_{\Dn{1} \otimes \Dn{j}} \Dn{1} \otimes \Dn{i}.
  \]
  Nous noterons $x$ la cellule principale de l'objet $\Dn{i}$
  apparaissant dans la source de~$\Dn{1} \otimes \nabla^i_j$, et $y$ et $z$
  les cellules principales des objets $\Dn{i}$ apparaissant de gauche à
  droite dans le but de $\Dn{1} \otimes \nabla^i_j$.

  En vertu du paragraphe \ref{paragr:desc_lambda_cocat} et avec ses
  notations, une base du complexe dirigé augmenté
  \[
    \lambda\big(\Dn{1} \otimes \Dn{i}
    \amalg_{\Dn{1} \otimes \Dn{j}} \Dn{1} \otimes \Dn{i}\big)
    \simeq
    \lambda(\Dn{1}) \otimes \lambda(\Dn{i} \amalgDn{j} \Dn{i})
  \]
  est donnée par les
  \[
    \atom{a^\ep_0 \otimes y^\e_i}, \quad
    \atom{a^\ep_0 \otimes z^\e_i}, \quad
    \atom{a \otimes y^\e_i}, \quad
    \atom{a \otimes z^\e_i},
  \]
  pour $0 \le k \le i$, $\e = 0, 1$ et $\ep = 0, 1$, modulo les
  identifications
  \[
      y^0_j = z^1_j, \quad
      y^\e_k = z^\e_k \quad\text{pour $0 \le k < j$ et $\e = 0, 1$,} \\
  \]
  ainsi que les identifications triviales $y^0_i = y^1_i$ et
  $z^0_i = z^1_i$.
\end{paragr}

\begin{lemme}
  Pour tout $k$ tel que $j < k \le i$ et $\e = 0, 1$, la $(k+1)$-flèche
  \[
    \begin{split}
      & \big(
          \atom{a^1_0 \otimes y^{\eta_1}_{j+1}} \ast_0 \atom{a \otimes z^0_0} \ast_1 \cdots
          \ast_{j-1} \atom{a \otimes z^0_{j-1}} \ast_j \atom{a \otimes z^\e_k}
        \big) \\
      & \phantom{=1} \quad
        \ast_{j+1} \big(
          \atom{a \otimes y^\e_k} \comp_j \atom{a \otimes y^1_{j-1}}
          \comp_{j-1} \cdots \comp_1 \atom{a \otimes y^1_0} \comp_0
          \atom{a^0_0 \otimes z^{\eta_0}_{j+1}}
        \big)
    \end{split}
  \]
  de $\Dn{1} \otimes \Dn{i} \amalg_{\Dn{1} \otimes \Dn{j}} \Dn{1} \otimes
  \Dn{i}$, où $\eta_0$ et $\eta_1$ valent $\e$ si $k = j + 1$ et $0$ et $1$
  respectivement sinon, est bien définie.
\end{lemme}

\begin{proof}
  En vertu de la proposition~\ref{prop:s_t_cyl}, la $j$-flèche
  \[
      \atom{a \otimes y^1_{j-1}} \ast_{j-1} \cdots \ast_1 \atom{a \otimes y^1_0}
        \ast_0 \atom{a^0_0 \otimes y^0_j}
  \]
  est bien définie.
  Puisque, pour $l \le j - 1$, on a
  \[
    t_l(\atom{a^0_0 \otimes z^{\eta_0}_{j+1}})
    =
    \atom{a^0_0 \otimes z^1_l}
    =
    \atom{a^0_0 \otimes y^1_l}
    =
    t_l(\atom{a^0_0 \otimes y^0_j}),
  \]
  on en déduit que la cellule
  \[
    \atom{a \otimes y^1_{j-1}} \comp_{j-1} \cdots \comp_1 \atom{a \otimes
    y^1_0} \comp_0 \atom{a^0_0 \otimes z^{\eta_0}_{j+1}}
  \]
  est également bien définie. De plus, on a
  \[
    \begin{split}
     \MoveEqLeft
      t_j\big(
        \atom{a \otimes y^1_{j-1}} \comp_{j-1} \cdots \comp_1 \atom{a \otimes
        y^1_0} \comp_0 \atom{a^0_0 \otimes z^{\eta_0}_{j+1}}
      \big) \\
      & =
        \atom{a \otimes y^1_{j-1}} \comp_{j-1} \cdots \comp_1 \atom{a \otimes
        y^1_0} \comp_0 \atom{a^0_0 \otimes z^{1}_{j}} \\
      & =
        \atom{a \otimes y^1_{j-1}} \comp_{j-1} \cdots \comp_1 \atom{a \otimes
        y^1_0} \comp_0 \atom{a^0_0 \otimes y^{0}_{j}} \\
      & =
      s_j(\atom{a \otimes y^\e_k}),
    \end{split}
  \]
  la dernière égalité résultant du lemme~\ref{lemme:s_t_cyl_iter},
  ce qui montre que la cellule
  \[
    \atom{a \otimes y^\e_k} \comp_j \atom{a \otimes y^1_{j-1}}
    \comp_{j-1} \cdots \comp_1 \atom{a \otimes y^1_0} \comp_0
    \atom{a^0_0 \otimes z^{\eta_0}_{j+1}}
  \]
  est bien définie. On montre de même que la cellule
  \[
    \atom{a^1_0 \otimes y^{\eta_1}_{j+1}} \ast_0 \atom{a \otimes z^0_0} \ast_1 \cdots
    \ast_{j-1} \atom{a \otimes z^0_{j-1}} \ast_j \atom{a \otimes z^\e_k}
  \]
  est bien définie. Enfin, on a
  \[
    \begin{split}
     \MoveEqLeft
     s_{j+1}\big(
      \atom{a^1_0 \otimes y^{\eta_1}_{j+1}} \ast_0 \atom{a \otimes z^0_0}
      \ast_1 \cdots \ast_{j-1} \atom{a \otimes z^0_{j-1}} \ast_j
      \atom{a \otimes z^\e_k}
     \big) \\
        & =
      \atom{a^1_0 \otimes y^{\eta_1}_{j+1}} \ast_0 \atom{a \otimes z^0_0}
      \ast_1 \cdots \ast_{j-1} \atom{a \otimes z^0_{j-1}} \ast_j
      s_{j+1}(\atom{a \otimes z^\e_k}) \\
      & =
      \atom{a^1_0 \otimes y^{\eta_1}_{j+1}} \ast_0 \atom{a \otimes z^0_0}
      \ast_1 \cdots \ast_{j-1} \atom{a \otimes z^0_{j-1}} \\*
      & \phantom{=1} \quad
        \ast_j \atom{a \otimes z^1_{j}} \ast_{j}
        \atom{a \otimes z^1_{j-1}} \comp_{j-1} \cdots \ast_1
        \atom{a \otimes z^1_0} \ast_0 \atom{a^0_0 \otimes z^{\eta_0}_{j+1}},
      \end{split}
    \]
    la dernière égalité résultant du lemme~\ref{lemme:s_t_cyl_iter}, et
    {
      \allowdisplaybreaks
      \begin{align*}
      \MoveEqLeft
        t_{j+1}\big(
          \atom{a \otimes y^\e_k} \comp_j \atom{a \otimes y^1_{j-1}}
          \comp_{j-1} \cdots \comp_1 \atom{a \otimes y^1_0} \comp_0
          \atom{a^0_0 \otimes z^{\eta_0}_{j+1}}
        \big) \\
      & =
        t_{j+1}(\atom{a \otimes y^\e_k}) \comp_j \atom{a \otimes y^1_{j-1}}
          \comp_{j-1} \cdots \comp_1 \atom{a \otimes y^1_0} \comp_0
          \atom{a^0_0 \otimes z^{\eta_0}_{j+1}} \\
      & =
        \atom{a^1_0 \otimes y^{\eta_1}_{j+1}} \ast_0 \atom{a \otimes y^0_0} \ast_1
        \cdots \ast_{j} \atom{a \otimes y^0_{j}} \\
      & \phantom{=1} \quad \comp_j \atom{a \otimes y^1_{j-1}}
          \comp_{j-1} \cdots \comp_1 \atom{a \otimes y^1_0} \comp_0
          \atom{a^0_0 \otimes z^{\eta_0}_{j+1}} \\*
      & \phantom{=1} \text{(en vertu du lemme~\ref{lemme:s_t_cyl_iter})} \\
      & =
        \atom{a^1_0 \otimes y^{\eta_1}_{j+1}} \ast_0 \atom{a \otimes y^0_0} \ast_1
        \cdots \ast_{j-1} \atom{a \otimes y^0_{j-1}} \\*
      & \phantom{=1} \quad \comp_j \atom{a \otimes y^0_j} \comp_j \atom{a
      \otimes y^1_{j-1}} \comp_{j-1} \cdots \comp_1 \atom{a \otimes y^1_0}
      \comp_0 \atom{a^0_0 \otimes z^{\eta_0}_{j+1}} \\
      & =
        \atom{a^1_0 \otimes y^{\eta_1}_{j+1}} \ast_0 \atom{a \otimes z^0_0} \ast_1
        \cdots \ast_{j-1} \atom{a \otimes z^0_{j-1}} \\*
      & \phantom{=1} \quad \comp_j \atom{a \otimes z^1_j} \comp_j \atom{a
      \otimes z^1_{j-1}} \comp_{j-1} \cdots \comp_1 \atom{a \otimes z^1_0}
      \comp_0 \atom{a^0_0 \otimes z^{\eta_0}_{j+1}},
    \end{align*}
  }%
  ce qui achève de montrer que la cellule de l'énoncé est bien définie.
\end{proof}

\begin{lemme}\label{lemme:s_t_nabla_cyl}
  Pour tout $k$ tel que $j < k \le i$, $\e = 0, 1$ et tout $l$
  tel que \hbox{$0 \le l \le k + 1$}, on a
  \[
    \begin{split}
    \MoveEqLeft
    (\Dn{1} \otimes \nabla^i_j)(\atom{a \otimes x^\e_k})^0_l = \\ & \quad
      \begin{cases}
        a^0_0 \otimes z^0_l + a \otimes y^1_{l-1} &
          \text{si $0 \le l \le j$,} \\
        a^0_0 \otimes y^{\eta_0}_l + a^0_0 \otimes z^{\eta_0}_l + a \otimes y^1_{l-1} &
          \text{si $l = j+1$,} \\
        a^0_0 \otimes y^{\eta_0}_l + a^0_0 \otimes z^{\eta_0}_l +
        a \otimes y^1_{l-1} + a \otimes z^1_{l-1} &
          \text{si $j+1 < l < k+1$,} \\
        a \otimes y^\e_{l-1} + a \otimes z^\e_{l-1} &
          \text{si $l = k+1$,} \\
      \end{cases}
    \end{split}
  \]
  où $\eta_0$ vaut $\e$ si $l = k$ et $0$ sinon, et
  \[
    \begin{split}
    \MoveEqLeft
    (\Dn{1} \otimes \nabla^i_j)(\atom{a \otimes x^\e_k})^1_l = \\ & \quad
      \begin{cases}
        a^1_0 \otimes y^1_l + a \otimes z^0_{l-1} &
          \text{si $0 \le l \le j$,} \\
        a^1_0 \otimes y^{\eta_1}_l + a^1_0 \otimes z^{\eta_1}_l + a \otimes z^0_{l-1} &
          \text{si $l = j+1$,} \\
        a^1_0 \otimes y^{\eta_1}_l + a^1_0 \otimes z^{\eta_1}_l +
        a \otimes y^0_{l-1} + a \otimes z^0_{l-1} &
          \text{si $j+1 < l < k+1$,} \\
        a \otimes y^\e_{l-1} + a \otimes z^\e_{l-1} &
          \text{si $l = k+1$,} \\
      \end{cases}
    \end{split}
  \]
  où $\eta_1$ vaut $\e$ si $l = k$ et $1$ sinon.
\end{lemme}

\begin{proof}
  Soit $\ep \in \{0, 1\}$. Pour $l < k + 1$, en utilisant le
  lemme~\ref{lemme:tab_cyl}, on a
  \[
    \begin{split}
    (\Dn{1} \otimes \nabla^i_j)(\atom{a \otimes x^\e_k})^\ep_l
    & =
    \lambda(\Dn{1} \otimes \nabla^i_j)(\atom{a \otimes x^\e_k}^\ep_l) \\
    & =
    \lambda(\Dn{1} \otimes \nabla^i_j)(a^\ep_0 \otimes x^{\eta_\ep}_l + a \otimes x^\epd_{l-1})
    \\
    & =
    a^\ep_0 \otimes \lambda(\nabla^i_j)(x^{\eta_\ep}_l) + a \otimes
      \lambda(\nabla^i_j)(x^\epd_{l-1})
    \end{split}
  \]
  et on obtient le résultat en utilisant la description de
  $\lambda(\nabla^i_j)$ donnée au paragraphe~\ref{paragr:desc_lambda_cocat}.
  Pour $l = k + 1$, en utilisant de nouveau le lemme~\ref{lemme:tab_cyl}, on
  a
  \[
    (\Dn{1} \otimes \nabla^i_j)(\atom{a \otimes x^\e_k})^\ep_l
    =
    \lambda(\Dn{1} \otimes \nabla^i_j)(a \otimes x^\e_k)
    =
    a \otimes \lambda(\nabla^i_j)(x^\e_k),
  \]
   et on conclut en invoquant de nouveau le
   paragraphe~\ref{paragr:desc_lambda_cocat}.
\end{proof}

\begin{prop}\label{prop:desc_cyl_nabla}
  Le \oo-foncteur $\Dn{1} \otimes \nabla^i_j : \Dn{1} \otimes \Dn{i} \to \Dn{1} \otimes \Dn{i}
  \amalg_{\Dn{1} \otimes \Dn{j}} \Dn{1} \otimes \Dn{i}$ est donné par
  {
    \allowdisplaybreaks
  \begin{alignat*}{3}{}
    \atom{a^\ep_0 \otimes x^0_k} & \mapsto \atom{a^\ep_0 \otimes z^0_k}
                                   && \qquad\text{pour $0 \le k \le j$ et
                                                    $\ep = 0, 1$,} \\
    \atom{a^\ep_0 \otimes x^1_k} & \mapsto \atom{a^\ep_0 \otimes y^1_k}
                                   && \qquad\text{pour $0 \le k \le j$ et
                                                    $\ep = 0, 1$,} \\
    \atom{a^\ep_0 \otimes x^\e_k} & \mapsto \atom{a^\ep_0 \otimes y^\e_k} \ast_j
                                   \atom{a^\ep_0 \otimes z^\e_k}
                                   && \qquad\text{pour $j < k \le i$,
                                      $\e = 0, 1$ et $\ep = 0, 1$,} \\
    \atom{a \otimes x^0_k} & \mapsto \atom{a \otimes z^0_k}
                                   && \qquad\text{pour $0 \le k \le j$,} \\
    \atom{a \otimes x^1_k} & \mapsto \atom{a \otimes y^1_k}
                                   && \qquad\text{pour $0 \le k \le j$,} \\
    \atom{a \otimes x^\e_k} & \mapsto u^\e_k && \qquad\text{pour $j < k \le i$
                                          et $\e = 0,1$},
  \end{alignat*}
  }%
  où
  \[
    \begin{split}
      u^\e_k & =
      \big(\atom{a^1_0 \otimes y^{\eta_1}_{j+1}} \ast_0 \atom{a \otimes z^0_0} \ast_1 \cdots
        \ast_{j-1} \atom{a \otimes z^0_{j-1}} \ast_j \atom{a \otimes
        z^\e_k}\big) \\
      & \phantom{=1} \quad
        \ast_{j+1} \big(
          \atom{a \otimes y^\e_k} \comp_j \atom{a \otimes y^1_{j-1}}
          \comp_{j-1} \cdots \comp_1 \atom{a \otimes y^1_0} \comp_0
          \atom{a^0_0 \otimes z^{\eta_0}_{j+1}}
        \big),
    \end{split}
  \]
  avec $\eta_0$ et $\eta_1$ valant $\e$ si $k = j+1$ et $0$ et $1$
  respectivement sinon.
\end{prop}

\begin{proof}
  Le cas des atomes de la forme $\atom{a^\ep_0 \otimes x^\e_k}$ résulte de la
  commutativité des carrés
  \[
    \xymatrix@C=3.5pc{
      \Dn{1} \otimes \Dn{i} \ar[r]^-{\Dn{1} \otimes \nabla^i_j} &
    \Dn{1} \otimes \big(\Dn{i} \amalg_{\Dn{j}} \Dn{i}\big) \\
    \Dn{0} \otimes \Dn{i} \ar[u]^{\partial \otimes \Dn{i}} \ar[r]_-{\Dn{0}
    \otimes \nabla^i_j} &
    \Dn{0} \otimes \big(\Dn{i} \amalgDn{j} \Dn{i}\big)
    \ar[u]_{\partial \otimes (\Dn{i} \amalgDn{j} \Dn{i})} \pbox{,}
    }
  \]
  où $\partial$ vaut $\sigma_1 : \Dn{0} \to \Dn{1}$ ou $\tau_1 : \Dn{0} \to
  \Dn{1}$, ainsi que de la description explicite du \oo-foncteur $\nabla^i_j
  : \Dn{i} \to \Dn{i} \amalgDn{j} \Dn{i}$ (voir le
  paragraphe~\ref{paragr:def_kappa_nabla}).

  Le cas des atomes de la forme $\atom{a \otimes x^\e_k}$ avec $0 \le k \le
  j$ est conséquence de la proposition~\ref{prop:tens_morph_atom} et des
  formules
  \[
    \nabla^i_j(\atom{x^0_k}) = \atom{z^0_k}
      \quadet
    \nabla^i_j(\atom{x^1_k}) = \atom{y^1_k},
  \]
  pour $0 \le k \le j$.

  Enfin, traitons le cas des atomes de la forme $\atom{a \otimes x^\e_k}$
  avec $k > j$.  Soient $l$ tel que $0 \le l \le k+1$ et $\ep = 0,1$. Il
  s'agit de montrer l'égalité $(\Dn{1} \otimes \nabla^i_j)(\atom{a \otimes x^\e_k})^\ep_l =
  (u^\e_k)^\ep_l$.  Le membre de gauche a été calculé dans le
  lemme~\ref{lemme:s_t_nabla_cyl}. Calculons celui de droite.

  Si $l \le j$, on a
  \[
    s_l(u^\e_k) = s_l\big(
          \atom{a \otimes y^1_{l-1}} \comp_{l-1} \cdots \comp_1 \atom{a
          \otimes y^1_0} \comp_0 \atom{a^0_0 \otimes z^{\eta_0}_{j+1}}
      \big)
  \]
  et donc
  \[
    (u^\e_k)^0_l
    = \atom{a \otimes y^1_{l-1}}^0_l + \atom{a^0_0 \otimes z^{\eta_0}_{j+1}}^0_l
    = a \otimes y^1_{l-1} + a^0_0 \otimes z^0_l.
  \]
  De même, on a
  \[
    (u^\e_k)^1_l
    = \atom{a^1_0 \otimes y^{\eta_1}_{j+1}}^1_l
      + \atom{a \otimes z^0_{l-1}}^1_l
    = a^1_0 \otimes y^1_l + a \otimes z^0_{l-1}.
  \]
  Si $l = j+1$, on a
  \[
    s_l(u^\e_k) = s_l\big(
          \atom{a \otimes y^\e_k} \comp_j \atom{a \otimes y^1_{j-1}}
          \comp_{j-1} \cdots \comp_1 \atom{a \otimes y^1_0} \comp_0
          \atom{a^0_0 \otimes z^{\eta_0}_{j+1}}
        \big)
  \]
  et donc, en utilisant le lemme~\ref{lemme:tab_cyl},
  \[
    (u^\e_k)^0_l
    = \atom{a \otimes y^\e_k}^0_l + \atom{a^0_0 \otimes z^{\eta_0}_{j+1}}^0_l
    = a^0_0 \otimes y^{\eta_0}_l + a \otimes y^1_{l-1} + a^0_0 \otimes z^{\eta_0}_l.
  \]
  De même, on a
  \[
    (u^\e_k)^1_l
    = \atom{a^1_0 \otimes y^{\eta_1}_{j+1}}^1_l + \atom{a \otimes z^\e_k}^1_l
    = a^1_0 \otimes y^{\eta_1}_{l} + a^1_0 \otimes z^{\eta_1}_l + a \otimes z^0_{l-1}.
  \]
  Si $j + 1 <  l \le k + 1$, on a
  \[
    (u^\e_k)^\ep_l = \atom{a \otimes z^\e_k}^\ep_l + \atom{a \otimes
    y^\e_k}^\ep_l
  \]
  et donc, pour $j + 1 < l < k + 1$, toujours en vertu du lemme~\ref{lemme:tab_cyl},
  \[
    (u^\e_k)^\ep_l =
      a^\ep_0 \otimes z^{\eta_\ep}_l + a \otimes z^\epd_{l-1} +
      a^\ep_0 \otimes y^{\eta_\ep}_l + a \otimes y^\epd_{l-1}
  \]
  et, pour $l = k + 1$,
  \[
    (u^\e_k)^\ep_l = a \otimes z^\e_k + a \otimes y^\e_k.
  \]
  On a bien retrouvé dans tous les cas la valeur de $(\Dn{1} \otimes
  \nabla^i_j)(\atom{a \otimes x^\e_k})^\ep_l$ obtenue dans le
  lemme~\ref{lemme:s_t_nabla_cyl}, ce qui achève la démonstration.
\end{proof}

\begin{prop}\label{prop:desc_cyl}
  Soit $C$ une \oo-catégorie. Fixons une $i$-flèche $(c, d, \alpha)$
  de~$\HomLax(\Dn{1}, C)$.
  \begin{enumerate}
    \item Si $i \ge 1$, on a $s(c, d, \alpha) = (s(c), s(d), \gamma)$, où
      \begin{alignat*}{3}{}
        \gamma^\e_k & = \alpha^\e_k
                   && \qquad\text{pour $0 \le k < i - 1$ et $\e = 0,1$,} \\
        \gamma_{i-1} & = \alpha^0_{i-1}.
      \end{alignat*}
    \item Si $i \ge 1$, on a $t(c, d, \alpha) = (t(c), t(d), \gamma)$, où
      \begin{alignat*}{3}{}
        \gamma^\e_k & = \alpha^\e_k
                   && \qquad\text{pour $0 \le k < i - 1$ et $\e = 0,1$,} \\
        \gamma_{i-1} & = \alpha^1_{i-1}.
      \end{alignat*}
    \item Si $i \ge 0$, on a $\id{(c, d, \alpha)} = (\id{c}, \id{d},
    \gamma)$, où
      \begin{alignat*}{3}{}
        \gamma^\e_k & = \alpha^\e_k
                   && \qquad\text{pour $0 \le k < i$ et $\e = 0,1$,} \\
        \gamma^\e_i & = \alpha_i
                   && \qquad\text{pour $\e = 0, 1$,} \\
        \gamma_{i+1} & = \id{\alpha_i}.
      \end{alignat*}
  \end{enumerate}
  Soit $(e, f, \beta)$ une seconde $i$-flèche de~$\HomLax(\Dn{1}, C)$.
  \begin{enumerate}[resume]
    \item Si $(c, d, \alpha)$ et $(e, f, \beta)$ sont $j$-composables pour un $j$
      tel que $0 \le j < i$, alors on~a $(c, d, \alpha) \comp_j (e, f,
      \beta) = (c \comp_j e, d \comp_j f, \gamma)$, où
      \begin{alignat*}{3}{}
        \gamma^0_k & = \beta^0_k
                   && \qquad\text{pour $0 \le k \le j$,} \\
        \gamma^1_k & = \alpha^1_k
                   && \qquad\text{pour $0 \le k \le j$}
      \end{alignat*}
      et
      \[
        \begin{split}
          \gamma^\e_k & =
          \big(d^{\eta_1}_{j+1} \comp_0 \beta^0_0 \comp_1 \cdots \comp_{j-1}
            \beta^0_{j-1} \comp_j \beta^\e_k\big) \\
          & \phantom{=1} \quad
            \comp_{j+1}
            \big(\alpha^\e_k \comp_j \alpha^1_{j-1} \comp_{j-1} \cdots
            \comp_1 \alpha^1_0 \comp_0 e^{\eta_0}_{j+1}\big)
        \end{split}
      \]
      pour $j < k \le i$ et $\e = 0, 1$, où $\eta_0$ et $\eta_1$ valent $\e$
      si $k = j + 1$ et $0$ et $1$ respectivement sinon.
  \end{enumerate}
\end{prop}

\begin{proof}
  Ces formules sont la traduction, à travers la bijection de la
  proposition~\ref{prop:desc_cyl_pol} et du
  paragraphe~\ref{paragr:desc_cyl_pol}, des formules obtenues dans les
  propositions~\ref{prop:desc_cyl_s_t}, \ref{prop:desc_cyl_k}
  et~\ref{prop:desc_cyl_nabla}.
\end{proof}

\begin{rem}\label{rem:HC}
  Il résulte de la description de $\HomLax(\Dn{1}, C)$ obtenue dans la
  proposition précédente que cette \oo-catégorie est isomorphe à
  la \oo-catégorie~$HC$ des cylindres dans $C$ introduite par Métayer
  dans~\cite{MetPolRes} (voir également~\cite{LafMetPolRes} pour une
  description alternative de cette \oo-catégorie).
\end{rem}

\section{Transformations oplax et produit tensoriel}

\begin{paragr}
  Fixons deux \oo-foncteurs $u, v : C \to D$. Le but de cette section
  est de montrer que les transformations oplax de $u$ vers $v$, telles que
  définies au paragraphe~\ref{paragr:def_trans}, sont en bijection avec les
  \oo-foncteurs $H : \Dn{1} \otimes C \to D$ rendant commutatif le diagramme
  \[
    \xymatrix@C=3pc{
    C \ar[dr]^u \ar[d]_{\sigma_1 \otimes C} \\
    \Dn{1} \otimes C \ar[r]^H & D \\
    C \ar[ur]_v \ar[u]^{\tau_1 \otimes C} & \pbox{,} \\
    }
  \]
  où, d'une part, on a identifié $C$ et $\Dn{0} \otimes C$ et, d'autre part,
  $\sigma_1, \tau_1 : \Dn{0} \to \Dn{1}$ désignent les \oo-foncteurs du
  paragraphe~\ref{paragr:def_disque}.
\end{paragr}

\begin{paragr}\label{paragr:cohomot_Gray}
  Par adjonction,
  un \oo-foncteur $H : \Dn{1} \otimes C \to D$, comme dans le
  paragraphe précédent, correspond à un \oo-foncteur $K : C \to
  \HomLax(\Dn{1}, D)$ rendant commutatif le diagramme de \oo-foncteurs
  \[
    \xymatrix@C=3pc{
    & D \\
    C \ar[ur]^u \ar[dr]_v \ar[r]^(.40)K & \HomLax(\Dn{1}, D)
    \ar[u]_{\pi^0} \ar[d]^{\pi^1} \\
    & D \pbox{,}
    }
  \]
  où on a posé
  \[
    \begin{split}
      \pi^0 & = \HomLax(\sigma_1, D) : \HomLax(\Dn{1}, D) \to
      \HomLax(\Dn{0}, D), \\
      \pi^1 & = \HomLax(\tau_1, D) : \HomLax(\Dn{1}, D) \to
      \HomLax(\Dn{0}, D)
    \end{split}
  \]
  \notindex{$\pi^0, \pi^1 : \HomLax(\Dn{1}, C) \to C$}%
  et identifié $\HomLax(\Dn{0}, D)$ à $D$ (ce qui est licite puisque
  $\Dn{0}$ est l'unité du produit tensoriel).

  L'action du \oo-foncteur $\pi^0$ sur les $i$-flèches, pour $i
  \ge 0$, est par définition induite par le \oo-foncteur
  $\sigma_1 \otimes \Dn{i} : \Dn{0} \otimes \Dn{i} \to \Dn{1} \otimes
  \Dn{i}$. On en déduit qu'en utilisant la description des $i$-flèches de
  $\HomLax(\Dn{1}, D)$ donnée au paragraphe~\ref{paragr:desc_cyl_pol}, on
  a $\pi^0(c, d, \alpha) = c$. De
  même, on a $\pi^1(c, d, \alpha) = d$.
\end{paragr}

\begin{paragr}\label{paragr:app_cohomot_trans}
  Considérons un morphisme de \oo-graphes $K : C \to \HomLax(\Dn{1}, D)$
  (c'est-à-dire une application des cellules de $C$ vers les cellules
  de $\HomLax(\Dn{1}, D)$ qui préserve la dimension des cellules et leurs
  sources et buts) rendant commutatif le diagramme du paragraphe précédent.
  La commutativité de ce diagramme signifie exactement que si $x$ est une
  cellule de $C$, les deux premières composantes de $K(x)$ sont $u(x)$ et
  $v(x)$. Notons $\alpha(x)$ la troisième composante et $\alpha_x$
  la $(i+1)$-flèche $\alpha(x)_i$, où $i$ est la dimension de $x$. Par
  définition, on a
  \[
     \alpha_x :
     \alpha(x)^1_{i-1} \comp_{i-1} \dots \comp_1 \alpha(x)^1_0 \comp_0 u(x)
     \to
     v(x) \ast_0 \alpha(x)^0_0 \ast_1 \dots \ast_{i-1} \alpha(x)^0_{i-1}.
  \]
  Or, en vertu de la proposition~\ref{prop:desc_cyl} et de la
  compatibilité de $K$ aux sources, on a
  \[
      \alpha(x)^0_l = s_l(\alpha(x))_l = \alpha(s_l(x))_l = \alpha_{s_l(x)}
  \]
  et, de même,
  \[
    \alpha(x)^1_l = \alpha_{t_l(x)},
  \]
  et donc
  \[
     \alpha_x :
     \alpha_{t_{i-1}(x)} \comp_{i-1} \dots \comp_1 \alpha_{t_0(x)} \comp_0 u(x)
     \to
     v(x) \ast_0 \alpha_{s_0(x)} \ast_1 \dots \ast_{i-1} \alpha_{s_{i-1}(x)}.
  \]
  Autrement dit, l'application $x \mapsto \alpha_x$ est une
  prétransformation oplax de $u$ vers $v$ au sens du paragraphe~\ref{paragr:def_trans}.
\end{paragr}

\begin{prop}
  Les morphismes de \oo-graphes $K : C \to \HomLax(\Dn{1}, D)$ comme dans
  le paragraphe précédent sont en bijection avec les prétransformations oplax
  de $u$ vers $v$ \forlang{via} l'application définie au paragraphe
  précédent.
\end{prop}

\begin{proof}
  Il suffit de construire une application inverse à l'application définie au
  paragraphe précédent. On vérifie facilement qu'on obtient un tel inverse
  en envoyant une prétransformation oplax $\alpha$ de $u$ vers $v$ sur le
  morphisme de \oo-graphes $K : C \to \HomLax(\Dn{1}, D)$ défini par, pour
  $x$ une $i$-flèche de $C$, pour $i \ge 0$,
  \[
      K(x) = (u(x), v(x), \alpha(x)),
  \]
  où, pour $0 \le l \le i$, on a posé
  \[
    \alpha(x)^0_l = \alpha_{s_l(x)}
    \quadet
    \alpha(x)^1_l = \alpha_{t_l(x)}. \qedhere
  \]
\end{proof}

\begin{prop}\label{prop:cotrans_abs}
  Les \oo-foncteurs $K : C \to \HomLax(\Dn{1}, D)$ comme dans le
  paragraphe~\ref{paragr:cohomot_Gray} sont en bijection avec les
  transformations oplax de $u$ vers $v$ \forlang{via} l'application définie
  au paragraphe~\ref{paragr:app_cohomot_trans}.
\end{prop}

\begin{proof}
  Soient $K : C \to \HomLax(\Dn{1}, D)$ un tel \oo-foncteur, $\alpha(x)$
  la troisième composante de $K(x)$ pour $x$ une cellule de $C$ et
  $\alpha$ la prétransformation oplax associée. Commençons par vérifier que
  $\alpha$ est une transformation oplax. Si $x$ est une $i$-flèche de $C$,
  pour un $i \ge 0$, on a, par fonctorialité de $K$,
  \[
    (u(\id{x}), v(\id{x}), \alpha(\id{x}))
    =
    \id{(u(x), v(x), \alpha(x))}
  \]
  et donc, en vertu de la proposition~\ref{prop:desc_cyl},
  \begin{equation}\label{eq:id}\tag{$\star_1$}
    \alpha(\id{x})_{i+1} = \id{\alpha(x)_i},
  \end{equation}
  ou encore
  \[
    \alpha_{\id{x}} = \id{\alpha_x}.
  \]
  Si maintenant $x$ et $y$ sont deux $i$-flèches $j$-composables de $C$,
  pour $i > j \ge 0$, alors on a, toujours par fonctorialité de $K$,
  \[
    (u(x \comp_j y), v(x \comp_j y), \alpha(x \comp_j y))
    =
    (u(x), v(x), \alpha(x)) \comp_j (u(y), v(y), \alpha(y)).
  \]
  En vertu de la proposition~\ref{prop:desc_cyl}, on a donc
  \begin{equation}\label{eq:comp}\tag{$\star_2$}
    \begin{split}
      \alpha(x \comp_j y)_i & =
            \big(t_{j+1}(v(x)) \comp_0 \alpha(y)^0_0 \comp_1 \cdots \comp_{j-1}
            \alpha(y)^0_{j-1} \comp_j \alpha(y)_i\big) \\*
          & \phantom{=1} \quad
            \comp_{j+1}
            \big(\alpha(x)_i \comp_j \alpha(x)^1_{j-1} \comp_{j-1} \cdots
            \comp_1 \alpha(x)^1_0 \comp_0 s_{j+1}(u(y))\big),
    \end{split}
  \end{equation}
  ou encore
  \[
    \begin{split}
      \alpha_{x \comp_j y} & =
      \left(v(t_{j+1}(x)) \comp_0 \alpha_{s_0(y)} \comp_1 \cdots
      \comp_{j-1} \alpha_{s_{j-1}(y)} \comp_j \alpha_y\right) \\*
      & \phantom{=1} \qquad
      \comp_{j+1} \left(\alpha_x \comp_j \alpha_{t_{j-1}(x)} \comp_{j-1}
      \cdots \comp_1 \alpha_{t_0(x)} \comp_0 u(s_{j+1}(y))\right),
    \end{split}
  \]
  ce qui montre que $\alpha$ est bien une transformation oplax.

  Pour conclure, en vertu de la proposition précédente, il nous suffit de
  montrer que si $K : C \to \HomLax(\Dn{1}, D)$ est un morphisme de
  \oo-graphes comme dans le paragraphe~\ref{paragr:app_cohomot_trans} qui
  satisfait de plus aux relations~\eqref{eq:id} et~\eqref{eq:comp}, alors $K$ est un
  \oo-foncteur. Si $x$ est une $i$-flèche, pour un $i \ge 0$, on a, pour $k$
  tel que $0 \le k \le i$,
  \[
    \alpha(\id{x})^0_k =
    \alpha(s_k(\id{x}))_k = \alpha(s_k(x))_k = \alpha(x)^0_k
  \]
  et de même, on a $\alpha(\id{x})^1_k = \alpha(x)^1_k$, ce qui, en vertu de
  la relation~\eqref{eq:id}, montre la compatibilité aux identités (voir la
  proposition~\ref{prop:desc_cyl}).
  Vérifions maintenant la compatibilité aux compositions. Soient donc $x$ et $y$
  deux $i$-flèches $j$-composables de $C$, pour~$i > j \ge 0$. Un calcul
  similaire à celui qu'on vient de mener montre que si $0 \le k \le j$, on
  a
  \[
    \alpha(x \comp_j y)^0_k = \alpha(y)^0_k
    \quadet
    \alpha(x \comp_j y)^1_k = \alpha(x)^1_k.
  \]
  Soit maintenant $k$ tel que $j < k < i$. En utilisant la
  relation~\eqref{eq:comp}, on obtient
  \[
    \begin{split}
      \MoveEqLeft
      \alpha(x \comp_j y)^0_k \\
      & = \alpha(s_k(x \comp_j y))_k
      = \alpha(s_k(x) \comp_j s_k(y))_k \\
      & =
      \big(t_{j+1}(v(s_k(x)) \comp_0 \alpha(s_k(y))^0_0 \comp_1 \cdots
      \comp_{j-1} \alpha(s_k(y))^0_{j-1} \comp_j \alpha(s_k(y))_k\big) \\*
      & \phantom{=1} \quad
      \comp_{j+1}
      \big(\alpha(s_k(x))_k \comp_j \alpha(s_k(x))^1_{j-1} \comp_{j-1} \cdots
      \comp_1 \alpha(s_k(x))^1_0 \comp_0 s_{j+1}(u(s_k(y)))\big) \\
      & =
      \big(t_{j+1}s_k(v(x)) \comp_0 \alpha(y)^0_0 \comp_1 \cdots
      \comp_{j-1} \alpha(y)^0_{j-1} \comp_j \alpha(y)^0_k\big) \\*
      & \phantom{=1} \quad
      \comp_{j+1}
      \big(\alpha(x)^0_k \comp_j \alpha(x)^1_{j-1} \comp_{j-1} \cdots
      \comp_1 \alpha(x)^1_0 \comp_0 s_{j+1}(u(y))\big).
    \end{split}
  \]
  Or, $t_{j+1}s_k(v(x))$ vaut $s_{j+1}(v(x))$ si $k = j + 1$ et
  $t_{j+1}(v(x))$ sinon, et on obtient donc bien la valeur de $\alpha(x
  \comp_j y)^0_k$ attendue (voir la proposition~\ref{prop:desc_cyl}). Un
  calcul similaire montre l'assertion analogue pour $\alpha(x \comp_j
  y)^1_j$, ce qui achève de montrer que $K$ est un \oo-foncteur.
\end{proof}

\begin{coro}\label{coro:trans_oplax_abs}
  Les \oo-foncteurs $H : \Dn{1} \otimes C \to D$
  rendant commutatif le diagramme
  \[
    \xymatrix@C=3pc{
    C \ar[dr]^u \ar[d]_{\sigma_1 \otimes C} \\
    \Dn{1} \otimes C \ar[r]^H & D \\
    C \ar[ur]_v \ar[u]^{\tau_1 \otimes C} \\
    }
  \]
  sont en bijection canonique avec les transformations oplax de $u$ vers
  $v$.
\end{coro}

\begin{proof}
  Cela résulte immédiatement de la proposition précédente et du
  paragraphe~\ref{paragr:cohomot_Gray}.
\end{proof}

\begin{coro}\label{coro:trans_lax_abs}
  Les \oo-foncteurs $H : C \otimes \Dn{1} \to D$
  rendant commutatif le diagramme
  \[
    \xymatrix@C=3pc{
    C \ar[dr]^u \ar[d]_{C \otimes \sigma_1} \\
    C \otimes \Dn{1} \ar[r]^H & D \\
    C \ar[ur]_v \ar[u]^{C \otimes \tau_1} \\
    }
  \]
  sont en bijection canonique avec les transformations lax de $u$ vers
  $v$.
\end{coro}

\begin{proof}
  En vertu de la proposition~\ref{prop:dual_tens}, en appliquant la dualité
  paire au diagramme de l'énoncé on obtient le diagramme
  \[
    \xymatrix@C=3pc{
    C^\co \ar[dr]^{u^\co} \ar[d]_{(\sigma_1)^\co \otimes C^\co} \\
    (\Dn{1})^\co \otimes C^\co \ar[r]^{H^\co} & D^\co \\
    C^\co \ar[ur]_{v^\co} \ar[u]^{(\tau_1)^\co \otimes C^\co} \\
    }
  \]
  et, puisque $(\Dn{1})^\co = \Dn{1}$, $(\sigma_1)^\co = \sigma_1$ et
  $(\tau_1)^\co = \tau_1$, ce diagramme n'est autre que
  \[
    \xymatrix@C=3pc{
    C^\co \ar[dr]^{u^\co} \ar[d]_{\sigma_1 \otimes C^\co} \\
    \Dn{1} \otimes C^\co \ar[r]^{H^\co} & D^\co \\
    C^\co \ar[ur]_{v^\co} \ar[u]^{\tau_1 \otimes C^\co} & \pbox{.} \\
    }
  \]
  Le \oo-foncteur $H^\co$ définit donc, en vertu du corollaire précédent,
  une transformation oplax de $u^\co$ vers $v^\co$, c'est-à-dire une
  transformation lax de $u$ vers $v$ (voir le
  paragraphe~\ref{paragr:def_trans}).
\end{proof}

\begin{rem}\label{rem:not_HomOpLax}
  Les deux corollaires précédents justifient les notations
  $\HomOpLax(C, D)$ et $\HomLax(C, D)$. En effet, les $1$-flèches de ces
  \oo-catégories sont respectivement les \oo-foncteurs $\Dn{1} \otimes C \to
  D$ et les \oo-foncteurs $C \otimes \Dn{1} \to D$ (voir le
  paragraphe~\ref{paragr:def_HomOpLax}) et correspondent donc
  respectivement, en vertu de ces corollaires, aux transformations oplax et
  aux transformations lax entre \oo-foncteurs de $C$ vers~$D$
  (voir le paragraphe~\ref{paragr:HomOpLax_ob_fl} pour plus de
  détails).
\end{rem}

\begin{rem}
  On peut facilement déduire la remarque~\ref{rem:dual_trans} des
  corollaires précédents en utilisant les dualités de $\ooCat$ (et notamment
  la proposition~\ref{prop:dual_tens}).
\end{rem}

\begin{prop}
  Soit $\alpha$ une transformation oplax de $u$ vers $v$. Notons
  \hbox{$H : \Dn{1} \otimes C \to D$} le \oo-foncteur correspondant. Pour
  toute $i$-flèche $x$ de $C$, pour un~$i \ge 0$, la $(i+1)$-flèche
  $\alpha_x$ est l'image par le \oo-foncteur
  \[
    \Dn{1} \otimes \Dn{i} \xto{\Dn{1} \otimes x} \Dn{1} \otimes
    C \xto{H} D
  \]
  de l'unique $(i+1)$-flèche non triviale de~$\Dn{1} \otimes \Dn{i}$.
\end{prop}

\begin{proof}
  Notons $K : C \to \HomLax(\Dn{1}, D)$ le \oo-foncteur correspondant à~$H$
  par adjonction.
  En vertu de la proposition~\ref{prop:cotrans_abs}, en notant~$\alpha(x)$
  la troisième composante de $K(x)$ (voir le
  paragraphe~\ref{paragr:desc_cyl_pol}), on a $\alpha_x = \alpha(x)_i$. Or,
  par adjonction, la $i$-flèche $K(x)$
  \[ \Dn{i} \xto{x} C \xto{K} \HomLax(\Dn{1}, D) \]
  correspond au composé
  \[
    \Dn{1} \otimes \Dn{i} \xto{\Dn{1} \otimes x} \Dn{1} \otimes
    C \xto{H} D
  \]
  et, en vertu de la proposition~\ref{prop:desc_cyl_pol}, la $(i+1)$-flèche
  $\alpha(x)_i$ est l'image par ce \oo-foncteur de l'unique $(i+1)$-flèche
  non triviale de $\Dn{1} \otimes \Dn{i}$, d'où le résultat.
\end{proof}

\begin{prop}\label{prop:trans_sesqui_abs}
  Soient $v_0, v_1 : C \to D$ deux \oo-foncteurs et $\alpha$ une
  transformation oplax de $v_0$ vers $v_1$. Notons $H : \Dn{1} \otimes C \to
  D$ le \oo-foncteur correspondant à~$\alpha$.
  \begin{enumerate}
    \item Pour tout \oo-foncteur $u : B \to C$, la transformation oplax
      $\alpha \comp u$ \noemph{(voir le
      paragraphe~\ref{paragr:def_trans_comp})} correspond au \oo-foncteur
      \[
        \Dn{1} \otimes B
        \xto{\Dn{1} \otimes u} \Dn{1} \otimes C \xto{H} D.
      \]
    \item Pour tout \oo-foncteur $w : D \to E$, la transformation oplax
     $w \comp \alpha$ \noemph{(voir également le
     paragraphe~\ref{paragr:def_trans_comp})} correspond au \oo-foncteur
      \[
        \Dn{1} \otimes C \xto{H} D \xto{w} E.
      \]
  \end{enumerate}
\end{prop}

\begin{proof}
  Commençons par démontrer la première assertion. Par adjonction, en notant
  $K : C \to \HomLax(\Dn{1}, C)$ le transposé de $H$, cela revient à
  montrer que la transformation oplax $\alpha \comp u$ correspond, dans la
  bijection de la proposition~\ref{prop:cotrans_abs}, au \oo-foncteur
  \[
    B \xto{u} C \xto{K} \HomLax(\Dn{1}, D).
  \]
  Or, par définition (voir le paragraphe~\ref{paragr:app_cohomot_trans}), si
  $x$ est une cellule de $B$, la transformation oplax $\beta$ correspondant à
  $Ku$ vérifie $\beta_x = \alpha_{u(x)}$ et on a donc bien $\beta = \alpha
  \comp u$.

  De même, établir la seconde assertion revient à montrer que la
  transformation oplax $w \comp \alpha$ correspond, dans la bijection de la
  proposition~\ref{prop:cotrans_abs}, au \oo-foncteur
  \[
    C \xto{K} \HomLax(\Dn{1}, D)
      \xto{\HomLax(\Dn{1}, w)} \HomLax(\Dn{1}, E),
  \]
  ce qui se vérifie comme ci-dessus en utilisant de nouveau le
  paragraphe~\ref{paragr:app_cohomot_trans}.
\end{proof}

\section{Composition verticale des transformations oplax}

\emph{Dans cette section, on fixe deux \oo-catégories $C$ et $D$.}

\begin{paragr}\label{paragr:HomOpLax_ob_fl}
  Considérons la \oo-catégorie $\HomOpLax(C, D)$ (voir le
  paragraphe~\ref{paragr:def_HomOpLax}). Par définition, pour $i
  \ge 0$, ses $i$-flèches sont les \oo-foncteurs $\Dn{i} \otimes C \to D$.
  En particulier, modulo l'identification $C \simeq \Dn{0} \otimes C$, ses
  objets sont les \oo-foncteurs $C \to D$ et ses $1$-flèches les
  \oo-foncteurs $\Dn{1} \otimes C \to D$, les objets source et but étant
  obtenus par précomposition par $\sigma_1 \otimes C : C \to \Dn{1} \otimes C$
  et $\tau_1 \otimes C : C \to \Dn{1} \otimes C$ respectivement, en
  identifiant toujours $C$ et $\Dn{0} \otimes C$. Ainsi, en vertu du
  corollaire~\ref{coro:trans_oplax_abs}, les $1$-flèches de~$\HomOpLax(C,
  D)$ d'un \oo-foncteur $u : C \to D$ vers un \oo-foncteur $v : C \to D$
  correspondent aux transformations oplax de $u$ vers $v$.
\end{paragr}

\begin{paragr}\label{paragr:def_comp_trans}
  En vertu du paragraphe précédent, la composition des $1$-flèches de
  la \oo-catégorie $\HomOpLax(C, D)$ définit une composition des
  transformations oplax. Plus précisément, fixons $u$, $v$ et $w$ trois
  \oo-foncteurs de $C$ vers $D$. Si $\alpha$ est une transformation oplax de
  $u$ vers $v$ et $\alpha'$ une transformation oplax de $v$ vers $w$, alors
  on dispose d'une transformation oplax de $u$ vers $w$ qu'on notera \nnot{$\alpha'
  \circ \alpha$}.

  Cette composition peut se décrire de la manière suivante.
  Les transformations oplax~$\alpha$ et $\alpha'$ correspondent à des
  \oo-foncteurs $H : \Dn{1} \otimes C \to D$ et \hbox{$H' : \Dn{1} \otimes C \to
  D$}.  Par ailleurs, le \oo-foncteur $\nabla^1_0 : \Dn{1} \to \Dn{1}
  \amalg_{\Dn{0}} \Dn{1}$ du paragraphe~\ref{paragr:def_kappa_nabla} induit
  un \oo-foncteur
  \[
    \nabla^1_0 \otimes C : \Dn{1} \otimes C \longto (\Dn{1} \amalg_{\Dn{0}}
    \Dn{1}) \otimes C \simeq (\Dn{1} \otimes C) \amalg_C (\Dn{1} \otimes C).
  \]
  La transformation oplax $\alpha' \circ \alpha$ correspond alors au
  \oo-foncteur composé
  \[
    \Dn{1} \otimes C
      \xto{\nabla^1_0 \otimes C}
    (\Dn{1} \otimes C) \amalg_C (\Dn{1} \otimes C)
      \xto{(H', H)}
    D.
  \]
\end{paragr}

\begin{paragr}\label{paragr:def_id_trans_abs}
  Soit $u : C \to D$ un \oo-foncteur. On peut voir $u$ comme un objet
  de~$\HomOpLax(C, D)$ et on dispose donc d'une $1$-flèche $\id{u}$ de
  $\HomOpLax(C, D)$ qu'on verra comme une transformation oplax de $u$ vers
  $u$. Explicitement, la transformation oplax~$\id{u}$ correspond au
  \oo-foncteur
  \[ \Dn{1} \otimes C \xto{\kappa_0 \otimes C} C \xto{u} D, \]
  induit par le \oo-foncteur $\kappa_0 : \Dn{1} \to \Dn{0}$ du
  paragraphe~\ref{paragr:def_disque}, où on a encore identifié $\Dn{0}
  \otimes C$ et~$C$. Par adjonction, elle correspond également, dans la
  bijection de la proposition~\ref{prop:cotrans_abs}, au \oo-foncteur
  \[
    C \xto{\HomLax(\kappa_0, C)} \HomLax(\Dn{1}, C)
     \xto{\HomLax(\Dn{1}, u)} \HomLax(\Dn{1}, D),
  \]
  où cette fois on a identifié $\HomLax(\Dn{0}, C)$ et $C$.
\end{paragr}

On a déjà défini, au paragraphe~\ref{paragr:def_trans_id}, une
transformation oplax $\id{u}$. Dans la suite de cette section, sauf
mention du contraire, $\id{u}$ fera toujours référence à la transformation
oplax que l'on vient d'introduire. On va montrer que les deux définitions
sont équivalentes.

\begin{prop}\label{prop:desc_cone_triv}
  Fixons $i \ge 0$ et notons $a$ la cellule principale de $\Dn{1}$, $b$
  celle de $\Dn{0}$ et $x$ celle de $\Dn{i}$. Alors le \oo-foncteur
  $\kappa_0 \otimes \Dn{i} : \Dn{1} \otimes \Dn{i} \to \Dn{0} \otimes
  \Dn{i}$ est donné par
  \begin{alignat*}{3}{}
    \atom{a^\ep_0 \otimes x^\e_k} & \mapsto \atom{b \otimes x^\e_k} &&
      \qquad\text{pour $0 \le k \le i$, $\e = 0, 1$ et $\ep = 0, 1$,}\\
    \atom{a \otimes x^\e_k} & \mapsto \id{\atom{b \otimes x^\e_k}} &&
      \qquad\text{pour $0 \le k \le i$ et $\e = 0, 1$.}
  \end{alignat*}
\end{prop}

\begin{proof}
  Le cas des atomes de la forme $\atom{a^\ep_0 \otimes x^\e_k}$ résulte, par
  fonctorialité, des égalités $\kappa_0\sigma_1 = \id{\Dn{0}} = \kappa_0\tau_1$.
  Par ailleurs, puisque $\lambda(\kappa_0)(a) = 0$, l'image d'un atome de
  la forme $\atom{a \otimes x^\e_k}$ est une identité. De plus, en utilisant
  le lemme~\ref{lemme:tab_cyl}, on~a
  \[
    \begin{split}
      s\big((\kappa_0 \otimes \Dn{i})(\atom{a \otimes x^\e_k})\big)_k
      & =
      (\kappa_0 \otimes \Dn{i})(s(\atom{a \otimes x^\e_k}))^{}_k \\
      & =
      (\lambda(\kappa_0) \otimes \lambda(\Dn{i}))(\atom{a \otimes
      x^\e_k}^0_k) \\
      & =
      (\lambda(\kappa_0) \otimes \lambda(\Dn{i}))(a^0_0 \otimes x^\e_k + a
      \otimes x^1_{k-1}) \\
      & =
      b \otimes x^\e_k,
    \end{split}
  \]
  où la dernière égalité résulte des relations $\lambda(\kappa_0)(a^0_0) = b$
  et $\lambda(\kappa_0)(a) = 0$. Or, la seule $k$-flèche $y$ de $\Dn{0} \otimes \Dn{i}
  \simeq \Dn{i}$ telle que $y_k = b \otimes x^\e_k$ est $\atom{b \otimes
  x^\e_k}$, d'où le résultat.
\end{proof}

\begin{prop}\label{prop:trans_id_abs}
  Soit $u : C \to D$ un \oo-foncteur. Pour toute cellule $x$ de $C$, on a
  \[ (\id{u})_x = \id{u(x)}. \]
\end{prop}

\begin{proof}
  En vertu de la proposition~\ref{prop:trans_sesqui_abs}, on peut supposer
  que $u$ est un \oo-foncteur identité. Si $x$ est une $i$-flèche de $C$,
  pour un $i \ge 0$, l'image de $x$ par le \oo-foncteur
  \[ \HomLax(\kappa_0, C) : C \to \HomLax(\Dn{1}, C) \]
  est obtenue en composant
  \[
    \Dn{1} \otimes \Dn{i} \xto{\kappa_0 \otimes \Dn{i}}
    \Dn{0} \otimes \Dn{i} \simeq \Dn{i} \xto{x} C.
  \]
  En vertu de la proposition précédente, cette image est $\id{x}$. On a donc
  $(\id{u})_x = \id{x}$, ce qu'on voulait démontrer.
\end{proof}

\begin{rem}
  La proposition précédente affirme que la transformation oplax~$\id{u}$
  définie dans cette section coïncide avec la transformation oplax
  $\id{u}$ définie au paragraphe~\ref{paragr:def_trans_id}.
\end{rem}

\section{Homotopies et transformations oplax}

\begin{paragr}\label{paragr:homot_abs}
  Soient $f, g : K \to L$ deux morphismes de complexes dirigés augmentés.
  La donnée d'une homotopie $h$ de $f$ vers $g$ est équivalente à celle d'un
  morphisme de complexes dirigés augmentés $H : \lambda(\Dn{1}) \otimes K
  \to L$ rendant le diagramme
  \[
    \xymatrix@C=3pc{
    K \ar[dr]^f \ar[d]_{\lambda(\sigma_1) \otimes K} \\
    \lambda(\Dn{1}) \otimes K \ar[r]^H & L \\
    K \ar[ur]_g \ar[u]^{\lambda(\tau_1) \otimes K}
    }
  \]
  commutatif. En effet, l'assertion analogue pour les complexes de chaînes
  est bien connue et on vérifie immédiatement qu'elle s'étend aux complexes
  dirigés augmentés. Explicitement, en notant $a$ la cellule principale de
  $\Dn{1}$, l'homotopie $h$ associée à $H$ est définie par $h(x) = H(a
  \otimes x)$ pour tout élément homogène $x$ de $K$.
\end{paragr}

\begin{rem}
  On déduit du paragraphe précédent, par dualité (en utilisant notamment la
  proposition~\ref{prop:dual_tens_cda}), comme dans la preuve du
  corollaire~\ref{coro:trans_lax_abs}, qu'une antihomotopie $h$ de $f$ vers $g$
  correspond à un morphisme de complexes dirigés augmentés $H : K \otimes
  \lambda(\Dn{1}) \to L$ rendant le diagramme
  \[
    \xymatrix@C=3pc{
    K \ar[dr]^f \ar[d]_{K \otimes \lambda(\sigma_1)} \\
    K \otimes \lambda(\Dn{1}) \ar[r]^H & L \\
    K \ar[ur]_g \ar[u]^{K \otimes \lambda(\tau_1)} \\
    }
  \]
  commutatif.
\end{rem}

\begin{rem}\label{rem:n-homot_abs}
  Plus généralement, pour $n \ge 1$, on peut vérifier que la donnée d'une
  $n$-homotopie (voir le paragraphe~\ref{paragr:def_n-homot}) de $K$ vers
  $L$ correspond à la donnée d'un morphisme
  \[
    \lambda(\Dn{n}) \otimes K \to L,
  \]
  les $(n-1)$-homotopies source et but correspondant aux morphismes obtenus
  en précomposant par
  \[
    \lambda(\sigma_n) \otimes K, \lambda(\tau_n) \otimes K :
    \lambda(\Dn{n-1}) \otimes K \to \lambda(\Dn{n}) \otimes K.
  \]
  On vérifie de même qu'on obtient l'assertion analogue pour les
  $n$-antihomotopies (voir également le paragraphe~\ref{paragr:def_n-homot})
  en inversant les facteurs du produit tensoriel.
\end{rem}

\begin{paragr}\label{paragr:homot_sesqui_abs_Cda}
  Soient $f_0, f_1 : K \to L$ deux morphismes de complexes dirigés augmentés
  et $h$ une homotopie de $f_0$ vers $f_1$. Notons $H : \lambda(\Dn{1})
  \otimes K \to L$ le morphisme correspondant.
  \begin{enumerate}
    \item Pour tout morphisme $g : J \to K$, l'homotopie $hg$
      (voir le paragraphe~\ref{paragr:def_antih_sesqui})
      correspond au morphisme
      \[
        \lambda(\Dn{1}) \otimes J
        \xto{\lambda(\Dn{1}) \otimes g} \lambda(\Dn{1}) \otimes K
        \xto{H} L.
      \]
      En effet, par naturalité, un tel morphisme correspond bien à une
      homotopie de $f_0g$ vers $f_1g$ et, en notant $a$ la cellule
      principale de $\Dn{1}$, pour $x$ un élément homogène de $J$, on a
      $H(\lambda(\Dn{1}) \otimes g)(a \otimes x) = H(a \otimes g(x)) = h(g(x)).$
    \item Pour tout morphisme $g : L \to M$, l'homotopie $gh$
      (voir également le paragraphe~\ref{paragr:def_antih_sesqui})
      correspond au morphisme
      \[
        \lambda(\Dn{1}) \otimes K \xto{H} L \xto{g} M.
      \]
      En effet, un tel morphisme correspond bien à une homotopie de $gf_0$
      vers $gf_1$ et, en notant $a$ la cellule principale de $\Dn{1}$, pour
      $x$ un élément homogène de $K$, on~a~$gH(a, x) = g(h(x))$.
  \end{enumerate}

  Par ailleurs, pour tout morphisme de complexes dirigés augmentés $f : K
  \to L$, l'homotopie $\id{f}$ (voir le
  paragraphe~\ref{paragr:def_antih_id})
  correspond au morphisme
  \[ \lambda(\Dn{1}) \otimes K \xto{\lambda(\kappa_0) \otimes K} K \xto{f} L, \]
  induit par le \oo-foncteur $\kappa_0 : \Dn{1} \to \Dn{0}$ du
  paragraphe~\ref{paragr:def_disque}. En effet, par naturalité, un tel
  morphisme correspond bien à une homotopie de $f$ vers $f$ et, en
  notant $a$ la cellule principale de $\Dn{1}$, pour $x$ un élément homogène
  de $K$, on~a~\hbox{$f(\lambda(\kappa_0) \otimes K)(a \otimes x) = f(0
  \otimes x) = 0$}.
\end{paragr}

\begin{paragr}\label{paragr:homot_comp_abs_Cda}
  Soient $f_0, f_1, f_2 : K \to L$ trois morphismes de complexes dirigés
  augmentés, $h$ une homotopie de $f_0$ vers $f_1$ et $h'$ une homotopie
  de $f_1$ vers $f_2$. D'après le paragraphe~\ref{paragr:homot_abs}, ces
  homotopies correspondent à des morphismes $H : \lambda(\Dn{1}) \otimes K
  \to L$ et $H' : \lambda(\Dn{1}) \otimes K \to L$ respectivement. Par
  ailleurs, comme dans le paragraphe~\ref{paragr:def_comp_trans}, on peut
  définir un morphisme $H''$ en composant
  \[
    \lambda(\Dn{1}) \otimes K
      \xto{\lambda(\nabla^1_0) \otimes K}
    (\lambda(\Dn{1}) \otimes K) \amalg_K (\lambda(\Dn{1}) \otimes K)
      \xto{(H', H)}
    L.
  \]
  Le morphisme $H''$ correspond à une homotopie de $f_0$ vers $f_2$
  qui n'est autre que l'homotopie $h' + h$ (voir le
  paragraphe~\ref{paragr:def_antih_comp}).
  En effet, en notant $a$, $b$ et $c$ les cellules principales des copies de
  $\lambda(\Dn{1})$ apparaissant, dans l'ordre, dans la source de
  $\lambda(\nabla^1_0) \otimes K$, puis de gauche à droite dans son but,
  on a, pour tout élément homogène $x$ de $K$,
  \[
     H''(a \otimes x) = (H', H)(b \otimes x + c \otimes x)
      = H'(b \otimes x) + H(c \otimes x) = h'(x) + h(x),
  \]
  d'où l'assertion.
\end{paragr}

\begin{paragr}\label{paragr:def_nu_homot}
  Soient $f, g : K \to L$ deux morphismes de complexes dirigés augmentés. À
  toute homotopie $h$ de $f$ vers $g$, on associe une transformation oplax
  \nnot[$\nu(h) : \nu(f) \Rightarrow \nu(g)$]{$\nu(h)$} de~$\nu(f)$
  vers~$\nu(g)$ de la manière suivante.  L'homotopie $h$ correspond, en
  vertu du paragraphe~\ref{paragr:homot_abs}, à un morphisme $H :
  \lambda(\Dn{1}) \otimes K \to L$. On obtient un \oo-foncteur $\Dn{1}
  \otimes \nu(K) \to \nu(L)$ en composant
  \[
    \Dn{1} \otimes \nu(K) \simeq \nu(\lambda(\Dn{1})) \otimes \nu(K)
    \to \nu(\lambda(\Dn{1}) \otimes K) \xto{\nu(H)} \nu(L),
  \]
  où la flèche du milieu est la contrainte du foncteur monoïdal lax $\nu$
  (voir la proposition~\ref{prop:lambda_nu_mon_tens}). Ce
  \oo-foncteur correspond à son tour, cette fois en vertu du
  corollaire~\ref{coro:trans_oplax_abs}, à une transformation oplax qui, par
  naturalité des flèches en jeu, va de $\nu(f)$ vers $\nu(g)$.
\end{paragr}

\begin{prop}\label{prop:homot_trans_oplax}
  Soient $f, g : K \to L$ deux morphismes de complexes dirigés augmentés
  avec $K$ un complexe de Steiner fort.  Le foncteur $\nu$ induit une
  bijection entre les homotopies de~$f$ vers $g$ et les transformations
  oplax de $\nu(f)$ vers $\nu(g)$.
\end{prop}

\begin{proof}
  Par adjonction, la correspondance $h \mapsto \nu(h)$ du paragraphe
  précédent peut se décrire comme le composé
  \[
    \begin{split}
      \Hom_{\Cda}(\lambda(\Dn{1}) \otimes K, L)
      & \to
      \Hom_{\Cda}(\lambda\nu(\lambda(\Dn{1}) \otimes K), L) \\
      & \simeq
      \Hom_{\ooCat}(\nu(\lambda(\Dn{1}) \otimes K), \nu(L)) \\
      & \to
      \Hom_{\ooCat}(\nu(\lambda(\Dn{1})) \otimes \nu(K), \nu(L)) \\
      & \simeq
      \Hom_{\ooCat}(\Dn{1} \otimes \nu(K), \nu(L)),
    \end{split}
  \]
  où la première application est induite par la coünité du couple de
  foncteurs adjoints~$(\lambda, \nu)$ et la troisième par la contrainte du
  foncteur monoïdal lax $\nu$. Mais, puisque $K$ est un complexe de Steiner
  fort, il en est de même de $\lambda(\Dn{1}) \otimes K$ (voir la
  proposition~\ref{prop:tens_Steiner}) et les théorèmes~\ref{thm:Steiner}
  et~\ref{thm:produit_tens} entraînent que ces deux applications
  sont des bijections, d'où le résultat.
\end{proof}

Montrons maintenant la compatibilité aux compositions et identités
de l'extension de $\nu$ aux homotopies.

\begin{prop}\label{prop:compat_mu_sesqui}
  Soient $K$ et $L$ deux complexes dirigés augmentés et $h$ une homotopie
  entre morphismes de $K$ vers $L$.
  \begin{enumerate}
    \item Pour tout morphisme de complexes dirigés augmentés $g : J \to K$,
      on a \[ \nu(hg) = \nu(h) \comp \nu(g). \]
    \item Pour tout morphisme de complexes dirigés augmentés $g : L \to M$,
      on a \[ \nu(gh) = \nu(g) \comp \nu(h). \]
  \end{enumerate}
\end{prop}

\begin{proof}
  Notons $H : \lambda(\Dn{1}) \otimes K \to L$ le morphisme correspondant
  à~$h$. En vertu de la proposition~\ref{prop:trans_sesqui_abs} et du
  paragraphe~\ref{paragr:homot_sesqui_abs_Cda}, les transformations oplax de
  la première assertion correspondent aux deux \oo-foncteurs
  $\Dn{1} \otimes \nu(J) \to \nu(L)$ qu'on peut définir à partir du
  diagramme
  \[
    \xymatrix{
      \Dn{1} \otimes \nu(J) \ar[d]_{\Dn{1} \otimes \nu(g)} \ar[r] &
      \nu(\lambda(\Dn{1}) \otimes J) \ar[d]^{\nu(\lambda(\Dn{1}) \otimes g)}\\
      \Dn{1} \otimes \nu(K) \ar[r] &
      \nu(\lambda(\Dn{1}) \otimes K) \ar[r]^-{\nu(H)} & \nu(L) \pbox{,}
    }
  \]
  où les flèches non nommées sont la contrainte du foncteur monoïdal lax
  $\nu$. Or, le carré de ce diagramme est commutatif par
  naturalité, d'où la première assertion.

  D'après les mêmes résultats, les transformations oplax de la seconde
  assertion correspondent aux deux manières de composer les trois
  \oo-foncteurs
  \[
    \Dn{1} \otimes \nu(K) \to \nu(\lambda(\Dn{1}) \otimes K) \xto{\nu(H)}
    \nu(L) \xto{\nu(g)} \nu(M),
  \]
  où, de nouveau, la flèche non nommée est la contrainte du foncteur
  monoïdal lax~$\nu$, d'où le résultat.
\end{proof}

\begin{prop}\label{prop:compat_mu_id}
  Soit $f : K \to L$ un morphisme de complexes dirigés augmentés. On
  a
  \[ \nu(\id{f}) = \id{\nu(f)}. \]
\end{prop}

\begin{proof}
  En vertu de la proposition~\ref{prop:trans_id_abs} et du
  paragraphe~\ref{paragr:homot_sesqui_abs_Cda}, cela résulte de la
  commutativité, par naturalité, du diagramme
  \[
    \xymatrix{
      \Dn{1} \otimes \nu(K) \ar[d]_{\kappa_0 \otimes \nu(K)} \ar[r] &
      \nu(\lambda(\Dn{1}) \otimes K) \ar[d]^{\nu(\lambda(\kappa_0) \otimes K)}\\
      \Dn{0} \otimes \nu(K) \ar[r] &
      \nu(\lambda(\Dn{0}) \otimes K) \ar[r]^-{\nu(f)} & \nu(L) \pbox{,}
    }
  \]
  où les flèches non nommées sont la contrainte du foncteur monoïdal lax~$\nu$.
\end{proof}

\begin{prop}\label{prop:compat_mu_comp}
  Soient $f_0, f_1, f_2 : K \to L$ trois morphismes de complexes dirigés
  augmentés, $h$ une homotopie de $f_0$ vers $f_1$ et $h'$ une homotopie de
  $f_1$ vers $f_2$. On a
  \[ \nu(h' + h) = \nu(h') \circ \nu(h). \]
\end{prop}

\begin{proof}
  Notons $H, H' : \lambda(\Dn{1}) \otimes K \to L$ les morphismes
  correspondant respectivement à~$h$ et $h'$. En vertu des
  paragraphes~\ref{paragr:def_comp_trans}
  et~\ref{paragr:homot_comp_abs_Cda}, il suffit de montrer la commutativité
  du diagramme
  \[
    \xymatrix{
      \Dn{1} \otimes \nu(K)
        \ar[d]_{\nabla^1_0 \otimes \nu(K)} \ar[r]
      &
      \nu(\lambda(\Dn{1}) \otimes K)
        \ar[d]^{\nu(\lambda(\nabla^1_0) \otimes K)}
      \\
      (\Dn{1} \otimes \nu(K)) \amalg_{\nu(K)} (\Dn{1} \otimes \nu(K))
        \ar[dr]_{(\nu(h'), \nu(h))\quad} \ar[r]
      &
      \nu\big((\lambda(\Dn{1}) \otimes K) \amalg_{K} (\lambda(\Dn{1}) \otimes
      K)\big) \ar[d]^{\nu((H', H))} \\
      & \nu(L) \pbox{,}
    }
  \]
  où les flèches non nommées sont induites par la contrainte du foncteur
  monoïdal lax~$\nu$ et, plus précisément, pour celle du bas, par le morphisme
  \[
    (\Dn{1} \amalg_{\Dn{0}} \Dn{1}) \otimes \nu(K)
    \to
    \nu(\lambda(\Dn{1} \amalg_{\Dn{0}} \Dn{1}) \otimes K).
  \]
  Or, le rectangle de ce diagramme est commutatif par naturalité. Pour
  conclure, il nous suffit donc de vérifier la commutativité du triangle.
  En précomposant ce triangle par les deux morphismes canoniques
  \[
    \Dn{1} \otimes \nu(K) \to
    (\Dn{1} \otimes \nu(K)) \amalg_{\nu(K)} (\Dn{1} \otimes \nu(K)),
  \]
  on obtient deux triangles qui commutent par définition de $\nu(h)$ et
  $\nu(h')$, d'où le résultat.
\end{proof}

\begin{rem}
  Dualement, si $h$ est une antihomotopie de complexes dirigés augmentés, on
  peut lui associer une transformation lax $\nu(h)$. Les
  propositions~\ref{prop:homot_trans_oplax}, \ref{prop:compat_mu_sesqui},
  \ref{prop:compat_mu_id} et~\ref{prop:compat_mu_comp} entraînent, par
  dualité, leurs analogues pour les antihomotopies et les transformations
  lax.
\end{rem}

\section{Joint et transformations oplax}

\begin{paragr}\label{paragr:incl_cone_cyl}
  Soient $C$ une \oo-catégorie et $c$ un objet de $C$. On définit un
  foncteur d'inclusion $\cotr{C}{c} \hookto \HomLax(\Dn{1}, C)$ de la
  manière suivante.  Pour tout $i \ge 0$, si $(d, \alpha)$ est une
  $i$\nbd-flèche de $\cotr{C}{c}$ (voir le
  paragraphe~\ref{paragr:desc_tr_pol}), on lui
  associe la $i$-flèche $(\id{c}, d, \alpha)$ de $\HomLax(\Dn{1}, C)$ (voir
  le paragraphe~\ref{paragr:desc_cyl_pol}), où $\id{c}$ désigne l'identité
  itérée de $c$ en dimension $i$. Le fait qu'on obtienne bien ainsi une
  $i$-flèche de $\HomLax(\Dn{1}, C)$ résulte des
  paragraphes~\ref{paragr:desc_tr_pol} et \ref{paragr:desc_cyl_pol}, ainsi
  que de la formule
  \[ \alpha^1_{k-1} \comp_{k-1} \cdots \comp_1 \alpha^1_0 \comp \id{c} =
  \alpha^1_{k-1}. \]
  Par ailleurs, la fonctorialité de cette correspondance découle des
  propositions~\ref{prop:desc_tr} et~\ref{prop:desc_cyl}, ainsi que de la
  formule ci-dessus.
\end{paragr}

\begin{prop}\label{prop:incl_cyl_cone}
  Pour toute \oo-catégorie $C$ et tout objet $c$ de $C$, le carré
  \[
    \xymatrix{
      \cotr{C}{c} \ar[r] \ar[d] & \HomLax(\Dn{1}, C) \ar[d]^{\pi^0} \\
      \Dn{0} \ar[r]_c & C \pbox{,}
    }
  \]
  où la flèche horizontale du haut est le \oo-foncteur défini au paragraphe
  précédent et $\pi^0$ désigne le \oo-foncteur du
  paragraphe~\ref{paragr:cohomot_Gray}, est cartésien.
\end{prop}

\begin{proof}
  Cela résulte immédiatement de la description de l'inclusion $\cotr{C}{c}
  \hookto \HomLax(\Dn{1}, C)$ donnée au paragraphe précédent.
\end{proof}

\begin{coro}
  Soient $C$ une \oo-catégorie et $c$ un objet de $C$. Pour toute
  \oo-catégorie $A$, on a une bijection naturelle entre les \oo-foncteurs
  $A \to \cotr{C}{c}$ et les transformations oplax entre \oo-foncteurs de $A$ vers $C$
  de source le \oo-foncteur constant de valeur $c$. De plus, le but de la
  transformation oplax associée à un tel \oo-foncteur est le composé de
  $A \to \cotr{C}{c} \xto{U} C$, où $U$ désigne le \oo-foncteur d'oubli.
\end{coro}

\begin{proof}
  La première assertion est conséquence immédiate de la proposition
  précédente et de la description des transformations oplax en termes de
  $\HomLax(\Dn{1}, C)$ (proposition~\ref{prop:cotrans_abs}). La seconde
  résulte de la commutativité du triangle
  \[
    \xymatrix@C=0.5pc{
      \cotr{C}{c} \ar[drrr]_U \ar[rrrrr]&&&&& \HomLax(\Dn{1}, C) \ar[dll]^{\pi^1} \\
      &&& C && \pbox{,}
    }
  \]
  où $U$ désigne le \oo-foncteur d'oubli et $\pi^1$ le \oo-foncteur du
  paragraphe~\ref{paragr:cohomot_Gray}.
\end{proof}

\begin{coro}
  Soient $C$ une \oo-catégorie et $c$ un objet de $C$. Pour tout $i \ge 0$,
  les $i$-flèches de $\cotr{C}{c}$ sont en correspondance bijective
  canonique avec les couples $(d, \alpha)$, où $d$ est une $i$-flèche de $C$
  et $\alpha$ est une transformation oplax du \oo-foncteur $\Dn{i} \to C$
  constant de valeur $c$ vers le \oo-foncteur $d : \Dn{i} \to C$.
\end{coro}

\begin{proof}
  Cela résulte du corollaire précédent appliqué à $A = \Dn{i}$.
\end{proof}

\begin{paragr}
  Soit $C$ une \oo-catégorie. On définit un \oo-foncteur $\Dn{1} \otimes C
  \to \Dn{0} \joint C$ de la manière suivante. Si $A$ est une \oo-catégorie,
  en utilisant les adjonctions des paragraphes~\ref{paragr:def_tranche}
  et~\ref{paragr:def_HomOpLax}, ainsi que le \oo-foncteur du
  paragraphe~\ref{paragr:incl_cone_cyl}, on obtient une application
  \[
    \begin{split}
      \Hom_{\ooCat}(\Dn{0} \joint C, A)
      & \xto{\sim} {\smash{\coprod_{a \in A_0}} \Hom_{\ooCat}(C, \cotr{A}{a})} \\
      & \qquad\qquad\qquad\big\downarrow \\
      & \Hom_{\ooCat}(C, \HomLax(\Dn{1}, A))
       \xto{\sim} \Hom_{\ooCat}(\Dn{1} \otimes C, A),
    \end{split}
  \]
  naturelle en $A$, et donc le \oo-foncteur recherché en vertu du lemme de
  Yoneda.
\end{paragr}

\begin{coro}
  Pour toute \oo-catégorie $C$, le carré
  \[
    \xymatrix@C=2.5pc{
      C \ar[d] \ar[r]^-{\sigma_1 \otimes C} & \Dn{1} \otimes C \ar[d] \\
      \Dn{0} \ar[r]_-{\iota_1} & \Dn{0} \joint C \pbox{,}
    }
  \]
  où la flèche verticale de droite est le \oo-foncteur défini au paragraphe
  précédent et $\sigma_1$ désigne le \oo-foncteur du paragraphe
  \ref{paragr:def_disque}, est cocartésien.
\end{coro}

\begin{proof}
  Si $A$ est une \oo-catégorie, en vertu des adjonctions utilisées dans le
  paragraphe précédent et de la proposition~\ref{prop:incl_cyl_cone}, on a
  {\let\ooCat\ooCatTiny
  \[
    \begin{split}
      \Hom_{\ooCat}(\Dn{0} \joint C, A)
      & \simeq \coprod_{a \in A_0} \Hom_{\ooCat}(C, \cotr{A}{a}) \\
      & \simeq \coprod_{a \in A_0} \Hom_{\ooCat}(C, \HomLax(\Dn{1}, A))
        \times_{\Hom_{\ooCat}(C, A)} \{a\} \\
      & \simeq \coprod_{a \in A_0} \Hom_{\ooCat}(\Dn{1} \otimes C, A)
        \times_{\Hom_{\ooCat}(C, A)} \{a\} \\
      & \simeq \Hom_{\ooCat}(\Dn{1} \otimes C, A)
        \times_{\Hom_{\ooCat}(C, A)} \Hom_{\ooCat}(\Dn{0}, A) \\
      & \simeq \Hom_{\ooCat}((\Dn{1} \otimes C) \amalg_C \Dn{0}, A),
    \end{split}
  \]
  }
  d'où le résultat.
\end{proof}

\section{Suspension et transformations oplax}

\begin{paragr}\label{paragr:incl_susp_cyl}
  Soit $C$ une \oo-catégorie et soient $c$ et $d$ deux objets de $C$. On
  notera \nnot{$\Homi_C(c, d)$} la \oo-catégorie
  des cellules de $C$ de $0$-source $c$ et de $0$-but $d$.  On définit un foncteur
  d'inclusion $\Homi_C(c,d)^\op \hookto \HomLax(\Dn{1}, C)$ de la
  manière suivante. Pour tout~$i \ge 0$, on associe à une $i$-flèche de
  $\Homi_C(c,d)^\op$, qui correspond par définition à une $(i+1)$\nbd-flèche~$f$
  de $C$ de $0$-source $c$ et de $0$-but $d$, la $i$-flèche $(\id{c},
  \id{d}, \alpha)$ de~$\HomLax(\Dn{1}, C)$ (voir le
  paragraphe~\ref{paragr:desc_cyl_pol}), où $\id{c}$ et $\id{d}$ désignent
  les identités itérées respectives de $c$ et $d$ en dimension $i$, et où
  $\alpha$ est défini par
  \[ \alpha^0_k = t^{}_{k+1}(f) \quadet \alpha^1_k = s^{}_{k+1}(f), \]
  pour $0 \le k \le i$ (en particulier, on a $\alpha_i = f$).
  Le fait qu'on obtienne bien ainsi une $i$-flèche de $\HomLax(\Dn{1}, C)$
  résulte du paragraphe \ref{paragr:desc_cyl_pol}, ainsi que des formules
  \[ \alpha^1_{k-1} \comp_{k-1} \cdots \comp_1 \alpha^1_0 \comp_0 \id{c} =
  \alpha^1_{k-1}
  \quadet
  \id{d} \ast_0 \alpha^0_0 \ast_1 \dots \ast_{k-1} \alpha^0_{k-1} =
  \alpha^0_{k-1}.
  \]
  Par ailleurs, la fonctorialité de cette correspondance découle de la
  proposition~\ref{prop:desc_cyl}, ainsi que des formules ci-dessus.
\end{paragr}

\begin{prop}\label{prop:incl_susp_cone}
  Pour toute \oo-catégorie $C$ et tous objets $c$ et $d$ de $C$, le
  \oo-foncteur d'inclusion $\Homi_C(c,d)^\op \hookto \HomLax(\Dn{1}, C)$
  induit un isomorphisme
  \[
    \Homi_C(c,d)^\op \simeq \{c\} \times_C \HomLax(\Dn{1}, C)
    \times_C \{d\},
  \]
  le membre de droite désignant la limite projective du diagramme
  \[
    \xymatrix@C=1.5pc{
      \Dn{0} \ar[dr]_c & & \HomLax(\Dn{1}, C) \ar[dl]^{\pi^0}
      \ar[dr]_{\pi^1} & & \Dn{0} \ar[dl]^d \\
             & C & & C & \pbox{,}
    }
  \]
  où $\pi^0$ et $\pi^1$ sont les \oo-foncteurs du
  paragraphe~\ref{paragr:cohomot_Gray}.
\end{prop}

\begin{proof}
  Cela résulte immédiatement de la description de l'inclusion
  $\Homi_C(c,d)^\op \hookto \HomLax(\Dn{1}, C)$ donnée au paragraphe
  précédent.
\end{proof}

\begin{coro}
  Soit $C$ une \oo-catégorie et soient $c$ et $d$ deux objets de~$C$. Pour toute
  \oo-catégorie $A$, on a une bijection naturelle entre les \oo-foncteurs
  \hbox{$A \to \Homi_C(c, d)^\op$} et les transformations oplax entre \oo-foncteurs
  de $A$ vers $C$ de source le \oo-foncteur constant de valeur $c$ et de but
  le \oo-foncteur constant de valeur~$d$.
\end{coro}

\begin{proof}
  L'assertion est conséquence immédiate de la proposition précédente et de
  la description des transformations oplax en termes de $\HomLax(\Dn{1}, C)$
  (proposition~\ref{prop:cotrans_abs}).
\end{proof}

\begin{coro}
  Soit $C$ une \oo-catégorie et soient $c$ et $d$ deux objets de $C$. Pour tout
  $i \ge 0$, les transformations oplax du \oo-foncteur $\Dn{i} \to C$
  constant de valeur~$c$ vers le \oo-foncteur $\Dn{i} \to C$ constant de
  valeur $d$ sont en bijection naturelle avec les $(i+1)$\nbd-flèches de $C$ de
  $0$-source $c$ et de $0$-but $d$. De plus, si une transformation
  oplax~$\alpha$ et une $(i+1)$\nbd-cellule $f$ de $C$ se correspondent dans cette
  bijection, alors, pour $i \ge 1$, les $i$-cellules source et but de $f$
  correspondent aux transformations oplax $\alpha \comp \tau_i$ et $\alpha
  \comp \sigma_i$, où $\sigma_i$ et $\tau_i$ désignent les \oo-foncteurs du
  paragraphe~\ref{paragr:def_disque}.
\end{coro}

\begin{proof}
  La première assertion résulte de la bijection du corollaire précédent dans
  le cas $A = \Dn{i}$ et la seconde de la naturalité de cette même bijection
  appliquée à $\sigma_i$ et $\tau_i$.
\end{proof}

\begin{paragr}
  Soit $C$ une \oo-catégorie. La \ndef[suspension!d'une
  $\infty$-catégorie]{suspension} de $C$ est la \oo-catégorie
  \nnot{$\susp{C}$} définie de la manière suivante:
  \begin{itemize}
    \item les objets de $\susp{C}$ sont $0$ et $1$ ;
    \item si $x$ et $y$ sont deux objets de $\susp{C}$,
    on a
    \[
      \Homi_{\susp{C}}(x, y) =
        \begin{cases}
          C^\op & \text{si $x = 0$ et $y = 1$,} \\
          \ast & \text{si $x = y$,} \\
          \vide & \text{sinon ;}
        \end{cases}
    \]
    \item les compositions et les identités sont définies de la façon
    évidente.
  \end{itemize}

  On définit un \oo-foncteur $\Dn{1} \otimes C \to \susp{C}$ comme suit.
  Soit $A$ une \oo-catégorie. On vérifie immédiatement que
  la donnée d'un \oo-foncteur de $\susp{C}$ vers $A$ correspond à la donnée de
  deux objets $a$ et $b$ de $A$, images respectives des objets $0$ et $1$,
  et d'un \oo-foncteur de $\Homi_{\susp{C}}(0, 1) = C^\op$ vers
  $\Homi_{A}(a, b)$. Ainsi, on dispose de bijections naturelles
  \[
    \begin{split}
      \Hom_{\ooCat}(\susp{C}, A)
      & \simeq
      \coprod_{a, b \in A_0} \Hom_{\ooCat}(C^\op, \Homi_{A}(a, b)) \\
      & \simeq
      \coprod_{a, b \in A_0} \Hom_{\ooCat}(C, \Homi_{A}(a, b)^\op).
    \end{split}
  \]
  En utilisant le \oo-foncteur du paragraphe~\ref{paragr:incl_susp_cyl} et
  l'adjonction du paragraphe~\ref{paragr:def_tranche}, on obtient donc une
  application
  \[
    \begin{split}
      \Hom_{\ooCat}(\susp{C}, A)
      & \xto{\sim} {\smash{\coprod_{a, b \in A_0}} \Hom_{\ooCat}(C,
      \Homi_{A}(a, b)^\op)} \\
      & \qquad\qquad\qquad\big\downarrow \\
      & \Hom_{\ooCat}(C, \HomLax(\Dn{1}, A))
       \xto{\sim} \Hom_{\ooCat}(\Dn{1} \otimes C, A),
    \end{split}
  \]
  naturelle en $A$, et donc un \oo-foncteur $\Dn{1} \otimes C \to \susp{C}$
  en vertu du lemme de Yoneda.
\end{paragr}

\begin{coro}
  Pour toute \oo-catégorie $C$, le \oo-foncteur $\Dn{1} \otimes C \to
  \susp{C}$ défini au paragraphe précédent induit un isomorphisme
  \[ \susp{C} \simeq \Dn{0} \amalg_C (\Dn{1} \otimes C) \amalg_C
  \Dn{0}, \]
  le membre de droite désignant la limite inductive du diagramme
  \[
    \xymatrix@C=2pc{
      \Dn{0} && \Dn{1} \otimes C && \Dn{0} \\
      & C \ar[ul] \ar[ur]^{\sigma_1 \otimes C} && C \ar[ul]_{\tau_1
      \otimes C} \ar[ur] & \pbox{,}
    }
  \]
  où $\sigma_1$ et $\tau_1$ sont les \oo-foncteurs du paragraphe
  \ref{paragr:def_disque}.
\end{coro}

\begin{proof}
  Si $A$ est une \oo-catégorie, en vertu du paragraphe précédent et de la
  proposition~\ref{prop:incl_susp_cone}, on a
  {\let\ooCat\ooCatTiny
  \[
    \begin{split}
      \MoveEqLeft
      \Hom_{\ooCat}(\susp{C}, A) \\
      & \simeq \coprod_{a,b \in A_0} \Hom_{\ooCat}(C, \Homi_A(a, b)^\op) \\
      & \simeq \coprod_{a,b \in A_0} \{a\} \times_{\Hom_{\ooCat}(C, A)} \Hom_{\ooCat}(C, \HomLax(\Dn{1}, A))
        \times_{\Hom_{\ooCat}(C, A)} \{b\} \\
      & \simeq \coprod_{a,b \in A_0} \{a\} \times_{\Hom_{\ooCat}(C, A)}
      \Hom_{\ooCat}(\Dn{1} \otimes C, A) \times_{\Hom_{\ooCat}(C, A)} \{b\} \\
      & \simeq
      \Hom_{\ooCat}(\Dn{0}, A)
      \underset{\substack{\quad\\\mathclap{\Hom_{\ooCat}(C, A)}}}{\times}
      \Hom_{\ooCat}(\Dn{1} \otimes C, A) 
      \underset{\substack{\quad\\\mathclap{\Hom_{\ooCat}(C, A)}}}{\times}
      \Hom_{\ooCat}(\Dn{0}, A) \\
      & \simeq
      \Hom_{\ooCat}(\Dn{0} \amalg_C (\Dn{1} \otimes C) \amalg_C \Dn{0}, A),
    \end{split}
  \]
  }
  d'où le résultat.
\end{proof}

\chapter{Fonctorialités des tranches : conjectures}
\label{sec:conj}

\kern-1pt

Le but de cet appendice est de dégager des conjectures de fonctorialité des
tranches généralisant les résultats obtenus dans le
chapitre~\ref{sec:fonct_tr}. Pour formuler ces conjectures, on a besoin,
entre autres, d'introduire les notions de \oo-sesquicatégorie, de
\oo-catégorie de Gray (catégorie enrichie dans $\ooCat$ muni du produit
tensoriel de Gray), de \oo-catégorie de Gray gauche, et de montrer que la
sesquicatégorie des \oo-catégories, \oo-foncteurs et transformations oplax
(resp. transformations lax) provient d'une \oo-catégorie de Gray $\ooCatGr$
(resp. d'une \oo-catégorie de Gray gauche $\ooCatGrg$). On commence par
développer les notions correspondantes dans un cadre monoïdal (non
nécessairement symétrique) général.

\begin{paragr}\label{paragr:def_sesqui}
Soit $\Catmod$ une catégorie munie d'un foncteur $P:\Catmod\to\Ens$ à
valeurs dans la catégorie des ensembles qu'on appellera \ndef{foncteur
points}. Une \ndef[$z$@$\Catmod$-sesquicatégorie]{$\Catmod$\nbd-sesquicatégorie}, ou
\ndef[sesquicatégorie enrichie]{sesquicatégorie enrichie dans~$\Catmod$},
est une catégorie $\C$ munie d'un foncteur
\[
\VHom_\C:\C^\op\times\C\to\Catmod
\]
\notindex{$\VHom_\C : \C^\op \times \C \to \Catmod$}%
et d'un isomorphisme de foncteurs
\[
\UseTwocells
\xymatrixcolsep{.0pc}
\xymatrixrowsep{-.1pc}
\xymatrix{
\C^\op\times\C\ar[rrrrrr]^{\VHom_\C}\ar[rrrdddddd]_{\Hom_\C}
&&&&&&\Catmod\ar[llldddddd]^{P}
% \\&&&&\ar@2{<-}[lldd]^{\gamma}_\sim
\\&&&&\ar@2{<-}[lldd]_\sim
\\&&&&&&&
\\&&&&&&&
\\
\\
\\
&&&\Ens
&&&,
}
\]
qui le plus souvent, pour simplifier, sera considéré comme étant l'identité.
On dit que $\C$ est la \ndef[catégorie sous-jacente!à une
$\Catmod$-sesquicatégorie]{catégorie sous-jacente} à une telle
$\Catmod$\nbd-sesquicatégorie, que les objets de $\C$ sont ses
\ndef[objet!d'une $\Catmod$-sesquicatégorie]{objets} et que, pour $x,y$ deux
objets, $\VHom_\C(x,y)$ est son
\ndef[objet de morphismes!d'une $\Catmod$-sesquicatégorie]{objet de
morphismes de~$x$ vers $y$}.

Si $\Catmod$ est une catégorie monoïdale d'objet unité $\CatmodI$, on
dispose d'un foncteur points canonique
\[
P=\Hom_\Catmod(\CatmodI,\var):\Catmod\to\Ens,
\]
d'où une notion de $\Catmod$\nbd-sesquicatégorie. Toutes les
$\Catmod$\nbd-sesquicatégories considérées pour $\Catmod$ une catégorie
monoïdale seront relatives à ce foncteur points. On remarque que le produit
tensoriel de la catégorie monoïdale ne joue aucun rôle dans cette
définition.
\end{paragr}

\begin{exem}\label{exem:bourbaki_sesqui}
Si $\Catmod$ est la catégorie des ensembles avec comme foncteur points le
foncteur identité, une $\Catmod$\nbd-sesquicatégorie n'est autre qu'une
catégorie. Si $\Catmod$ est la catégorie $\Cat$ des petites
catégories munie du foncteur points $\Ob:\Cat\to\Ens$, alors on vérifie
facilement qu'une $\Catmod$\nbd-sesquicatégorie est ce qu'on appelle
habituellement une sesquicatégorie (voir par
exemple~\cite[section 2]{StreetCatStruct}).
\end{exem}

\begin{paragr}\label{paragr:def_oosesqui}
Une \ndef[$\infty$-sesquicatégorie]{\oo-sesquicatégorie} est une
sesquicatégorie enrichie dans la catégorie $\ooCat$ des \oo-catégories munie
du foncteur points
$\Ob:\ooCat\to\Ens$, associant à une \oo-catégorie l'ensemble de ses objets
ou, de façon équivalente, est une sesqui\-catégorie enrichie dans la catégorie
monoïdale $\ooCat$ munie du produit tensoriel défini, au choix, par le
produit cartésien ou par le produit tensoriel de Gray (voir le
paragraphe~\ref{paragr:def_tens}), puisque dans ces deux cas l'objet unité
est un objet final de $\ooCat$.

Soit $\C$ une \oo-sesquicatégorie. Pour $i\ge1$, une \ndef[$i$-flèche!d'une
$\infty$-sesquicatégorie]{$i$\nbd-flèche}
de $\C$ est une $(i-1)$\nbd-flèche de la \oo-catégorie $\VHom_\C(x,y)$,
pour des objets $x,y$ de $\C$, et on dit que $x$ est sa
\ndef[]{$0$\nbd-source} et $y$ son \ndef[]{$0$\nbd-but}. Pour $0<j\le i$, sa
\ndef[$j$-source d'une $i$-flèche!d'une
$\infty$-sesqui\-catégorie]{$j$\nbd-source} et son \ndef[$j$-but d'une
$i$-flèche!d'une $\infty$-sesquicatégorie]{$j$\nbd-but}
sont respectivement sa $(j-1)$\nbd-source et son $(j-1)$\nbd-but comme
cellule de $\VHom_\C(x,y)$.  La commutativité du triangle du
paragraphe~\ref{paragr:def_sesqui} montre que les $1$\nbd-flèches de $\C$
sont les flèches de sa catégorie sous-jacente. Pour $0\le k<i,j$, on dit
qu'une $i$\nbd-flèche $\alpha$ de $\C$ et une $j$\nbd-flèche $\beta$ de $\C$
sont \ndef[$i$-flèches $j$-composables!d'une
$\infty$-sesquicatégorie]{$k$\nbd-composables} si la $k$\nbd-source de
$\alpha$ est égale au $k$\nbd-but de $\beta$, et alors si $k>0$, leur
\ndef[$j$-composé de $i$-flèches!d'une
$\infty$-sesqui\-catégorie]{$k$\nbd-composé} \nnot{$\alpha\comp_k\beta$} est défini
comme étant leur $(k-1)$\nbd-composé dans la \oo-catégorie $\VHom_\C(x,y)$,
où $x$ et $y$ sont respectivement la $0$\nbd-source, nécessairement commune,
et le $0$\nbd-but, nécessairement commun, de $\alpha$ et $\beta$. Si $k=0$,
leur $0$\nbd-composé $\alpha\comp_0\beta$ n'est défini que si $i=1$ ou $j=1$.
Dans ce cas, soient $x$ la $0$\nbd-source de $\beta$, $y$ son $0$\nbd-but,
qui est alors égal à la $0$\nbd-source de $\alpha$, et $z$ le $0$\nbd-but de
$\alpha$. Si $i=1$, le \ndef[]{$0$\nbd-composé} $\alpha\comp_0\beta$ est
défini comme étant l'image de la cellule $\beta$ de $\VHom_\C(x,y)$ par le
\oo-foncteur
\[
\VHom_\C(x,\alpha):\VHom_C(x,y)\to\VHom_\C(x,z)
\]
et, si $j=1$, comme étant l'image de la cellule $\alpha$ de $\VHom_\C(y,z)$ par le \oo-foncteur
\[
\VHom_\C(\beta,z):\VHom_C(y,z)\to\VHom_\C(x,z).
\]
Si $x$ est un objet de $\C$, l'\ndef[unité d'une cellule!d'une
$\infty$-sesqui\-catégorie]{unité} de $x$ est l'unité \nnot[$1_x$,
$1_\alpha$]{$1_x$} de $x$ dans la catégorie sous-jacente à $\C$ (qui est une
$1$\nbd-flèche de $\C$) et si, pour $i>0$, $\alpha$ est une $i$\nbd-flèche
de $\C$ de $0$\nbd-source~$x$ et $0$\nbd-but~$y$, l'\ndef[]{unité} de
$\alpha$ est son unité $1_\alpha$ comme cellule de $\VHom_\C(x,y)$ (qui est
une $(i+1)$\nbd-flèche de $\C$).  Ces données satisfont aux mêmes axiomes
que les données correspondantes d'une \oo-catégorie, sauf la règle de
Godement pour la $0$\nbd-composition (autrement dit, pour $i,j>1$, si
$\alpha$ est une $i$\nbd-flèche de $1$\nbd-source $a_0$ et $1$\nbd-but $a_1$
et $\beta$ une $j$\nbd-flèche de $1$\nbd-source~$b_0$ et $1$\nbd-but $b_1$
et si $\alpha$ et $\beta$ sont $0$\nbd-composables
\[
\UseTwocells
\xymatrixcolsep{2.5pc}
\xymatrix{
\protect\phantom{\bullet}\rtwocell<5>^{b_0}_{b_1}{\,\beta}&
\protect\phantom{\bullet}\rtwocell<5>^{a_0}_{a_1}{\,\alpha}&
\protect\phantom{\bullet}\pbox{,}
}
\]
le composé $(\alpha\comp_0b_1)\comp_1(a_0\comp_0\beta)$ n'est pas en général
égal au composé $(a_1\comp_0\beta)\comp_1(\alpha\comp_0 b_0)$, ce qui
empêche de définir le $0$\nbd-composé de $\alpha$ et $\beta$). Ces données
soumises à ces axiomes fournissent une définition équivalente des
\oo-sesquicatégories.
\end{paragr}

\begin{paragr}\label{paragr:sesqui_tranches}
Soient $C$ une \oo-catégorie et $c$ un objet de $C$. Dans la
description explicite de la \oo-catégorie tranche $\cotr{C}{c}$ donnée dans
le chapitre~\ref{sec:desc_expl}, les seules $0$\nbd-compositions dans $C$
figurant dans les formules définissant les cellules de $\cotr{C}{c}$, ainsi
que leurs compositions, sont de la forme $\alpha\comp_0\beta$, où $\alpha$
ou $\beta$ est une $1$\nbd-flèche de $C$ (voir le
paragraphe~\ref{paragr:desc_tr_pol} et la proposition~\ref{prop:desc_tr}).
De plus, si $(x,f:c\to x)$ et \hbox{$(x',f':c\to x')$} sont deux objet de
$\cotr{C}{c}$ et $i>1$, toute $i$\nbd-flèche non triviale de $C$ qui est une
composante d'une cellule de la \oo-catégorie
$\Homi_{\cotr{C}{c}}((x,f),(x',f'))$ a comme $0$\nbd-source $c$ ou $x$ et
comme $0$-but~$x'$. Ainsi, aucune utilisation de la règle de Godement pour
la $0$\nbd-composition de $C$ n'intervient pour vérifier que les formules du
paragraphe~\ref{paragr:desc_tr_pol} et de la proposition~\ref{prop:desc_tr}
définissent une \oo-catégorie $\Homi_{\cotr{C}{c}}((x,f),(x',f'))$.

On en déduit que si $\C$ est une \oo-sesquicatégorie, $c$ un objet de $\C$
et \hbox{$f:c\to x$}, $f':c\to x'$ deux $1$\nbd-flèches de $\C$ de source
$c$, on peut définir une \oo-catégorie notée
$\VHom_{\cotr{\C}{c}}((x,f),(x',f'))$
\notindex{$\VHom_{\cotr{\C}{c}}$, $\Homi_{\cotrm{\C}{c}}$,
$\Homi_{\tr{\C}{c}}$, $\Homi_{\trm{\C}{c}}$}%
par les mêmes formules que dans le cas d'une \oo-catégorie~$C$.

En revanche, il
n'est pas possible en général d'assembler, pour $(x,f)$ et~$(x',f')$
variables, les \oo-catégories $\VHom_{\cotr{\C}{c}}((x,f),(x',f'))$ en une
\oo-sesquicatégorie.
En effet, étant donné une \oo-catégorie $C$ et un objet $c$ de $C$, on
observe que dans la vérification du fait que, si $\alpha$ et $\beta$ sont une
$2$\nbd-flèche et une $1$\nbd-flèche $0$\nbd-composables de~$\cotr{C}{c}$,
alors les formules de la proposition~\ref{prop:desc_tr} définissant le
composé $\alpha\comp_0\beta$ définissent bien une $2$\nbd-flèche de
$\cotr{C}{c}$,
on utilise la règle de Godement pour la $0$\nbd-composition de $C$.

De façon analogue, on peut définir une \oo-catégorie
\smash{$\Homi_{\cotrm{\C}{c}}((x,f),(x',f'))$} et, pour $g:x\to c$ et
$g':x'\to c$ deux $1$\nbd-flèches de $\C$ de but $c$, des \oo-catégories
$\Homi_{\tr{\C}{c}}((x,g),(x',g'))$ et
\smash{$\Homi_{\trm{\C}{c}}((x,g),(x',g'))$} mais, par contre,  pas de
\oo-sesqui\-catégories associées. Pour pouvoir le faire, on a besoin d'une
structure plus forte, celle de \oo-catégorie de Gray ou de \oo-catégorie de
Gray gauche (voir le paragraphe~\ref{paragr:Graycat} et la
conjecture~\ref{conject:conj_prelim}).
\end{paragr}

\begin{paragr}\label{paragr:cat_enrich}
Soit $\Catmod$ une catégorie monoïdale de produit tensoriel $\otimes$ et
d'objet unité $\CatmodI$. On rappelle qu'une \ndef[catégorie
enrichie]{catégorie enrichie dans $\Catmod$}, ou plus simplement
\ndef[$z$@$\Catmod$-catégorie]{$\Catmod$\nbd-catégorie}, $\Vcat{C}$ est la
donnée
\begin{itemize}
\item d'une classe \nnot{$\Ob(\Vcat{C})$} dont les éléments sont appelés les
\ndef[objet!d'une $\Catmod$-catégorie]{objets} de $\Vcat{C}$; 
\item pour tous $x,y$ objets de $\Vcat{C}$, d'un objet
\nnot{$\VHom_{\Vcat{C}}(x,y)$} de $\Catmod$, appelé \ndef[objet de
morphismes!d'une $\Catmod$-catégorie]{objet de morphismes de $x$ vers $y$};
\item pour tous $x,y,z$ objets de $\Vcat{C}$, d'un morphisme
\[ \Vcomp_{z,y,x}:\VHom_{\Vcat{C}}(y,z)\otimes\VHom_{\Vcat{C}}(x,y)\to\VHom_{\Vcat{C}}(x,z) \]
\notindex{$\Vcomp_{z,y,x} : \VHom_{\Vcat{C}}(y,z)\otimes\VHom_{\Vcat{C}}(x,y)\to\VHom_{\Vcat{C}}(x,z)$}%
de $\Catmod$, appelé \ndef[morphisme de composition d'une
$\Catmod$-catégorie]{morphisme de composition};
\item pour tout $x$ objet de $\Vcat{C}$, d'un morphisme
\nnot{$\VId_x:\CatmodI\to\VHom_{\Vcat{C}}(x,x)$} de $\Catmod$, appelé
\ndef[morphisme d'unité d'une $\Catmod$-catégorie]{morphisme d'unité},
\end{itemize}
ces données satisfaisant aux conditions usuelles d'associativité et d'unité,
autrement dit rendant commutatifs les diagrammes
\[
\xymatrixrowsep{2.5pc}
\xymatrixcolsep{3.5pc}
\xymatrix{
\VHom_{\Vcat{C}}(z,t)\otimes\VHom_{\Vcat{C}}(y,z)\otimes\VHom_{\Vcat{C}}(x,y)\ar[r]^-{\Vcomp_{t,z,y}\otimes1}\ar[d]_{1\otimes\Vcomp_{z,y,x}}
&\VHom_{\Vcat{C}}(y,t)\otimes\VHom_{\Vcat{C}}(x,y)\ar[d]^{\Vcomp_{t,y,x}}
\\
\VHom_{\Vcat{C}}(z,t)\otimes\VHom_{\Vcat{C}}(x,z)\ar[r]_-{\Vcomp_{t,z,x}}
&\VHom_{\Vcat{C}}(x,t) \pbox{,}
}
\]
\[
\xymatrix{
\VHom_{\Vcat{C}}(x,y)\otimes\CatmodI\ar[r]^-{\simeq}\ar[d]_{1\otimes\VId_x}
&\VHom_{\Vcat{C}}(x,y)
&\CatmodI\otimes\VHom_{\Vcat{C}}(x,y)\ar[l]_-{\simeq}\ar[d]^{\VId_y\otimes1}
\\
\VHom_{\Vcat{C}}(x,y)\otimes\VHom_{\Vcat{C}}(x,x)\ar[ru]_{\,\,\Vcomp_{y,x,x}}
&&\VHom_{\Vcat{C}}(y,y)\otimes\VHom_{\Vcat{C}}(x,y)\ar[lu]^{\Vcomp_{y,y,x}}
\pbox{,}
}
\]
pour $x,y,z,t$ objets de $\Vcat{C}$ (où on omet pour simplifier les
contraintes d'associativité et d'unité de $\Catmod$).

Une \ndef[$z$@$\Catmod$-catégorie gauche]{$\Catmod$\nbd-catégorie gauche}
est une catégorie enrichie dans la catégorie monoïdale \ndef[catégorie
monoïdale!transposée]{transposée} de $\Catmod$, c'est-à-dire dans la
catégorie monoïdale ayant même catégorie sous-jacente et même unité que
$\Catmod$ mais comme produit tensoriel le foncteur
\[
  \begin{split}
    \Catmod\times\Catmod & \to\Catmod \\
     (X,Y) & \mapsto Y\otimes X \pbox{,}
  \end{split}
\]
obtenu à partir de $\otimes$ en composant avec l'isomorphisme de symétrie du
produit cartésien. Les notions de $\Catmod$\nbd-catégorie et
$\Catmod$\nbd-catégorie gauche sont bien distinctes, sauf dans le cas où la
catégorie monoïdale $\Catmod$ est symétrique, auquel cas ces deux notions
sont essentiellement équivalentes, la symétrie définissant une bijection
entre $\Catmod$\nbd-catégories et $\Catmod$\nbd-catégories gauches.
\end{paragr}

\begin{exem}\label{exem:bourbaki_enrich}
Une catégorie enrichie dans la catégorie monoïdale des ensembles, avec comme
produit tensoriel le produit cartésien, n'est autre qu'une catégorie
ordinaire.
\end{exem}

\begin{paragr}\label{paragr:imdir_enrich}
Soient $\Catmod$ et $\Catmod'$ deux catégories monoïdales de produits
tensoriels et unités $\otimes$, $\CatmodI$ et $\otimes'$, $\CatmodI'$
respectivement, et $F:\Catmod\to\Catmod'$ un foncteur monoïdal lax. À toute
$\Catmod$\nbd-catégorie $\Vcat{C}$, le foncteur $F$ associe une
$\Catmod'$\nbd-catégorie \nnot[$F_\ast(\Vcat{C})$]$\Vcat{C}'=F_*(\Vcat{C})$,
appelée \ndef[image directe d'une $\Catmod$-catégorie]{image directe de
$\Vcat{C}$ par $F$}, définie comme suit :
\begin{itemize}
\item les objets de $\Vcat{C}'$ sont les mêmes que ceux de $\Vcat{C}$;
\item pour $x,y$ objets de $\Vcat{C}'$, l'objet de morphismes de $x$ vers $y$ est
$F(\VHom_{\Vcat{C}}(x,y))$;
\item pour $x,y,z$ objets de $\Vcat{C}'$, le morphisme de composition est le composé
\[
F(\VHom_{\Vcat{C}}(y,z))\otimes'F(\VHom_{\Vcat{C}}(x,y))\to
F(\VHom_{\Vcat{C}}(y,z)\otimes\VHom_{\Vcat{C}}(x,y))\to
F(\VHom_{\Vcat{C}}(x,z)),
\]
la flèche de gauche venant de la contrainte de composition du foncteur
monoïdal lax $F$ et la flèche de droite étant $F(\Vcomp_{z,y,x})$; \item
pour $x$ objet de $\Vcat{C}'$, le morphisme d'unité est le composé
\[
\CatmodI'\to F(\CatmodI)\to F(\VHom_{\Vcat{C}}(x,x)),
\]
la flèche de gauche venant de la contrainte d'unité du foncteur monoïdal lax
$F$ et la flèche de droite étant $F(\VId_x)$.
\end{itemize}
\end{paragr}

\begin{exem}\label{exem:cat_sousj}
Soit $\Catmod$ une catégorie monoïdale de produit tensoriel $\otimes$ et
d'objet unité $\CatmodI$. Le foncteur
\[
\Hom_{\Catmod}(\CatmodI,\var):\Catmod\to\Ens
\]
s'enrichit en un foncteur monoïdal lax à valeurs dans la catégorie des
ensembles munie du produit cartésien en définissant pour $X,Y$ objets de
$\Catmod$ la contrainte de composition
\[
\Hom_{\Catmod}(\CatmodI,X)\times\Hom_{\Catmod}(\CatmodI,Y)\to\Hom_{\Catmod}(\CatmodI,X\otimes Y)
\]
par
\[
(x:\CatmodI\to X,y:\CatmodI\to Y)\mapsto x\otimes y:\CatmodI\simeq \CatmodI\otimes\CatmodI\to X\otimes Y
\]
et la contrainte d'unité
\[
\{*\}\to\Hom_{\Catmod}(\CatmodI,I)\qquad\hbox{par}\qquad *\mapsto 1_\CatmodI.
\]
Pour toute $\Catmod$-catégorie $\Vcat{C}$, l'image directe de $\Vcat{C}$ par
ce foncteur monoïdal lax est, en vertu de
l'exemple~\ref{exem:bourbaki_enrich}, une catégorie ordinaire appelée
\ndef[catégorie sous-jacente!à une $\Catmod$-catégorie]{catégorie
sous-jacente} à la $\Catmod$\nbd-catégorie $\Vcat{C}$.

La \ndef[catégorie sous-jacente!à une $\Catmod$-catégorie gauche]{catégorie
sous-jacente} à une $\Catmod$\nbd-catégorie gauche $\Vcat{C}'$ est la
catégorie sous-jacente à $\Vcat{C}'$, vue comme catégorie
enrichie dans la catégorie monoïdale transposée de $\Catmod$.
\end{exem}

\begin{paragr}\label{paragr:sesquicat_sousj}
Soient $\Catmod$ une catégorie monoïdale de produit tensoriel $\otimes$ et
d'objet unité $\CatmodI$, et $\Vcat{C}$ une $\Catmod$-catégorie de catégorie
sous-jacente $\C$. On peut munir la catégorie $\C$ d'une structure canonique
de $\Catmod$\nbd-sesquicatégorie, appelée
\ndef[$z$@$\Catmod$-sesquicatégorie!sous-jacente!à une
$\Catmod$-catégorie]{$\Catmod$\nbd-sesquicatégorie sous-jacente} à la
$\Catmod$\nbd-catégorie $\Vcat{C}$, comme suit. Le foncteur
$\VHom_\C:\C^\op\times\C\to\Catmod$ associe à un objet $(x,y)$ de
$\C^\op\times\C$ l'objet $\VHom_{\Vcat{C}}(x,y)$ de $\Catmod$. Si $f:x'\to
x$ et $g:y\to y'$ sont des morphismes de $\C$, la flèche $\VHom_{\C}(f,g)$
de $\Catmod$ est définie par le diagramme commutatif
\[
\xymatrixcolsep{3pc}
\xymatrix{
\VHom_{\Vcat{C}}(x,y)\ar[r]^-{\VHom_{\C}(f,g)}
&\VHom_{\Vcat{C}}(x',y')\\
\CatmodI\otimes\VHom_{\Vcat{C}}(x,y)\otimes\CatmodI\ar[r]_-{g\otimes1\otimes f}\ar[u]^\wr
&\VHom_{\Vcat{C}}(y,y')\otimes\VHom_{\Vcat{C}}(x,y)\otimes\VHom_{\Vcat{C}}(x',x)\ar[u]_{\Vcomp_{y',y,x'}(1\otimes\Vcomp_{y,x,x'})}
}
\]
(où la flèche verticale de gauche vient de la contrainte d'unité de
$\Catmod$ et où on néglige les contraintes d'associativité). Il est alors
immédiat que le triangle
\[
\xymatrix{
\C^\op\times\C\ar[rr]^{\VHom_\C}\ar[rd]_{\Hom_\C}
&&\Catmod\ar[ld]^{\Hom_{\Catmod}(\CatmodI,\var)}
\\
&\Ens
}
\]
du paragraphe~\ref{paragr:def_sesqui} est strictement commutatif.

La \ndef[$z$@$\Catmod$-sesquicatégorie!sous-jacente!à une $\Catmod$-catégorie
gauche]{$\Catmod$\nbd-sesquicatégorie sous-jacente} à une
$\Catmod$\nbd-catégorie gauche $\Vcat{C}'$ est la
$\Catmod$\nbd-sesqui\-catégorie sous-jacente à $\Vcat{C}'$, vue comme
catégorie enrichie dans la catégorie monoïdale transposée de $\Catmod$ (on
rappelle que le produit tensoriel de $\Catmod$ n'intervient pas dans la
définition d'une $\Catmod$\nbd-sesquicatégorie).
\end{paragr}

\begin{exem}\label{exem:standard_enrich}
Soit $\Catmod$ une catégorie monoïdale fermée à droite (voir le
paragraphe~\ref{paragr:biferme}) de produit tensoriel $\otimes$ et d'objet
unité $\CatmodI$, de sorte qu'il existe un foncteur $\Homd_\Catmod :
\Catmod^\op \times \Catmod \to\Catmod$ et une bijection
\[
\Hom_\Catmod(X \otimes Y, Z) \simeq \Hom_\Catmod(X, \Homd_\Catmod(Y, Z)),
\]
naturelle en $X$, $Y$ et $Z$ dans $\Catmod$. On définit une
$\Catmod$\nbd-catégorie \nnot[$\VCatd$, $\VCatg$]{$\VCatd$} ayant mêmes
objets que $\Catmod$ et telle que $\VHom_\VCatd(X,Y)=\Homd_\Catmod(X,Y)$
comme suit. La bijection ci-dessus, appliquée à $X=\Homd_\Catmod(Y, Z)$,
fournit un morphisme de $\Catmod$, dit d'\ndef[morphisme d'évaluation dans
une catégorie monoïdale fermée à droite]{évaluation},
\[
\ev_{Z,Y}:\Homd_\Catmod(Y, Z)\otimes Y\to Z
\]
correspondant par cette bijection à $1_{\Homd_\Catmod(Y, Z)}$. Le composé
$\ev_{Z,Y}(1\otimes\ev_{Y,X})$
\[
\Homd_\Catmod(Y, Z)\otimes\Homd_\Catmod(X,Y)\otimes X\to\Homd_\Catmod(Y,Z)\otimes Y\to Z
\]
définit par adjonction le morphisme de composition. De même, l'isomorphisme
\hbox{$\CatmodI\otimes X\simeq X$} définit par adjonction le morphisme
d'unité $\CatmodI\to\Homd_\Catmod(X,X)$. Une vérification simple mais
fastidieuse prouve les propriétés d'associativité et d'unité (voir par
exemple \cite{Kelly}). La catégorie sous-jacente à la
$\Catmod$\nbd-catégorie $\VCatd$ n'est autre que la catégorie sous-jacente à
la catégorie monoïdale $\Catmod$ et la $\Catmod$\nbd-sesquicatégorie
sous-jacente à $\VCatd$ est définie par le foncteur $\Homd_\Catmod :
\Catmod^\op \times \Catmod \to\Catmod$.

Si la catégorie monoïdale $\Catmod$, au lieu d'être fermée à droite est
fermée à gauche, de sorte qu'il existe un foncteur $\Homg_\Catmod :
\Catmod^\op \times \Catmod \to \Catmod$ et une bijection
\[
\Hom_\Catmod(X \otimes Y, Z) \simeq \Hom_\Catmod(Y, \Homg_\Catmod(X, Z)),
\]
naturelle en $X$, $Y$ et $Z$ dans $\Catmod$, il faut se garder de croire
qu'on puisse définir en général une $\Catmod$\nbd-catégorie ayant mêmes
objets que $\Catmod$ et dont l'objet de morphismes soit donné par
$\Homg_\Catmod$. En revanche, en appliquant ce qui précède à la catégorie
monoïdale~$\Catmod'$ transposée de $\Catmod$, on déduit l'existence d'une
$\Catmod$\nbd-catégorie \emph{gauche} $\VCatg$ ayant mêmes objets que
$\Catmod$ et dont l'objet de morphismes est défini par $\Homg_\Catmod$. En
effet, la catégorie monoïdale $\Catmod'$ est alors fermée à droite et, pour
$X,Y$ objets de $\Catmod$, on~a
\[
\Homg_\Catmod(X,Y)=\Homd_{\Catmod'}(X,Y).
\]
La catégorie sous-jacente à cette $\Catmod$\nbd-catégorie gauche est la
catégorie sous-jacente à la catégorie monoïdale $\Catmod$ et sa
$\Catmod$\nbd-sesquicatégorie sous-jacente est définie par le foncteur
$\Homg_\Catmod : \Catmod^\op \times \Catmod \to\Catmod$.
\end{exem}

\begin{paragr}\label{paragr:fonct_enr}
Soient $\Catmod$ une catégorie monoïdale de produit tensoriel $\otimes$ et
d'objet unité~$\CatmodI$, et $\Vcat{C}$ et $\Vcat{C'}$ deux catégories
enrichies dans $\Catmod$ (voir le paragraphe~\ref{paragr:cat_enrich}). On
rappelle qu'un \ndef{foncteur enrichi}, ou plus simplement
\ndef[$z$@$\Catmod$-foncteur]{$\Catmod$\nbd-foncteur}, $\Vcat{F}$ de
$\Vcat{C}$ vers $\Vcat{C'}$ est
la donnée
\begin{itemize}
  \item d'une application
    \nnot{$\Vcat{F}_0:\Ob(\Vcat{C})\to\Ob(\Vcat{C'})$};
  \item pour tous $x,y$ objets de $\Vcat{C}$, d'un morphisme de $\Catmod$
%
% INDEXCHECK
\notindex{$\Vcat{F}_{y, x} : \VHom_{\Vcat{C}}(x,y)\to\VHom_{\Vcat{C'}}(\Vcat{F}_0(x),\Vcat{F}_0(y))$}%
\[
\Vcat{F}_{y,x}:\VHom_{\Vcat{C}}(x,y)\to\VHom_{\Vcat{C'}}(\Vcat{F}_0(x),\Vcat{F}_0(y)),
\]
\end{itemize}
ces données satisfaisant aux compatibilités usuelles aux compositions et aux
unités, autrement dit rendant commutatifs les diagrammes
\[
\xymatrixcolsep{5pc}
\xymatrixrowsep{3pc}
\xymatrix{
\VHom_{\Vcat{C}}(y,z)\otimes\VHom_{\Vcat{C}}(x,y)\ar[r]^-{\Vcomp_{z,y,x}}\ar[d]_{\Vcat{F}_{z,y}\otimes\,\Vcat{F}_{y,x}}
&\VHom_{\Vcat{C}}(x,z)\ar[d]^{\Vcat{F}_{z,x}}
\\
\VHom_{\Vcat{C'}}(\Vcat{F}_0(y),\Vcat{F}_0(z))\otimes\VHom_{\Vcat{C'}}(\Vcat{F}_0(x),\Vcat{F}_0(y))\ar[r]^-{\Vcomp_{\Vcat{F}_0(z),\Vcat{F}_0(y),\Vcat{F}_0(x)}}
&\VHom_{\Vcat{C'}}(\Vcat{F}_0(x),\Vcat{F}_0(z)) \pbox{,}
}
\]

\[
\xymatrixcolsep{3pc}
\xymatrixrowsep{.8pc}
\xymatrix{
&&\VHom_{\Vcat{C}}(x,x)\ar[dd]^{\Vcat{F}_{x,x}}
\\
&\CatmodI\ar[ru]^(.4){\VId_x}\ar[rd]_(.35){\VId_{\Vcat{F}_0(x)}}
\\
&&\VHom_{\Vcat{C'}}(\Vcat{F}_0(x),\Vcat{F}_0(x)) \pbox{,}
&
}
\]
pour $x,y,z$ objets de $\Vcat{C}$.

Si $\Vcat{C}$ et $\Vcat{C'}$ sont deux $\Catmod$\nbd-catégories gauches, un
\ndef[$z$@$\Catmod$-foncteur gauche]{$\Catmod$\nbd-foncteur gauche} de
$\Vcat{C}$ vers $\Vcat{C'}$ est un
foncteur enrichi dans la catégorie monoïdale transposée de $\Catmod$ de
$\Vcat{C}$ vers $\Vcat{C'}$.
\end{paragr}

\begin{paragr}\label{paragr:fonct_enr_sousj}
Soient $\Catmod$ une catégorie monoïdale d'objet unité~$\CatmodI$,
$\Vcat{C}$~et~$\Vcat{C'}$ deux catégories enrichies dans $\Catmod$ et
$\Vcat{F}$ un foncteur enrichi de $\Vcat{C}$ vers $\Vcat{C'}$. Le foncteur
enrichi $\Vcat{F}$ induit un foncteur ordinaire $F$ de la catégorie $\C$
sous-jacente à $\Vcat{C}$ vers la catégorie $\C'$ sous-jacente à
$\Vcat{C'}$, appelé foncteur sous-jacent à $\Vcat{F}$, défini comme suit.
Sur les objets, le foncteur $F$ est défini par l'application $\Vcat{F}_0$.
Si \hbox{$f:x\to y$} est un morphisme de $\C$, autrement dit une flèche
$f:\CatmodI\to\VHom_{\Vcat{C}}(x,y)$ de $\Catmod$, le morphisme~$F(f)$ de
$\C'$ est défini par le composé
\[
\xymatrix{
\CatmodI\ar[r]^-f
&\VHom_{\Vcat{C}}(x,y)\ar[r]^-{\Vcat{F}_{y,x}}
&\VHom_{\Vcat{C'}}(\Vcat{F}_0(x),\Vcat{F}_0(y)).
}
\]
Le foncteur sous-jacent à un $\Catmod$\nbd-foncteur gauche est son foncteur
sous-jacent comme foncteur enrichi dans la catégorie monoïdale transposée de
$\Catmod$.
\end{paragr}

\begin{paragr}\label{paragr:sous-Vcat}
Soient $\Catmod$ une catégorie monoïdale et $\Vcat{C}$ une
$\Catmod$\nbd-catégorie. Pour toute partie de $\Ob(\Vcat{C})$, on définit
une $\Catmod$\nbd-catégorie, appelée \ndef[sous-$\Catmod$-catégorie
pleine]{sous-$\Catmod$\nbd-catégorie
pleine de $\Vcat{C}$ définie par cette partie}, dont la classe des objets
est cette partie et dont les objets de morphismes, les morphismes de
composition et les morphismes d'unité sont induits par ceux de $\Vcat{C}$.
L'inclusion de cette partie dans la classe des objets de $\Vcat{C}$
s'enrichit en un $\Catmod$\nbd-foncteur, appelé \ndef[$z$@$\Catmod$-foncteur!
d'inclusion]{$\Catmod$\nbd-foncteur d'inclusion}. De même, si $\Vcat{C}$ est
une $\Catmod$\nbd-catégorie gauche, toute partie de $\Ob(\Vcat{C})$ définit
une \ndef[sous-$\Catmod$-catégorie gauche
pleine]{sous-$\Catmod$\nbd-catégorie gauche pleine de $\Vcat{C}$} et un
\ndef[$z$@$\Catmod$-foncteur gauche!d'inclusion]{$\Catmod$\nbd-foncteur
gauche d'inclusion}.
\end{paragr}

\begin{paragr}\label{paragr:transp_Vcat}
Soit $\Catmod$ une catégorie monoïdale. À une $\Catmod$\nbd-catégorie
$\Vcat{C}$, on associe une $\Catmod$\nbd-catégorie gauche
\nnot{$\transp\Vcat{C}$},
\ndef[transposée!d'une $\Catmod$-catégorie]{transposée} de $\Vcat{C}$, ayant
mêmes objets que $\Vcat{C}$ et, pour $x,y$ objets de $\Vcat{C}$, comme objet
de morphismes de $x$ vers~$y$
\[
\VHom_{\transp\Vcat{C}}(x,y)=\VHom_{\Vcat{C}}(y,x).
\]
De même, à une $\Catmod$\nbd-catégorie gauche $\Vcat{C}$, on associe de
façon analogue une $\Catmod$\nbd-catégorie \ndef[transposée!d'une
$\Catmod$-catégorie gauche]{transposée}
$\transp\Vcat{C}$. Pour $\Vcat{C}$ une $\Catmod$\nbd-catégorie ou une
$\Catmod$\nbd-catégorie gauche, la catégorie sous-jacente à
$\transp\Vcat{C}$ est la catégorie opposée de la catégorie sous-jacente à
$\Vcat{C}$ et on a $\transp\transp\Vcat{C}=\Vcat{C}$.

Si $\Vcat{C}$ et $\Vcat{C}'$ sont deux $\Catmod$\nbd-catégories (resp. deux
$\Catmod$\nbd-catégories gauches) et $\Vcat{F}:\Vcat{C}\to\Vcat{C}'$ est un
$\Catmod$\nbd-foncteur (resp. un $\Catmod$\nbd-foncteur gauche), on définit
le $\Catmod$\nbd-foncteur gauche (resp. le $\Catmod$\nbd-foncteur)
\ndef[transposé!d'un $\Catmod$-foncteur]{transposé}
\termindex{transposé!d'un $\Catmod$-foncteur gauche}%
$\transp\Vcat{F}:\transp\Vcat{C}\to\transp\Vcat{C}'$ de la façon évidente.
\end{paragr}

\begin{exem}\label{exem:fonct_enr}
Soit $\Catmod$ une catégorie monoïdale bifermée (voir le
paragraphe~\ref{paragr:biferme}) de produit tensoriel $\otimes$,
de sorte qu'il existe des foncteurs
\[
\Homd_\Catmod,\Homg_\Catmod : \Catmod^\op \times \Catmod \to\Catmod
\]
et des bijections
\[
\Hom_\Catmod(Y, \Homg_\Catmod(X, Z))\simeq\Hom_\Catmod(X \otimes Y, Z) \simeq \Hom_\Catmod(X, \Homd_\Catmod(Y, Z)),
\]
naturelles en $X$, $Y$ et $Z$ dans $\Catmod$, et soient $\VCatd$ la
$\Catmod$\nbd-catégorie et $\VCatg$ la $\Catmod$\nbd-catégorie gauche
associées (voir l'exemple~\ref{exem:standard_enrich}). Pour $X$ un objet de
$\Catmod$, les affirmations suivantes sont purement formelles
(voir~\cite{Kelly}).
\begin{enumerate}
\item Il existe un $\Catmod$\nbd-foncteur canonique $\VCatd\to\VCatd$ dont le foncteur sous-jacent est le foncteur $\var\otimes X:\Catmod\to\Catmod$.
\item Il existe un $\Catmod$\nbd-foncteur gauche canonique $\VCatg\to\VCatg$ dont le foncteur sous-jacent est le foncteur $X\otimes\var:\Catmod\to\Catmod$.
\item Il existe un $\Catmod$\nbd-foncteur canonique $\VCatd\to\VCatd$ dont
le foncteur sous-jacent est le foncteur $\Homd_\Catmod(X,\var):\Catmod\to\Catmod$.
\item Il existe un $\Catmod$\nbd-foncteur gauche canonique
$\transp\VCatd\to\VCatg$ dont le foncteur sous-jacent est le foncteur
$\Homd_\Catmod(\var,X):\Catmod^\op\to\Catmod$.
\item Il existe un $\Catmod$\nbd-foncteur gauche canonique $\VCatg\to\VCatg$
dont le foncteur sous-jacent est le foncteur $\Homg_\Catmod(X,\var):\Catmod\to\Catmod$.
\item Il existe un $\Catmod$\nbd-foncteur canonique $\transp\VCatg\to\VCatd$
dont le foncteur sous-jacent est le foncteur $\Homg_\Catmod(\var,X):\Catmod^\op\to\Catmod$.
\end{enumerate}
\end{exem}

\begin{paragr}\label{paragr:Graycat}
Une \ndef[$\infty$-catégorie de Gray]{\oo-catégorie de Gray} est une
catégorie enrichie dans la
catégorie monoïdale des \oo-catégories avec comme produit tensoriel le
produit tensoriel de Gray (voir le paragraphe~\ref{paragr:def_tens}). Une
\ndef[$\infty$-catégorie de Gray gauche]{\oo-catégorie de Gray gauche} est
la variante gauche de cette notion,
autrement dit une catégorie enrichie dans la catégorie monoïdale des
\oo-catégories avec comme produit tensoriel le foncteur
\[
(X,Y)\mapsto Y\otimes X \simeq (X^\opp\otimes Y^\opp)^\opp\simeq (X^\co\otimes Y^\co)^\co
\]
(voir la proposition~\ref{prop:dual_tens}). Le produit tensoriel de Gray dans
$\ooCat$ n'étant pas symétrique, les notions de \oo-catégorie de Gray et
\oo-catégorie de Gray gauche ne sont pas équivalentes. Si $\Vcat{C}$ est une
\oo-catégorie de Gray ou une \oo-catégorie de Gray gauche et $x,y$ deux
objets de $\Vcat{C}$, alors par définition $\VHom_{\Vcat{C}}(x,y)$ est une
\oo-catégorie. On vérifie immédiatement que les objets de
$\VHom_{\Vcat{C}}(x,y)$ sont les morphismes de source~$x$ et but~$y$ dans la
catégorie sous-jacente à $\Vcat{C}$ (voir l'exemple~\ref{exem:cat_sousj}). De
plus, si $\C$ désigne la \oo-sesquicatégorie sous-jacente à $\Vcat{C}$, on
a $\VHom_\C(x,y)=\VHom_{\Vcat{C}}(x,y)$ (voir les
paragraphes~\ref{paragr:def_oosesqui} et~\ref{paragr:sesquicat_sousj}).
Pour $i\geq1$, une \ndef[$i$-flèche!d'une $\infty$-catégorie de
Gray]{$i$\nbd-flèche}
\termindex{$i$-flèche!d'une $\infty$-catégorie de Gray gauche}%
de $\Vcat{C}$ sera par définition une $i$\nbd-flèche
de sa \oo-sesquicatégorie sous-jacente $\C$ et de même pour les sources,
buts, compositions et unités.
\end{paragr}

\begin{exem}\label{exem:bourbaki_Gray}
  % TOCHECK
  \abovedisplayskip=3.0pt
  \belowdisplayskip=3.0pt
Le foncteur identité de $\ooCat$ s'enrichit en un foncteur monoïdal lax de
source $\ooCat$ munie du produit tensoriel défini par le produit cartésien
et de but $\ooCat$ munie du produit tensoriel de Gray. Pour $X,Y$ deux
\oo-catégories, la contrainte de composition est
\[
(1_X\otimes p^{}_Y,p^{}_X\otimes1_Y):X\otimes Y\to X\times Y,
\]
où pour $Z$ une \oo-catégorie, $p^{}_Z$ désigne l'unique flèche de $Z$ vers
la \oo-catégorie finale, et la contrainte d'unité est l'identité (en vertu
du paragraphe~\ref{paragr:def_tens}, l'unité du produit tensoriel de Gray
est la catégorie finale). Or, une catégorie enrichie dans la catégorie
monoïdale des \oo-catégories avec comme produit tensoriel le produit
cartésien est simplement une \oo-catégorie (non nécessairement petite mais
localement petite). Ainsi, à une telle \oo-catégorie $\C$, on associe une
\oo-catégorie de Gray $\Vcat{C}$, image directe de $\C$ par ce foncteur
monoïdal lax, satisfaisant à l'égalité
\[
\VHom_{\Vcat{C}}(x,y)=\Homi_\C(x,y).
\]
\end{exem}

\begin{exem}\label{exem:standard_Gray}
On rappelle que la catégorie monoïdale $\ooCat$, munie du produit tensoriel
de Gray, est bifermée (voir le théorème~\ref{thm:produit_tens}), de sorte que
dans les notations du paragraphe~\ref{paragr:def_HomOpLax}, on a des
bijections
\[
\Hom_{\ooCat}(B,\HomLax(A,C))\simeq\Hom_{\ooCat}(A\otimes B,C)\simeq
\Hom_{\ooCat}(A,\HomOpLax(B,C))
\]
naturelles en les \oo-catégories $A,B,C$. En vertu de
l'exemple~\ref{exem:standard_enrich}, on a donc une \oo-catégorie de Gray
\nnot[$\ooCatGr$, $\ooCatGrg$]{$\ooCatGr$}, appelée \ndef[$\infty$-catégorie
de Gray!des $\infty$-catégories]{\oo-catégorie de Gray des \oo-catégories},
dont les objets sont les \oo-catégories et telle que pour tout couple $A,B$
de \oo-catégories, l'objet de morphismes de $A$ vers $B$ soit
$\HomOpLax(A,B)$.  La catégorie sous-jacente à $\ooCatGr$ n'est autre que
$\ooCat$. En vertu de la remarque~\ref{rem:not_HomOpLax}, les
$2$\nbd-flèches de $\ooCatGr$ sont les transformations oplax. Plus
généralement, pour $i\geq0$, les $(i+1)$\nbd-flèches de $\ooCatGr$ sont
appelées des \ndef[$i$-transformation!oplax]{$i$\nbd-transformations oplax} (une
$0$\nbd-transformation oplax est donc un \oo-foncteur et une
$1$\nbd-transformation oplax une transformation oplax ordinaire). De même,
on a une \oo-catégorie de Gray gauche~$\ooCatGrg$, appelée
\ndef[$\infty$-catégorie de Gray gauche!des
$\infty$-caté\-gories]{\oo-catégorie de Gray
gauche des \oo-catégories}, dont les objets sont les
\oo-catégories et telle que pour tout couple $A,B$ de \oo-catégories,
l'objet de morphismes de $A$ vers $B$ soit $\HomLax(A,B)$. La catégorie
sous-jacente à~$\ooCatGrg$ est $\ooCat$, et les $2$\nbd-flèches de
$\ooCatGrg$ sont les transformations lax (voir la
remarque~\ref{rem:not_HomOpLax}). Plus généralement, pour $i\geq0$, les
$(i+1)$\nbd-flèches de $\ooCatGr$ sont appelées des
\ndef[$i$-transformation!lax]{$i$\nbd-transformations lax}.
\end{exem}

\begin{paragr}\label{paragr:def_fonct_Gr}
Soient $\Vcat{C}$ et $\Vcat{C'}$ deux \oo-catégories de Gray. Un
\ndef[$\infty$-foncteur de Gray]{\oo-foncteur de Gray} de $\Vcat{C}$ vers $\Vcat{C'}$ est un foncteur
enrichi de $\Vcat{C}$ vers $\Vcat{C'}$ (voir le
paragraphe~\ref{paragr:fonct_enr}). De même, si~$\Vcat{C}$ et $\Vcat{C'}$
sont deux \oo-catégories de Gray gauches, un \ndef[$\infty$-foncteur de Gray
gauche]{\oo-foncteur de Gray gauche} de $\Vcat{C}$ vers $\Vcat{C'}$ est un
foncteur enrichi de $\Vcat{C}$ vers $\Vcat{C'}$.
\end{paragr}

\begin{paragr}\label{paragr:sous-GrCat}
En vertu du paragraphe~\ref{paragr:sous-Vcat}, si $\Vcat{C}$ est une
\oo-catégorie de Gray, toute partie de la classe des objets de $\Vcat{C}$
définit une \ndef[sous-$\infty$-catégorie de Gray pleine]{sous-\oo-catégorie
de Gray pleine} de $\Vcat{C}$ et un \ndef[$\infty$-foncteur de
Gray!d'inclusion]{\oo-foncteur de Gray d'inclusion}.
De même, si $\Vcat{C}$ est une \oo-catégorie de Gray gauche, toute partie de
la classe des objets de $\Vcat{C}$ définit une \ndef[sous-$\infty$-catégorie
de Gray gauche pleine]{sous-\oo-catégorie de Gray gauche pleine} de
$\Vcat{C}$ et un \ndef[$\infty$-foncteur de Gray
gauche!d'inclusion]{\oo-foncteur de Gray gauche d'inclusion}.
\end{paragr}

\begin{paragr}\label{paragr:dual_Gray}
Soit $\Vcat{C}$ une \oo-catégorie de Gray. On définit une \oo-catégorie de
Gray \nnot[$\Vcat{C}^\opp$, $\Vcat{C}^\co$,
$\Vcat{C}^\op$]{$\Vcat{C}^\opp$}, ayant mêmes objets que $\Vcat{C}$, en
posant
\[
\VHom_{\Vcat{C}^\opp}(x,y)=\VHom_{\Vcat{C}}(y,x)^\co,
\]
les morphismes de composition et d'unité étant définis de la
manière évidente (en utilisant la proposition~\ref{prop:dual_tens}),
et deux \oo-catégories de Gray gauches $\Vcat{C}^\co$ et $\Vcat{C}^\op$,
ayant également mêmes objets que $\Vcat{C}$, en posant
\[\begin{aligned}
&\VHom_{\Vcat{C}^\co}(x,y)=\VHom_{\Vcat{C}}(x,y)^\opp,\cr
\noalign{\vskip 3pt}
&\VHom_{\Vcat{C}^\op}(x,y)=\VHom_{\Vcat{C}}(y,x)^\op
\end{aligned}\]
(voir la proposition~\ref{prop:dual_tens}),
les morphismes de composition et d'unité étant toujours définis de la
manière évidente. En particulier, à une \oo-catégorie
de Gray $\Vcat{C}$, on en associe \emph{trois} autres $\Vcat{C}^\opp$,
$\transp\Vcat{C}^\co$ et $\transp\Vcat{C}^\op$ (voir le
paragraphe~\ref{paragr:transp_Vcat}). On définit de façon analogue, pour
$\Vcat{C}$ une \oo-catégorie de Gray gauche, une \oo-catégorie de Gray
gauche~$\Vcat{C}^\opp$ et deux \oo-catégories de Gray $\Vcat{C}^\co$ et
$\Vcat{C}^\op$.

Pour $\Vcat{C}$ une \oo-catégorie de Gray ou une \oo-catégorie de Gray
gauche, de catégorie sous-jacente $\C$, la catégorie sous-jacente à
$\Vcat{C}^\opp$ et à $\Vcat{C}^\op$ est $\C^\op$, la catégorie sous-jacente
à $\C^\co$ est $\C$ et on a des égalités
\[\begin{aligned}
&\Vcat{C}^{\opp\,\opp}=\Vcat{C}^{\co\,\co}=\Vcat{C}^{\op\,\op}=\Vcat{C}\,,\cr
&\Vcat{C}^{\opp\,\co}=\Vcat{C}^{\co\,\opp}=\Vcat{C}^{\op},\cr
&\Vcat{C}^{\opp\,\op}=\Vcat{C}^{\op\,\opp}=\Vcat{C}^{\co},\cr
&\Vcat{C}^{\co\,\op}=\Vcat{C}^{\op\,\co}=\Vcat{C}^{\opp}.
\end{aligned}\]

Si $\Vcat{C}$ et $\Vcat{C}'$ sont deux \oo-catégories de Gray et
$\Vcat{F}:\Vcat{C}\to\Vcat{C}'$ est un \oo-foncteur de Gray, on définit de
la façon évidente un \oo-foncteur de Gray
$\Vcat{F}^\opp:\Vcat{C}^\opp\to\Vcat{C}'^\opp$ et des \oo-foncteurs de Gray
gauches $\Vcat{F}^\co:\Vcat{C}^\co\to\Vcat{C}'^\co$ et
$\Vcat{F}^\op:\Vcat{C}^\op\to\Vcat{C}'^\op$. De même, si~$\Vcat{C}$
et~$\Vcat{C}'$ sont deux \oo-catégories de Gray gauches et
$\Vcat{F}:\Vcat{C}\to\Vcat{C}'$ est un \oo-foncteur de Gray gauche, on
définit un \oo-foncteur de Gray gauche
$\Vcat{F}^\opp:\Vcat{C}^\opp\to\Vcat{C}'^\opp$ et des \oo-foncteurs de Gray
$\Vcat{F}^\co:\Vcat{C}^\co\to\Vcat{C}'^\co$ et
$\Vcat{F}^\op:\Vcat{C}^\op\to\Vcat{C}'^\op$.
\end{paragr}

\begin{rem}\label{rem:dual_Gray}
Les isomorphismes naturels
\[
    \begin{split}
      \HomOpLax(A, B)^\opp & \simeq \HomLax(A^\opp, B^\opp), \\
      \HomOpLax(A, B)^\co & \simeq \HomLax(A^\co, B^\co), \\
      \HomOpLax(A, B)^\op & \simeq \HomOpLax(A^\op, B^\op),
    \end{split}
\]
pour $A,B$ des \oo-catégories (voir la proposition~\ref{prop:dual_HomLax}),
induisent des isomorphismes de \oo-catégories de Gray
\[
\begin{aligned}
&\xymatrix{
(\ooCatGrg)^\co\ar[r]^-{\mathbf{op}}_-\sim&\ooCatGr
},
\cr
&\xymatrix{
(\ooCatGr)^\opp\ar[r]^-{\mathbf{co}}_-\sim&\transp(\ooCatGrg)
},
\cr
&\xymatrix{
(\ooCatGrg)^\op\ar[r]^-{\mathbf{o}}_-\sim&\transp(\ooCatGrg)
}
\end{aligned}
\]
et des isomorphismes de \oo-catégories de Gray gauches
\[
\begin{aligned}
&\xymatrix{
(\ooCatGr)^\co\ar[r]^-{\mathbf{op}}_-\sim&\ooCatGrg
},
\cr
&\xymatrix{
(\ooCatGrg)^\opp\ar[r]^-{\mathbf{co}}_-\sim&\transp(\ooCatGr)
},
\cr
&\xymatrix{
(\ooCatGr)^\op\ar[r]^-{\mathbf{o}}_-\sim&\transp(\ooCatGr)
}
\end{aligned}
\]
de foncteurs sous-jacents respectifs $C\mapsto C^\opp$, $C\mapsto C^\co$
et $C\mapsto C^\op$.
\end{rem}

\begin{exem}\label{exem:fonct_enr_part}
En se souvenant que la catégorie monoïdale $\ooCat$, munie du produit
tensoriel de Gray, est bifermée (voir le théorème~\ref{thm:produit_tens}), de
sorte que dans les notations du paragraphe~\ref{paragr:def_HomOpLax} on a
des bijections
\[
\Hom_{\ooCat}(B,\HomLax(A,C))\simeq\Hom_{\ooCat}(A\otimes B,C)\simeq
\Hom_{\ooCat}(A,\HomOpLax(B,C))
\]
naturelles en les \oo-catégories $A$, $B$ et $C$, on obtient comme cas particulier
de l'exemple~\ref{exem:fonct_enr} les assertions suivantes pour $C$ une
\oo-catégorie.
\begin{enumerate}
\item Il existe un \oo-foncteur de Gray canonique $\ooCatGr\to\ooCatGr$ dont le foncteur sous-jacent est le foncteur $\var\otimes C:\ooCat\to\ooCat$.
\item Il existe un \oo-foncteur de Gray gauche canonique $\ooCatGrg\to\ooCatGrg$ dont le foncteur sous-jacent est le foncteur $C\otimes\var:\ooCat\to\ooCat$.
\item Il existe un \oo-foncteur de Gray canonique $\ooCatGr\to\ooCatGr$ dont le foncteur sous-jacent est le foncteur $\HomOpLax(C,\var):\ooCat\to\ooCat$.
\item Il existe un \oo-foncteur de Gray gauche canonique $\transp\ooCatGr\to\ooCatGrg$ dont le foncteur sous-jacent est le foncteur $\HomOpLax(\var,C):\ooCat^\op\to\ooCat$.
\item Il existe un \oo-foncteur de Gray gauche canonique $\ooCatGrg\to\ooCatGrg$ dont le foncteur sous-jacent est le foncteur $\HomLax(C,\var):\ooCat\to\ooCat$.
\item Il existe un \oo-foncteur de Gray canonique $\transp\ooCatGrg\to\ooCatGr$ dont le foncteur sous-jacent est le foncteur $\HomLax(\var,C):\ooCat^\op\to\ooCat$.
\end{enumerate}
\end{exem}

\begin{conj}\label{conject:conj_prelim}
\begin{enumerate}
  \item Soient $\Vcat{C}$ une \oo-catégorie de Gray de \oo-sesquicatégorie sous-jacente $\C$ et $c$ un objet de $\Vcat{C}$.
\begin{enumerate}[wide]
\item Il existe une \oo-catégorie de Gray canonique $\tr{\Vcat{C}}{c}$
  \notindex{$\tr{\Vcat{C}}{c}$, $\cotrm{\Vcat{C}}{c}$, $\trm{\Vcat{C}}{c}$,
  $\cotr{\Vcat{C}}{c}$}%
  dont les objets sont les couples $(x,g)$ formés d'un objet $x$ de $\Vcat{C}$ et d'une $1$\nbd-flèche $g:x\to c$ de $\Vcat{C}$, et telle que si $(x,g)$ et $(x',g')$ sont deux tels couples, on ait
\[
\VHom_{\tr{\Vcat{C}}{c}}((x,g),(x',g'))=\VHom_{\tr{\C}{c}}((x,g),(x',g'))
\]
\noemph{(voir le paragraphe~\ref{paragr:sesqui_tranches})}.
\item Il
existe une \oo-catégorie de Gray canonique $\cotrm{\Vcat{C}}{c}$ dont les
objets sont les couples $(x,f)$ formés d'un objet $x$ de $\Vcat{C}$ et d'une
$1$\nbd-flèche $f:c\to x$ de $\Vcat{C}$, et telle que si $(x,f)$ et
$(x',f')$ sont deux tels couples, on ait
\[
\VHom_{\cotrm{\Vcat{C}}{c}}((x,f),(x',f'))=\VHom_{\cotrm{\C}{c}}((x,f),(x',f'))
\]
\noemph{(voir le paragraphe~\ref{paragr:sesqui_tranches})}.
\end{enumerate}

\item Soient $\Vcat{C}$ une \oo-catégorie de Gray gauche de \oo-sesquicatégorie sous-jacente $\C$ et $c$ un objet de $\Vcat{C}$.
\begin{enumerate}[wide]
\item Il existe une \oo-catégorie de Gray gauche canonique $\trm{\Vcat{C}}{c}$ dont les objets sont les couples $(x,g)$ formés d'un objet $x$ de~$\Vcat{C}$ et d'une $1$\nbd-flèche $g:x\to c$ de~$\Vcat{C}$, et telle que si $(x,g)$ et $(x',g')$ sont deux tels couples, on ait
\[
\VHom_{\trm{\Vcat{C}}{c}}((x,g),(x',g'))=\VHom_{\trm{\C}{c}}((x,g),(x',g'))
\]
\noemph{(voir le paragraphe~\ref{paragr:sesqui_tranches})}.
\item Il existe une \oo-catégorie de Gray gauche canonique $\cotr{\Vcat{C}}{c}$ dont les objets sont les couples $(x,f)$ formés d'un objet $x$ de $\Vcat{C}$ et d'une $1$\nbd-flèche $f:c\to x$ de~$\Vcat{C}$, et telle que si $(x,f)$ et $(x',f')$ sont deux tels couples, on ait
\[
\VHom_{\cotr{\Vcat{C}}{c}}((x,f),(x',f'))=\VHom_{\cotr{\C}{c}}((x,f),(x',f'))
\]
\noemph{(voir le paragraphe~\ref{paragr:sesqui_tranches})}.
\end{enumerate}

\item Si $\Vcat{C}$ est une \oo-catégorie de Gray et $c$ un objet de $\Vcat{C}$, on a des isomorphismes canoniques de \oo-catégories de Gray
\[\begin{aligned}
&\tr{\Vcat{C}}{c}\simeq(\cotrm{\Vcat{C}^\opp}{c})^\opp\simeq(\trm{\Vcat{C}^\co}{c})^\co\simeq(\cotr{\Vcat{C}^\op}{c})^\op,\cr
&\cotrm{\Vcat{C}}{c}\simeq(\tr{\Vcat{C}^\opp}{c})^\opp\simeq(\cotr{\Vcat{C}^\co}{c})^\co\simeq(\trm{\Vcat{C}^\op}{c})^\op\cr
\end{aligned}\]
et, si $\Vcat{C}$ est une \oo-catégorie de Gray gauche et $c$ un objet de $\Vcat{C}$, on a des isomorphismes canoniques de \oo-catégories de Gray gauches
\[\begin{aligned}
&\trm{\Vcat{C}}{c}\simeq(\cotr{\Vcat{C}^\opp}{c})^\opp\simeq(\tr{\Vcat{C}^\co}{c})^\co\simeq(\cotrm{\Vcat{C}^\op}{c})^\op,\cr
&\cotr{\Vcat{C}}{c}\simeq(\trm{\Vcat{C}^\opp}{c})^\opp\simeq(\cotrm{\Vcat{C}^\co}{c})^\co\simeq(\tr{\Vcat{C}^\op}{c})^\op,\cr
\end{aligned}\]
tous ces isomorphismes étant compatibles entre eux \emph{via} les dualités
$\var^\opp$, $\var^\co$ et~$\var^\op$.
\end{enumerate}
\end{conj}

\begin{rem}
On remarque que pour démontrer cette conjecture, il suffit de démontrer une
seule des assertions (\emph{a.1}), (\emph{a.2}),
(\emph{b.1}) ou (\emph{b.2}). En effet, il suffit alors de définir les
autres « tranches » par les formules pertinentes de l'assertion~(\emph{c}).
\end{rem}

\begin{exem}\label{exem:part_conj}
On va s'intéresser plus spécialement au cas particulier de la conjecture
appliquée à la \oo-catégorie de Gray $\ooCatGr$ et à la \oo-catégorie de
Gray gauche $\ooCatGrg$. Ainsi, conjecturalement, pour toute \oo-catégorie
$C$, on dispose d'une~\oo-catégorie de Gray $\tr{\ooCatGr}{C}$
\notindex{$\tr{\ooCatGr}{C}$, $\cotrm{\ooCatGr}{C}$,
$\trm{\ooCatGrg}{C}$, $\cotr{\ooCatGrg}{C}$}%
des \oo-catégories au-dessus de $C$ et d'une \oo-catégorie de Gray
\smash{$\cotrm{\ooCatGr}{C}$} des \oo-catégories au-dessous de $C$. De même,
on dispose d'une \oo-catégorie de Gray gauche \smash{$\trm{\ooCatGrg}{C}$}
des \oo-catégories au-dessus de $C$ et d'une \oo-catégorie de Gray gauche
$\cotr{\ooCatGrg}{C}$ des \oo-catégories au-dessous de $C$. On note
\[
\tr{\ooCatopl}{C},\quad \cotrm{\ooCatopl}{C},\quad \trm{\ooCatlax}{C}
\quadtext{et} \cotr{\ooCatlax}{C}
\]
\notindex{$\tr{\ooCatopl}{C}$, $\cotrm{\ooCatopl}{C}$, $\trm{\ooCatlax}{C}$,
$\cotr{\ooCatlax}{C}$}%
les catégories sous-jacentes à
\[
\tr{\ooCatGr}{C},\quad \cotrm{\ooCatGr}{C},\quad \trm{\ooCatGrg}{C}
\quadtext{et}
\cotr{\ooCatGrg}{C}
\]
respectivement. Les objets de $\tr{\ooCatopl}{C}$ et
de \smash{$\trm{\ooCatlax}{C}$} sont les mêmes que ceux de $\tr{\ooCat}{C}$.
De même, les objets de \smash{$\cotrm{\ooCatopl}{C}$} et de
$\cotr{\ooCatlax}{C}$ sont les mêmes que ceux de $\cotr{\ooCat}{C}$. Dans le
tableau suivant, on indique les $2$\nbd-diagrammes définissant les
morphismes d'un objet vers un objet «~prime~» dans ces quatre catégories :
\[
\UseTwocells
\xymatrixcolsep{.0pc}
\xymatrixrowsep{-.1pc}
\xymatrix{
A\ar[rrrrrr]^{u}\ar[rrrdddddd]_{g}
&&&&&&A'\ar[llldddddd]^{g'}
&\kern 10pt&
A\ar[rrrrrr]^{u}\ar[rrrdddddd]_{g}
&&&&&&A'\ar[llldddddd]^{g'}
&\kern 10pt&
&&&C\ar[llldddddd]_f\ar[rrrdddddd]^{f'}
&&&
&\kern 10pt&
&&&C\ar[llldddddd]_f\ar[rrrdddddd]^{f'}
&&&
\\
&&&&\ar@2{->}[lldd]^{\alpha}_{}
&&&&
&&&&\ar@2{<-}[lldd]^{\beta}_{}
&&&&
&&&&&&
&&
\\&&&&&&&&&&&&&&&&&&&&&&&&
\\&&&&&&&&&&&&&&&&&&&&\ar@2{<-}[lldd]^{\gamma}_{}
&&
&&
&&&&\ar@2{->}[lldd]^{\delta}_{}
\\&&&&&&&&&&&&&&&&&&&&&&&&&&&&&&
\\&&&&&&&&&&&&&&&&&&&&&&&&&&&&&&
\\
&&&C
&&&&&
&&&C
&&&&&
A\ar[rrrrrr]_u
&&&&&&A'
&&
A\ar[rrrrrr]_u
&&&&&&A'
\\
\\
&&&\ar@{}[r]_{\displaystyle\tr{\ooCatopl}{C}\ \ \vrule height 12pt width 0pt}
&&&&&
&&&\ar@{}[r]_{\displaystyle\trm{\ooCatlax}{C}\ \ }
&&&&&
&&&\ar@{}[r]_{\displaystyle\cotrm{\ooCatopl}{C}\ \ }
&&&&&
&&&\ar@{}[r]_{\displaystyle\cotr{\ooCatlax}{C} \pbox{,}\ \ \vrule height 12pt width 0pt}
&&&&&
}
\]
$\alpha$ et $\gamma$ étant des transformations oplax et $\beta$ et $\delta$
des transformations lax. On vérifie aussitôt
qu'on a des foncteurs d'inclusion canoniques
\[\begin{aligned}
&\tr{\ooCat}{C}\to\tr{\ooCatopl}{C}\,,\quad\tr{\ooCat}{C}\to\trm{\ooCatlax}{C}\,,\cr
\noalign{\vskip 4pt}
&\cotr{\ooCat}{C}\to\cotrm{\ooCatopl}{C}\,,\quad\cotr{\ooCat}{C}\to\cotr{\ooCatlax}{C}
\end{aligned}\]
(en admettant que les morphismes dans les catégories buts se composent de la
manière évidente) induisant l'identité sur les objets et envoyant les
triangles commutatifs sur les $2$\nbd-triangles de même contour munis de la
transformation oplax ou lax unité.
\end{exem}

\begin{conj}\label{conj:fond}
Soit $C$ une \oo-catégorie.
\begin{enumerate}
\item Il existe un \oo-foncteur de Gray canonique \smash{$\ooCatGr\to\cotrm{\ooCatGr}{C}$} dont le foncteur sous-jacent est le composé
\[
\ooCat\to\cotr{\ooCat}{C}\to\cotrm{\ooCatopl}{C}\,,
\]
où la flèche de gauche est le foncteur
\[
\var\joint C:A\mapsto(A\joint C,\,i_2:C\to A\joint C)
\]
et la flèche de droite le foncteur d'inclusion canonique.
\item Il existe un \oo-foncteur de Gray gauche canonique \smash{$\ooCatGrg\to\cotr{\ooCatGrg}{C}$} dont le foncteur sous-jacent est le composé
\[
\ooCat\to\cotr{\ooCat}{C}\to\cotr{\ooCatlax}{C}\,,
\]
où la flèche de gauche est le foncteur
\[
C\joint\var:A\mapsto(C\joint A,\,i_1:C\to C\joint A)
\]
et la flèche de droite le foncteur d'inclusion canonique.
\item Il existe un \oo-foncteur de Gray canonique \smash{$\cotrm{\ooCatGr}{C}\to\ooCatGr$} dont le foncteur sous-jacent composé avec l'inclusion
$\cotr{\ooCat}{C}\to\cotrm{\ooCatopl}{C}$ est le foncteur
\[
  \begin{split}
    \cotr{\ooCat}{C} & \to \ooCat\\
    (A,\,f:C\to A) & \mapsto \trm{A}{f} \pbox{.}
  \end{split}
\]
\item Il existe un \oo-foncteur de Gray gauche canonique \smash{$(\tr{\ooCatGr}{C})^\op\to\ooCatGrg$} dont le foncteur sous-jacent composé avec l'inclusion
$$(\tr{\ooCat}{C})^\op\to(\tr{\ooCatopl}{C})^\op$$ est le foncteur
\[
  \begin{split}
    (\tr{\ooCat}{C})^\op & \to\ooCat\\
    (A,\,f:A\to C) & \mapsto \trm{C}{f} \pbox{.}
  \end{split}
\]
\item Il existe un \oo-foncteur de Gray gauche canonique \smash{$\cotr{\ooCatGrg}{C}\to\ooCatGrg$} dont le foncteur sous-jacent composé avec l'inclusion
$\cotr{\ooCat}{C}\to\cotr{\ooCatlax}{C}$ est le foncteur
\[
  \begin{split}
    \cotr{\ooCat}{C} & \to\ooCat\\
    (A,\,f:C\to A) & \mapsto \cotr{A}{f} \pbox{.}
  \end{split}
\]
\item Il existe un \oo-foncteur de Gray canonique \smash{$(\trm{\ooCatGrg}{C})^\op\to\ooCatGr$} dont le foncteur sous-jacent composé avec l'inclusion
$(\tr{\ooCat}{C})^\op\to(\trm{\ooCatlax}{C})^\op$ est le foncteur
\[
  \begin{split}
    (\tr{\ooCat}{C})^\op & \to\ooCat\\
    (A,\,f:A\to C) & \mapsto \cotr{C}{f} \pbox{.}
  \end{split}
\]
\end{enumerate}
\end{conj}

\begin{rem}\label{}
Les assertions (\emph{a})--(\emph{f}) de la conjecture ci-dessus sont les
analogues pour le joint des assertions formelles (\emph{a})--(\emph{f}) de
l'exemple~\ref{exem:fonct_enr_part} relatives au produit tensoriel de Gray.
Par ailleurs, en utilisant les isomorphismes de la
remarque~\ref{rem:dual_Gray} ainsi que ceux de l'assertion (\emph{c}) de la
conjecture~\ref{conject:conj_prelim}, on remarque que pour prouver
l'assertion (\emph{b}) (resp. (\emph{e}), resp. (\emph{f})) de la
conjecture~\ref{conj:fond}, il suffit de démontrer l'assertion (\emph{a})
(resp. (\emph{c}), resp. (\emph{d})) et réciproquement. De même, la
conjecture~\ref{conj:fond} est équivalente à la conjecture analogue
pour le foncteur $\joint'$ et les «~tranches~»~$\tr{C}{f}$ et
\smash{$\cotrm{C}{f}$} (voir la remarque~\ref{rem:cojoint_ooCat}).
\end{rem}

\begin{rem}
L'assertion (\emph{f}) de la conjecture~\ref{conj:fond} est une vaste
généralisation des résultats du chapitre~\ref{sec:fonct_tr}. Par exemple,
si
\[
    \shorthandoff{;}
    \xymatrix@C=1.5pc{
      A \ar[rr]^f \ar[dr]_{c}_{}="f" & & A' \ar[dl]^(0.42){c'} \\
      & C
      \ar@{}"f";[ur]_(.15){}="ff"
      \ar@{}"f";[ur]_(.55){}="oo"
      \ar@<-0.5ex>@2"ff";"oo"^{\alpha}
    }
  \]
est un diagramme dans $\ooCat$, avec $\alpha$ une transformation lax,
autrement dit un diagramme représentant une $1$\nbd-flèche
$(f,\alpha):(A,c)\to(A',c')$ de \smash{$\trm{\ooCatGrg}{C}$} (qui est une
flèche de sa catégorie sous-jacente \smash{$\trm{\ooCatlax}{C}$}), en
appliquant le foncteur \smash{$(\trm{\ooCatlax}{C})^\op\to\ooCat$}
sous-jacent au \oo-foncteur de Gray canonique
\smash{$(\trm{\ooCatGrg}{C})^\op\to\ooCatGr$}, on obtient un \oo-foncteur
$(f,\alpha)^\ast:\cotr{C}{c'} \to \cotr{C}{c}$, généralisant celui du
théorème~\ref{thm:img_tri}. Par fonctorialité, on obtient une généralisation
des résultats de la section~\ref{sec:fonct_tri}. De même, si
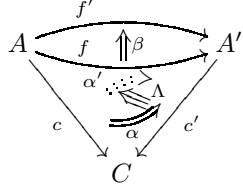
\[
      \shorthandoff{;:}
        \xymatrix@C=2.1pc@R=3pc{
        A \ar@/^2ex/[rr]^(.33){f'}_{}="1" \ar@/_2ex/[rr]^(.30)f_{}="0"
        \ar[dr]_{}="f"_{\phantom{c'}c}
        \ar@2"0";"1"_\beta
        & & A' \ar[dl]^{c'} \\
        & C
        \ar@{}"f";[ur]_(.15){}="ff"
        \ar@{}"f";[ur]_(.55){}="oo"
        \ar@<-0.5ex>@/^1ex/@{:>}"ff";"oo"^(.18){\alpha'\!\!}_(.30){}="h'"
        \ar@<-2.0ex>@/^-1ex/@2"ff";"oo"_(.36){\alpha}_(.80){}="h"
        \ar@3"h";"h'"_(.16){\,\Lambda{}}
        }
    \]
est un diagramme dans $\ooCat$, avec $\alpha$, $\alpha'$ et $\beta$ des
transformations lax et $\Lambda$ une $2$\nbd-transformation lax, autrement
dit un diagramme représentant une $2$\nbd-flèche
\hbox{$(\beta,\Lambda):(f,\alpha)\to(f',\alpha')$} de
\smash{$\trm{\ooCatGrg}{C}$}, en appliquant le \oo-foncteur de Gray
canonique \smash{$(\trm{\ooCatGrg}{C})^\op\to\ooCatGr$}, on obtient une
$2$\nbd-flèche de $\ooCatGr$, autrement dit une $2$-transformation oplax de
source $(f',h')^\ast$ et but $(f,h)^\ast$, généralisant ainsi le
théorème~\ref{thm:img_cone}. Enfin, une description plus précise de la
\oo-sesquicatégorie sous-jacente à \smash{$\trm{\ooCatGrg}{C}$} fournirait
une généralisation des résultats de fonctorialité de la
section~\ref{sec:fonct_tr_cone}.
\end{rem}

\begin{paragr}\label{paragr:conj_Cda}
Les conjectures~\ref{conject:conj_prelim} et~\ref{conj:fond} sont étayées
par le fait que des conjectures analogues sont vraies pour les complexes
dirigés augmentés. Plus précisément, on rappelle que la catégorie monoïdale
$\Cda$ des complexes dirigés augmentés, munie du produit tensoriel, est
bifermée (voir le paragraphe~\ref{paragr:def_produit_cda}), de sorte qu'il
existe des foncteurs
\nnot[$\HomCdad(K, L)$, $\HomCdag(K, L)$]%
  {$\HomCdad,\HomCdag:\Cda^\op\times\Cda^{}\to\Cda^{}$} et des bijections
\[
\Hom_{\Cda}(L,\HomCdag(K,M))\simeq\Hom_{\Cda}(K\otimes
L,M)\simeq\Hom_{\Cda}(K,\HomCdad(L,M)),
\]
naturelles en les complexes dirigés augmentés $K,L,M$. En vertu de
l'exemple~\ref{exem:standard_enrich}, on dispose donc d'une
$\Cda$\nbd-catégorie ayant mêmes objets que $\Cda$ et dont l'objet de
morphismes est défini par $\HomCdad$, et d'une $\Cda$\nbd-catégorie gauche
ayant également les mêmes objets que $\Cda$ et dont l'objet de morphismes
est défini par $\HomCdag$. Comme le foncteur $\nu : \Cda \to \ooCat$
s'enrichit en un foncteur monoïdal lax (voir la
proposition~\ref{prop:lambda_nu_mon_tens}), on peut appliquer la
construction de l'image directe du paragraphe~\ref{paragr:imdir_enrich} pour
obtenir une \oo-catégorie de Gray \nnot[$\CdaGr$, $\CdaGrg$]{$\CdaGr$}
(resp. une \oo-catégorie de Gray gauche~$\CdaGrg$) ayant comme objets les
complexes dirigés augmentés et telle que pour tous objets $K,L$, % on ait
\[
\VHom_{\CdaGr}(K,L)=\nu(\HomCdad(K,L))\quad \hbox{(resp. }\
\VHom_{\CdaGrg}(K,L)=\nu(\HomCdag(K,L))\text{)}.
\]
On remarque que, essentiellement par définition, pour $i\geq0$, les
$i$\nbd-flèches de $\VHom_{\CdaGr}(K,L)$ (resp. de $\VHom_{\CdaGrg}(K,L)$)
sont les $i$\nbd-homotopies (resp. les $i$\nbd-anti\-homo\-topies) de
complexes dirigés augmentés de $K$ vers~$L$ (voir le
paragraphe~\ref{paragr:def_n-homot}). En
particulier, les objets de $\VHom_{\CdaGr}(K,L)$ ou de
$\VHom_{\CdaGrg}(K,L)$ sont les morphismes de $K$ vers $L$. De plus, il
résulte facilement des théorèmes~\ref{thm:Steiner} et~\ref{thm:produit_tens}
ainsi que des propriétés d'adjonction que si $K$ est un complexe de Steiner fort,
on a des isomorphismes canoniques
\[
\VHom_{\CdaGr}(K,L)\simeq\HomOpLax(\nu(K),\nu(L))\ \hbox{ et\,}\
\VHom_{\CdaGrg}(K,L)\simeq\HomLax(\nu(K),\nu(L)).
\]
En particulier, la sous-\oo-catégorie de Gray pleine de $\CdaGr$ définie par
les complexes de Steiner forts s'identifie à la sous-\oo-catégorie de
Gray pleine de $\ooCatGr$ définie par les \oo-catégories de Steiner fortes (voir
le paragraphe~\ref{paragr:def_ooCat_Stf}), et de même pour les
sous-\oo-catégories de Gray gauches pleines de $\CdaGrg$ et
$\ooCatGrg$ correspondantes.

Par ailleurs, pour tout complexe dirigé augmenté $L$ et tout couple de
morphismes $g:K\to L$, $g:K'\to L$ (resp. $f:L\to K$, $f':L\to K'$) de
$\Cda$ de but $L$ (resp. de source $L$), on sait définir des complexes
dirigés augmentés explicites
\[\begin{aligned}
&\Homd_L((K,g),(K',g'))\quad\hbox{et}\quad\Homg_L((K,g),(K',g'))\cr
\hbox{(resp.}\quad&\Homi_{\mathrm{d}}^L((K,f),(K',f'))
\quad\hbox{et}\quad
\Homi_{\mathrm{g}}^L((K,f),(K',f'))\text{)},\phantom{(resp.}\
\end{aligned}\]
pouvant s'assembler pour former deux $\Cda$\nbd-catégories dont les objets
sont les complexes dirigés augmentés respectivement au-dessus et au-dessous
de $L$ et dont les objets de morphismes sont définis par
$\Homd_L((K,g),(K',g'))$ et $\Homi_{\mathrm{d}}^L((K,f),(K',f'))$, et deux
$\Cda$\nbd-catégories gauches dont les objets sont aussi respectivement les
complexes dirigés augmentés au-dessus et au-dessous de $L$ et dont les
objets de morphismes sont définis par $\Homg_L((K,g),(K',g'))$ et
$\Homi_{\mathrm{g}}^L((K,f),(K',f'))$. En appliquant la construction de
l'image directe du paragraphe~\ref{paragr:imdir_enrich} par le foncteur
monoïdal lax $\nu$, on obtient deux \oo-catégories de Gray notées
  \nnot[$\tr{\CdaGr}{L}$, $\cotrm{\CdaGr}{L}$, $\trm{\CdaGrg}{L}$,
$\cotr{\CdaGrg}{L}$]{$\tr{\CdaGr}{L}$} et \smash{$\cotrm{\CdaGr}{L}$} et
deux \oo-catégories de Gray gauches notées \smash{$\trm{\CdaGrg}{L}$} et
$\cotr{\CdaGrg}{L}$. On peut démontrer que celles-ci satisfont les
affirmations de la conjecture~\ref{conject:conj_prelim} relatives à la
\oo-catégorie de Gray $\CdaGr$ et à la \oo-catégorie de Gray gauche
$\CdaGrg$. En particulier, cela implique le cas de la
conjecture~\ref{conject:conj_prelim} exposé dans
l'exemple~\ref{exem:part_conj} pour la sous-\oo-catégorie de Gray pleine de
$\ooCatGr$ et la sous-\oo-catégorie de Gray gauche pleine de $\ooCatGrg$
définies par les \oo-catégories de Steiner fortes.

De plus, si le complexe dirigé augmenté $L$ est décent (voir le
paragraphe~\ref{paragr:def_decent}), on peut définir explicitement les
morphismes de complexes dirigés augmentés pertinents pour construire les
\oo-foncteurs de Gray et \oo-foncteurs de Gray gauches correspondant à
l'analogue de la conjecture~\ref{conj:fond} pour les complexes dirigés
augmentés. Ceci permet comme ci-dessus d'établir la
conjecture~\ref{conj:fond} si on se restreint aux \oo-catégories de Steiner
fortes.

Ces résultats seront exposés plus en détails ailleurs. On espère pouvoir en
déduire, par un argument de densité et une version « enrichie » de la
généralisation du théorème de Day démontrée dans ce texte
(théorème~\ref{thm:Day_loc}), la conjecture~\ref{conj:fond} en toute
généralité.
\end{paragr}

\def\bibname{Références}
\bibliography{biblio}
\bibliographystyle{mysmfplain}
\begin{flushleft}
\printindex[not]
\printindex[term]
\end{flushleft}

\end{document}

%% file: macros.tex
\theoremstyle{plain}
\newtheorem{thm}{Théorème}
\newtheorem{prop}[thm]{Proposition}
\newtheorem{lemme}[thm]{Lemme}
\newtheorem{coro}[thm]{Corollaire}
\newtheorem{conj}[thm]{Conjecture}
\theoremstyle{remark}
\newtheorem{rem}[thm]{Remarque}

\newtheorem{exem}[thm]{Exemple}

\theoremstyle{definition}

\newtheorem{paragr}[thm]{}
\newtheorem*{notations}{Notations et terminologie}
\newtheorem*{remerciements}{Remerciements}
\newtheorem*{organisation}{Organisation de l'article}

\numberwithin{equation}{thm}

\preto\chapter{\numberwithin{thm}{chapter}}
\preto\section{\numberwithin{thm}{section}}

\newindex[Index des notations]{not}
\newindex[Index terminologique]{term}
\AtWriteToIndex{not}{\let\thepage\thenoentry}
\AtWriteToIndex{term}{\let\thepage\thenoentry}

\makeatletter

\newcommand\termindex[1]{%
  \Hy@raisedlink{\hypertarget{\thenoentry}{}}%
  \label{\thenoentry}%
  \expandafter\xdef\csname idx\thenoentry\endcsname{\thethm}%
  \sindex[term]{#1|linkindex}%
  \stepcounter{noentry}%
}

\newcounter{noentry}

\newcommand\notindex[1]{%
  \Hy@raisedlink{\hypertarget{\thenoentry}{}}%
  \label{\thenoentry}%
  \expandafter\xdef\csname idx\thenoentry\endcsname{\thethm}%
  \sindex[not]{\thenoentry @\detokenize{#1}|linkindex}%
  \stepcounter{noentry}%
}

\newcommand\linkindex[1]{%
  \ifnum\pdf@strcmp{\@nameuse{idx#1}}{0.0}=0%
    \hyperlink{#1}{\hbox{p.~\pageref*{#1}, \S~Notations et terminologie}}%
  \else
    \hyperlink{#1}{p.~\pageref*{#1},~\S~\@nameuse{idx#1}}%
  \fi
}
\newcommand\ndef[2][\@empty]{%
  \@ifmtarg{#1}{}{%
  \Hy@raisedlink{\hypertarget{\thethm}{}}%
  \ifx\@empty#1%
    \termindex{#2}%
  \else%
    \termindex{#1}%
  \fi%
  }%
  \emph{#2}%
}
\newcommand\nnot[2][\@empty]{%
  \@ifmtarg{#1}{}{%
  \Hy@raisedlink{\hypertarget{\thethm}{}}%
  \ifx\@empty#1%
    \notindex{#2}%
  \else%
    \notindex{#1}%
  \fi%
  }%
  #2%
}

\makeatother

\newcommand\forlang\emph
\newcommand\noemph\emph
\newcommand\nbd\nobreakdash
\newcommand\nquad{\mkern-18mu}

\makeatletter

\def\xpoint{\futurelet\@let@token\@xpoint}
\def\@xpoint{%
  \ifx\@let@token.\else
    .%
  \fi
  \xspace}

\newcommand{\setstretch}[1]{%
  \def\baselinestretch{#1}%
  \@currsize
}

\makeatother

\newcommand\lp(
\newcommand\rp)

\newcommand\zbox[1]{\makebox[0pt][l]{#1}}
\newcommand\pbox[1]{\zbox{\quad#1}}

\newcommand\defssi{\overset{\text{\tiny déf}}{\Longleftrightarrow}}

\newcommand\quadtext[1]{\quad\text{#1}\quad}
\newcommand\quadmath[1]{\quadtext{$#1$}}
\newcommand\quadimpl{\quadmath{\Rightarrow}}
\newcommand\quadet{\quadtext{et}}
\newcommand\quadssi{\quadtext{si et seulement si}}
\newcommand\quaddefssi{\quadmath{\defssi}}

\renewcommand\le\leqslant
\renewcommand\ge\geqslant
\renewcommand\epsilon\varepsilon
\renewcommand\phi\varphi

\newcommand\longto\longrightarrow
\newcommand\ot\leftarrow
\newcommand\longot\longleftarrow
\newcommand\hookto\hookrightarrow
\newcommand\xto\xrightarrow
\newcommand\xot\xleftarrow
\newcommand\xlongto[1]{\xrightarrow{\,\,#1\,\,}}

\newcommand\e{\epsilon}

\newcommand\ep{{\e'}}
\newcommand\epd{{1-\e'}}

\newcommand\var\bullet

\newcommand\lec{\preccurlyeq}

\newcommand\leN{\le_\N}
\newcommand\lNv{<^\vide_\N}
\newcommand\leNv{\le^\vide_\N}
\newcommand\geNv{\ge^\vide_\N}

\newcommand{\sauf}{\mathchoice{\raise 1.8pt\hbox{${\scriptstyle\kern
2.5pt\smallsetminus\kern 2.5pt}$}}{\raise 1.8pt\hbox{${\scriptstyle\kern
2.5pt\smallsetminus\kern 2.5pt}$}}{\raise
1.8pt\hbox{${\scriptscriptstyle\kern 1.5pt\smallsetminus\kern
1.5pt}$}}{\raise 1.8pt\hbox{${\scriptscriptstyle\kern
1.5pt\smallsetminus\kern 1.5pt}$}}}

\newcommand\Homig{\Homi^{\mathrm{g}}}

\newcommand\HomOpLax{\Homi_{\mathrm{oplax}}}
\newcommand\HomLax{\Homi_{\mathrm{lax}}}

\newcommand\joint\star
\newcommand\jointc{\star^{\text{c}}}
\newcommand\vide\varnothing
\newcommand\otimesG{\otimes_\text{G}}
\newcommand\susp[1]{\Sigma{#1}}

\newcommand\aug{e}

\newcommand\opp{\mathrm{op}}
\newcommand\co{\mathrm{co}}
\newcommand\coop{\mathrm{coop}}

\newcommand\pdfoo{\texorpdfstring{$\infty$}{\unichar{"221E}}}
\newcommand\pdfn{\texorpdfstring{$n$}{n}}
\newcommand\pdfTheta{\texorpdfstring{$\Theta$}{Theta}}

\newcommand\Z{\mathbb{Z}}
\newcommand\N{\mathbb{N}}

\newcommand\limind\varinjlim
\newcommand\limproj\varprojlim

\renewcommand\C{\mathcal{C}}
\newcommand\D{\mathcal{D}}
\newcommand{\Hom}{\operatorname{\mathsf{Hom}}}

\newcommand{\Homi}{\operatorname{\kern.5truept\underline{\kern-.5truept\mathsf{Hom}\kern-.5truept}\kern1truept}}

\newcommand{\Homg}{\Homi^{\mathrm{g}}}
\newcommand{\Homd}{\Homi^{\mathrm{d}}}

\newcommand{\pref}[1]{{\widehat{ #1 }}}
\newcommand{\faisc}[1]{{\widehat{#1}_{\mathrm{glob}}}}
\newcommand{\faiscp}[1]{{(\widehat{#1})_{\mathrm{glob}}}}
\newcommand\id[1]{1_{#1}}
\newcommand\op\circ
\newcommand{\Ob}{\operatorname{\mathsf{Ob}}}
\newcommand{\Fl}{\operatorname{\mathsf{Fl}}}

\newcommand{\Ens}{{\mathcal{E}\mspace{-2.mu}\it{ns}}}
\newcommand{\Ab}{{\mathcal{A}\mspace{-2.mu}\it{b}}}

\newcommand{\Cat}{{\mathcal{C}\mspace{-2.mu}\it{at}}}
\newcommand{\nCatTiny}[1]{\text{$#1$-$\Cat$}}
\newcommand{\nCat}[1]{{#1}\hbox{\protect\nbd-}\kern1pt\Cat}
\newcommand{\ooCat}{\nCat{\infty}}
\newcommand{\ooCatTiny}{\nCatTiny{\infty}}

\newcommand\oo{$\infty$\nbd}

\newcommand{\FlTord}{\operatorname{\widetilde{\mathsf{Fl}}}}
\newcommand{\FlTordG}{\operatorname{\widetilde{\mathsf{Fl}}}_{\text{g}}}
\newcommand{\FlTordD}{\operatorname{\widetilde{\mathsf{Fl}}}_{\text{d}}}

\newcommand\comp\ast
\newcommand{\tb}[1]{\tau_{\le #1}^{\mathrm b}}
\newcommand{\ti}[1]{\tau_{\le #1}^{\mathrm i}}
\newcommand\dual[1]{D_{#1}}

\newcommand\cocatD{\mathrm{D}_{\bullet}}

\newcommand\Dn[1]{\mathrm{D}_{#1}}

\newcommand\amalgDn[1]{\amalg_{\Dn{#1}}}
\newcommand\Homcocat\Hom

\newcommand\ThetanAug[1]{(\Theta_{#1})_+}
\newcommand\ThetaAug{\Theta_+}
\newcommand\DeltaAug{\cDelta_+}

\newcommand\cDelta{\mathbf{\Delta}}
\newcommand\cDeltaAug{\mathbf{\Delta}_{+}}
\newcommand\Deltan[1]{\varDelta_{#1}}

\newcommand\OAug{\mathcal{O}_+}
\newcommand\On[1]{\mathcal{O}_{#1}}

\newcommand{\tr}[2]{\mathchoice
  {#1\raise -1.8pt\vbox{\hbox{$\kern -.8pt/#2$}}}
  {#1\raise -1.8pt\vbox{\hbox{$\kern -.8pt/#2$}}\kern .8pt}
  {#1\raise -1.8pt\vbox{\hbox{$\scriptstyle\kern -.8pt /#2$}}}
  {#1\raise -1.8pt\vbox{\hbox{$\scriptscriptstyle\kern -.8pt /#2$}}}}

\newcommand{\trm}[2]{\mathchoice
  {#1\raise -1.8pt\vbox{\hbox{$\kern -.8pt\!\stackrel{\,\,\,\rm co}{/}\!\!#2$}}}
  {#1\raise -1.8pt\vbox{\hbox{$\kern -.8pt\!\stackrel{\,\,\,\rm co}{/}\!\!#2$}}\kern .8pt}
  {#1\raise -1.8pt\vbox{\hbox{$\scriptstyle\kern -.8pt\!\!\stackrel{\,\,\,\rm co}{/}\!\!#2$}}\kern .8pt}
  {TODO}}

\newcommand{\cotr}[2]{\mathchoice
  {\raise -1.8pt\vbox{\hbox{$#2\backslash$}}#1}
  {\raise -1.8pt\vbox{\hbox{$#2\backslash$}}#1}
  {\raise -1.8pt\vbox{\hbox{$\scriptstyle#2\backslash$}}#1}
  {\raise -1.8pt\vbox{\hbox{$\scriptscriptstyle#2\backslash$}}#1}}

\newcommand{\cotrm}[2]{\mathchoice
  {\raise -1.8pt\vbox{\hbox{$#2\!\stackrel{\!\!\!\!\rm co}{\backslash}$}}#1}
  {\raise -1.8pt\vbox{\hbox{$#2\!\stackrel{\!\!\!\!\rm co}{\backslash}$}}#1}
  {\raise -1.8pt\vbox{\hbox{$\scriptstyle#2\stackrel{\!\!\!\rm
  co}{\backslash}$}}#1}
  {TODO}}

\newcommand{\tru}[2]{\mathchoice
  {#1\raise -1.8pt\vbox{\hbox{$\kern -.8pt/^1#2$}}}
  {#1\raise -1.8pt\vbox{\hbox{$\kern -.8pt\overset{\,\scriptscriptstyle1}{/}#2$}}\kern .8pt}
  {#1\raise -1.8pt\vbox{\hbox{$\scriptstyle\kern -.8pt /^1#2$}}}
  {#1\raise -1.8pt\vbox{\hbox{$\scriptscriptstyle\kern -.8pt /^1#2$}}}}

\newcommand{\cotru}[2]{\mathchoice
  {\raise -1.8pt\vbox{\hbox{$#2^1\backslash$}}#1}
  {\raise -1.8pt\vbox{\hbox{$#2\overset{\!\scriptscriptstyle1}{\backslash}$}}#1}
  {\raise -1.8pt\vbox{\hbox{$\scriptstyle#2^1\backslash$}}#1}
  {\raise -1.8pt\vbox{\hbox{$\scriptscriptstyle#2^1\backslash$}}#1}}

\newcommand{\cotrCeg}[2]{#2/\hspace{-3pt}/#1}

\newcommand{\Cda}{\mathcal{C}_{\mathrm{da}}}
\newcommand{\CdaSigma}{\mathcal{C}_{\mathrm{da}}^\Sigma}
\newcommand{\CdaSSigma}{\mathcal{C}_{\mathrm{da}}^{\phantom{\Sigma}}}
\newcommand{\nCda}[1]{{#1}\hbox{\protect\nbd-}\mathcal{C}_{\mathrm{da}}}

\newcommand\Stf{{\mathcal{S}t}_{\mathrm{f}}}
\newcommand{\atom}[1]{\langle{#1}\rangle}
\newcommand{\tablm}[2]{\begin{pmatrix}#1^0_0 &\dots
  &#1^0_{#2}\cr\noalign{\vskip 3pt} #1^1_0 &\dots
  &#1^1_{#2}\end{pmatrix}}

\newcommand{\tabld}[2]{\begin{pmatrix}#1^0_0 &\dots &#1^0_{#2-1}
  &#1^0_{#2}\cr\noalign{\vskip 3pt} #1^1_0 &\dots &#1^1_{#2-1}
  &#1^1_{#2}\end{pmatrix}}
\newcommand{\tablnu}[3]{\begin{pmatrix}#1^0_0 &\dots &#1^0_{#2}
  &#3\cr\noalign{\vskip 3pt} #1^1_0 &\dots &#1^1_{#2} &#3\end{pmatrix}}
\newcommand{\tabll}[2]{\begin{pmatrix}#1^0_0 &#1^0_1 &\dots &#1^0_{#2-1}
  &#1^0_{#2}\cr\noalign{\vskip 3pt} #1^1_0 &#1^1_1 &\dots &#1^1_{#2-1}
  &#1^1_{#2}\end{pmatrix}}
\newcommand{\supp}{\operatorname{supp}}
\newcommand\Zdec{\underline{\mathbb{Z'}\kern -2.5pt}\kern 2pt}

\newcommand\Catmod{\mathcal{V}}    % catégorie monoïdale de base
\newcommand\VCatd{{\Vcat{V}^{\kern 1pt\mathrm d}}}       % catégorie enrichie définie par le hom interne droit de la catégorie monoïdale de base
\newcommand\VCatg{{\Vcat{V}^{\kern 1pt\mathrm g}}}       % catégorie enrichie gauche définie par le hom interne gauche de la catégorie monoïdale de base
\newcommand\CatmodI{I}             % unité d'une catégorie monoïdale
\newcommand\Vcat[1]{\mathbb{#1}}   % typographie des catégories enrichies
\newcommand\VHom{\operatorname{\kern.5truept\underline{\kern-.5truept\mathsf{Hom}\kern-1truept}\kern1.5truept}^{}}    % Hom externe
\newcommand\Vcomp{\circ}           % composition d'une catégorie enrichie.
\newcommand\VId{\mathsf{id}}       % unité d'une catégorie enrichie
\newcommand\ev{\mathsf{ev}}        % morphisme d'évaluation
\newcommand\ooCatGr{\infty\hbox{\protect\nbd-}\mathbb{C}\mathsf{at}^{}_\mathrm{oplax}}              % oo-catégorie de Gray des oo-catégories
\newcommand\ooCatGrg{\infty\hbox{\protect\nbd-}\mathbb{C}\mathsf{at}^{}_\mathrm{lax}}               % oo-catégorie de Gray gauche des oo-catégories
\newcommand\ooCatopl{\ooCat^{}_\mathrm{oplax}}
\newcommand\ooCatlax{\ooCat^{}_\mathrm{lax}}
\newcommand\transp{{}^t}            % transposition
\newcommand\HomCdad{\operatorname{\kern.5truept\underline{\kern-.5truept\mathsf{Hom}\kern-1.5truept}\kern1.9truept}^{\mathrm{d}}_{\Cda}}   % Hom interne droit de CDA
\newcommand\HomCdag{\operatorname{\kern.5truept\underline{\kern-.5truept\mathsf{Hom}\kern-1.5truept}\kern1.9truept}^{\mathrm{g}}_{\Cda}}   % Hom interne gauche de CDA
\newcommand{\CdaGr}{\Vcat{C}^{\mathrm{d}}_{\mathsf da}}   % oo-catégorie de Gray des CDA
\newcommand{\CdaGrg}{\Vcat{C}^{\mathrm{g}}_{\mathsf da}}  % oo-catégorie de Gray gauche des CDA